\documentclass[11pt,reqno]{amsart}
\usepackage{geometry}
\usepackage{amssymb,amsmath,amsthm}
\usepackage[utf8]{inputenc}
\usepackage[T1]{fontenc}
\usepackage{enumerate}
\usepackage[all]{xy}
\usepackage{fullpage}
\usepackage{comment}
\usepackage{array}
\usepackage{longtable}
\usepackage{bm}
\usepackage{stmaryrd}
\usepackage{mathrsfs} 
\usepackage{xcolor}
\usepackage[dvipsnames]{xcolor} 
\usepackage{mathdots}
\usepackage{mathtools}
\usepackage{tikz}
\usepackage{float}
\usepackage{siunitx}
\usepackage{xcolor}
\usepackage{babel}

\def\Z{\mathbb{Z}}
\def\R{\mathbb{R}}
\def\N{\mathbb{N}}

\def\K{\mathbb{K}}
\def\F{\mathbb{F}}
\def\id{\mathcal{R}}
\def\A{\mathcal{A}}

\def\deg{\mathrm{deg}}

\def\int{\mathrm{Int}}

\def\dim{\mathrm{dim}}

\newcommand{\rd}{\color{red}}
\newcommand{\bk}{\color{black}}
\newcommand{\bl}{\color{blue}}
\newcommand{\gn}{\color{ForestGreen}}
\definecolor{LightRed}{rgb}{1, 0.5, 0.5} 
\definecolor{DarkRed}{rgb}{0.5, 0, 0}
\definecolor{LightBlue}{rgb}{0.5, 0.5, 1} 
\definecolor{DarkBlue}{rgb}{0, 0, 0.5}
\definecolor{LightGreen}{rgb}{0.5, 1, 0.5} 
\definecolor{DarkGreen}{rgb}{0, 0.3, 0}
\newcommand{\lrd}{\color{LightRed}}
\newcommand{\drd}{\color{DarkRed}}
\newcommand{\lbl}{\color{LightBlue}}
\newcommand{\dbl}{\color{DarkBlue}}
\newcommand{\lgn}{\color{LightGreen}}
\newcommand{\dgn}{\color{DarkGreen}}

\allowdisplaybreaks

\usepackage{hyperref}
\hypersetup{
    colorlinks=true,
    linkcolor=blue,
    filecolor=red,  
    citecolor=green,
    urlcolor=cyan,
    pdftitle={Escape of Mass of the p Cantor},
    pdfpagemode=FullScreen,
    }

\usepackage{cleveref}
\crefformat{section}{\S#2#1#3}
\crefformat{subsection}{\S#2#1#3}
\crefformat{subsubsection}{\S#2#1#3}
\usepackage{enumitem}
\usepackage{tikz}

\newtheorem{theorem}{Theorem}[section]
\newtheorem{lemma}[theorem]{Lemma}
\newtheorem{corollary}[theorem]{Corollary}
\newtheorem{proposition}[theorem]{Proposition}
\newtheorem{example}[theorem]{Example}
\newtheorem{conjecture}[theorem]{Conjecture}

\theoremstyle{definition}
\newtheorem{definition}[theorem]{Definition}

\theoremstyle{remark}
\newtheorem{remark}[theorem]{Remark}

\makeatletter
\newcommand{\subalign}[1]{%
  \vcenter{%
    \Let@ \restore@math@cr \default@tag
    \baselineskip\fontdimen10 \scriptfont\tw@
    \advance\baselineskip\fontdimen12 \scriptfont\tw@
    \lineskip\thr@@\fontdimen8 \scriptfont\thr@@
    \lineskiplimit\lineskip
    \ialign{\hfil$\m@th\scriptstyle##$&$\m@th\scriptstyle{}##$\hfil\crcr
      #1\crcr
    }%
  }%
}
\makeatother
\makeatletter
\def\@tocline#1#2#3#4#5#6#7{\relax
  \ifnum #1>\c@tocdepth 
  \else
    \par \addpenalty\@secpenalty\addvspace{#2}%
    \begingroup \hyphenpenalty\@M
    \@ifempty{#4}{%
      \@tempdima\csname r@tocindent\number#1\endcsname\relax
    }{%
      \@tempdima#4\relax
    }%
    \parindent\z@ \leftskip#3\relax \advance\leftskip\@tempdima\relax
    \rightskip\@pnumwidth plus4em \parfillskip-\@pnumwidth
    #5\leavevmode\hskip-\@tempdima
      \ifcase #1
       \or\or \hskip 1em \or \hskip 2em \else \hskip 3em \fi%
      #6\nobreak\relax
    \dotfill\hbox to\@pnumwidth{\@tocpagenum{#7}}\par
    \nobreak
    \endgroup
  \fi}
\makeatother

\numberwithin{equation}{section}

\newcommand{\nth}{\textup{th}}

\title{Fractals Emerging from the Toepltiz Determinants of the $p$-Cantor Sequence}
\author{Noy Soffer Aranov}
\email{noy.sofferaranov@tugraz.at}
\address{Graz University of Technology, Institute of Analysis and Number Theory, 8010 Graz, Austria}

\author{Steven Robertson}
\email{
steven.robertson@manchester.ac.uk
}
\address{Department of Mathematics, University of Manchester, Manchester, United Kingdom}
\begin{document}
\setcounter{MaxMatrixCols}{15}
\begin{abstract}
    \noindent  This is the first of a pair of papers, whose collective goal is to disprove a conjecture of Kemarsky, Paulin, and Shapira (KPS) on the \textit{escape of mass of Laurent series}. This paper lays the foundations on which its sibling, \cite{SAR}, builds.\\

    \noindent In particular, the $p$-Cantor sequence is introduced which generalises the classical Cantor sequence into a $p$-automatic sequence for any odd prime $p$. Two main results are then established, both of which play a key role in the disproof of the KPS conjecture.\\

    \noindent First, the two-dimensional sequence comprised of the Toeplitz determinants of the $p$-Cantor sequence over $\F_p$ is extensively studied. Indeed, the so-called \textit{profile} of this sequence (which encodes the zero regions) is shown to be $[p,p]$-automatic. In the process of deriving this, the theory of so-called \textit{number walls} is developed greatly. Many of these results are stated in full generality, as the authors expect them to be useful when tackling similar problems going forward.\\
    
    \noindent Secondly, a natural process is described that converts number wall of an automatic sequence into a unique fractal. When this sequence is the aforementioned $p$-\textit{Cantor sequence}, this fractal is shown to have Hausdorff dimension $\log((p^2+1)/2)/\log(p).$
\end{abstract}
\maketitle

\tableofcontents
\section{Introduction}
\noindent Throughout this paper, \begin{itemize}
    \item $q$ is a positive power of a prime,
    \item $\F_q$ is the finite field of cardinality $q$,
    \item $\mathcal{R}$ is an arbitrary integral domain.
\end{itemize}

\subsection{Motivation}\hfill\\
The following is a brief introduction to the subject of \textit{escape of mass}. For more information, see \cite{AS,PS,EW, KPS, SAR} and the references contained therein. 
\subsubsection{\textbf{Escape of Mass over }$\R$}\hfill\\
\noindent Every real irrational number $\alpha$ has a unique continued fraction given by \[\alpha=a_0+\frac{1}{a_1+\frac{1}{a_2+\frac{1}{\ddots}}}=[a_0;a_1,a_2,\dots],\] where $a_0\in\Z$ and $a_i\in\N$ for $i\ge1$. A classical theorem of Lagrange states that $\alpha\in\R$ is a quadratic irrational if and only if its continued fraction expansion is eventually periodic. That is, there exists $m_\alpha, \ell_\alpha\in\N$ such that \[\alpha=[b_0;b_1,\dots,b_{m_\alpha},\overline{a_{0}, a_1,\dots,a_{\ell_{\alpha}}}]\] where $b_0\in\Z$, $b_i\in\N$ for $i\ge1$ and $a_i\in\N$ for $i\ge0$. \\

\noindent The set of quadratic irrationals is closed under multiplication by primes. That is, for any prime $p$ and quadratic irrational $\alpha\in\R$, one has that $p\alpha$ is also a quadratic irrational. It is natural then, to ask how the continued fraction expansion of $p^k\alpha$ evolves as $k\in\N$ grows to infinity.  \\

\noindent In particular, it is known that no single term in the periodic part of the continued fraction expansion of $p^k\alpha$ makes up a positive logarithmic proportion of the total mass.

\begin{theorem}{\cite[Theorem 1.2]{AS}}
\label{thm:ASEsc}
    Let $\alpha\in \mathbb{R}$ be a quadratic irrational and let $p$ be a prime. Additionally, for $k\in\N$ let $\ell_k$ be the length of the periodic part of the continued fraction expansion of $p^k\alpha$. Then, 
    \begin{equation}
    \label{eqn:noEscR}
\lim_{k\rightarrow\infty}\frac{\max_{i=1,\dots,\ell_k}\log\left(a_i(p^k\alpha)\right)}{\sum_{i=1}^{\ell_k}\log(a_i(p^k\alpha))}=0.
\end{equation}
\end{theorem}
\noindent That is, the sequence $\{p^k\alpha\}_{k\ge0}$ has \textbf{no escape of mass}.
\subsubsection{\textbf{Escape of Mass over }$\F_p(\!(t^{-1})\!)$}\hfill\\
\noindent Any statement over $\R$ has a natural analogue over function fields. That is, one considers polynomials (formal Laurent series, respectively) in place of integers (real numbers, respectively). More formally, let \[\F_q[t]=\left\{\sum_{n=0}^{h}a_nt^{n}:a_n\in \mathbb{F}_q,\,h\in \mathbb{N},\, a_h\neq0\right\}\] be the ring of polynomials over $\F_q$ and let \[\F_q\left(\!\left(t^{-1}\right)\!\right)=\left\{\sum_{n=-h}^{\infty}a_nt^{-n}:a_n\in \mathbb{F}_q,h\in \mathbb{Z},a_{-h}\neq 0\right\}\]be the field of formal Laurent series over $\F_q$. The counterpart to prime numbers is given by the set of irreducible polynomials. \\

\noindent In analogy with the reals, every Laurent series $\Theta(t)\in\F_q\left(\!\left(t^{-1}\right)\!\right)$ has a unique continued fraction expansion given by\begin{equation}\nonumber
    \Theta(t)=A_0^{[\Theta]}(t)+\frac{1}{A_1^{[\Theta]}(t)+\frac{1}{A_2^{[\Theta]}(t)+\frac{1}{\ddots}}}=\left[A_0^{[\Theta]}(t);A_1^{[\Theta]}(t),A_2^{[\Theta]}(t),\dots\right],
\end{equation}where $A_i^{[\Theta]}(t)\in \mathbb{F}_q[t]$ for $i\geq 0$ and $\deg\left(A_i^{[\Theta]}(t)\right)\geq 1$ for $i\geq 1$. \\

\noindent Additionally, a Laurent series if called \textbf{irrational} if it is not expressible as the ratio of two polynomials.\\

\noindent Continuing the similarities with the real case, $\Theta(t)\in\F_q(\!(t^{-1})\!)$ is a quadratic irrational if and only if its continued fraction expansion is eventually periodic \cite[Theorem 3.2]{DAFFsurvey}. That is, if there exist $m_{\Theta},\ell_{\Theta}\geq 1$, and there exist polynomials $B_i^{[\Theta]}(t)\in\F_q[t]$ for $0\le i \le m_{\Theta}\in\N$ and $A_i^{[\Theta]}(t)\in\F_q[t]$ for $0\le i \le \ell_{\Theta}\in\N$ such that
\begin{equation}\nonumber
    \Theta(t)=\left[B_0^{[\Theta]}(t);B_1^{[\Theta]}(t),\dots,B_{m_{\Theta}}^{[\Theta]}(t),\overline{A_0^{[\Theta]}(t),\dots ,A_{\ell_{\Theta}}^{[\Theta]}(t)}\right].
\end{equation}
\noindent However, the analogy between $\F_q(\!(t^{-1})\!)$ and $\R$ breaks down in view of Theorem \ref{thm:ASEsc}.
\begin{theorem}{\cite[Theorem 1]{PS}}
    \label{thm:PS}
    Let $P(t)\in \mathbb{F}_q[t]$ be an irreducible polynomial, let $\Theta(t)\in \mathbb{F}_q(\!(t^{-1})\!)$ be a quadratic irrational, and let $\ell_k$ be the length of the periodic part of the continued fraction expansion of $P^k(t)\cdot \Theta(t)$. Then, there exists $0<c\leq 1$, depending only on $P(t)$ and $\Theta(t)$, such that \begin{equation}\label{eqn:MassEqn}
    \lim_{n\rightarrow \infty}\liminf_{k\rightarrow \infty}\frac{\max\left\{\deg\left(A_{i}^{[\Theta\cdot P^k]}(t)\right)-n,0\right\}}{\sum_{i=1}^{\ell_{k}}\deg\left(A_{i}^{[\Theta\cdot P^k]}(t)\right)}\geq c.
    \end{equation}
\end{theorem}
\noindent In fact, the above theorem is stronger than the precise function field analogue of Theorem \ref{thm:ASEsc} due to the additional `minus $n$' term. Indeed, Theorem \ref{thm:PS} is saying that one can subtract \textit{any} natural number $n$ from the numerator of equation (\ref{eqn:noEscR}) and the analogue of Theorem \ref{thm:ASEsc} is still not true. This led Kemarsky, Paulin, and Shapira to make the following conjecture:
\begin{conjecture}{\cite[Conjecture 6]{KPS}}
    \label{conj:KPS}
    For every irreducible polynomial $P(t)\in \mathbb{F}_q[t]$ and every quadratic irrational Laurent series $\Theta(t)\in \mathbb{F}_q(\!(t^{-1})\!)$, one has \begin{equation}\nonumber\label{eqn:MassEqn}
    \lim_{n\rightarrow \infty}\liminf_{k\rightarrow \infty}\frac{\max\left\{\deg\left(A_{i}^{[\Theta\cdot P^k]}(t)\right)-n,0\right\}}{\sum_{i=1}^{\ell_{k}}\deg\left(A_{i}^{[\Theta\cdot P^k]}(t)\right)}=1,
    \end{equation}
    where, $\ell_k$ is the length of the periodic part of the continued fraction expansion of $P^k(t)\cdot \Theta(t)$.
\end{conjecture}
\subsection{Overview of Main Results}\hfill\\

\noindent In 2025, Nesharim, Shapira and the first named author \cite{AN} disproved Conjecture \ref{conj:KPS} over $\F_2$ in the case $P(t)=t$. The combined goal of both this paper and \cite{SAR} is to disprove Conjecture \ref{conj:KPS} in \textit{all remaining cases}. \begin{theorem}
    Conjecture \ref{conj:KPS} fails over any choice finite field $\F_p$ and for every irreducible polynomial $P(t)\in\F_q[t]$. 
\end{theorem}
\noindent Not only is Conjecture \ref{conj:KPS} shown to be false, an explicit counterexample is provided for each of the possible cases. When $\F_q$ has characteristic equal to a prime $p$, these counterexamples are given by the so-called $p$-\textit{Cantor sequence}. This family of sequences is introduced in Definition \ref{pcant}, and it provides a generalisation of the classical Cantor sequence. \\

\noindent The disproof of Conjecture \ref{conj:KPS} was too long and varied to reasonably fit into a single paper. Therefore, the decision was made to split it in two.  This paper is dedicated to establishing all of the necessary groundwork so that the second paper, \cite{SAR}, is able to focus solely on Conjecture \ref{conj:KPS} itself. \\

\noindent All this is not to say that the results in this paper are of interest only in the context of Conjecture \ref{conj:KPS}; this is not the case. Indeed, this work contains two main theorems, each of which are valuable in their own right. These are the subjects of subsections \ref{Sect: 1.3} and \ref{Sect: 1.4}, respectively. 
  
\subsection{Main Result 1: The Number Wall of the $p$-Cantor Sequence}\label{Sect: 1.3}\hfill\\
\noindent The disproof of Conjecture \ref{conj:KPS} relies on so-called \textit{number walls}. To define these, one must first define Toeplitz matrices. 
\subsubsection{\textbf{Toeplitz Matrices Generated by Sequences.}}\hfill\\
\noindent Let $m$ be a natural number. A square matrix $M=(s_{i,j})_{0\le i,j \le m}$ is called \textbf{Toeplitz} if $s_{i,j}=s_{i+1,j+1}$ for all $0\le i,j \le m-1$. That is, all entries in any given diagonal of $M$ are equal to one another.\\

\noindent There is an immediate link between sequences and Toeplitz matrices: given a doubly infinite sequence $\mathbf{\textbf{S}}:= (s_i)_{i\in\mathbb{Z}}$ over some integral domain $\id$, a natural number $m$ and an integer $n$, define the $(m+1)\times (m+1)$ Toeplitz matrix $T_\textbf{S}(n;m):= (s_{i-j+n})_{0\le i,j \le m}$ as\\
\begin{align}T_\textbf{S}(n;m):=\begin{pmatrix}
s_n & s_{n+1} & \dots & s_{n+m}\\
s_{n-1} & s_{n} & \dots & s_{n+m-1}\\
\vdots &\vdots&\ddots& \vdots\\
s_{n-m} & s_{n-m+1} & \dots & s_{n} 
\end{pmatrix}.\label{eqn: toe_def}\end{align} From $\textbf{S}$, one constructs the infinite family of Toeplitz matrices $(T_{\textbf{S}}(n;m))_{n\in\Z,m\in\N}$. 

\subsubsection{\textbf{The Number Wall of a Sequence}}\hfill\\
\noindent Attempting to examine the properties of any individual matrix in $(T_{\textbf{S}}(n;m))_{n\in\Z,m\in\N}$ quickly becomes unwieldy. Instead, one studies the entire family of Toeplitz matrices at once by considering the so-called \textit{number wall of} $\textbf{S}$: \begin{definition}\label{nw}
Let $\mathbf{S}=(s_i)_{i\in\mathbb{Z}}$ be a doubly infinite sequence over an integral domain $\mathcal{R}$. The \textbf{number wall} of the sequence $\mathbf{S}$ is defined as the two dimensional array of numbers $W_\id(\mathbf{S})=(W_{\id}(\mathbf{S})[m,n])_{m,n\in\mathbb{Z}}$ with \begin{equation*}
    W_{\id}(\mathbf{S})[m,n]=\begin{cases}\det(T_S(n;m)) &\textup{ if } m\ge0,\\
    1 & \textup{ if } m=-1,\\
    0 & \textup{ if } m<-1. \end{cases}
\end{equation*}  
Above, $W_\id(\textbf{S})[m,n]$ is the entry of the number wall in row $m$ and column $n$. When $\id=\F_q$ for some prime power $q\in\N$, $W_\id(\mathbf{S})$ is abbreviated to $W_q(\mathbf{S})$.
\end{definition} 
\noindent To stay consistent with standard matrix notation, $m$ increases as the rows go down the page and $n$ increases from left to right. \\

\noindent Although Toeplitz matrices have been the focus of study since the $19^\nth$ century, the first instance of \textit{number walls} in the literature is due to Conway and Guy \cite[pp. 85--89]{CG}. Since then, the subject has been greatly developed by Lunnon \cite{L01,L09} and has been applied to problems in Diophantine approximation \cite{ANL,R23,RG, P(t)_LC}, dynamics \cite{AN} and discrepancy theory \cite{Rob24}. \\

\noindent There is also precedent for the study of number walls in their own right. For example, the number wall of the famous \textit{Thue-Morse} sequence over $\F_2$ was the subject of a celebrated paper by Allouche, Peyri\`ere, Wen and Wen \cite{APWW}, and the number wall of the classical Cantor sequence has been studied by Wen and Wu \cite{WW}.

\subsubsection{\textbf{Number Walls of Automatic Sequences}}\hfill\\
\noindent For the purposes of this introduction, a rough definition of automatic sequences is provided. For a rigorous definition, see Section \ref{Sect: 2}. \\

\noindent A sequence defined over a finite alphabet is said to be \textit{automatic} if it is generated by a substitution rule\footnote{That is, a uniform $k$-morphism for some $k\in\N$. Formal language is used from Section \ref{Sect: 2} onward.}. 

\begin{example}\label{Ex: Cantor}The \textbf{Cantor Sequence}\footnote{\hyperlink{https://oeis.org/A088917}{A088917} in the Online Encyclopedia of Integer Sequences} is constructed by the rule $\phi:\{0,1\}\to\{0,1\}^3$,\[\phi:1\mapsto 101, ~~~ \phi:0\mapsto 000.\] To attain the Cantor sequence, $\phi$ is applied to 1 infinitely many times, as illustrated below: \begin{alignat*}{5}\rd1\bk\xrightarrow{\phi}& \rd101\bk&\\&\rd1\bl0\gn1\bk \xrightarrow{\phi} \rd101\bl000\gn101&\\&\hspace{1.2cm}\lrd1\rd0\drd1\lbl0\bl0\dbl0\lgn1\gn0\dgn1\bk\xrightarrow{\phi}\lrd101\rd000\drd101\lbl000\bl000\dbl000\lgn101\gn000\dgn101\bk\xrightarrow{\phi}\cdots\end{alignat*}
\noindent Above, each colour represents an application of $\phi$, with the input (output respectively) being on the left (right, respectively) of the arrow. \\
\end{example}
\noindent The definition of automatic sequences has a natural extension into two dimensions: a two-dimensional sequence is automatic if it is generated by a substitution rule that replaces each letter with a $d\times k$ rectangle of letters, for some natural numbers $d,k$ such that $d,k\ge2$.\\

\noindent There are numerous examples of number walls inheriting automaticity from the sequences that generate them. For example, the number walls of the Thue-Morse sequence \cite{APWW, AN}, the Paperfolding sequence \cite{ANL, GaRo, P(t)_LC} and the Cantor sequence \cite{WW} are known to be two dimensional automatic sequences. This motivated the following conjecture by Garrett and the second named author.

\begin{conjecture}[{{\cite[Conjecture 5.1]{RG}}}]\label{conj: auto_nw}
    Let $q$ be a positive power of a prime and let $\textbf{S}$ be an automatic sequence over $\F_q$. Then, $W_q(\textbf{S})$ is itself a 2-dimensional automatic sequence.
\end{conjecture}  

\noindent Conjecture \ref{conj: auto_nw} has since appeared in \cite{P(t)_LC} and \cite{AN}. \\ 

\noindent The first main result of this work improves upon what is known by providing further evidence towards Conjecture \ref{conj: auto_nw} for the \textit{infinite family} of so-called $p$-\textit{Cantor sequences}. The definition of these sequences is postponed to Section \ref{Sect: 3} as it requires a rigorous introduction to automatic sequences and is not required to state the results of this paper. \\

\noindent The theorem below is presented in plain language so to avoid cluttering the introduction with definitions. To see this theorem written formally, see Theorem \ref{thm: morphism}.\\

\noindent \textbf{Main Result 1:} Let $p$ be any odd prime, and let $W$ be the number wall of the $p$-Cantor sequence over $\F_p$ when no distinction is made between nonzero entries. Then, $W$ is a 2-dimensional automatic sequence defined on an alphabet of 12 letters. \\

\noindent The above result marks the first time the number wall of \textit{infinitely many} different non-trivial sequences has been fully understood. Furthermore, the proof is not computer assisted, which is a departure from most other number wall results. \\

\noindent For pictures of the number wall of the $p$-Cantor sequence over $\F_p$, see Appendix \hyperref[Sect: Appendix_B]{B}.\\

\noindent In order to prove Theorem \ref{thm: morphism}, the theory of number walls has been developed greatly. Indeed, Section \ref{sect: 7} contains numerous lemmata that are stated in full generality and are applicable to \textit{any} number wall. It is the author's belief that these will be valuable to anyone studying number walls in the future.

\subsection{Main Result 2: Number Walls and Fractals}\label{Sect: 1.4}\hfill\\

\noindent The uses of Toeplitz matrices are as vast as they are varied, with applications in both applied mathematics (error correcting codes \cite{StBa}, numerical analysis \cite{GoGL} and statistics \cite{Da}, for example) and pure mathematics (operator space theory \cite{Gi}, Functional Analysis \cite{AlSi} and linear algebra \cite{Gr}, for example). The goal of this subsection is to build another such connection, this time between Toeplitz matrices and fractal geometry.\\

\noindent To describe the connection between number walls and fractals, the following fundamental theorem is required.

\begin{theorem}[Square Window Theorem, Lunnon, \cite{L09}]\label{window}
Zero entries in a Number Wall can only occur within windows; that is, within
square regions with horizontal and vertical edges.
\end{theorem}

\noindent Let $X\subset\mathcal{R}$ be a finite set, $k\ge2$ be a natural number and let $\textbf{S}=(s_i)_{i\ge0}$ be an automatic sequence generated by applying some substitution rule $\phi:X\to X^k$ to some $\sigma\in X$, where the first letter of $\phi(\sigma)$ is $\sigma$ itself. That is, $\phi(\sigma)=\sigma w_2\dots w_{k-1}$ for some $w_2,\dots,w_{k-1}\in X$, and  $$\textbf{S}=\lim_{r\to\infty} \phi^r(\sigma),~~~\text{ where }~~~\phi^r(\sigma)=\overbrace{\phi(\phi(\cdots\phi}^{r~\text{ times}}(\sigma)\cdots)).$$
\noindent For any $r\in\N$, it is clear that $\phi^r(\sigma)$ has length $k^r$. Consider the sequence of length $3k^r$ given by \[\widetilde{\textbf{S}}_r=(x_i)_{-k^r\le i <2k^r}~~~\text{ where }~~~ x_i=\begin{cases}0&\text{ if }i<0\text{ or }i\ge k^r\\ s_{i} &\text{ otherwise}.\end{cases}\]Then, by the Square Window Theorem (Theorem \ref{window}), the number wall of $\widetilde{\textbf{S}}_r$ has the following appearance. \begin{figure}[H]
    \centering
    \includegraphics[width=0.5\linewidth]{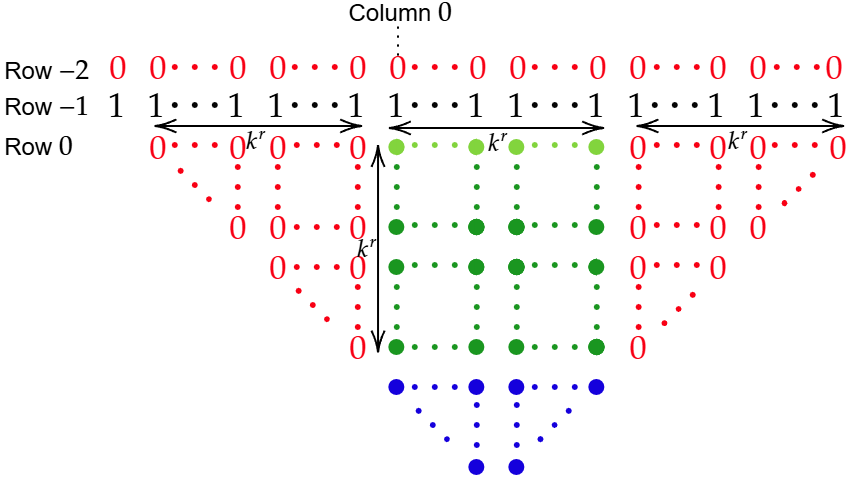}
    \caption{The number wall of $\widetilde{\textbf{S}}_r$. Each entry is either denoted with a number (where it is known) or a coloured dot (where it depends on $\textbf{S}$). The sequence $(s_i)_{0\le i <k^r}$ is coloured in light green. The other colours serve only to distinguish one part of the number wall from another. }
    \label{fig: 1.2}
\end{figure}
\noindent Next, consider only the square portion of the number wall given by $(W_\mathcal{R}(\widetilde{\textbf{S}}_r)[m,n])_{0\le m,n<k^r}$. That is, the green portion in Figure \ref{fig: 1.2}. This is used to construct a unique set $\mathcal{J}_r\subset [0,1]^2$ in the following way \\

\noindent Let $0\le m,n<k^r$ be natural numbers and let $$I_{m,n}=\left[\frac{m}{k^r},\frac{m+1}{k^r}\right)\times \left[\frac{n}{k^r},\frac{n+1}{k^r}\right).$$ Then, define \begin{equation*}\mathcal{J}_r:=\bigcup_{\substack{0\le m,n <k^r\\W_\mathcal{R}(\widetilde{\textbf{S}}_r)[m,n]\neq0}} I_{m,n}.\end{equation*}
\noindent Finally, the fractal set is defined as \begin{equation}\label{eqn: fractal_def}\mathcal{F}(\textbf{S},\mathcal{R}):=\bigcap_{r\ge0}\mathcal{J}_r.\end{equation}
\noindent The second main result of this work calculates the \textit{Hausdorff dimension} (denoted $\dim_H$) of the Fractal generated by the $p$-Cantor sequence. For more information on Hausdorff dimension, see \cite[Chapter 2]{Fal}.\\

\noindent \begin{theorem}[\textbf{Main Result 2}]\label{thm: frac_nw}
    Let $p$ be an odd prime and $C^{(p)}$ be the $p$-Cantor sequence over $\F_p$. Then, \[\dim_H\left(\mathcal{F}\left(C^{(p)},\F_p\right)\right)=\frac{\log\left(\frac{p^2+1}{2}\right)}{\log \left(p\right)}\]
\end{theorem}
\subsection*{Structure of the Paper}\hfill\\
The following flow diagram illustrates the structure of the paper. Here, an arrow going from Section A to Section B indicates that the information in Section A is required to understand Section B. \begin{figure}[H]
    \centering
    \includegraphics[width=0.75\linewidth]{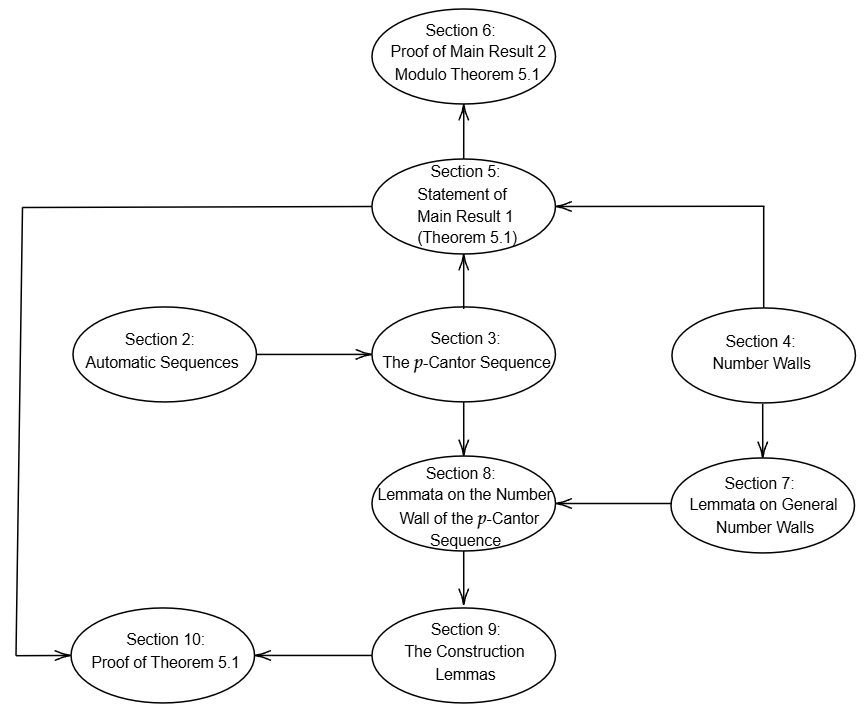}
\end{figure}

\subsection*{Acknowledgements}\hfill\\
The first named author would like to thank Uri Shapira for proposing a question which led to this line of research. She would also like to thank Erez Nesharim for multiple discussions which led to the ideas in this paper. The second named author would like to thank Fred Lunnon for his discussions relating to the number wall of the $p$-Cantor sequence. Both authors would also like to thank Faustin Adiceam for suggesting that they work together on this question. 

\section{The Essentials of Automatic Sequences}\label{Sect: 2}
\noindent This section provides a formal introduction to one and two-dimensional automatic sequences, which are crucial to the definition of the $p$-Cantor sequence and the statement of Main Result 1 (Theorem \ref{thm: morphism}). The content contained within this section closely follows \cite{AllSha}. 
 \subsubsection*{\textbf{Alphabets and Words}}\hfill\\\noindent To begin, the following elementary definitions are made.\\

\noindent A \textbf{finite alphabet} $\Sigma$ is a finite collection of symbols. Given a finite alphabet $\Sigma$, a \textbf{finite word} over $\Sigma$ is a concatenation of finitely many letters $\sigma\in\Sigma$. The set of all finite words made from $\Sigma$ is denoted by $\Sigma^*$. To combine words, the following operation is defined. \begin{definition}\label{def: concat}
    Let $\sigma_0\cdots\sigma_n,~ \sigma'_0\cdots\sigma'_{m}\in\Sigma^*$ be words. Define the \textbf{concatenation operation} $*$ as $$*:\Sigma^*\to\Sigma^*,\hspace{2cm}\sigma_0\cdots \sigma_n*\sigma'_0\cdots\sigma'_{m}=\sigma_0\cdots\sigma_n\sigma'_0\cdots\sigma'_{m}.$$
\end{definition}  \noindent The set $\Sigma^*$ is a monoid where the identity is given by the \textbf{empty word} $\epsilon$, defined by the relationship $w*\epsilon=w$ for every finite word $w\in \Sigma^*$.  Similarly, an \textbf{infinite word} is a concatenation of countably many letters from the alphabet. Denote the set of infinite words made from $\Sigma$ as $\Sigma^\omega$. In most cases, one is interested in both finite and infinite words. Hence, define $\Sigma^\infty:=\Sigma^*\cup\Sigma^\omega.$\\

\noindent For the majority of this work, the focus is on sequences over $\F_p$ for some odd prime $p$. It is often useful to think of those sequences instead as infinite words over the alphabet comprised of the elements of $\F_p$. For this reason, define the finite alphabet \begin{equation}\label{eqn: gamma_p}\Gamma_p=\{0,1,2,\dots,p-1\}.\end{equation}

  \subsubsection*{\textbf{Codings and uniform morphisms}}\hfill\\
\noindent The \textbf{length} of a word $w\in\Sigma^*$, denoted $L(w)$, is a function $L:\Sigma^*\to\N$ that returns the number of (non-empty) letters in $w$. By definition, $L(\epsilon)=0$.

\begin{definition}
\noindent Given finite alphabets $\Sigma$ and $\Delta$, a function $\tau:\Sigma^\infty\to\Delta^\infty$ is called a \textbf{coding}\footnote{Typically, a coding is defined as a uniform 1-morphism. However, for the purposes of this paper, it is more convenient to define it as above. The role that codings play in the construction of automatic sequences has not changed.} if for any word $w=\sigma_0\sigma_1\cdots\in\Sigma^\infty$,  $\tau(\sigma_0\sigma_1\dots)=\tau(\sigma_0)*\tau(\sigma_1)*\cdots$. Note that, it is sufficient to define $\tau$ on the letters of $\Sigma$.  \\

\noindent If $d$ is a natural number, then $\tau$ is a $d$-\textbf{coding} if $L(\tau(w))=d\cdot L(w)$ for every word $w\in\Sigma^*$. In the case where $\Delta=\Sigma$, the $d$-coding $\tau$ is instead called a \textbf{uniform }$d$-\textbf{morphism}. Usually, uniform morphisms are denoted by $\phi:\Sigma^\infty\to\Sigma^\infty$. Furthermore, for a natural number $i<L(w)$, let $w_i$ denote the $i^\nth$ letter of $w$.  \\

\noindent Let $d\in\N$ and let $\phi:\Sigma^\infty\to\Sigma^\infty$ be a uniform $d$-morphism. A letter $\sigma\in\Sigma$ is called $\phi$-\textbf{prolongable} if it is the first letter of its own image under $\phi$: that is, $\phi(\sigma)_0=\sigma$. If an infinite word $w=\sigma_0\sigma_1\sigma_2\cdots$ in $\Sigma^\omega$ satisfies $\phi(w)=w$, then $\sigma_0$ is $\phi$-prolongable. In this case, $w$ is called a \textbf{fixed point} of $\phi$ and denoted by $w=\phi^\infty(\sigma_0)$. From another point of view, one defines $\phi^\infty(\sigma_0)$ inductively by\begin{equation*}
    \phi^\infty(\sigma_0)=\lim_{n\to\infty} \phi^n(\sigma_0), ~~~\text{ where }\phi^n(\sigma)=\phi(\phi^{n-1}(\sigma_0)) ~~~\text{ and }~~~ \phi^0(\sigma_0)=\sigma_0.
\end{equation*}
\end{definition}

  \subsubsection*{\textbf{Cobham's Theorem and Automatic Sequences}}\hfill\\
\noindent Originally, a sequence was called \textit{automatic} if it was the output of some deterministic finite automaton (see \cite[Chapter 4]{AllSha} for details). However, the following theorem of Cobham \cite{C} (rephrased in modern language in \cite[Theorem 6.3.2]{AllSha}) provides a more practical definition of an automatic sequence:
\begin{theorem}[Cobham's Theorem]
    A sequence $\textbf{S}$ is \textbf{automatic} if and only if it there exist natural numbers $k$ and $d$ and finite alphabets $\Sigma$ and $\Delta$, such that $\textbf{S}$ is the fixed point of some uniform $k$-morphism $\phi:\Sigma^\infty\to\Sigma^\infty$ under a $d$-coding $\tau:\Sigma^\infty\to\Delta^\infty$.
\end{theorem}
  \subsubsection*{\textbf{Example: The Thue-Morse Sequence}}\hfill\\
\noindent The simplest example of a non-periodic automatic sequence is given by the Thue-Morse sequence. \begin{example}\label{Ex: TM}
    Let $\Sigma=\Gamma_2=\{0,1\}$ be a finite alphabet. Define the \textbf{Thue-Mose} sequence as $\phi^\infty(0)$ where $\phi:\Sigma^\infty\to\Sigma^\infty$ is a 2-morphism defined as \begin{align*}
        \phi(0)=01~~~\text{ and }~~~\phi(1)=10.
    \end{align*} 
\end{example}
\noindent To make the Thue-Morse sequence comply with Cobham's Theorem, one uses the trivial 1-coding with $\Delta=\Sigma$. Some $k$-automatic sequences are impossible to state without a non-trivial $d$-coding; the $p$-Singer sequence (see Section \ref{Sect: 4}) is one such example. 

  \subsubsection*{\textbf{Two-Dimensional Words}}\hfill\\

\noindent Automatic sequences have a natural generalisation to higher dimensions. For the purposes of this work, only two-dimensional automatic sequences are required and thus, are defined below. \\

\noindent Let $\Sigma$ be a finite alphabet. A \textbf{two-dimensional finite word} $\mathbf{w}$ over $\Sigma$ is defined as a matrix $$\textbf{w}=(\sigma_{m,n})_{0\le m \le M, 0\le n \le N} ~~~~\text{ where }~~~~ M,N\ge0.$$ To be consistent with standard matrix terminology, $\sigma_{m,n}\in\Sigma$ denotes the entry of $\mathbf{w}$ in row $m$ and column $n$. \\

\noindent Let $\Sigma^*_2$ denote the set of finite two-dimensional words over $\Sigma$. Similarly, let $\Sigma^\omega_2$ be the set of two-dimensional infinite words, each of which is an infinite two-dimensional matrix $(\sigma_{m,n})_{m,n\in\N}$. Finally, define $\Sigma_2^\infty:=\Sigma_2^*\cup\Sigma_2^\omega$. 
  \subsubsection*{\textbf{Two-Dimensional Morphisms and Codings}}\hfill\\
\noindent Given $k$ and, $l\in\N$, a \textbf{two-dimensional uniform $[k,l]$-coding }$\tau:\Sigma_2^\infty\to\Sigma_2^\infty$ is a function that satisfies the following two properties: \begin{enumerate}
    \item For any letter $\sigma\in\Sigma$, $\tau(\sigma)$ returns a $k\times l$ matrix in $\Sigma^*_2$.
    \item Let $M,N\in\N$ and let $\textbf{w}=(\sigma_{m,n})_{0\le m \le M,0\le n \le N}\in\Sigma_2^*$ be a two-dimensional word. Then, the function $\tau$ satisfies \[\tau\begin{pmatrix}
        \sigma_{0,0}&\sigma_{0,1}&\cdots &\sigma_{0,N}\\\vdots&&&\vdots\\\sigma_{M,0}&\sigma_{M,1}&\cdots&\sigma_{M,N}
    \end{pmatrix}=~\begin{pmatrix}
        \tau(\sigma_{0,0})&\tau(\sigma_{0,1})&\cdots &\tau(\sigma_{0,N})\\\vdots&&&\vdots\\\tau(\sigma_{M,0})&\tau(\sigma_{M,1})&\cdots&\tau(\sigma_{M,N})
    \end{pmatrix}\]
\end{enumerate} Similarly to the one-dimensional case, a $[k,l]$-coding is called a $[k,l]$-\textbf{morphism} if $\Delta=\Sigma$. \\

\noindent A letter $\sigma_0\in\Sigma$ is called $\phi$\textbf{-prolongable} if $\phi(\sigma_0)_{0,0}=\sigma_0$. Just as in the one-dimensional case, define $\phi^k(\sigma_0)$ iteratively as $\phi^k(\sigma_0)=\phi(\phi^{k-1}(\sigma_0))$ and $\phi^0(\sigma_0)=\sigma_0$. Then, the limiting sequence $\phi^\infty(\sigma_0)=\lim_{k\to\infty}\phi^k(\sigma_0)$ exists if and only if $\sigma_0$ is $\phi$-prolongable. 
  \subsubsection*{\textbf{Example: Thue-Morse in 2-dimensions}}\hfill\\
\noindent The 2-dimensional version of the Thue Morse sequence from Example \ref{Ex: TM} is defined.
\begin{example}\label{2dim}
    Let $\Sigma=\Gamma_2=\{0,1\}$, and define the two-dimensional $[2,2]$-morphism $\phi:\Sigma^\infty_2\to\Sigma^\infty_2$ by \begin{align*}
        &\phi(0)=\begin{matrix}
            0&1\\1&0
        \end{matrix}& &\phi(1)=\begin{matrix}
            1&0\\0&1
        \end{matrix}.&
    \end{align*}
\noindent Then, applying $\phi$ repeatedly to the $\phi$-prolongable element $0\in\Sigma$ gives \begin{align*}
    &\phi^0(0)=0& &\phi^1(0)=\begin{matrix}
            0&1\\1&0
        \end{matrix}& &\phi^2(0)=\begin{matrix}
            0&1&1&0\\1&0&0&1\\1&0&0&1\\0&1&1&0
        \end{matrix}&&\phi^3(0)=\begin{matrix}
            0&1&1&0&1&0&0&1\\1&0&0&1&0&1&1&0\\1&0&0&1&0&1&1&0\\0&1&1&0&1&0&0&1\\1&0&0&1&0&1&1&0\\0&1&1&0&1&0&0&1\\0&1&1&0&1&0&0&1\\1&0&0&1&0&1&1&0
        \end{matrix}.&
\end{align*}
\noindent Once again, the trivial $[1,1]$-coding is applied to generate the 2-dimensional Thue-Morse sequence. However, for the sake of example, let $\Delta=\{a,b\}$ and define the $[1,2]$-coding $\tau:\Sigma_2^\infty\to\Delta_2^\infty$ by \begin{align*}
    &\tau(0)=~\begin{matrix}a&a\end{matrix}& &\tau(1)=~\begin{matrix}b&b\end{matrix}.& 
\end{align*} 
\noindent Then, $\tau$ applied to $\phi^2(0)$ is given below.
\[\tau\left(\phi^2(0)\right)=\begin{matrix}
            a&a&b&b&b&b&a&a\\b&b&a&a&a&a&b&b\\b&b&a&a&a&a&b&b\\a&a&b&b&b&b&a&a
        \end{matrix}.\]
\end{example}
  \subsubsection*{\textbf{Cobham's Theorem in Two Dimensions}}\hfill\\
\noindent The concept of a two-dimensional 2-automatic sequence is defined similarly to the one-dimensional case, using a generalised version of Cobham's Theorem by Allouche and Shallit \cite[Theorem 14.2.3]{AllSha}.\begin{theorem}[Cobham's Theorem for Two-Dimensional Sequences:]
    A two dimensional sequence $\mathbf{S}$ is \textbf{automatic} if and only if there exist natural numbers $k,l,k',l'$ such that $\mathbf{S}$ is the image under a $[k',l']$-coding of a fixed point of a two-dimensional uniform $[k,l]$-morphism.
\end{theorem}

\section{The $p$-Cantor and $p$-Singer Sequences}\label{Sect: 3}

\noindent Let $p$ be an odd prime. The purpose of this section is to introduce the aforementioned $p$-Cantor sequence, along with the so-called $p$-\textit{Singer} sequence. These are both $p$-automatic sequences over $\Gamma_p$, and the former is a generalisation of the `Cantor' sequence seen in Example \ref{Ex: Cantor}. There are four subsections: the first states some preliminary functions, the second and third introduces the $p$-Cantor and $p$-Singer sequences respectively, and the fourth explores the relationship between the two sequences.
\subsection{Preliminary Functions}
\subsubsection*{\textbf{The Choose function for non-natural numbers}}\hfill\\
\noindent The following generalisation of the classical `choose' function is used often: for $a,b\in\R$, let \begin{equation}{a\choose b}=\begin{cases} \frac{a!}{b!\cdot (a-b)!}&\text{ if }a,b\in\N\cup\{0\}\text{ and }a\ge b\\0&\text{ otherwise}.
\end{cases}\label{eqn: choose}\end{equation}
\subsubsection*{\textbf{Shorthand for} $(p-1)/2$}\hfill\\
\noindent For any odd prime $p$ and $k\in\N$, let $$p^{(k)}_2:=\frac{p^k-1}{2}\cdotp$$ Furthermore, denote $p_2:=p^{(1)}_2$. 
\subsubsection*{\textbf{Multiplying a sequence by a constant}}\hfill\\
Let $\textbf{S}=(s_i)_{i\ge0}$ be a sequence over some integral domain $\id$. For any $n\in\id$, define \[n\cdot \textbf{S}=(n\cdot s_i)_{i\ge0}.\]
\subsubsection*{\textbf{Concatenating sequences}}\hfill\\
Given $l_S,l_R\in\N$ and two finite sequences $\mathbf{S}=(s_i)_{0\le i \le l_S}$ and $\textbf{R}=(r_i)_{0\le i \le l_R}$, define $\mathbf{S}\oplus \textbf{R}$  as the concatenation of $\textbf{S}$ and $\textbf{R}$. Explicitly, \begin{equation}\label{eqn: oplus}
        \mathbf{S}\oplus \textbf{R} = (k_i)_{0\le i \le l_S+l_R+1}, ~~~\text{ where }~~~ k_i=\begin{cases}s_i &\text{ if } 0\le i \le l_S,\\ r_{i-l_S-1} &\text{ if } l_S < i \le l_S+l_S+1.\end{cases}
    \end{equation}

\subsubsection*{\textbf{Turning one-sided sequences into two-sided sequences}}\hfill\\
\noindent For an infinite sequence $\textbf{S}=(s_i)_{i\ge0}$, one adapts $\textbf{S}$ into a two-sided infinite sequence as follows; \begin{equation}\textbf{S}^L:= (r_i)_{i\in\Z} ~~~\text{such that}~~~ r_i=\begin{cases}
    s_i &\text{ if }i\ge0\\ 0&\text{ otherwise,}
\end{cases}\label{eqn: ^L}\end{equation} 
\begin{remark}
    Throughout, many variations on a given sequence in a similar style to $\textbf{S}^L$ are defined. For convenience, these are collected into Appendix \hyperref[Sect: Appendix_A]{A}.
\end{remark}
\subsection{The $p$-Cantor Sequence}\hfill\\
\noindent The so-called $p$-Cantor sequence is the protagonist of this paper. The goal of this section is to define it and study its elementary properties.
\subsubsection{\textbf{The Definition of the} $p$\textbf{-Cantor Sequence}}\hfill\\
\noindent Recall the alphabet $\Gamma_p$ from equation (\ref{eqn: gamma_p})
\begin{definition}\label{pcant}
    Let $p$ be an odd prime. Then, for every $n\in\Sigma$ and for $0\le i <p$, define a $p$-morphism $\phi_p:\Gamma_p^\infty\to\Gamma_p^\infty$ as \begin{equation}
        \phi_p(n)_i = n\cdot {p_2 \choose\frac{i}{2}}\mod p.
    \end{equation} 
The \textbf{$\mathbf{p}$-Cantor sequence} $\textbf{C}^{(p)}=\left(c^{(p)}_i\right)_{i\ge0}$ is defined by $$\textbf{C}^{(p)}:=\lim_{n\to\infty}\phi_p^n(1).$$ 
\end{definition}

\begin{example}
    The 7-Cantor sequence is given by the 7-morphism $\phi:\Gamma_7^\infty\to\Gamma_7^\infty$ defined by:\begin{align*}
         &\phi_7(0)=0000000& &\phi_7(1)=1030301& &\phi_7(2)=2060602& \\&\phi_7(3)=3020203& &\phi_7(4)=4050504& &\phi_7(5)=5010105& \\&&&\phi_7(6)=6040406.&
    \end{align*}
\end{example}

\subsubsection{\textbf{The Symmetry of the} $p$\textbf{-Cantor Sequence}}\hfill\\
\noindent One crucial property of the $p$-Cantor sequence is its symmetry:
\begin{lemma}\label{sym}
    For any $k\in\N$, the $p$-Cantor sequence satisfies $c^{(p)}_{i}=c^{(p)}_{p^k-1-i}$ for all $0\le i \le p^k-1$.
\end{lemma}
\begin{proof}
\noindent The claim is proved by induction on $k\in\N$.\\

\noindent The base case, given by $k=0$, is trivial. Assume that for some $n\in\N$ and for every integer $0\le i<p^n$, one has $c^{(p)}_i=c^{(p)}_{p^n-1-i}$. Then, by the definition of the $p$-Cantor sequence, one has $$\left(c^{(p)}_{pi},c^{(p)}_{pi+1},\dots,c^{(p)}_{pi+p-1}\right)=\left(c^{(p)}_{i}\cdot{p_2\choose0},0,c^{(p)}_{i}\cdot{p_2\choose1},0,\dots,0,c^{(p)}_{i}\cdot{p_2\choose p_2}\right),$$ and similarly $$\hspace{-5cm}\left(c^{(p)}_{p^{n+1}-pi-p},c^{(p)}_{p^{n+1}-pi-p+1},\dots,c^{(p)}_{p^{n+1}-pi-1}\right)=$$ $$\left(c^{(p)}_{p^n-i-1}\cdot{p_2\choose0},0,c^{(p)}_{p^n-i-1}\cdot{p_2\choose1},0,\dots,0,c^{(p)}_{p^n-i-1}\cdot{p_2\choose p_2}\right).$$ Hence, for $0\le j \le p-1$, $$c^{(p)}_{pi+j}=c_i^{(p)} {p_2\choose \frac{j}{2}}\stackrel{(*)}{=}c_{p^n-i-1}^{(p)}{p_2\choose p_2-\frac{j}{2}}=c^{(p)}_{p^{n+1}-pi-j-1},$$  where $(*)$ is a consequence the induction hypothesis. This concludes the proof.\end{proof} 

\subsubsection{\textbf{The} $p$\textbf{-Cantor Sequence is Self-Similar}}\hfill\\
\noindent For $n\in\N$, let $\textbf{C}^{(p)}_n:= \left(c^{(p)}_i\right)_{0\le i <p^n}$ be the first $p^n$ entries.\\

\noindent The next useful property of the $p$-Cantor sequence is that for every $n\in\N$, $\textbf{C}^{(p)}_n$ is comprised of $p$ copies of $\textbf{C}^{(p)}_{n-1}$, the $j^\nth $ of which is multiplied by ${p_2\choose j/2}$.
\begin{lemma}\label{pcant_geom}
    Let $n\in\N$. Then, $$\textbf{C}^{(p)}_n=\bigoplus_{j=0}^{p-1} {p_2\choose \frac{j}{2}}\cdot \textbf{C}^{(p)}_{n-1}.$$In other words, for every $0\le j \le p-1$, \[\left(c^{(p)}_i\right)_{j\cdot p^{n-1} \le i \le (j+1)p^{n-1}-1}=\left(\begin{pmatrix}p_2\\\frac{j}{2}\end{pmatrix}\cdot c_i\right)_{0 \le i \le p^{n-1}-1}.\] 
\end{lemma}
\begin{proof}
    The proof proceeds by induction on $n$. The base case, when $n=1$, is immediate by the definition of the $p$-Cantor Sequence (Definition \ref{pcant}). \\

    \noindent Next, since $\textbf{C}^{(p)}$ is the fixed point of the $p$-morphism $\phi_p$ (from Definition \ref{pcant}), one has that for any $a,b\in\N$,  \begin{equation}\label{eqn: lem_3.5_1}
        \phi_p\left(\left(c^{(p)}_i\right)_{a\le i \le b}\right)=\left(c^{(p)}_i\right)_{pa\le i\le pb+p-1}.
    \end{equation} Therefore, if one assumes  that \begin{equation}c^{(p)}_{jp^n\le i \le (j+1)p^n-1}=\left(\begin{pmatrix}p_2\\\frac{j}{2}\end{pmatrix}\cdot c^{(p)}_i\right)_{0\le i \le p^{n-1}-1},\label{eqn: lem_3.5_2}\end{equation}  then \begin{align*}\left(c^{(p)}_i\right)_{jp^{n+1}\le i \le (j+1)p^{n+1}-1} &\substack{(\ref{eqn: lem_3.5_1})\\=}~ \phi_p\left(\left(c^{(p)}_i\right)_{jp^n\le i \le (j+1)p^n-1}\right)\\&\substack{(\ref{eqn: lem_3.5_2})\\=}~\phi_p\left(\begin{pmatrix}p_2\\\frac{j}{2}\end{pmatrix}\cdot (c^{(p)}_i)_{0\le i \le p^{n-1}-1}\right)=\begin{pmatrix}p_2\\\frac{j}{2}\end{pmatrix}\cdot (c^{(p)}_i)_{0\le i \le p^n-1}\mod p\end{align*}as required, where the final equality follows from the fact that $\phi_p$ respects multiplication (Definition \ref{pcant}).
\end{proof}
\begin{corollary}\label{cor: p^h=0}
    For every $h\in\N$ and every $p^h\le i\le 2p^h-1$, $c^{(p)}_{i}=0$.
\end{corollary}
\begin{proof}
    Since $p^h\le i\le 2p^h-1$, Lemma \ref{pcant_geom} implies that $c_{i}^{(p)}={p_2\choose \frac{1}{2}}\cdot c_i^{(p)}=0$ from equation (\ref{eqn: choose}).
\end{proof}
\subsection{The $p$-Singer Sequence}\hfill\\
\noindent It is seen in Section \ref{Sect: 8} that to study the number wall\footnote{See section \ref{Sect: 4}.} of the $p$-Cantor sequence, one must also study the so-called $p$-Singer\footnote{The second-named author thanks Fred Lunnon for the name suggestion.} sequence. Hence, the $p$-Singer sequence is now defined. 

\subsubsection{\textbf{The Definition of the }$p$\textbf{-Singer Sequence}}\hfill\\
\noindent Like the $p$-Cantor sequence, the $p$-Singer sequence is $p$-automatic.
\begin{definition}
    Define a $p$-morphism $\varphi:\Gamma_p^\infty\to\Gamma_p^\infty$ as \begin{equation}\label{p_sing_eqn}
        \varphi_p(n)_i=n\cdot 
            {p_2\choose i}\mod p.
    \end{equation}
    \noindent Additionally, define a $2p$-coding $\tau_p:\Sigma^\infty\to\Sigma^\infty$ as \begin{equation*}
        \tau_p(n)_i=n\cdot 
            {p_2+1\choose\frac{i}{2}}
        \mod p.
    \end{equation*}
    \noindent Then, define the $\mathbf{p}$-\textbf{Singer} sequence $\textbf{S}(p)=\left(s^{(p)}_i\right)_{i\ge0}:=\tau_p\left(\varphi_p^\infty(1)\right)$. Also, define the \textbf{pseudo}-$p$-\textbf{Singer} sequence as $\mathcal{S}(p)=(s_{p,i})_{i\ge0}=\varphi_p^\infty (1)$. That is, $\textbf{S}(p)=\tau_p(\mathcal{S}(p))$. 
\end{definition}
\noindent The pseudo-$p$-Singer sequence is only used briefly in Section \ref{sect: 3.3.3}, and hence, the similarities between the notations $\mathcal{S}(p)$ and $\textbf{S}(p)$ does not cause confusion.
\subsubsection{\textbf{An Example of the} $p$\textbf{-Singer Sequence}}\hfill\\
\noindent As the $p$-Singer sequence is significantly more complicated than the $p$-Cantor sequence, a detailed example is given.
\begin{example}
    Let $p=7$ and recall $\Gamma_7$ from equation (\ref{eqn: gamma_p}). Then the 7-morphism $\varphi_7:\Gamma_7^\infty\to\Gamma_7^\infty$ is given by:\begin{align*}
        &\varphi(0)=0000000& &\varphi(1)=1331000& &\varphi(2)=2662000& \\&\varphi(3)=3223000& &\varphi(4)=4554000&
        &\varphi(5)=5115000& \\&&&\varphi(6)=6446000&
    \end{align*}
    and the 14-coding $\tau_7:\Gamma_7^\infty\to\Gamma_7^\infty$ is given by \begin{align*}
        &\tau_7(0)=00000000000000& &\tau_7(1)=10406040100000& &\tau_7(2)=20105010200000&\\&\tau_7(3)=30504050300000& &\tau_7(4)=40203020400000& &\tau_7(5)=50602060500000&\\
        &&&\tau_7(6)=60301030600000.&
    \end{align*}
    Hence, the pseudo-7-Singer sequence $\mathcal{S}(7)=(s_{7,i})_{i\ge0}=\varphi_7^\infty(1)$ begins as \begin{align*}\varphi_7^0(1)=1,~~\bk\varphi^1_7(1)=133100\bk,~~\bk\varphi_7^2(1)=\rd{1}\bl{3}\rd{3}\bl{1}\rd{0}\bl{0}\rd{0}\gn{3}\bk{2}\gn{2}\bk{3}\gn{0}\bk{0}\gn{0}\rd{3}\bl{2}\rd{2}\bl{3}\rd{0}\bl{0}\rd{0}\gn{1}\bk{3}\gn{3}\bk{1}\gn{0}\bk{0}\gn{0}\rd{0}\bl{0}\rd{0}\bl{0}\rd{0}\bl{0}\rd{0}\gn{0}\bk{0}\gn{0}\bk{0}\gn{0}\bk{0}\gn{0}\rd{0}\bl{0}\rd{0}\bl{0}\rd{0}\bl{0}\rd{0}\bk\end{align*} Above, each pair of colours (red and blue, green and black) denotes the image of each element of $\varphi_7^1(1)$ under $\varphi_7$. For example, the first green and black section is the image of the first $3$ in $\varphi_7^1(1)$ under $\varphi_7$. The colours within each pair serve no other purpose than to distinguish one entry from another. The coloured text below indicates the image of the corresponding number from $\varphi_7^2(1)$ under $\tau_7$.  \[\rd{10406040100000}\bl{30504050300000}\rd{30504050300000}\bl{10406040100000}\rd{00000000000000}\bl{00000000000000}\]\vspace{-0.7cm}\[\rd{00000000000000}\gn{30504050300000}\bk{20105010200000}\gn{20105010200000}\bk{30504050300000}\gn{00000000000000}\]\vspace{-0.7cm}\[\bk{00000000000000}\gn{00000000000000}\rd{30504050300000}\bl{20105010200000}\rd{20105010200000}\bl{30504050300000}\]\vspace{-0.7cm}\[\rd{00000000000000}\bl{00000000000000}\rd{00000000000000}\gn{10406040100000}\bk{30504050300000}\gn{30504050300000}\]\vspace{-0.7cm}\[\bk{10406040100000}\gn{00000000000000}\bk{00000000000000}\gn{00000000000000}\rd{00000000000000}\bl{00000000000000}\]\vspace{-0.7cm}\[\rd{00000000000000}\bl{00000000000000}\rd{00000000000000}\bl{00000000000000}\rd{00000000000000}\gn{00000000000000}\]\vspace{-0.7cm}\[\bk{00000000000000}\gn{00000000000000}\bk{00000000000000}\gn{00000000000000}\bk{00000000000000}\gn{00000000000000}\]\vspace{-0.7cm}\[\rd{00000000000000}\bl{00000000000000}\rd{00000000000000}\bl{00000000000000}\rd{00000000000000}\bl{00000000000000}\]\vspace{-0.7cm}\[\rd{00000000000000}\bk\phantom{00000000000000}\phantom{00000000000000}\phantom{00000000000000}\phantom{00000000000000}\phantom{00000000000000}\]
    \noindent This is the first 243 digits of the 7-Singer sequence, $\textbf{S}(7)=(s^{(7)}_i)_{i\ge0}=\tau_7(\varphi_7^\infty(1))$.
\end{example}

\subsubsection{\textbf{Symmetries of the Pseudo-}$p$\textbf{-Singer Sequence}}\label{sect: 3.3.3}\hfill\\
\noindent The goal is to establish an analogue of Lemma \ref{sym} for the $p$-Singer sequence. Before this can be achieved, a similar lemma must be stated for the \textit{pseudo}-$p$-Singer sequence.
\begin{lemma}\label{lem: sym_sing_1}
    For any $k\in\N$ and any $0\le i \le p_2^{(k)}$, one has that $s_{p,i}=s_{p,p^{(k)}_2-i}$.
\end{lemma}
\begin{proof}
    \noindent This is done by induction. In the base case where $k=1$, one has $0\le i \le p_2$, and therefore, \[s_{p,i}={p_2\choose i}\mod p={p_2\choose p_2-i}\mod p=s_{p,p_2-i},\] as required.\\
    
    \noindent For the inductive step, assume that the values of $s_{p,i}$ for $0\le i \le p^{(k)}_2$ satisfy $s_{p,i}=s_{p,p^{(k)}_2-i}$. For $n,m\in\N$, define the finite word $B(n,m)$ of length $m+1\in\N$, where for $0\le i \le m$, $$B(n,m)_i=n\cdot{m\choose i}\mod p,$$ and define $\textbf{0}_n=0\dots0$ as the finite word of length $n$ comprised only of zeroes. Note that, for $n\in\Gamma_p$, one has that $\varphi_p(n)=B\left(n,p_2\right)*\textbf{0}_{p_2}$, where $*$ is from Definition \ref{def: concat}. Hence, applying $\varphi_p$ to the word $s^{(p)}_{0}*s^{(p)}_{1}*\dots*s^{(p)}_{p^{(k)}_2}$ yields \[B\left(s^{(p)}_{0},p_2\right)*\textbf{0}_{p_2}*B\left(s^{(p)}_{1},p_2\right)*\textbf{0}_{p_2}*\dots* B\left(s^{(p)}_{p^{(k)}_2}*p_2\right)*\textbf{0}_{p_2}. \] This has length $\frac{p^{k+1}+p}{2}$. Removing the final $p_2$ entries leaves a word of length $\frac{p^{k+1}+1}{2}$, which is symmetric by the inductive hypothesis and the symmetry of the binomial coefficients.
\end{proof}

\subsubsection{\textbf{The Symmetries of the} $p$\textbf{-Singer Sequence}}\hfill\\
\noindent The counterpart to Lemma \ref{sym} is now established.
\begin{lemma}\label{sym_sing}
    For any $k\in\N$, the $p$-Singer sequence $\textbf{S}(p)=\left(s_i^{(p)}\right)_{i\ge0}$ satisfies $s^{(p)}_i=s^{(p)}_{p^k+1-i}$ for all $0\le i \le p^k+1$.
\end{lemma}
\begin{proof}[Proof of Lemma \ref{sym_sing}]
\noindent The claim is proved by induction on $k\in\N$. The base case where $k=1$ is immediate from the symmetry of the binomial coefficients.\\

\noindent A similar argument to that of Lemma \ref{lem: sym_sing_1} is now used to show that $s^{(p)}_i=s^{(p)}_{p^k+1-i}$ for all $i$. For $n,m\in\N$, define the finite word $$S(n,m)=n\cdot{m\choose 0}*0*n\cdot{m\choose 1}*0*n\cdot{m\choose 2}*0*\dots*0*n\cdot{m\choose m}.$$ Above, each $n\cdot {m\choose i}$ term is taken modulo $p$. Note that $S(n,m)$ has length $2m+1$ and is symmetric. For $n\in\Gamma_p$, one has that $\tau_p(n)=S(n,\frac{p+1}{2})*\textbf{0}_{p-2}$, which has length $2p$. Hence, the symmetric (by Lemma \ref{lem: sym_sing_1}) finite word $s_{p,0}* s_{p,1}*\dots*s_{p,p^{(k)}_2}$ is mapped by $\tau_p$ to \[S\left(s_{p,0},\frac{p+1}{2}\right), \textbf{0}_{p-2}, S\left(s_{p,1},\frac{p+1}{2}\right), \textbf{0}_{p-2},\dots, S\left(s_{p,\frac{p^k-1}{2}},\frac{p+1}{2}\right), \textbf{0}_{p-2}. \] This word has length $p^{k+1}+p$. Hence, by Lemma \ref{lem: sym_sing_1}, removing the final $\textbf{0}_{p-2}$ returns a symmetric sequence of length $p^{k+1}+2$ given by $\left(s^{(p)}_i\right)$ for $0\le i \le p^{k+1}+1$.
\end{proof}

\subsection{The Relationship Between the $p$-Cantor and $p$-Singer Sequences}\label{Sect: 3.4}\hfill\\
\noindent The similarities in the construction of the $p$-Cantor and $p$-Singer sequences suggests a possible relationship between the two. In this subsection, this is made explicit through the language of \textit{formal Laurent series}. \\

\noindent To this end, define the field of formal Laurent series over $\F_p$ as \[\F_p((t^{-1}))=\left\{\sum_{i=-h}^\infty a_i t^{-i}: h\in\Z, a_i\in\F_p\right\}.\]
\noindent Next, let \[\Theta_p(t)=\sum_{i=0}^\infty c_i^{(p)} t^{-i}~~~\text{ and }~~~\Xi_p(t)=\sum_{i=0}^\infty s_i^{(p)}t^{-i}.\] This subsection establishes the following lemma: \begin{lemma} \label{lem: inv_LS}
    For any odd prime $p$, \begin{equation} \label{eqn: inv_LS}\Theta_p(t)\cdot \Xi_p(t)=1.\end{equation} That is, the $p$-Cantor and $p$-Singer sequences are `inverse' to one another from the point of view of Laurent series.
\end{lemma}
\noindent Lemma \ref{lem: inv_LS} is an immediate corollary of Lemmata (\ref{quad}) and (\ref{p-sing-inv}).
\subsubsection{\textbf{The Laurent Series} $\Theta_p(t)$}\hfill\\
\noindent To establish Lemma \ref{lem: inv_LS}, a closed form for $\Theta_p(t)$ is computed. 
\begin{lemma}\label{quad}
    For any odd prime $p$, one has $\Theta_p(t)=(1+t^{-2})^{-\frac{1}{2}}\in\F_p((t^{-1}))$. 
\end{lemma}
\begin{proof}

\noindent It suffices to prove that \begin{equation}\label{quad2}\Theta_p(t)=\Theta_p(t)^p \sum_{i=0}^{p_2} {p_2 \choose i}t^{-2i}.\end{equation} Indeed, rearranging this equation gives \begin{equation}\label{quad1}\Theta(t)^{p-1}=\left(\sum_{i=0}^{p_2} {p_2 \choose i}  t^{-2i}\right)^{-1}.\end{equation} The denominator of the right hand side of equation (\ref{quad1}) is the binomial expansion of $$(1+t^{-2})^{p_2}.$$ Thus, \[\Theta_p(t)=(1+t^{-2})^{-\frac{1}{2}},\] as required. All that remains is to prove (\ref{quad2}). To achieve this, one extends the definition of a $p$-morphism to apply to Laurent series in the following natural way: \\
    
    \noindent Let $\phi:\Gamma_p^\infty\to\Gamma_p^\infty$ be a $p$-morphism and let $\Theta(t)=\sum_{i=0}^\infty s_it^{-i}\in\F_p((t^{-1}))$ be a Laurent series whose coefficients are given by a sequence $(s_i)_{i\in\N}$ over $\F_p$. To reduce notation, $x\in\F_p$ is identified with $x\in\Gamma_p$. Then, define \begin{equation}\phi(\Theta(t))=\sum_{i=0}^\infty t^{-pi}\sum_{j=0}^{p-1} \phi(s_i)_j t^{-j}.\label{eqn: morph_to_LS}\end{equation}  Let $\phi_p$ be as in Definition \ref{pcant}. As $\mathbf{C}^{(p)}=\phi_p\left(\mathbf{C}^{(p)}\right)$ by definition, one has that $\Theta_p(t)=\phi_p(\Theta_p(t))$ and hence, \begin{align*}\Theta_p(t):=\sum_{j=0}^\infty c^{(p)}_j t^{-j}=&\sum_{j=0}^\infty \sum_{i=0}^p {p_2\choose \frac{i}{2}}c_j^{(p)}t^{-(pj+i)}\\=&\sum_{j=0}^{\infty}\sum_{i=0}^{p_2} {p_2\choose i}c_j^{(p)}t^{-(pj+2i)}\\
    =&\sum_{i=0}^{p_2} {p_2 \choose i} \sum_{j=0}^\infty c_j^{(p)} t^{-(pj+2i)}\\
    =&\sum_{i=0}^{p_2} {p_2 \choose i} \Theta_p(t^p) t^{-2i}.\end{align*} Finally, as $\Theta(t^p)=\Theta(t)^p$ holds for any Laurent series $\Theta(t)\in\F_p((t^{-1}))$, the proof is complete.
\end{proof}
\subsubsection{\textbf{The} $p$\textbf{-Singer Sequence is the Inverse to the }$p$\textbf{-Cantor Sequence}}\hfill\\

\noindent Recall that $\left(s_{p,i}\right)_{i\ge0}=\varphi^\infty_p(1)$ is the $p$-Singer sequence before the coding $\tau_p$ is applied.
\begin{lemma}\label{p-sing-inv}
   For any odd prime $p$, $\Xi_p(t)=\left(1+t^{-2}\right)^\frac{1}{2}\in\F_p((t^{-1}))$.
\end{lemma}
\begin{proof}

    \noindent Define $\overline{\Xi}(t):=\sum_{i=0}^\infty s_{p,i} t^{-i}\in\F_p((t^{-1}))$. As $\left(s_{p,i}\right)_{i\ge0}$ is the fixed point of $\varphi_p$, one has that \begin{align}
        \overline{\Xi}(t)&=\sum_{i=0}^\infty s_{p,i} t^{-pi}\sum_{j=0}^{p-1}{p_2\choose j}\cdot t^{-j}\nonumber\\
        &=\overline{\Xi}(t^p)\cdot(1+t^{-1})^{p_2}\nonumber\\
        &=\overline{\Xi}(t)^p\cdot(1+t^{-1})^{p_2}\nonumber.
    \end{align}
    \noindent Hence,\[\overline\Xi(t)=\left(1+t^{-1}\right)^{-\frac{1}{2}}.\]Applying the coding $\tau_p$ to $\overline{\Xi}(t)$ (in the sense of equation (\ref{eqn: morph_to_LS})) gives the Laurent series \begin{align}\Xi_p(t)&=\sum_{i=0}^\infty s_{p,i}\left(\sum_{j=0}^{p_2+1} {p_2+1\choose j}t^{-2j}\right)t^{-2pi}\\&=\overline{\Xi}(t^{2p})\cdot (1+t^{-2})^{p_2+1}\nonumber\\
    &= \left(\frac{(1+t^{-2})^{p+1}}{1+t^{-2p}}\right)^\frac{1}{2}\nonumber\\
    &= \left(1+t^{-2}\right)^\frac{1}{2}\cdot\left(\frac{(1+t^{-2})^{p}}{1+t^{-2p}}\right)^\frac{1}{2}.\end{align}
    \noindent Since $\Phi(t^p)=\Phi(t)^p$ for any $\Phi(t)\in\F_p((t^{-1}))$,  $\Xi_p(t)=\left(1+t^{-2}\right)^\frac{1}{2}$ as required.
\end{proof}

\section{Introduction to Number Walls}\label{Sect: 4}
\noindent Recall the definition of a number wall (Definition \ref{nw}). The goal of this section is to introduce the fundamental properties of number walls that are used in this work. More specific lemmata on number walls are established in Sections \ref{sect: 7} and \ref{Sect: 8}. 
\subsection{Properties of Number Walls}\hfill\\
\noindent Recall the Square Window Theorem (Theorem \ref{window}). As zero regions must always appear in squares, every such zero region is bordered by a square of nonzero entries. This motivates the following definition.
\begin{definition}
The entries of a number wall surrounding a window are referred to as the \textbf{inner frame}. The entries surrounding the inner frame form the \textbf{outer frame}. 
\end{definition}
\noindent The entries of the inner frame are extremely well structured:
\begin{theorem}[Lunnon, {{\cite[Page 11]{L01}}}]\label{ratio ratio}
The inner frame of a window with side length $l\ge1$ is comprised of 4 geometric sequences. These are along the top, left, right and bottom edges and they have ratios $P,Q,R$ and $S$ respectively with origins at the top left and bottom right. Furthermore, these ratios satisfy the relation \[\frac{PS}{QR}=(-1)^{l}.\]
\end{theorem}
\noindent See Figure \ref{Fig: basicwindow} for an example of a window of side length $l$. For $i\in\{0,\dots,l+1\}$, the inner and outer frames are labelled by the entries $A_i,B_i,C_i,D_i$ and $E_i,F_i,G_i,H_i$ respectively. The ratios of the geometric sequences comprising the inner frame are labelled as $P,Q,R$ and $S$. This notation is used frequently throughout this work.
\begin{figure}[H]
\centering
\includegraphics[width=0.65\textwidth]{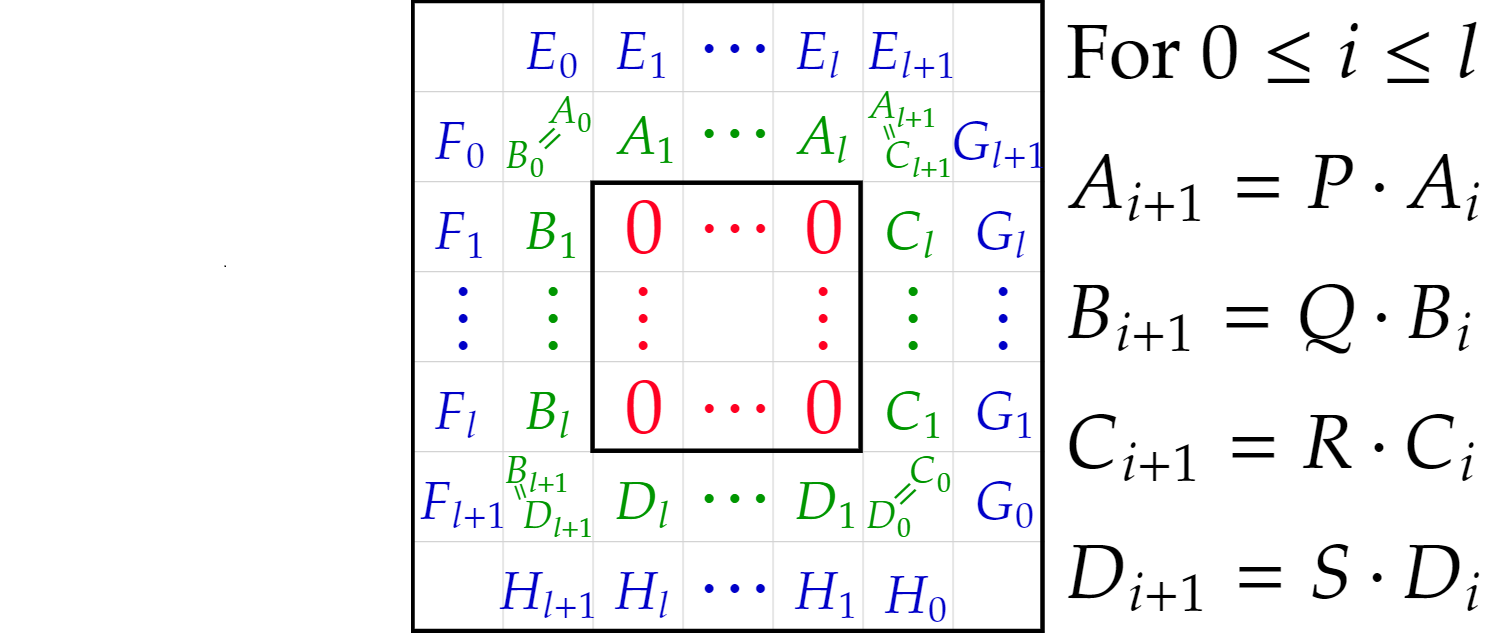}
\caption{Illustration of a window in a number wall. The window, inner frame and outer frame are in red, green and blue, respectively.}
\label{Fig: basicwindow}
\end{figure}
\subsection{The Frame Constraints}\hfill\\
\noindent Calculating a number wall from its definition is a computationally exhausting task. The following theorem gives a simple and far more efficient way to calculate the $m^\nth$ row using the previous rows.
\begin{theorem}[\textbf{Frame Constraints}, Adiceam, Nesharim and Lunnon, {{\cite[Corollary 3.6]{ANL}}}] \label{FC}
Given a doubly infinite sequence $\mathbf{S}=(s_{i})_{i\in\mathbb{Z}}$ over an integral domain $\id$, the number wall $W_{\id}(\mathbf{S}):=(W_\id(\mathbf{S})[m,n])_{n,m\in\mathbb{Z}}$ can be generated by a recurrence in row $m\in\mathbb{Z}$ in terms of the
previous rows. To state this precisely, recall the notation of Figure \ref{Fig: basicwindow} and let $(W_{m,n})_{m,n\in\Z}$ be the 2-dimensional sequence satisfying the initial conditions\begin{equation*}
    W_{m,n}=\begin{cases}
    \begin{aligned}
    &0 &&\textup{ if } m<-1 &\\
    &1 &&\textup{ if } m=-1;&\\
    &s_n &&\textup{ if } m=0;&\end{aligned}\end{cases}
\end{equation*} and the recurrence relations \begin{equation}
    W_{m,n}=\begin{cases}
    \begin{aligned}
    &0 &&\textup{ if } W_{m,n}\textup{ is within a window;}&\\
    &\frac{W_{m-1,n}^2-W_{m-1,n-1}W_{m-1n+1}}{W_{m-2,n}} &&\textup{ if } m>0\textup{ and }W_{m-2,n}\neq0;&\\
    &D_k=\frac{(-1)^{l\cdot k}B_kC_k}{A_k}& &\textup{ if } m>0\textup{ and }W_{m-2,n}=0=W_{m-1,n};&\\
    &H_k=\frac{\frac{QE_k}{A_k}+(-1)^{k}\frac{PF1_k}{B_k}-(-1)^k\frac{SG_k}{C_k}}{R/D_k}& &\textup{ if } m>0 \textup{ and } W_{n,m-2}=0\neq W_{n,m-1}.&
    \end{aligned}
    \end{cases}\label{eqn: FC}
\end{equation}
\noindent Above, the value of $k$ above is found in the natural way from the value of $m,n$ and the side length $l$. Additionally, ``$W_{m,n}$ is within a window" means that if $W_{m,n}\neq0$ then the sequence $(W_{m,n})_{m,n\in\Z}$ would have a non-square zero portion. \\

\noindent Then, for all $m,n\in\Z$, $W_{m,n}=W_\id (\textbf{S})[m,n]$.
\end{theorem}
 \noindent The final three equations in (\ref{eqn: FC}) are known as the \hypertarget{FC1}{\textbf{First}}, \hypertarget{FC2}{\textbf{Second}} and \hypertarget{FC3}{\textbf{Third}}\textbf{ Frame Constraint Equations}. These allow the number wall of a sequence to be considered independently of its original definition in terms of Toeplitz matrices. Throughout this paper, the three Frame Constraints are frequently referred to as \hyperlink{FC1}{FC1}, \hyperlink{FC2}{FC2} and \hyperlink{FC3}{FC3} respectively. 

\subsection{Finite Number Walls}\label{Sect: 3.3}\hfill\\
\noindent Whilst a number wall is defined for doubly infinite sequences, the definition can easily be extended to allow for finite sequences. Indeed, given a finite sequence $\mathbf{S}$ over an integral domain $\id$, the \textbf{finite number wall} $W_\id(\mathbf{S})$ is a two-dimensional infinite array defined as follows: \begin{equation*}
    W_\id(\mathbf{S})[m,n]=\begin{cases}\det(T_\textbf{S}(m,n)) &\textup{ if } m\ge0 \text{ and every entry of }T_\textbf{S}(m,n)\text{ is defined},\\
     \text{an unknown variable }&\textup{ if } m\ge0 \text{ and not all entries of }T_\textbf{S}(m,n)\text{ are defined},\\
    1 & \textup{ if } m=-1,\\
    0 & \textup{ if } m<-1. \end{cases}
\end{equation*} An example of a finite number wall is illustrated below. 

\begin{figure}[H]
    \centering
    \includegraphics[width=0.65\linewidth]{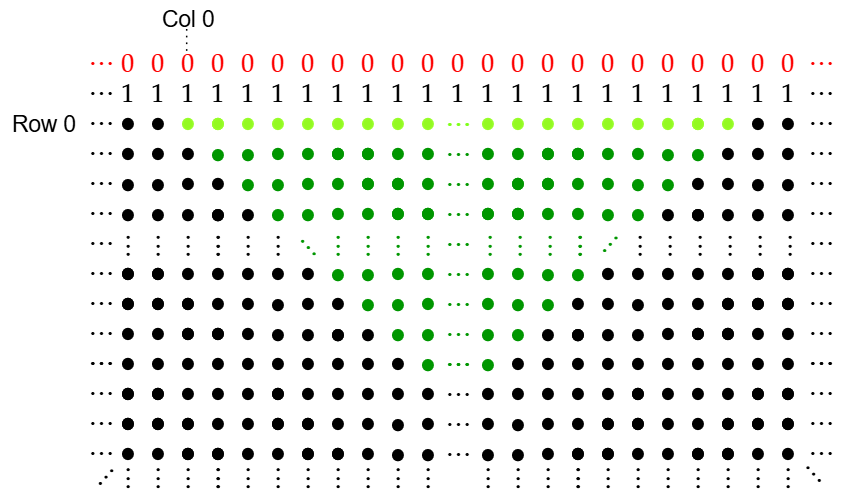}
    \caption{Each dot represents an entry in a number wall. The finite sequence (light green dots) that generates the finite number wall (whole picture) is on row zero. Each dot represents an entry in the finite number wall, with the dark green dots being those that are known explicitly and the black dots being those that are still variables.}
    \label{fin_nw}
\end{figure}  
\noindent It is elementary to see that a finite sequence of length $r$ defines the portion of a finite number wall in the shape of an isosceles triangle with the maximal row index being equal to $\left\lfloor\frac{r-1}{2}\right\rfloor$. Therefore, it may seem paradoxical that a \textit{finite} number wall is actually an \textit{infinite} two-dimensional array instead a finite triangular array. However, defining it this way is advantageous later on when one wishes to rotate, translate, reflect and composite finite number walls. In practice though, only the known entries of $W_\id(\mathbf{S})$ are of interest. Therefore, the unknown entries are often ignored in figures, as below. \begin{figure}[H]
    \centering
    \includegraphics[width=0.35\linewidth]{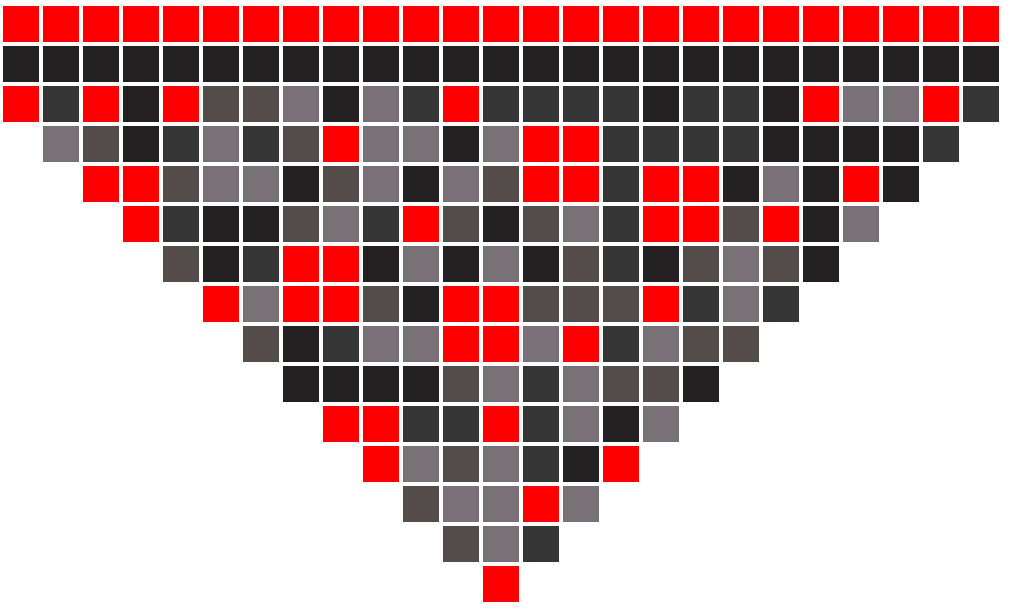}
    \caption{The finite number wall of a sequence over $\F_5$. The zero entries are in red, whereas the numbers 1-4 are illustrated in progressively darker shades of gray. The top row has index $-2$. }
\end{figure}

\subsection{The Profile of a Number Wall}
\noindent In all known applications of number walls, the zero entries carry the most important information. That is, it only matters whether an entry is zero or nonzero, with the precise value in the latter case being irrelevant. This motivates the following definition: 

\begin{definition}\label{profile_def}
    Let $W_\id(\textbf{S})$ be the number wall of a sequence $\textbf{S}$ over an integral domain $\id$. The \textbf{profile} of $W_\id(\textbf{S})$, denoted $\chi(W_\id(\mathbf{S}))$, is a two dimensional array with the same shape as $W_\id(\textbf{S})$ and with the entry in row $m$ and column $n$ being given by\[\chi(W_\id(\mathbf{S}))[m,n]=\begin{cases}
        0&\text{ if }W(\mathbf{S})[m,n]=0\\ X&\text{ otherwise.}
    \end{cases}\] Above, $X$ should be interpreted as an arbitrarily chosen symbol. 
\begin{figure}[H]
    \centering
    \includegraphics[width=0.5\linewidth]{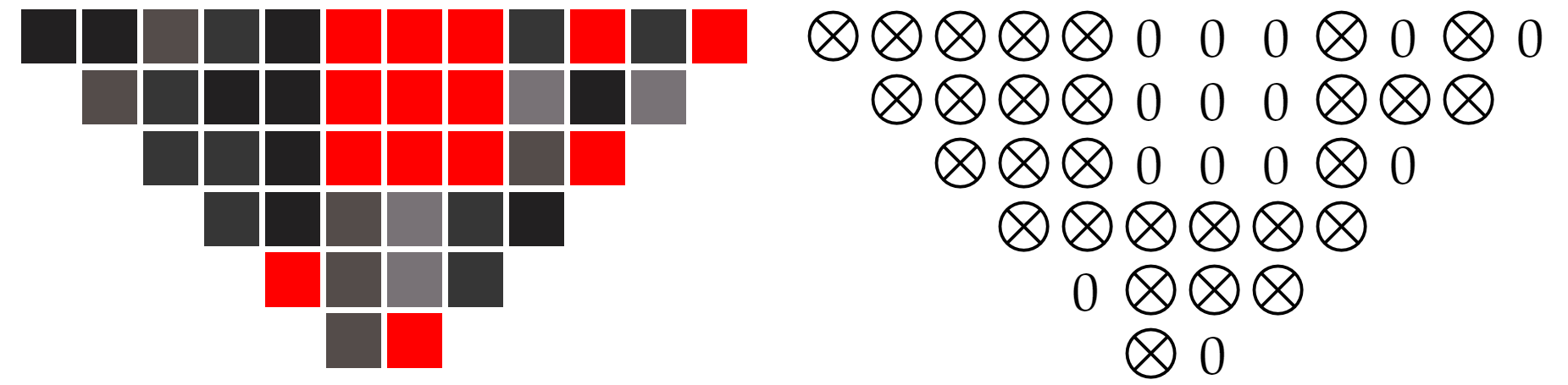}
    \caption{\textbf{Left}: The number wall $W_5(\textbf{S})$ for the sequence $\textbf{S}=(1,1,3,2,1,0,0,0,2,0,2,0)\in\F_5^{12}$, where the zeroes are coloured in red. \textbf{Right}: The profile of $W_5(\textbf{S})$.}
\end{figure}
\end{definition}

\section{The Number Wall of $\mathbf{C}^{(p)}$ as a Two-Dimensional Automatic Sequence}

\noindent  The goal of this section is to define the two-dimensional morphism that generates $\chi\left(W_p\left(\mathbf{C}^{(p)}\right)\right)$. To this end, one must first clarify precisely what part of $\chi\left(W_p\left(\mathbf{C}^{(p)}\right)\right)$ is claimed to be automatic.\\

\noindent For pictures of $W_p\left(\textbf{C}^{(p)}\right)$ itself, see Appendix \hyperref[Sect: Appendix_B]{B}.
\subsection{The Right-Side of $W_p\left(\mathbf{C}^{(p)}\right)$.}\hfill\\
\noindent Recall that the $p$-Cantor sequence $\textbf{C}^{(p)}$ is extended to a doubly infinite sequence $\textbf{C}^{(p,L)}$ by defining $c^{(p)}_{i}=0$ for all $i<0$. Therefore, by the Square Window Theorem (Theorem \ref{window}), $W_p\left(\mathbf{C}^{(p,L)}\right)[m,n]=0$ for all $m\ge0$ and $n<0$. That is, $W_p\left(\mathbf{C}^{(p,L)}\right)$ has an infinite window containing all columns of negative index. Hence, one focuses only on the two-dimensional sequence $\left(W_p\left(\mathbf{C}^{(p,L)}\right)[m,n]\right)_{m,n\ge0}$. 
 \begin{figure}[H]
    \centering
    \includegraphics[width=0.8\linewidth]{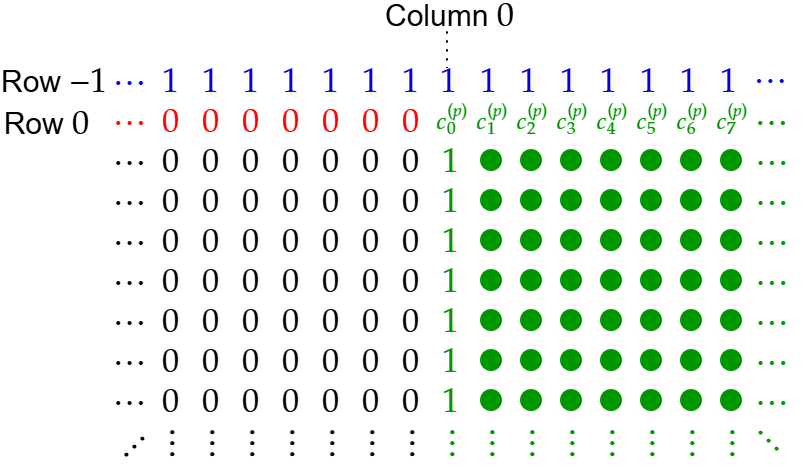}
    \caption{An illustration of $W_p(\mathbf{C}^{(p,L)})$. Here, the dots represent an arbitrary element of $\F_p$. The profile of the green portion is what is shown to be automatic. Row $-1$ is coloured in blue, and $W_p(\mathbf{C}^{(p,L)})[0,n]$ for $n<0$ is in red.  }
\end{figure}

\noindent As seen in Section \ref{Sect: 1.4}, the fractal generated by a number wall of a sequence $\textbf{S}$ depends only on the profile of the number wall, $\chi(W_\K(\textbf{S}))$. Hence, the two-dimensional sequence $$\chi\left(\left(W_p\left(\mathbf{C}^{(p,L)}\right)[m,n]\right)_{m,n\ge0}\right)$$ is the focus of this section. This is shown to be the image of a $[p,p]$-morphism under a $[1,1]$-coding.

\subsection{The Morphism that Generates The Right-Side of $W_p\left(\mathbf{C}^{(p,L)}\right)$}\label{subsect: morphism}\hfill\\
\noindent This subsection details the finite alphabet $\A$, the $[p,p]$-morphism $\Phi_p:\A_2^\infty\to\A_2^\infty$ and the $[1,1]$-coding $\Pi:\A_2^\infty\to\{0,X\}_2^\infty$ that generates $\left(\chi\left(W_p\left(\mathbf{C}^{(p,L)}\right)[m,n]\right)\right)_{m,n\ge0}$. The alphabet $\A$ is defined first.
\subsubsection{\textbf{The Finite Alphabet} $\mathcal{A}$}\label{sec: 5}\hfill\\

\noindent Regardless of the choice of the prime $p$, the alphabet $\A$ consists of $12$ letters - $8$ of which are labelled by the cardinal or intercardinal directions: \begin{equation*}
        \mathcal{A}:=\{A,~ B,~ F,~ 0,~ E_N,~ E_E,~ E_S,~ E_W,~ C_{NE},~ C_{SE},~ C_{SW},~ C_{NW}\}.
    \end{equation*}
\noindent To break the definition of $\Phi_p$ and $\Pi$ up into manageable words, the alphabet $\A$ is split into four sub-alphabets: \begin{align*}
    &\A_{units}:=\{A,B\},& &\A_{zeroes}=\{F,0\},& \\&\A_{edges}=\{E_N,E_E,E_S,E_W\},& &\A_{corners}=\{C_{NE},C_{SE},C_{SW},C_{NW}\}.&
\end{align*} 

\noindent Above, $E$ and $C$ are intended to stand for `edge' and `corner', respectively. For example, $E_N$ and $C_{SW}$ should be read as `northern edge' and `south-western corner'. These labels are explained in sections \ref{sect: 5.zeroes}, \ref{sect: 5.edges} and \ref{sect: 5.corners}. \\

\noindent The image of each letter under $\Phi_p$ is both defined explicitly and illustrated. With the exception of $A,B$ and $F$, the images of all the above letters under $\Phi_p$ are very simple. For those three letters, it is recommended to look at the picture \textit{before} reading the definition.\\

\noindent For the remainder of this section, $0\le m,n\le p-1$ are natural numbers and let $\Phi_p(\cdot)_{m,n}$ be the entry in the $m^\nth$ row and the $n^\nth$ column of the image of the input under $\Phi_p$.
\subsubsection{\textbf{The Image of} $\Phi_p$ \textbf{on} $\A_{units}$}\label{sect: 5.zeroes}\hfill\\
\noindent The $[p,p]$-morphism $\Phi_p$ acts on the letters $A$ and $B$ by
\begin{align*}
    &\Phi_p(A)_{m,n}=\begin{cases}A &\text{ if } m\equiv 0 \equiv n \mod 2\\
    0 &\text{ if } m\equiv 0\not\equiv  n\mod 2\\
    F &\text{ if } m\not\equiv 0\equiv  n\mod 2\\
    B &\text{ if } m\equiv 1\equiv  n\mod 2\\\end{cases}& &\Phi_p(B)_{m,n}=\begin{cases}
    B &\text{ if } m\equiv 0 \equiv n \mod 2\\
    F &\text{ if } m\equiv 0\not\equiv  n\mod 2\\
    0 &\text{ if } m\not\equiv 0\equiv  n\mod 2\\
    A &\text{ if } m\equiv 1\equiv  n\mod 2\\\end{cases}&,
\end{align*}
\noindent that is, \begin{align*}
    &\Phi_p(A):=~\begin{matrix}
        A&0&A&0&A &\cdots& A&0&A\\
        F&B&F&B&F&\cdots &F&B&F\\
        A&0&A&0&A &\cdots& A&0&A\\
        F&B&F&B&F&\cdots &F&B&F\\
        A&0&A&0&A &\cdots& A&0&A\\
        \vdots &\vdots&\vdots&\vdots&\vdots&&\vdots&\vdots&\vdots\\
        A&0&A&0&A &\cdots& A&0&A\\
        F&B&F&B&F&\cdots &F&B&F\\
        A&0&A&0&A &\cdots& A&0&A\\
    \end{matrix}& &\Phi_p(B):=~\begin{matrix}
    B&F&B&F&B&\cdots&B&F&B\\
    0&A&0&A&0&\cdots &0&A&0\\
    B&F&B&F&B&\cdots&B&F&B\\
    0&A&0&A&0&\cdots &0&A&0\\
    B&F&B&F&B&\cdots&B&F&B\\
    \vdots &\vdots&\vdots&\vdots&\vdots&&\vdots&\vdots&\vdots\\
    B&F&B&F&B&\cdots&B&F&B\\
    0&A&0&A&0&\cdots &0&A&0\\
    B&F&B&F&B&\cdots&B&F&B\\
    \end{matrix}.
\end{align*}
\subsubsection{\textbf{The Image of} $\Phi_p$ \textbf{on} $\A_{zeroes}$}\hfill\\
\noindent Define $\Phi_p(F)$ and $\Phi_p(0)$ as\begin{align*}
    &\Phi_p(F)_{m,n}=\begin{cases}
        C_{NW} &\text{ if }~~~m=n=0,\\
        C_{NE} &\text{ if }~~~m=0~\text{ and }~n=p-1,\\
        C_{SW} &\text{ if }~~~m=p-1~\text{ and }~n=0,\\
        C_{SE} &\text{ if }~~~m=n=p-1,\\
        E_N &\text{ if }~~~m=0~\text{ and }~0\neq n \neq p-1,\\
        E_E &\text{ if }~~~0\neq m\neq p-1~\text{ and }~n=p-1,\\
        E_S &\text{ if }~~~m=p-1~\text{ and }~0\neq n \neq p-1,\\
        E_W &\text{ if }~~~0\neq m\neq p-1~\text{ and }~n=0,\\
        0 &\text{ otherwise},
    \end{cases}& &\Phi_p(0)_{m,n}=0 \text{   for all }m,n,&
\end{align*}\noindent  that is,\begin{align*}
    &\Phi_p(F):=~\begin{matrix}
        C_{NW}&E_N&E_N&\cdots&E_N&E_N&C_{NE}\\
        E_W&0&0&\cdots&0&0&E_E\\
        E_W&0&0&\cdots&0&0&E_E\\
        \vdots&\vdots&\vdots&&\vdots&\vdots&\vdots&\\
        E_W&0&0&\cdots&0&0&E_E\\
        E_W&0&0&\cdots&0&0&E_E\\
        C_{SW}&E_S&E_S&\cdots&E_S&E_S&C_{SE}
    \end{matrix}& &\Phi_p(0):=~\begin{matrix}
        0&0&0&\cdots&0&0&0\\
        0&0&0&\cdots&0&0&0\\
        0&0&0&\cdots&0&0&0\\
        \vdots&\vdots&\vdots&&\vdots&\vdots&\vdots&\\
        0&0&0&\cdots&0&0&0\\
        0&0&0&\cdots&0&0&0\\
        0&0&0&\cdots&0&0&0\\
    \end{matrix}&.
\end{align*}
\begin{remark}
    The letter $F$ was chosen to stand for the word ``Frame", as the image of $F$ under $\Phi_p$ is a $p-2 \times p-2$ square of zeroes \textit{framed} by the letters labelled with the corresponding cardinal or intercardinal directions. Similarly, the letters in $\mathcal{A}_{edges}$ and $\mathcal{A}_{corners}$ make up the edges and corners of $\Phi_p(F)$.
\end{remark}  
\subsubsection{\textbf{The Image of} $\Phi_p$\textbf{ on} $\A_{edges}$}\label{sect: 5.edges}\hfill\\
\noindent The $[p,p]$-morphism $\Phi_p$ acts on the elements of $\A_{edges}$ as,  \begin{align*}
    &\Phi_p(E_N)_{m,n}=\begin{cases}E_N &\text{ if } m= 0 \\
    0 &\text{ otherwise } 
\end{cases}& &\Phi_p(E_E)_{m,n}=\begin{cases}
    E_E &\text{ if } n=p-1\\
    0 &\text{ otherwise } \end{cases}&\\&\Phi_p(E_S)_{m,n}=\begin{cases}E_S &\text{ if } m= p-1 \\
    0 &\text{ otherwise } 
\end{cases}& &\Phi_p(E_W)_{m,n}=\begin{cases}
    E_W &\text{ if } n=0\\
    0 &\text{ otherwise } \end{cases}&,
\end{align*}
\noindent that is,
\begin{align*}
    &\begin{matrix}
        &E_N&E_N&E_N&\cdots&E_N&E_N&E_N\\
        &0&0&0&\cdots&0&0&0\\
        &0&0&0&\cdots&0&0&0\\
       \Phi_p(E_N):=&\vdots&\vdots&\vdots&&\vdots&\vdots&\vdots\\
        &0&0&0&\cdots&0&0&0\\
        &0&0&0&\cdots&0&0&0\\
        &0&0&0&\cdots&0&0&0\\
        \\
        &0&0&0&\cdots&0&0&0\\
        &0&0&0&\cdots&0&0&0\\
        &0&0&0&\cdots&0&0&0\\
        \Phi_p(E_S):=&\vdots&\vdots&\vdots&&\vdots&\vdots&\vdots\\
        &0&0&0&\cdots&0&0&0\\
        &0&0&0&\cdots&0&0&0\\
        &E_S&E_S&E_S&\cdots&E_S&E_S&E_S
    \end{matrix}&  &\begin{matrix}
        &0&0&0&\cdots&0&0&E_E\\
        &0&0&0&\cdots&0&0&E_E\\
        &0&0&0&\cdots&0&0&0_E\\
        \Phi_p(E_E):=&\vdots&\vdots&\vdots&&\vdots&\vdots&\vdots\\
        &0&0&0&\cdots&0&0&E_E\\
        &0&0&0&\cdots&0&0&E_E\\
        &0&0&0&\cdots&0&0&E_E\\
        \\
        &E_W&0&0&\cdots&0&0&0\\
        &E_W&0&0&\cdots&0&0&0\\
        &E_W&0&0&\cdots&0&0&0\\
        \Phi_p(E_W):=&\vdots&\vdots&\vdots&&\vdots&\vdots&\vdots\\
        &E_0&0&0&\cdots&0&0&0\\
        &E_W&0&0&\cdots&0&0&0\\
        &E_W&0&0&\cdots&0&0&0
    \end{matrix}&
\end{align*}
\subsubsection{\textbf{The Image of} $\Phi_p$ \textbf{on} $\A_{corners}$}\label{sect: 5.corners}\hfill\\
\noindent The $[p,p]$-morphism $\Phi_p$ acts on the elements of $\A_{corners}$ as  \begin{align*}
    &\Phi_p(C_{NE})_{m,n}=\begin{cases}C_{NE} &\text{ if } m= 0 ~\text{ and }~n=p-1 \\
    E_N &\text{ if } m=0 ~\text{ and }~n\neq p-1\\
    E_E &\text{ if }m\neq0~\text{ and }~n=p-1\\
    0 &\text{ otherwise } 
\end{cases}& &\Phi_p(C_{SE})_{m,n}=\begin{cases}C_{SE} &\text{ if } m=n= p-1 \\
    E_S &\text{ if } m=p-1 \neq n\\
    E_E &\text{ if }m\neq p-1=n~\\
    0 &\text{ otherwise } 
\end{cases}&\\&\Phi_p(C_{SW})_{m,n}=\begin{cases}C_{SW} &\text{ if } m= p-1~\text{ and }~n=0 \\
    E_S &\text{ if } m=p-1 ~\text{ and } ~n\neq0\\
    E_E &\text{ if }m\neq p-1~\text{ and }~n=0\\
    0 &\text{ otherwise } 
\end{cases}& &\Phi_p(C_{NW})_{m,n}=\begin{cases}C_{NW} &\text{ if } m=n= 0 \\
    E_N &\text{ if } m=0 \neq n\\
    E_W &\text{ if }m\neq 0=n~\\
    0 &\text{ otherwise, } 
\end{cases}&
\end{align*}
\noindent that is,
\begin{align*}
    & \Phi_p(C_{NW}):=\hspace{-0.2cm}\begin{matrix}
        C_{NW}&E_N&E_N&\cdots&E_N&E_N&E_N\\
        E_W&0&0&\cdots&0&0&0\\
        E_W&0&0&\cdots&0&0&0\\
       \vdots&\vdots&\vdots&&\vdots&\vdots&\vdots\\
        E_W&0&0&\cdots&0&0&0\\
        E_W&0&0&\cdots&0&0&0\\
        E_W&0&0&\cdots&0&0&0
    \end{matrix}\Phi_p(C_{NE}):=\hspace{-0.2cm}\begin{matrix}
        E_N&E_N&E_N&\cdots&E_N&E_N&C_{NE}\\
        0&0&0&\cdots&0&0&E_E\\
        0&0&0&\cdots&0&0&E_E\\
       \vdots&\vdots&\vdots&&\vdots&\vdots&\vdots\\
        0&0&0&\cdots&0&0&E_E\\
        0&0&0&\cdots&0&0&E_E\\
        0&0&0&\cdots&0&0&E_E
    \end{matrix}&
\end{align*}
\begin{align*}
    &\Phi_p(C_{SW}):=\begin{matrix}
        E_W&0&0&\cdots&0&0&0\\
        E_W&0&0&\cdots&0&0&0\\
        E_W&0&0&\cdots&0&0&0\\
        \vdots&\vdots&\vdots&&\vdots&\vdots&\vdots\\
        E_W&0&0&\cdots&0&0&0\\
        E_W&0&0&\cdots&0&0&0\\
        C_{SW}&E_S&E_S&\cdots&E_S&E_S&E_S
    \end{matrix}~\Phi_p(C_{SE}):=\begin{matrix}
        0&0&0&\cdots&0&0&E_E\\
        0&0&0&\cdots&0&0&E_E\\
        0&0&0&\cdots&0&0&E_E\\
        \vdots&\vdots&\vdots&&\vdots&\vdots&\vdots\\
        0&0&0&\cdots&0&0&E_E\\
        0&0&0&\cdots&0&0&E_E\\
        E_S&E_S&E_S&\cdots&E_S&E_S&C_{SE}
    \end{matrix}&
\end{align*}

\subsubsection{\textbf{The} $[1,1]$\textbf{-coding }$\Pi$ \textbf{on} $\A$}\hfill\\
\noindent Let $L\in\A$. Then, the $[1,1]$-coding $\Pi:\mathcal{A}^\infty_2\to \{0,X\}_2^\infty$ is defined by,  \begin{align*}
    &\Pi(L)_{m,n}=\begin{cases}
    0 &\text{ if } L = 0\\
    X &\text{ otherwise } \end{cases}& 
\end{align*}
\subsubsection{\textbf{Statement of Theorem \ref{thm: morphism}}}\hfill\\
\noindent All the necessary material has been covered to allow for a full statement of Main Result 1.
\begin{theorem}\label{thm: morphism}
    Let $\A$, $\Phi_p:\mathcal{A}^\infty_2\to\A_2^\infty$ and $\Pi:\A_2^\infty\to\{0,X\}_2^\infty$ be as described above. Then, \[\chi\left(\left(W_p\left(\mathbf{C}^{(p,L)}\right)[m,n]\right)_{m,n\in\N}\right)=\Pi(\Phi_p^\infty(A)).\]
\end{theorem}
\noindent The proof of Theorem \ref{thm: morphism} is completed in Section \ref{Sect: 10}. The main goal of Sections \ref{sect: 7}-\ref{Sect: 9} is to provide the machinery needed to undertake this proof. In Section \ref{Sect: 6}, Theorem \ref{thm: frac_nw} is deduced from Theorem \ref{thm: morphism}.

\section{Proof of Theorem \ref{thm: frac_nw}}\label{Sect: 6}

\noindent Recall the fractal  $\mathcal{F}\left(\textbf{C}^{(p)},\mathbb{F}_p\right)$, defined in equation (\ref{eqn: fractal_def}). To prove Theorem \ref{thm: frac_nw}, two fractals $\mathcal{F}_{0,\infty}, \mathcal{F}_{F,\infty}\subseteq [0,1]^2$ are constructed, which satisfy the following properties:
\begin{enumerate}
    \hypertarget{link: frac_1}{\item} $\mathcal{F}_{0,\infty}\subseteq \mathcal{F}\left(\textbf{C}^{(p)},\mathbb{F}_p\right)\subseteq \mathcal{F}_{F,\infty}$ and
    \hypertarget{link: frac_2}{\item} $\frac{\log\left(\frac{p^2+1}{2}\right)}{\log p}\le \dim_H(\mathcal{F}_{0,\infty})$ and $\dim_H(\mathcal{F}_{F,\infty})\le\frac{\log\left(\frac{p^2+1}{2}\right)}{\log p}$.
\end{enumerate}

\noindent Claim \hyperlink{link: frac_2}{2} is established first.
\subsection{The Fractal $\mathcal{F}_{0,\infty}$}\hfill\\
\noindent First, a simplified version of the morphism $\Phi_p$ is defined. Let $\Phi_{0,p}:\{0,A\}_2^\infty\rightarrow \{0,A\}^{\infty}_2$ be the $[p,p]$-morphism defined by 
\begin{align}
    \Phi_{0,p}(\mathbf{0}_{p\times p})=\mathbf{0}_{p\times p}\,&& \Phi_{0,p}(A)=\begin{matrix}
        A&0&A&\dots&A&0&A\\
        0&A&0&\dots&0&A&0\\
        A&0&A&\dots&A&0&A\\
        \vdots&\vdots&\vdots&&\vdots&\vdots&\vdots\\
        A&0&A&\dots&A&0&A\\
        0&A&0&\dots&0&A&0\\
        A&0&A&\dots&A&0&A
    \end{matrix}.
\end{align}
To bound $\dim_H\mathcal{F}_{0,\infty}$ from below, \cite[Mass Distribution Principle 4.2]{Fal} is utilized. For $U\subseteq [0,1]^2$, define the \textbf{diameter} of $U$ by $\vert U\vert:=\sup\{\Vert \mathbf{x}-\mathbf{y}\Vert_2:\mathbf{x},\mathbf{y}\in U\}$.
\begin{proposition}{\cite[Mass Distribution Principle 4.2]{Fal}}
\label{prop:MassDist}
    Let $X$ be a set, and let $\mu$ be a mass distribution on $X$. Suppose that there exist $C>0$ and $\varepsilon>0$, such $\mu(U)\leq C\vert U\vert^s$, for any set $U$ with $\vert U\vert<\varepsilon$. Then, $\dim_HX\geq s$.
\end{proposition}
\noindent To apply Proposition \ref{prop:MassDist}, the number of non-zero entries in $\Phi_{0,p}^k(A)$ is computed.
\begin{lemma}
\label{lem:phi_0,p^k[m,n]=0}
    Let $k\in \mathbb{N}$, let $m=\sum_{j=0}^{k-1}m_jp^j$, and let $n=\sum_{j=0}^{k-1}n_jp^j$, where $m_j,n_j\in \{0,1,\dots,p-1\}$. Then, $\Phi_{0,p}^k(A)[m,n]\neq 0$ if and only if $m_j\equiv n_j\mod 2$ for every $j=0,1,\dots,k-1$.
\end{lemma}
\noindent To prove Lemma \ref{lem:phi_0,p^k[m,n]=0}, the following definition is made.
\begin{definition}
    Let $n=\sum_{j=0}^{\lfloor \log_p n\rfloor}n_jp^j$, where $n_j\in \{0,1,\dots,p-1\}$ for every $j\in\{0,\dots \lfloor\log_p n\rfloor\}$. Define the \textbf{prefix} of $n$ by $\operatorname{pref}(n)=\sum_{j=0}^{\lfloor \log_p n\rfloor-1}n_jp^j$.
\end{definition}
\begin{proof}[Proof of Lemma \ref{lem:phi_0,p^k[m,n]=0}]
    The proof proceeds by induction. The case $k=1$ follows from the definition of $\Phi_{0,p}(A)$. Assume that the claim holds for $k$. Let $m,n\leq p^{k+1}$ be of the form $m=\sum_{j=0}^km_jp^j$ and $n=\sum_{j=0}^kn_jp^j$. Note that 
    \begin{equation}
    \label{eqn:phi_0,p^k}
        \Phi_{0,p}^{k+1}(A)=\begin{matrix}
        \Phi_{0,p}^k(A)&\Phi_{0,p}^{k}(0)&\Phi_{0,p}^k(A)&\dots&\Phi_{0,p}^k(A)&\Phi_{0,p}^k(0)&\Phi_{0,p}^k(A)\\
        \Phi_{0,p}^k(0)&\Phi_{0,p}^k(A)&\Phi_{0,p}^k(0)&\dots&\Phi_{0,p}^k(0)&\Phi_{0,p}^k(A)&\Phi_{0,p}^k(0)\\
        \Phi_{0,p}^k(A)&\Phi_{0,p}^k(0)&\Phi_{0,p}^k(A)&\dots&\Phi_{0,p}^k(A)&\Phi_{0,p}^k(0)&\Phi_{0,p}^k(A)\\
        \vdots&\vdots&\vdots&&\vdots&\vdots&\vdots\\
        \Phi_{0,p}^k(A)&\Phi_{0,p}^k(0)&\Phi_{0,p}^k(A)&\dots&\Phi_{0,p}^k(A)&\Phi_{0,p}^k(0)&\Phi_{0,p}^k(A)\\
        \Phi_{0,p}^k(0)&\Phi_{0,p}^k(A)&\Phi_{0,p}^k(0)&\dots&\Phi_{0,p}^k(0)&\Phi_{0,p}^k(A)&\Phi_{0,p}^k(0)\\
        \Phi_{0,p}^k(A)&\Phi_{0,p}^k(0)&\Phi_{0,p}^k(A)&\dots&\Phi_{0,p}^k(A)&\Phi_{0,p}^k(0)&\Phi_{0,p}^k(A)
    \end{matrix}.
    \end{equation}
    As a consequence, if $m_k\neq n_k\mod 2$, then, by \eqref{eqn:phi_0,p^k}, $\Phi_{0,p}^{k+1}(A)[m,n]=0$. Otherwise, if $m_k\equiv n_k\mod 2$, then, by \eqref{eqn:phi_0,p^k} and the induction hypothesis, $\Phi_{0,p}^{k+1}(A)[m,n]=\Phi_{0,p}^k(A)[\operatorname{pref}(m),\operatorname{pref}(n)]$. By the induction hypothesis, $\Phi_{0,p}^k(A)[\operatorname{pref}(m),\operatorname{pref}(n)]\neq 0$ if and only if $m_i\equiv n_i\mod 2$ for every $i=0,\dots,k-1$. As a consequence, $\Phi_{0,p}^{k+1}(A)[m,n]\neq 0$ if and only if $m_i\equiv n_i\mod 2$ for every $i=0,\dots,k$.
\end{proof}
\noindent By Lemma \ref{lem:phi_0,p^k[m,n]=0}, for every $k\in \mathbb{N}$,
$$\mathcal{I}_k:=\bigcup_{\substack{m,n\in\N\\\Phi_{0,p}^k(A)[m,n]\neq 0}}\left[\frac{m}{p^k},\frac{m+1}{p^k}\right)\times \left[\frac{n}{p^k},\frac{n+1}{p^k}\right)=\bigcup_{(m,n)\in \mathcal{N}_k}\left[\frac{m}{p^k},\frac{m+1}{p^k}\right)\times \left[\frac{n}{p^k},\frac{n+1}{p^k}\right),$$
where
\begin{align*}\mathcal{N}_k&=\left\{(m,n)\in \mathbb{N}^2:m=\sum_{i=0}^{k-1}m_ip^i,n=\sum_{i=0}^{k-1}n_ip^i:\forall i=0,\dots,k-1, m_i\equiv n_i\mod 2\right\}\\&\subseteq \{(m,n)\in\N^2: 0\le m,n\le p^k-1\}.\end{align*}
Define
$$\mathcal{F}_{0,\infty}:=\bigcap_{k=1}^{\infty}\mathcal{I}_k=\left\{(x,y)\in[0,1]^2:x=\sum_{n=1}^{\infty}x_np^{-n}, y=\sum_{n=1}^{\infty}y_np^{-n},x_n\equiv y_n\mod 2,\forall n\in\N\right \}.$$
The size of $\mathcal{N}_k$ is now able to be computed.
\begin{lemma}
\label{lem:NumN_r}
    For every $k\in \mathbb{N}$, $\#\mathcal{N}_k=\left(\frac{p^2+1}{2}\right)^k$.    
\end{lemma}
\begin{proof}
Note that there are $\frac{p+1}{2}$ even digits and $\frac{p-1}{2}$ odd digits in the set $\{0,1,\dots,p-1\}$. From the definition of $\mathcal{N}_k$, one has $(m,n) \in\mathcal{N}_k$ if and only if
$n_i=m_i \mod 2$ for every $0\leq i\leq k-1$. For any such $i$, there are $\left(\frac{p-1}{2}\right)^2$ choices for the pair $(n_i,m_i)$ when both $n_i$ and $m_i$ are odd and
$\left(\frac{p+1}{2}\right)^2$ choices for the pair $(n_i,m_i)$, when both $m_i$ and $n_i$ are even. Hence, for every pair $(n_i,m_i)$, there are $\left(\frac{p-1}{2}\right)^2+\left(\frac{p+1}{2}\right)^2=\frac{p^2+1}{2}$ choices for pair $(n_i,m_i)$. Therefore, $\#\mathcal{N}_k=\left(\frac{p^2+1}{2}\right)^k$.
\end{proof}
\noindent Define a mass distribution $\mu$ on $[0,1]^2$, such that for every $m,n\in \{0,1,\dots,p^k-1\}$, we have 
$$\mu\left(\left[\frac{m}{p^k},\frac{m+1}{p^k}\right)\times \left[\frac{n}{p^k},\frac{n+1}{p^k}\right)\right)=\begin{cases}
    \left(\frac{p^2+1}{2}\right)^{-k}&(m,n)\in \mathcal{N}_k,\\
    0&\text{else}.
\end{cases}$$ 
Let $U\subseteq [0,1]^2$ be a set satisfying $p^{-(k+1)}<\vert U\vert<p^{-k}$. Then, $U$ can intersect at most four intervals comprising $\mathcal{I}_k$. Hence, 
$$\mu(U)\leq 4\left(\frac{p^2+1}{2}\right)^{-k}=4\left(p^{-k}\right)^{\frac{\log\left(\frac{p^2+1}{2}\right)}{\log p}}\leq 4\left(p\vert U\vert\right)^{\frac{\log\left(\frac{p^2+1}{2}\right)}{\log p}}.$$
Thus, by Proposition \ref{prop:MassDist},
$$\dim_H(\mathcal{F}_{0,\infty})\geq \frac{\log_p\left(\frac{p^2+1}{2}\right)}{\log p}$$ and the first part of claim \hyperlink{link: frac_2}{2} is established.
\subsection{The Fractal $\mathcal{F}_{F,\infty}$}\hfill\\
\noindent Recall the alphabet $\mathcal{A}$ from Section \ref{sec: 5}. The fractal $\mathcal{F}_{F,\infty}$ is obtained through the $[p,p]$-morphism $$\Phi_{F,p}:(\mathcal{A}\setminus \{B\})_2^{\infty}\rightarrow (\mathcal{A}\setminus \{B\})_2^{\infty},$$ which is defined in the following manner:
\begin{align}
    \Phi_{F,p}(A)=\begin{matrix}
        A&F&A&\dots&A&F&A\\
        F&A&F&\dots&F&A&F\\
        A&F&A&\dots&A&F&A\\
        \vdots&\vdots&\vdots&\dots&\vdots&\vdots&\vdots\\
        A&F&A&\dots&A&F&A\\
        F&A&F&\dots&F&A&F\\
        A&F&A&\dots&A&F&A
    \end{matrix},& & \forall X\in \mathcal{A}\setminus \mathcal{A}_{units},\,~ &\Phi_{F,p}(X)=\Phi_p(X).
\end{align}
By the definition of a $[p,p]$-morphism, 
\begin{equation}
    \Phi_{F,p}^k(A)=\begin{matrix}
        \Phi_{F,p}^{k-1}(A)&\Phi_{F,p}^{k-1}(F)&\Phi_{F,p}^{k-1}(A)&\dots&\Phi_{F,p}^{k-1}(A)&\Phi_{F,p}^{k-1}(F)&\Phi_{F,p}^{k-1}(A)\\
        \Phi_{F,p}^{k-1}(F)&\Phi_{F,p}^{k-1}(A)&\Phi_{F,p}^{k-1}(F)&\dots&\Phi_{F,p}^{k-1}(F)&\Phi_{F,p}^{k-1}(A)&\Phi_{F,p}^{k-1}(F)\\
        \Phi_{F,p}^{k-1}(A)&\Phi_{F,p}^{k-1}(F)&\Phi_{F,p}^{k-1}(A)&\dots&\Phi_{F,p}^{k-1}(A)&\Phi_{F,p}^{k-1}(F)&\Phi_{F,p}^{k-1}(A)\\
        \vdots&\vdots&\vdots&&\vdots&\vdots&\vdots\\
        \Phi_{F,p}^{k-1}(A)&\Phi_{F,p}^{k-1}(F)&\Phi_{F,p}^{k-1}(A)&\dots&\Phi_{F,p}^{k-1}(A)&\Phi_{F,p}^{k-1}(F)&\Phi_{F,p}^{k-1}(A)\\
        \Phi_{F,p}^{k-1}(F)&\Phi_{F,p}^{k-1}(A)&\Phi_{F,p}^{k-1}(F)&\dots&\Phi_{F,p}^{k-1}(F)&\Phi_{F,p}^{k-1}(A)&\Phi_{F,p}^{k-1}(F)\\
        \Phi_{F,p}^{k-1}(A)&\Phi_{F,p}^{k-1}(F)&\Phi_{F,p}^{k-1}(A)&\dots&\Phi_{F,p}^{k-1}(A)&\Phi_{F,p}^{k-1}(F)&\Phi_{F,p}^{k-1}(A)
    \end{matrix}.
\end{equation}
Moreover, by a simple induction one can see that for every $0\leq m,n\leq p^k-1$, 
\begin{align*}
   \Phi_{F,p}^k(F) =\Phi^k_p(F)[m,n]=\begin{cases}
        C_{NW} &\text{ if }~~~m=n=0,\\
        C_{NE} &\text{ if }~~~m=0~\text{ and }~n=p^k-1,\\
        C_{SW} &\text{ if }~~~m=p^k-1~\text{ and }~n=0,\\
        C_{SE} &\text{ if }~~~m=n=p^k-1,\\
        E_N &\text{ if }~~~m=0~\text{ and }~0\neq n \neq p^k-1,\\
        E_E &\text{ if }~~~0\neq m\neq p^k-1~\text{ and }~n=p^k-1,\\
        E_S &\text{ if }~~~m=p^k-1~\text{ and }~0\neq n \neq p^k-1,\\
        E_W &\text{ if }~~~0\neq m\neq p^k-1~\text{ and }~n=0,\\
        0 &\text{ otherwise}.
    \end{cases}
\end{align*}\noindent
Define $$\mathcal{I}_k:=\bigcup_{\substack{0\leq m,n<p^k\\\Phi_{F,p}^k(m,n)\neq 0}}\left[\frac{m}{p^k},\frac{m+1}{p^k}\right)\times \left[\frac{n}{p^k},\frac{n+1}{p^k}\right),$$ and let $\mathcal{F}_{F,\infty}:=\bigcap_{k\geq 1}\mathcal{I}_k$. Now \cite[Proposition 4.1]{Fal} is used to bound the Hausdorff dimension of $\mathcal{F}_{F,\infty}$ from above.
\begin{proposition}{\cite[Proposition 4.1]{Fal}}
\label{prop:DimUpp}
    Suppose a set $X$ can be covered by $n_k$ sets of diameter at most $\delta_k$, where $\delta_k\rightarrow 0$ as $k\rightarrow \infty$. Then,
    $$\dim_HX\leq \liminf_{k\rightarrow \infty}\frac{\log n_k}{-\log \delta_k}.$$
\end{proposition}
\noindent To apply Proposition \ref{prop:DimUpp}, the number of non-zero entries in $\Phi_{F,p}^k(A)$ is counted. Note that for $k\geq 1$, $\Phi_{F,p}^k(F)$ contains $4(p^k-1)$ non-zero entries, and $\Phi_{F,p}^k(A)$ contains $2\frac{p-1}{2}\cdot \frac{p+1}{2}=\frac{p^2-1}{2}$ copies of $\Phi_{F,p}^{k-1}(F)$ and $\left(\frac{p+1}{2}\right)^2+\left(\frac{p-1}{2}\right)^2=\frac{p^2+1}{2}$ copies of $\Phi_{F,p}^{k-1}(A)$. Let $a_k$ denote the number of non-zero entries in $\Phi_{F,p}^k(A)$. Then, 
\begin{align*}
    a_{k+1}=\frac{p^2+1}{2}a_k+2(p^2-1)(p^k-1),\,\,\,&
    a_1=p^2.
\end{align*}
Therefore, 
\begin{align*}
    a_k=&p^2\left(\frac{p^2+1}{2}\right)^{k-1}+2(p^2-1)\sum_{i=1}^{k-1}\left(\frac{p^2+1}{2}\right)^{i-1}(p^{k-i}-1)\\
    =&\left(\frac{p^2+1}{2}\right)^k\left(\frac{2p^2}{p^2+1}+\frac{2}{p-1}-\frac{4(p+1)}{p-1}\left(\frac{2p}{p^2+1}\right)^k+4\left(\frac{2}{p^2+1}\right)^k\right).
\end{align*}
Then, $\mathcal{I}_k$ can be covered with $a_k$ boxes of side length $p^{-k}$, and hence, $\mathcal{F}_{F,\infty}$ can also be covered by $a_k$ boxes of side length $p^{-k}$. Therefore, by Proposition \ref{prop:DimUpp},
\begin{align*}
    \begin{split}
    \dim_H(\mathcal{F}_{F,\infty})\leq& \liminf_{k\rightarrow \infty}\frac{k\log\left(\frac{p^2+1}{2}\right)+\log\left(\frac{2p^2}{p^2+1}+\frac{2}{p-1}-\frac{4(p+1)}{p-1}\left(\frac{2p}{p^2+1}\right)^k+4\left(\frac{2}{p^2+1}\right)^k\right)}{k\log p}\\
    &=\frac{\log\left(\frac{p^2+1}{2}\right)}{\log p},
    \end{split}
\end{align*}
\noindent proving the second inequality in claim \hyperlink{link: frac_2}{2}.
\subsection{Proof of the Inclusion}\hfill\\
Claim \hyperlink{link: frac_1}{1} is now established. Note that to obtain $\Phi_{0,p}$ from $\Phi_p$, one identifies $A$ and $B$ and then, replaces all other letters with $0$. Therefore, $\mathcal{F}_{0,\infty}\subseteq \mathcal{F}(\textbf{C}^{(p)},\mathbb{F}_p)$, and hence, $$\dim_H\left(\mathcal{F}(\textbf{C}^{(p)},\mathbb{F}_p)\right)\geq \dim_H(\mathcal{F}_{0,\infty})\ge\frac{\log\left(\frac{p^2+1}{2}\right)}{\log p}.$$
Similarly, since $\Phi_{F,p}(A)$ is obtained from $\Phi_p(A)$ by replacing $0$ with $F$ and $B$ with $A$, then, $\mathcal{F}(\textbf{C}^{(p)},\F_p)\subseteq \mathcal{F}_{F,\infty}$. Therefore, 
$$\frac{\log\left(\frac{p^2+1}{2}\right)}{\log p}\leq \dim_H\left(\mathcal{F}(C^{(p)},\mathbb{F}_p)\right)\leq \dim_H(\mathcal{F}_{F,\infty})\leq \frac{\log \left(\frac{p^2+1}{2}\right)}{\log p}.$$
Thus, Theorem \ref{thm: frac_nw} follows.

\section{Lemmata on Number Walls of Generic Sequences}\label{sect: 7}
\noindent With the proof of Theorem \ref{thm: frac_nw} now complete, all that remains is to prove Theorem \ref{thm: morphism}. The first step towards this is to establish a series of lemmas around the behaviour of number walls generated by non-specific sequences. Section \ref{Sect: 8} then refines these lemmas to the number wall of the $p$-Cantor and $p$-Singer sequence. \\ 

\noindent Before the aforementioned lemmata are stated, the following natural generalisation of a number wall is defined. 
\subsection{Geometric Transforms of Number Walls}\hfill\\
\noindent Let $\textbf{S}$ be a doubly infinite sequence over a field\footnote{See Remark \ref{remark: field_not_ID} as to why one uses a field $\K$ instead of an integral domain $\id$.} $\K$. By Definition \ref{nw}, $W_\K(\textbf{S})[m,n]=0$ for all $n\in\Z$ and $m<-1$. In the language of number walls, one can consider these rows to comprise an \textit{infinite window} in $W_\K(\textbf{S})$, denoted $\mathcal{W}_0$. From this perspective, $(W_\K(\textbf{S})[-1,n])_{n\in\Z}$ is southern inner frame of $\mathcal{W}_0$. Also by Definition \ref{nw}, $W_\K(\textbf{S})[-1,n]=1$ for all $n\in\Z$. But this is merely a convention - $(W_\K(\textbf{S})[-1,n])_{n\in\Z}$ need only be a geometric sequence to comply with Theorem \ref{ratio ratio}. This motivates the following generalisation of the number wall.
\begin{definition}\label{r-a-nw}
     Let $\K$ be a field, let $r\in\K\backslash\{0\}$, $a\in\K$ and let $\mathbf{S}=(s_n)_{n\in\Z}$ be a sequence over $\K$. Additionally, recall the notation from Figure \ref{Fig: basicwindow}. Then, the $(r,a)$\textbf{-Number Wall of }$\mathbf{S}$\textit{ over }$\K$, denoted $W_\K^{(r,a)}(\textbf{S})=(W_{\K}^{(r,a)}(\textbf{S})[m,n])_{n,m\in\Z}$, is defined as the two dimensional sequence generated by applying the Frame Constraints (equations (\ref{eqn: FC})) to the initial conditions \begin{equation*}
        W_{m,n}=W_\K^{(r,a)}(\textbf{S})[m,n]=\begin{cases}
    \begin{aligned}
    &0 &&\textup{ if } m<-1\\
    &a\cdot r^n &&\textup{ if } m=-1;&\\
    &s_n &&\textup{ if } m=0;&\\
    
    \end{aligned}
    \end{cases}
    \end{equation*}\end{definition}
\noindent In other words, Definition \ref{r-a-nw} states that $W_\K^{(r,a)}(\textbf{S})$ is calculated identically to $W_\K(\textbf{S})$, but with $W_\K(\textbf{S})[-1,n]=a\cdot r^n$ instead of $W_\K(\textbf{S})[-1,n]=1$ for all $n\in\Z$. Hence, it is immediate that for any sequence $\textbf{S}$, $W_\K^{(1,1)}(\textbf{S})=W_\K(\textbf{S})$.
\begin{remark}\label{remark: defined}
    Whilst equations (\ref{eqn: FC}) do guarantee that zeroes in an $(r,a)$-number wall appear in square windows, it is not immediate that \hyperlink{FC3}{FC3} remains well defined in this setup. That is, there is no reason \textit{a priori} that the inner frame of any window in an $(r,a)$-number wall is comprised of geometric sequences, which is necessary to define $P,Q,R$ and $S$ in \hyperlink{FC3}{FC3}. This is addressed in Lemma \ref{rswall}.
\end{remark}

\begin{remark}\label{remark: field_not_ID}
    Note, that $(r,a)$-number walls are only defined over fields and not integral domains. Standard number walls are well defined over an integral domain $\id$, since each entry is equal to the determinant of a Toeplitz matrix (Definition \ref{nw}), which are calculated using only addition and multiplication and hence always elements of $\id$. In contrast, $(r,a)$-number walls have no such connection to Toeplitz matrices, and therefore, $\id$ is replaced by a field $\K$ to ensure the division in equation (\ref{eqn: FC}) is well defined.   
\end{remark}
\subsubsection{\textbf{Geometric Transform of a Sequence}}\hfill\\
\noindent The following transformation of a sequence is also crucial to this section. 

\begin{definition}\label{r-a-trans}
    Let $\K$ be a field, let $r,a\in\K$ and let $\mathbf{S}=(s_i)_{i\in\Z}$ be a sequence over $\K$. Then, define the $(r,a)$-\textbf{geometric transform of} $\mathbf{S}$ is the sequence $\textbf{S}{(r,a)}=(a\cdot r^i\cdot s_i)_{i\in\Z}$.
\end{definition}
\subsubsection{\textbf{Relationship between Number Walls and }$(r,a)$\textbf{-Number Walls}}\hfill\\

\noindent For $r_0,a_0, r_1, a_1\in\K$, and $\mathbf{S}$ be a sequence over $\K$, the following lemma states how $W_\K^{(r_0,a_0)}\left(\mathbf{S}{(r_1,a_1)}\right)$ relates to $W_\K(\textbf{S})$.
\begin{lemma}\label{rswall}
    Let $\K$ be a field, let $r_0,a_0,r_1,a_1\in\K$ with $r_0,a_0\neq 0$ and let $\mathbf{S}$ be a sequence over $\K$. Then, the two-dimensional sequence $W_\K^{(r_0,a_0)}\left(\mathbf{S}{(r_1,a_1)}\right)$ is well defined in the sense of Remark \ref{remark: defined} and it satisfies\begin{equation}\label{rs}W_\K^{(r_0,a_0)}\left(\mathbf{S}{(r_1,a_1)}\right)[m,n]=\frac{r_1^{n(m+1)}a_1^{m+1}}{r_0^{nm}a_0^{m}}W_{\K}(\textbf{S})[m,n].\end{equation}
\end{lemma}

\begin{proof}

     \noindent If either $a_1=0$ or $r_1=0$, the conclusion follows immediately since $\mathbf{S}{(r_1,a_1)}=\{0\}_{i\in\Z}$, and hence $W_\K^{(r_0,a_0)}\left(\mathbf{S}{(r_1,a_1)}\right)[m,n]=0$ for all $n\in\Z$ and $m\in\Z\backslash\{-1\}$. Therefore, assume that $a_1,r_1\neq0$.\\
     
     \noindent First, it is confirmed that $P,Q,R$ and $S$ from \hyperlink{FC3}{FC3} are well defined. Indeed, assume that \begin{equation*}
        (W_\K(\textbf{S})[m+i,n])_{0\le i \le k\in\N}=x\cdot v^i ~~~\text{ or }~~~ (W_\K(\textbf{S})[m,n+i])_{0\le i \le k\in\N}=y\cdot u^i
    \end{equation*} for some $x,y,v,u\in\K$. In other words, assume that the above parts of the number wall are geometric sequences. Then, assuming equation (\ref{rs}) holds, \begin{align*}
        \left(W^{(r_0,a_0)}_\K\left(\mathbf{S}{(r_1,a_1)}\right)[m+i,n]\right)_{0\le i \le k\in\N}&=\frac{xr_1^{n(m+1)}a_1^{m+1}}{r_0^{nm}a_0^m}\cdot \left(\frac{r_1^n\cdot a_1}{r_0^n\cdot a_0} \cdot v\right)^i & ~~~\text{ or }~~~\\ \left(W^{(r_0,a_0)}_\K\left(\mathbf{S}{(r_1,a_1)}\right)[m,n+i]\right)_{0\le i \le k\in\N}&=\frac{yr_1^{n(m+1)}a_1^{m+1}}{r_0^{nm}a_0^m}\cdot \left(\frac{r_1^{m+1}}{r_0^m}\cdot u\right)^i.
    \end{align*} That is, equation (\ref{rs}) preserves geometric sequences and hence $P,Q,R$ and $S$ are well defined in \hyperlink{FC3}{FC3}.  \\
    
    \noindent Equation (\ref{rs}) is now established by induction on $m$. The base case is given by $m=-1$ and $m=0$, and it follows immediately from Definitions \ref{r-a-nw} and \ref{r-a-trans}. Therefore, let $M$ be a natural number, and assume that for rows $m<M$ formula (\ref{rs}) holds. \\
    
    \noindent Equation (7.1) is now verified for an entry on row $M$ in each of the three of the possible non-trivial cases. The strategy is the same for each: first apply one of the Frame Constraints from Definition \ref{r-a-nw}, then use the induction hypothesis, and finally show that it satisfies Lemma \ref{rswall}.\\

    \noindent \textbf{Case 1:} $W_\K^{(r_0,a_0)}\left(\mathbf{S}{(r_1,a_1)}\right)[{M-2,n}]\neq0$.\\
    \noindent In this case, \hyperlink{FC1}{FC1} is applied, giving $W_\K^{(r_0,a_0)}[M,n]\left(\mathbf{S}{(r_1,a_1)}\right)=$  \begin{equation}\label{rs1}\frac{\left(W_\K^{(r_0,a_0)}\left(\mathbf{S}{(r_1,a_1)}\right)[{M-1,n}]\right)^2-W_\K^{(r_0,a_0)}\left(\mathbf{S}{(r_1,a_1)}\right)[{M-1,n-1}]\cdot W_\K^{(r_0,a_0)}\left(\mathbf{S}{(r_1,a_1)}\right)[{M-1,n+1}]}{W_\K^{(r_0,a_0)}\left(\mathbf{S}{(r_1,a_1)}\right)[{M-2,n}]}.\end{equation} By the inductive hypothesis, \begin{align*}
    &W_\K^{(r_0,a_0)}\left(\mathbf{S}{(r_1,a_1)}\right)[M-1,n]=\frac{r_1^{nM}a_1^M}{r_0^{n(M-1)}a_0^{M-1}}W_\K(\mathbf{S})[M-1,n],& \\&W_\K^{(r_0,a_0)}\left(\mathbf{S}{(r_1,a_1)}\right)[M-2,n]=\frac{r_1^{n(M-1)}a_1^{M-1}}{r_0^{n(M-2)}a_0^{M-2}}W_\K(\mathbf{S})[M-2,n],&\\
    &W_\K^{(r_0,a_0)}\left(\mathbf{S}{(r_1,a_1)}\right)[{M-1,n-1}]=\frac{r_1^{(n-1)M}a_1^M}{r_0^{(n-1)(M-1)}a_0^{M-1}}W_\K(\mathbf{S})[{M-1,n-1}],&\end{align*}\noindent and \begin{align*}&W_\K^{(r_0,a_0)}\left(\mathbf{S}{(r_1,a_1)}\right)[M-1,n+1]=\frac{r_1^{(n+1)M}a_1^M}{r_0^{(n+1)(M-1)}a_0^{M-1}}W_\K(\mathbf{S})[{M-1,n+1}].&
    \end{align*} 
    Therefore, \begin{align*}
        (\ref{rs1})&=\frac{r_1^{n(M+1)}a_1^{M+1}}{r_0^{nM}a_0^M}\cdot \frac{\left(W_\K(\mathbf{S})[{M-1,n}]\right)^2-W_\K(\mathbf{S})[{M-1,n-1}]\cdot W_\K(\mathbf{S})[{M-1,n+1}]}{W_\K(\mathbf{S})[{M-2,n}]}\\&\substack{\hyperlink{FC1}{FC1}\\=}~\frac{r_1^{n(M+1)}a_1^{M+1}}{r_0^{nM}a_0^M}W_\K(\mathbf{S})[{n,M}],
    \end{align*}
    \noindent which finishes the proof in this case.\\

    \noindent \textbf{Case 2:} $W_\K^{(r_0,a_0)}\left(\mathbf{S}{(r_1,a_1)}\right)[{M-2,n}]=0=W_\K^{(r_0,a_0)}\left(\mathbf{S}{(r_1,a_1)}\right)[{M-1,n}]$. This case is a simpler application of the method seen in Case 3 and hence, for the sake of avoiding repetition, the proof is omitted. The reader can verify equation (\ref{rs}) in this case by following along with Case 3, making only 2 minor changes: \begin{itemize}
        \item Every instance of $M$ replaced with $M+1$;
        \item \hyperlink{FC2}{FC2} is used in place of \hyperlink{FC3}{FC3}.
    \end{itemize}
    
    \noindent \textbf{Case 3:} $W^{(r_0,a_0)}\left(\mathbf{S}{(r_1,a_1)}\right)[{M-2,n}]=0\neq W^{(r_0,a_0)}\left(\mathbf{S}{(r_1,a_1)}\right)[{M-1,n}]$ \\
    \noindent In this case, one applies \hyperlink{FC3}{FC3}. 
    \noindent Given a window in $W_\K^{(r_0,a_0)}\left(\mathbf{S}{(r_1,a_1)}\right)$ of size $l\in\N$, let $A^{\times}_k$ be the entry of the inner frame that corresponds to $A_k$ (recall Figure \ref{Fig: basicwindow}). Define similarly for other values of the inner and outer frame. \\
    
    \noindent Recall Figure \ref{Fig: basicwindow}. Using notation from \hyperlink{FC3}{F3}, the goal is to calculate $W_\K^{(r_0,a_0)}\left(\mathbf{S}{(r_1,a_1)}\right)[{M,n}]=H_k^\times$. To do so, one requires the values of $k,A^{\times}_k,B^{\times}_k,C^{\times}_k,D^{\times}_k,E^{\times}_k,F^{\times}_k,G^{\times}_k,$ $P^{\times}, Q^{\times},R^{\times}$ and $S^{\times}$, where  the latter four values are the ratios for the geometric sequences $(A_i^{\times})_{0\le i \le l+1}$, $(B_i^{\times})_{0\le i \le l+1}$, $(C_i^{\times})_{0\le i \le l+1}$ and $(D_i^{\times})_{0\le i \le l+1}$ respectively. The following diagram illustrates which rows and columns of $W_\K^{(r_0,a_0)}\left(\mathbf{S}{(r_1,a_1)}\right)$ these are located in. 
    \begin{figure}[H]
        \centering
        \includegraphics[width=0.7\linewidth]{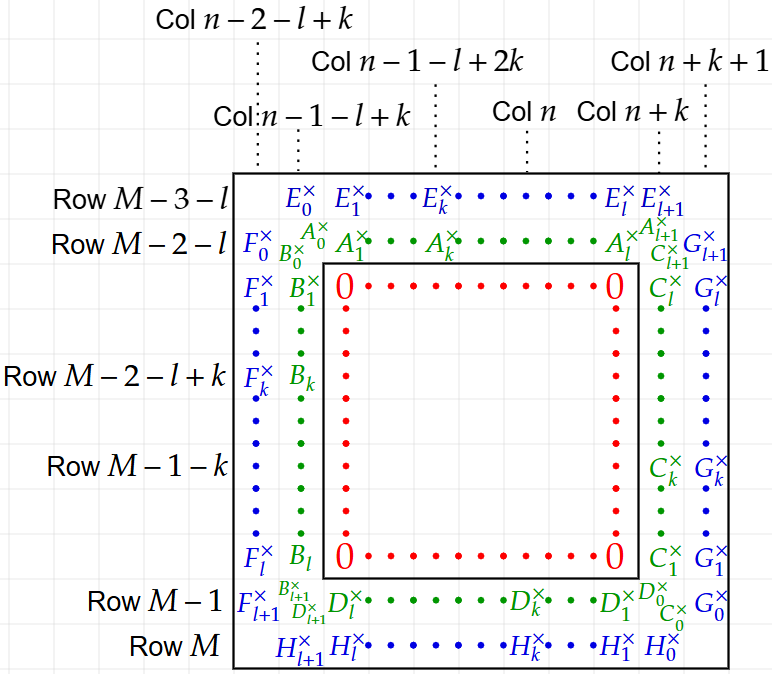}
        \caption{If $H_k$ is on row $M$ and column $n$, then this diagram shows the relative position of the other entries in the inner and outer frame.}
        \label{fig:rs window}
    \end{figure}
    \noindent Figure \ref{fig:rs window} shows that, for instance $$D_k^{\times}=W_\K^{(r_0,a_0)}\left(\mathbf{S}{(r_1,a_1)}\right)[M-1,n].$$ One then applies the induction hypothesis and the fact that equation (\ref{rs}) preserves geometric sequences to calculate that \begin{align*}
        W_\K^{(r_0,a_0)}\left(\mathbf{S}{(r_1,a_1)}\right)[M-1,n]&=\frac{r_1^{nM}a_1^{M}}{r_0^{n(M-1)}a_0^{(M-1)}}W_\K(\mathbf{S})[M-1,n]\\&=\frac{r_1^{nM}a_1^{M}}{r_0^{n(M-1)}a_0^{(M-1)}}D_k.\end{align*}
      \noindent By the same method, one obtains that\begin{align*}
        &A^{\times}_k=\frac{r_1^{(n-1-l+2k)(M-1-l)}a_1^{M-1-l}}{r_0^{(n-1-l+2k)(M-2-l)}a_0^{(M-2-l)}}A_k&&P^{\times}=\frac{r_1^{M-1-l}}{r_0^{M-2-l}}P& \\ &E^{\times}_k=\frac{r_1^{(n-1-l+2k)(M-2-l)}a_1^{M-2-l}}{r_0^{(n-1-l+2k)(M-3-l)}a_0^{M-3-l}}E_k& \\&B^{\times}_k=\frac{r_1^{(n-1-l+k)(M-1-l+k)}a_1^{(M-1-l+k)}}{r_0^{(n-1-l+k)(M-2-l+k)}a_0^{(M-2-l+k)}}B_k&&Q^{\times}=\frac{r_1^{n-1-l+k}a_1}{r_0^{n-1-l+k}a_0}Q& \\&F^{\times}_k=\frac{r_1^{(n-2-l+k)(M-1-l+k)}a_1^{(M-1-l+k)}}{r_0^{(n-2-l+k)(M-2-l+k)}a_0^{(M-2-l+k)}}F_k& \\&C^{\times}_k=\frac{r_1^{(n+k)(M-k)}a_1^{M-k}}{r_0^{(n+k)(M-1-k)}a_0^{(M-1-k)}}C_k& &R^{\times}=\frac{r_0^{n+k}a_0}{r_1^{n+k}a_1}R&\\&G^{\times}_k=\frac{r_1^{(n+k+1)(M-k)}a_1^{(M-k)}}{r_0^{(n+k+1)(M-1-k)}a_0^{(M-1-k)}}G_k& \\ &D^{\times}_k=\frac{r_1^{nM}a_1^{M}}{r_0^{n(M-1)}a_0^{(M-1)}}D_k& 
        &S^{\times}=\frac{r_0^{M-1}}{r_1^{M}}\mathbf{S}&
    \end{align*} 
    \noindent Finally, one calculates $H_k^{\times}$ using \hyperlink{FC3}{FC3}: \[H_k^{\times}=\frac{\frac{Q^{\times}E_k^{\times}}{A_k^{\times}}+(-1)^{k}\frac{P^{\times}F_k^{\times}}{B^{\times}_k}-(-1)^k\frac{S^{\times}G^{\times}_k}{C^{\times}_k}}{R^{\times}/D^{\times}_k}\cdotp\] Substituting the above values gives \begin{align*}
        &\frac{Q^{\times}E_k^{\times}}{A_k^{\times}}=\frac{r_0^{k} }{r_1^{k}}\cdot \frac{QE_k}{A_k}& &\frac{P^{\times}F_k^{\times}}{B_k^{\times}}=\frac{r_0^{k} }{r_1^{k}}\cdot \frac{PF_k}{B_k}& \\&\frac{S^{\times}G_k^{\times}}{C_k^{\times}}=\frac{r_0^{k} }{r_1^{k}}\cdot \frac{SG_k}{C_k}& &\frac{R^{\times}}{D_k^{\times}}=\frac{r_0^{nM+k} a_0^M}{r_1^{(n(M+1)+k)}a_1^{m+1}}\cdot \frac{R}{D_k}&
    \end{align*}
    and hence \begin{align*}H_k^\times&=\frac{r_1^{n(m+1)}a_1^{m+1}}{r_{0}^{nm}a_0^m}\cdot \frac{\frac{QE_k}{A_k}+(-1)^{k}\frac{PF_k}{B_k}-(-1)^k\frac{SG_k}{C_k}}{R/D_k}\\&\substack{\hyperlink{FC3}{FC3}\\=}~~\frac{r_1^{n(m+1)}a_1^{m+1}}{r_{0}^{nm}a_0^m}\cdot H_k\end{align*} as required.

\end{proof}
\noindent Lemma \ref{rswall} implies the following corollary that relates the profile of $W_\K(\textbf{S})$ to the profile of $W_\K^{(r_0,a_0)}(\mathbf{S}{(r_1,a_1)})$.
\begin{corollary}\label{cor: profile}
For any nonzero choice of nonzero $r_0,r_1, a_0$ and $a_1$ in $\K$, one has that \[\chi\left(W_\K^{(r_0,a_0)}\left(\mathbf{S}{(r_1,a_1)}\right)\right)=\chi(W_\K(\mathbf{S})).\]
\end{corollary}

\subsection{Reflecting Finite Number Walls}\hfill\\
\noindent Just as Lemma \ref{rswall} showed that number walls behave well under geometric transforms, this subsection shows the same but for a different transformation: reflection. There are two axes along which a finite number wall is reflected: the vertical and the horizontal. Here, \textbf{vertical} and \textbf{horizontal} refer to the line of reflection. The former is dealt with first.

\subsubsection{\textbf{Finite Number Walls under Vertical Reflection}}\hfill\\
\noindent Let $\ell\in\N$ and consider a finite sequence $\mathbf{S}=(s_i)_{0\le i \le \ell}$ over a field $\K$. The following lemma shows that $W_\K(\mathbf{S})$ behaves well when the sequence $\mathbf{S}$ is `reflected' vertically around its midpoint.
\begin{lemma}[Vertical Reflections of Finite Number Walls]\label{reflect}
    Let $W_\K(\mathbf{S})$ be the finite number wall of the finite sequence $\mathbf{S}=(s_i)_{0\le i \le \ell}$ over a field $\K$. Define $\overset{\textup{\tiny$\bm\leftrightarrow$}}{\textbf{\textup{S}}}=(\overset{\text{\tiny$\bm\leftrightarrow$}}{s}_{i})_{0\le i \le \ell}$ as the `reflection' of $\mathbf{S}$; that is, $\overset{\text{\tiny$\bm\leftrightarrow$}}{s}_i:=s_{\ell-i}$ for all $0\le i \le \ell$. Then $W_\K\big(\overset{\text{\tiny$\bm\leftrightarrow$}}{\textbf{\textup{S}}}\big)$ is the vertical reflection of $W_\K(\mathbf{S})$. More precisely, let $m$ and $n$ be natural numbers satisfying $0\le m\le \left\lfloor\frac{\ell}{2}\right\rfloor$ and $m\le n\le \ell-m$. Then, $$W_\K\big(\overset{\text{\tiny$\bm\leftrightarrow$}}{\textbf{\textup{S}}}\big)[{m,\ell-n}]=W_\K(\mathbf{S})[{m,n}].$$
\end{lemma}

\noindent Lemma \ref{reflect} is illustrated by Figure \ref{reflectpic}:\begin{figure}[H]
    \centering
    \includegraphics[width=1\linewidth]{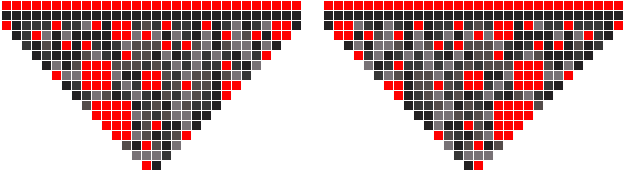}
    \caption{Left: The number wall of a sequence $\mathbf{S}$ over $\F_5$. The zero entries are coloured in red, and the nonzero entries are in shades of grey. Right: The number wall of \(\overset{\textup{\tiny$\bm\leftrightarrow$}}{\textbf{\textup{S}}}\), defined as in Lemma \ref{reflect}.}
    \label{reflectpic}
\end{figure}
\begin{proof}
    
    By Definition \ref{nw}, $$W_\K(\mathbf{S})[m,n]=\det(T_S(n;m)).$$ Similarly, $$W_\K(\overset{\text{\tiny$\bm\leftrightarrow$}}{\textbf{\textup{S}}})[m,\ell-n]=\det(T_{\overset{\text{\tiny$\bm\leftrightarrow$}}{\textbf{\textup{S}}}}(\ell-n;m)).$$ It is easy to verify that, $T_{\overset{\text{\tiny$\bm\leftrightarrow$}}{\textbf{\textup{S}}}}(\ell-n;m)$ is the transpose of $T_S(n;m)$. As the determinant of a matrix is invariant under transposition, the proof is complete.
\end{proof}

\noindent There are two immediate corollaries of Lemma \ref{reflect}: \begin{corollary}[Symmetry of Finite Number Walls Generated by Symmetrical Sequences]
     Let $\mathbf{S}=(s_i)_{0\le i \le \ell}$ be a finite sequence over a field $\K$ satisfying $s_i=s_{\ell-i}$ for every $0\le i \le \ell$. Additionally, define $\overset{\text{\tiny$\bm\leftrightarrow$}}{\textbf{\textup{S}}}=(\overset{\text{\tiny$\bm\leftrightarrow$}}{s}_{i})_{0\le i \le \ell}$ as in Lemma \ref{reflect}.  Then, for all $0\le m\le \left\lfloor\frac{\ell}{2}\right\rfloor$ and $m\le n\le \ell-m$, $$W_\K\big(\overset{\text{\tiny$\bm\leftrightarrow$}}{\textbf{\textup{S}}}\big)[{m,n}]=W_\K(\mathbf{S})[{m,n}].$$
\end{corollary}
\noindent The second corollary is obtained by combining Lemma \ref{reflect} and Lemma \ref{rswall}.\begin{corollary}[Reflections of Finite $(r,a)$-Number Walls]\label{lem: reflect_r_a}
    Let $\mathbf{S}=(s_i)_{0\le i \le \ell}$ be a finite sequence over $\K$ and let $r_0,a_0,r_1,a_1\in\K$ be such that $a_0,r_0\neq0$. Furthermore, define $\overset{\text{\tiny$\bm\leftrightarrow$}}{\textbf{\textup{S}}}$ as in Lemma \ref{reflect}. Then, for all $0\le m\le \left\lfloor\frac{\ell}{2}\right\rfloor$ and $m\le n\le \ell-m$,\[W^{(r_0,a_0)}_\K\left(\mathbf{S}{(r_1,a_1)}\right)[m,n] = W_\K^{(r_0^{-1}, a_0r^{\ell}_0)}\left(\overset{\text{\tiny$\bm\leftrightarrow$}}{\textbf{\textup{S}}}^{(r_1^{-1},r_1^{_\ell}a_1)}\right)[m, \ell-n].\]
\end{corollary}
\begin{proof}
    By Lemma \ref{rswall}, one has that \begin{align}
 W^{(r_0,a_0)}_\K\left(\mathbf{S}{(r_1,a_1)}\right)[m,n]&=\frac{r_1^{n(m+1)}a_1^{m+1}}{r_0^{nm}a_0^m}W_\K(\mathbf{S})[m,n]\nonumber\\
 &=\frac{(r_1^{-1})^{(\ell-n)(m+1)}\cdot r_1^{\ell(m+1)}\cdot a_1^{m+1}}{(r_0^{-1})^{(\ell-n)m}\cdot r_0^{\ell m}\cdot a_0^m}W_\K(\mathbf{S})[m,n].\label{eqn: r-s-reflect_1}\end{align} \noindent Lemma \ref{reflect} implies that\begin{align}
 (\ref{eqn: r-s-reflect_1})&=\frac{(r_1^{-1})^{(\ell-n)(m+1)}\cdot r_1^{\ell(m+1)}\cdot a_1^{m+1}}{(r_0^{-1})^{(\ell-n)m}\cdot r_0^{\ell m}\cdot a_0^m}W_\K(\overset{\text{\tiny$\bm\leftrightarrow$}}{\textbf{S}})[m,\ell-n]\label{eqn: r-s-reflect_2}\end{align} and a final application of Lemma \ref{rswall} gives \begin{align*} 
(\ref{eqn: r-s-reflect_2})&=W_\K^{(r_0^{-1}, r^{\ell}_0a_0)}\left(\overset{\text{\tiny$\bm\leftrightarrow$}}{\textbf{S}}^{(r_1^{-1},r_1^{_\ell}a_1)}\right)[m, \ell-n] 
    \end{align*}
\end{proof}
\noindent Lemma \ref{lem: reflect_r_a} shows that $(r,a)$-number walls behave well under vertical reflections. This motivates the following definition.\begin{definition}    Let $\mathbf{S}=(s_i)_{0\le i \le l}$ be a finite sequence over a field $\K$ and let $r_0,a_0,r_1,a_1\in\K$. Then, define \begin{equation}\mathcal{V}\left(W^{(r_0,a_0)}_\K\left(\mathbf{S}{(r_1,a_1)}\right)\right) := W_\K^{(r_0^{-1}, a_0r^{\ell}_0)}\left(\overset{\text{\tiny$\bm\leftrightarrow$}}{\textbf{S}}^{(r_1^{-1},r_1^{_\ell}a_1)}\right)\label{eqn: reflect_V}\end{equation} That is, $\mathcal{V}\left({W^{(r_0,a_0)}_\K\left(\mathbf{S}{(r_1,a_1)}\right)}\right)$ is the vertical reflection of $W^{(r_0,a_0)}_\K\left(\mathbf{S}{(r_1,a_1)}\right)$.
\end{definition}
\subsubsection{\textbf{Horizontal Reflections of Finite Number Walls}}\hfill\\
\noindent Recall, that a finite number wall is an infinite 2-dimensional array over $\K$ with only finitely many entries defined (Definition \ref{fin_nw}). This allows horizontal reflections of finite number walls to be well defined.
\begin{definition}\label{def: hor_reflect}
    Let $\textbf{S}$ be a sequence over a field $\K$ and let $\K\times \K $ be the space of 2-dimensional infinite arrays over $\K$. Then, define $\mathcal{H}:\K^{\Z\times \Z}\to\K^{\Z\times \Z}$ as \[\mathcal{H}((W_\K(\textbf{S})[m,n])_{m,n\in\Z}):=(W_\K(\textbf{S})[-m,n])_{m,n\in\Z}.\]
\end{definition}
\begin{remark}
    One rarely considers horizontal reflections of number walls on their own. Instead, they are often paired with a rotation, as defined in the next subsection. 
\end{remark}

\subsection{Rotating Finite Number Walls}\hfill\\
\noindent The goal of this subsection is to show how infinite number walls can often be deconstructed into infinitely many finite numbers walls, each rotated by some multiple of $90^\circ$ and potentially reflected horizontally or vertically. To this end, the following function is introduced that rotates an infinite array over $\K$ (recall that a finite number wall is an infinite array, as in Figure \ref{fin_nw}).
\begin{definition}\label{def:rotate}
    Let $(W_\K(\textbf{S})[m,n])_{m,n\in\Z}\in\K^{\Z\times \Z}$ be the number wall of a sequence $\textbf{S}$ over $\K$ and let $\K^{\Z\times\Z}$ be the space of 2-dimensional arrays over $\K$. Then, define the function $\rho:\K^{\Z\times \Z}\to\K^{\Z\times \Z}$ by \[(\rho(W_\K(\textbf{S}))[m,n])_{m,n\in\Z}:= (W_\K(\textbf{S})[-n,m])_{m,n\in\Z}.\] That is, $\rho$ is the clockwise rotation of $90^\circ$ of $W_\K(\textbf{S})[m,n]$ around $W_\K(\textbf{S})[0,0]$. 
\end{definition}

\noindent The next part of this subsection shows that rotated finite number walls always appear on the edges of windows.
\subsubsection{\textbf{Rotating the Frame Constraints}}\hfill\\
\noindent For an illustration of the following setup, see Figure \ref{fig: rot_lem}.\\

\noindent Let $\mathbf{S}$ be a sequence over a field $\K$. For $m,l\in\N$ and $n\in\Z$, let $(W_\K(\mathbf{S})[m+i,n+j])_{0\le i,j \le l+3}$ be an $(l+4)\times (l+4)$ square of entries in $W_\K(\mathbf{S})$ comprised of the following parts: \begin{itemize}
    \item An $l\times l$ window $\mathcal{W}$ given by $(W_\K(\mathbf{S})[m+i,n+j])_{2\le i,j\le l+1}=(0)_{1\le i,j\le l}$,
    \item The north inner and outer frame of $\mathcal{W}$ are given respectively by \begin{align*}
        &\textbf{A}:=(A_i)_{0\le i \le l+1}:=(W_\K(\mathbf{S})[m+1,n+i])_{1\le i \le l+2}&\end{align*} and\begin{align*} &\textbf{E}:=(E_i)_{0\le i \le l+1}:=(W_\K(\mathbf{S})[m,n+i])_{1\le i \le l+2}.& 
    \end{align*} 
    \item The west inner and outer frame of $\mathcal{W}$ are given respectively by \begin{align*}
        &\textbf{B}:=(B_i)_{0\le i \le l+1}:=(W_\K(\mathbf{S})[m+i,n+1])_{1\le i \le l+2}&\end{align*} and\begin{align*} &\textbf{F}:=(F_i)_{0\le i \le l+1}:=(W_\K(\mathbf{S})[m+i,n])_{1\le i \le l+2}.& 
    \end{align*} 
    \item The east inner and outer frame of $\mathcal{W}$ are given respectively by \begin{align*}
        &\textbf{C}:=(C_{i})_{0\le i \le l+1}:=(W_\K(\mathbf{S})[m+l-i+3, n+l+2])_{1 \le i \le l+2}&\end{align*} and\begin{align*} &\textbf{G}:=(G_{i})_{0\le i \le l+1}:=(W_\K(\mathbf{S})[m+l-i+3, n+l+3])_{1 \le i \le l+2}.& 
    \end{align*}
    \item The south inner and outer frame of $\mathcal{W}$ are given respectively by \begin{align*}
        &\textbf{D}:=(D_{i})_{0\le i \le l+1}:=(W_\K(\mathbf{S})[m+l+2, n+l-i+3])_{1\le i \le l+2}&\end{align*} and\begin{align*} &\textbf{H}:=(H_{i})_{0\le i \le l+1}:=(W_\K(\mathbf{S})[m+l+3, n+l-i+3])_{1\le i \le l+2}.& 
    \end{align*}

\end{itemize} 
\noindent Furthermore, define the inner frame ratios by\begin{align*}
    &P:=\frac{A_1}{A_0}& &Q:=\frac{B_1}{B_0}& &R:=\frac{C_1}{C_0}& &S:=\frac{D_1}{D_0}\cdotp&
\end{align*} 
\subsubsection{\textbf{Triangular Portions Adjacent to the Outer Frame of} $\mathcal{W}$}\hfill\\\noindent Next, define four further portions of $W_\K(\mathbf{S})$ as follows:\begin{align*}
    \mathfrak{P}_{south}&:= \left\{W_\K(\mathbf{S})[m+l+3+i, n+1+j]:{0\le i \le \left\lfloor\frac{l+1}{2}\right\rfloor, ~i \le j \le l+1-i}\right\}\\
    \mathfrak{P}_{east}&:= \left\{W_\K(\mathbf{S})[m+1+i, n+l+3+j]:{0\le j \le \left\lfloor\frac{l+1}{2}\right\rfloor, ~ j\le i \le l+1-j}\right\}\\
    \mathfrak{P}_{north}&:= \left\{W_\K(\mathbf{S})[m-i,n+1+j]:{0\le i \le \left\lfloor\frac{l+1}{2}\right\rfloor, ~i \le j \le l+1-i}\right\}\\
    \mathfrak{P}_{west}&:= \left\{W_\K(\mathbf{S})[m+1+i, n+1-j]:{0\le j \le \left\lfloor\frac{l+1}{2}\right\rfloor, ~ j\le i \le l+1-j}\right\}
\end{align*}

\noindent The definitions from this subsection thus far are illustrated in the following diagram:\begin{figure}[H]
    \centering
    \includegraphics[width=1\linewidth]{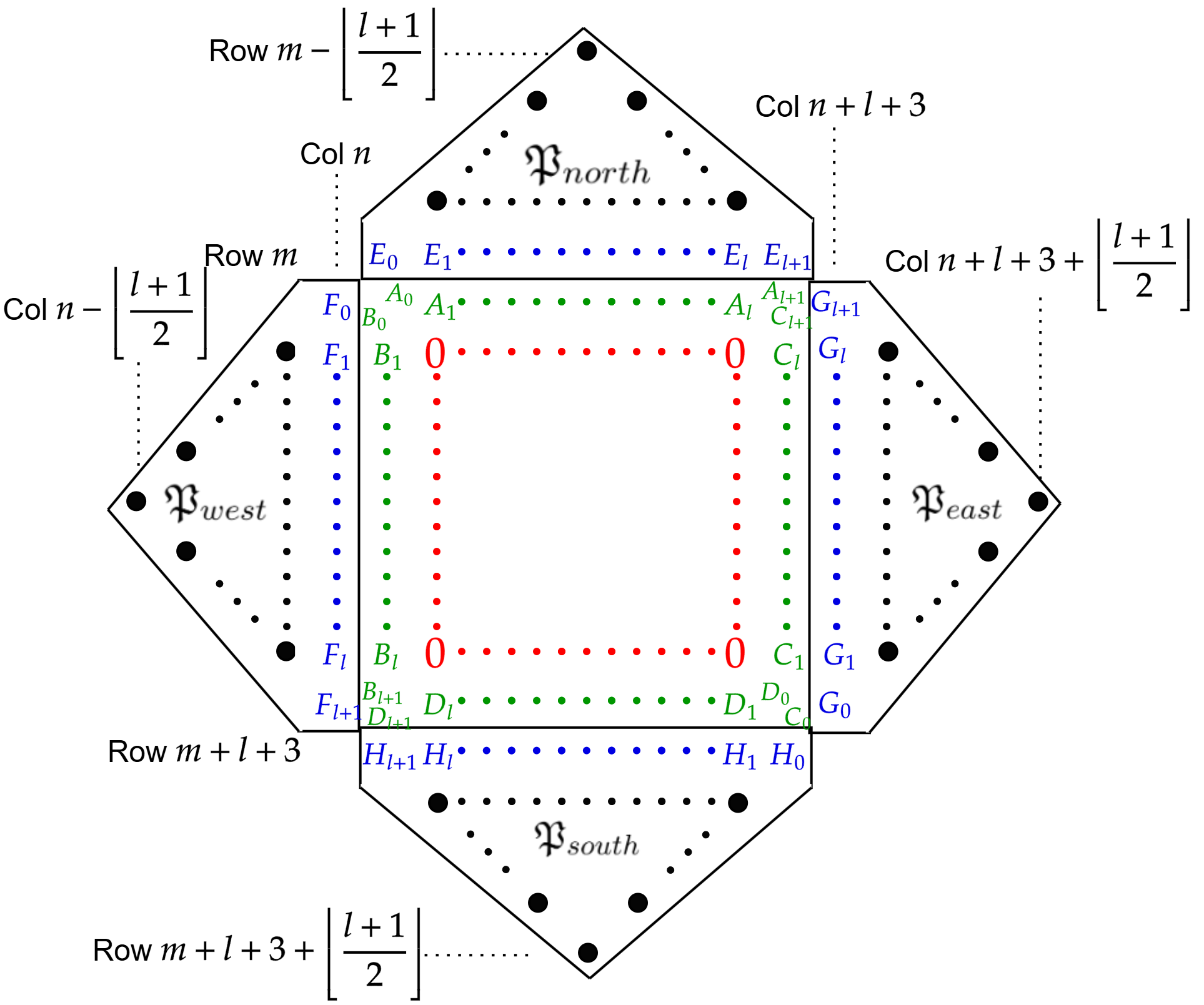}
    \caption{An illustration of the window $\mathcal{W}$ (red square of zeroes), its inner and outer frame (green and blue, respectively) and the portions $\mathfrak{P}_{north}$, $\mathfrak{P}_{east}$, $\mathfrak{P}_{south}$ and $\mathfrak{P}_{west}$ (black outlines).} 
    \label{fig: rot_lem}
\end{figure}

\begin{lemma}\label{lem: rotate}
    Consider the $(l+4)\times (l+4)$ square portion of $W_\K(\mathbf{S})$ as defined above. Then, \begin{align*}
        &\hypertarget{lem1}{1)}~~ W_\K(\textbf{S})[m+l+3+i,n+l+2-j]={W_\K^{(S,D_0)}(\textbf{H})}[i,j]\\&~~~~~\text{ for }~~{0\le i \le \left\lfloor\frac{l+1}{2}\right\rfloor, i \le j \le l+1-i};\\
        &\hypertarget{lem2}{2)}~~ W_\K(\textbf{S})[m+l+2-i,n+l+3+j] = W_\K^{(R,C_0)}(\textbf{G})[j,i]\\&~~~~~\text{ for }~~{0\le j \le \left\lfloor\frac{l+1}{2}\right\rfloor, j \le i \le l+1-j};\\
        &\hypertarget{lem3}{3)}~~ W_\K(\textbf{S})[m-i,n+1+j]={W_\K^{(P,A_0)}(\textbf{E})}[i,j]\\&~~~~~\text{ for }~~{0\le i \le \left\lfloor\frac{l+1}{2}\right\rfloor, i \le j \le l+1-i};\\
        &\hypertarget{lem4}{4)}~~ W_\K(\textbf{S})[m+1+i,n-j] = W_\K^{(Q,B_0)}(\textbf{F})[j,i]\\&~~~~~\text{ for }~~{0\le j \le \left\lfloor\frac{l+1}{2}\right\rfloor, j \le i \le l+1-j}.
    \end{align*}
    \noindent In other words, \begin{align*}
        &1)~~ \mathfrak{P}_{south}=\mathcal{V}\left(\left(W_\K^{(S,D_0)}(\textbf{H})[i,j]\right)_{{0\le i \le \left\lfloor\frac{l+1}{2}\right\rfloor, i \le j \le l+1-i}}\right),&\\
        &2)~~ \mathfrak{P}_{east}=\rho^{-1}\left(\left(W_\K^{(R,C_0)}(\textbf{G})[i,j]\right)_{0\le i \le \left\lfloor\frac{l+1}{2}\right\rfloor, i \le j \le l+1-i}\right),&\\
        &3)~~ \mathfrak{P}_{north} = \mathcal{H}\left(\left(W_\K^{(P,A_0)}(\textbf{E})[i,j]\right)_{0\le i \le \left\lfloor\frac{l+1}{2}\right\rfloor, i \le j \le l+1-i}\right),&\\
        &4)~~ \mathfrak{P}_{west}=\rho\left(\left(W_\K^{(Q,B_0)}(\textbf{F})[i,j]\right)_{0\le i \le \left\lfloor\frac{l+1}{2}\right\rfloor, i \le j \le l+1-i}\right).&
    \end{align*}
\end{lemma}

    \noindent The first step towards proving Lemma \ref{lem: rotate} is establishing that the triangular portions $\mathfrak{P}_{south}$, $\mathfrak{P}_{east}$, $\mathfrak{P}_{north}$ and $\mathfrak{P}_{west}$ depend only on the inner and outer frame of $\mathcal{W}$ that are adjacent to it. This is the goal of the following Lemma, which is illustrated immediately after it is stated. In this upcoming lemma, the $\pm$ and $\mp$ symbols are used to indicate that there are two separate cases.
\begin{lemma} \label{lem: non_zero_row}Let $\K$ be a field and let $\mathbf{S}$ be a doubly infinite sequence over $\K$. If $m,l\in\N$ with $l\ge5$ are such that for every $0\le i < l$, $W_\K(\mathbf{S})[m,n+i]\neq0$, then, the values of \begin{equation*}(W_\K(\mathbf{S})[m\pm j,n+i])_{2 \le j \le \left\lfloor\frac{l-1}{2}\right\rfloor,~ j-1\le i < l-j+1}\label{eqn: lem_7.6}\end{equation*} are calculated solely using $$(W_\K(\mathbf{S})[m\pm j,n+i])_{j \in\{0,1\},~0\le i <l}~~~\text{ and }~~~(W_\K(\mathbf{S})[m\mp 1,n+i])_{~1\le i <l-1}.$$ Similarly, the same is also true when the whole setup is rotated 90 degrees around its centre. That is, if for every $0\le i < l$, $(W_\K(\mathbf{S})[m+i,n])\neq0$, then the entries of $$(W_\K(\mathbf{S})[m+i+j,n\pm j])_{2 \le j \le \left\lfloor\frac{l-1}{2}\right\rfloor,~ j-1\le i < l-j+1}$$ are calculated using only $(W_\K(\mathbf{S})[m+i,n\pm j])_{j\in\{0,1\},~0\le i <l}~~~\text{ and }~~~ (W_\K(\mathbf{S})[m+i,n\mp1])_{0\le i <l}$.\end{lemma}
\noindent Lemma \ref{lem: non_zero_row} is depicted below.
\begin{figure}[H]
    \centering
    \includegraphics[width=1\linewidth]{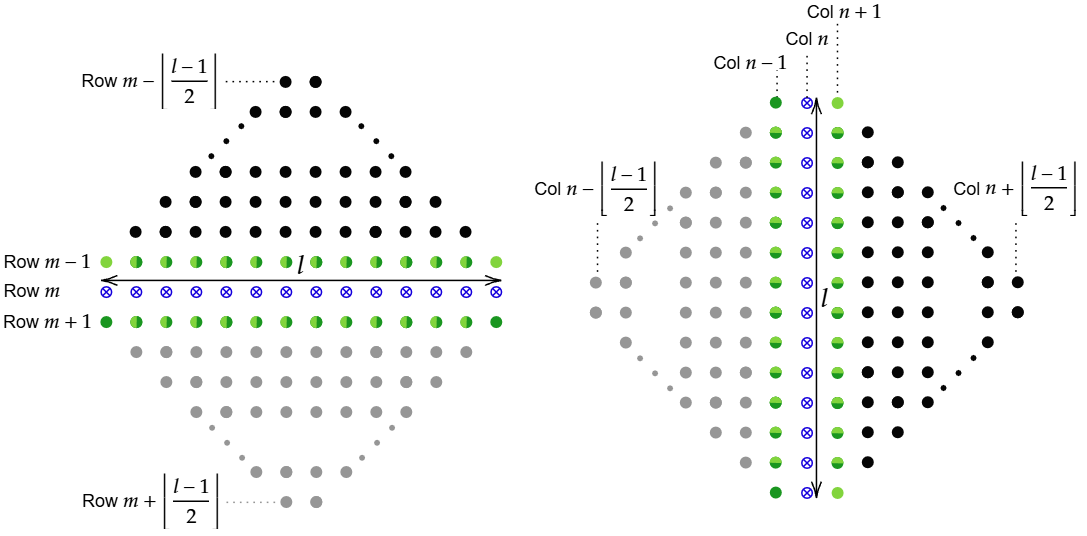}
    \caption{Each dot represents an entry in $W_\K(\mathbf{S})$, with the crossed blue dots being nonzero. Some dots have two different colours. The left image depicts the first part of Lemma \ref{lem: non_zero_row}, and the right image the second part. Lemma \ref{lem: non_zero_row} says that the grey (black, respectively) entries depend only on the dark (light, respectively) green  and blue entries.}
    \label{Fig: 7.2.1}
\end{figure}
\begin{proof}
    \noindent The `plus' case of the first part of Lemma \ref{lem: non_zero_row} is established first. The proof proceeds by induction on $j$, starting with $j=2$. For the base case, recall \hyperlink{FC1}{FC1}. By the assumption that $W_\K(\mathbf{S})[m,n+i]\neq0$ for all $0\le i <l$, \hyperlink{FC1}{FC1} implies $W_K(\mathbf{S})[m+2,n+i]$ depends only on $W_\K(\mathbf{S})[m+ k,n+i]$ for $k\in\{0,1\}$ and $0\le i <l$, as required.\\
    
    \noindent For the induction step, let $j\le \left\lfloor\frac{l-1}{2}\right\rfloor$ be a natural number and assume that the statement of Lemma \ref{lem: non_zero_row} holds for $(W_\K(\mathbf{S})[m+j',n+i])_{2 \le j'<j ,~ j'-1\le i < l-j'+1}$. This is now split into two cases.\begin{itemize}
        \item[] \textbf{Case 1:}   $W_\K(\mathbf{S})[m+j-2,n+i]\neq0$;
        \item[] \textbf{Case 2:} $W_\K(\mathbf{S})[m+j-2,n+i]=0$.
    \end{itemize} 
    \noindent \textbf{Case 1:} The conclusion follows from \hyperlink{FC1}{FC1} as in the base case. \\
    
    \noindent \textbf{Case 2:} There exists some window $\mathcal{W}$ such that $W_\K(\mathbf{S})[m+ j-2,n+i]$ is part of $\mathcal{W}$. Two further subcases are now stated: \begin{itemize}
        \item[]\textbf{Case 2.1:}  $W_\K(\mathbf{S})[m+ j,n+i]=0$ by the Square Window Theorem (Theorem \ref{window});
        \item[] \textbf{Case 2.2:}  $W_\K(\mathbf{S})[m+ j,n+i]$ is part of the inner or outer frame of $\mathcal{W}$.
    \end{itemize}
    \noindent \textbf{Case 2.1:} Since $W_\K[m,n+i]\neq0$ for all $0\le i <\ell$, the top row of $\mathcal{W}$ has index less than $m$ and hence the conclusion follows.\\
    
    \noindent \textbf{Case 2.2:} By \hyperlink{FC2}{FC2} and \hyperlink{FC3}{FC3}, the value of $W_\K(\mathbf{S})[m+ j,n+i]$ depends only on the inner and outer frame of $\mathcal{W}$. Hence, if $W_\K(\mathbf{S})[m+ j,n+i]$ depends on some entry in $(W_\K(\mathbf{S})[m- k,n])_{k\ge2}$, either the inner or the outer frame of $\mathcal{W}$ are part of $(W_\K(\mathbf{S})[m- k,n])_{k\ge2}$. However, this contradicts the assumption that $W_\K(\mathbf{S})[m,n+i]\neq0$ for all $0\le i <l$. \\
    
    \noindent To complete Case 2.2, one must also establish that the value of $W_\K(\mathbf{S})[m+ j, n+i]$ does not depend on $W_\K(\mathbf{S})[m-1,n+i]$ for $i\in\{0,l-1\}$. This is shown by contradiction: indeed, if the calculation of $W_\K(\mathbf{S})[m+ j, n+i]$ depends on $W_\K(\mathbf{S})[m-1,n+i]$ for $i\in\{0,l-1\}$, then this implies the existence of a non-square window. The proof is completed when $i=0$, and the case where $i=l-1$ follows from symmetry of Lemma \ref{reflect}. The following diagram aids in the proof: \begin{figure}[H]
        \centering
        \includegraphics[width=0.5\linewidth]{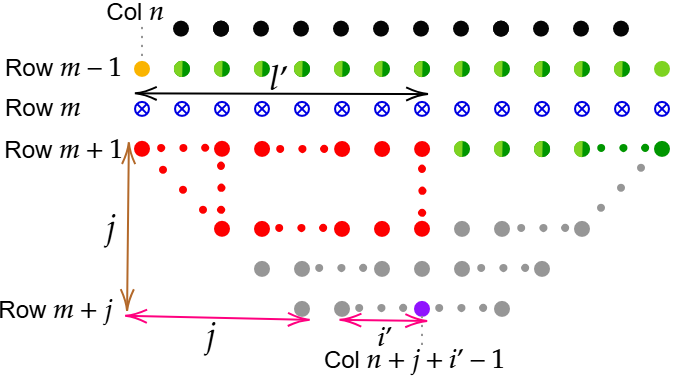}
        \caption{The left side of Figure \ref{Fig: 7.2.1}, with additional details that are described below.}
        \label{Fig :7.2.2}
    \end{figure}

    \noindent Assume that, for some $0\le i'<l-2j$, $W_\K(\mathbf{S})[m+j,n+j+i'-1]$ (the purple dot, above) depended on $W_\K(\mathbf{S})[m-1,n]$ (the yellow dot, above). This necessitates that $W_\K(\mathbf{S})[m-1,n]$ is in the north outer frame of some window, $\mathcal{W}$ of size $l'$ (red dots, above), as the difference in the row index is too great to apply \hyperlink{FC1}{FC1} (from the assumption that $j>2$), and since $W_\K(\mathbf{S})[m-1,n]$ being in any other part of the inner or outer frame of $\mathcal{W}$ would imply that $W_\K(\mathbf{S})[m,n+i']=0$ for some $0\le i' <l'$, contradicting the assumptions of the lemma. \\
    
    \noindent As illustrated by the brown arrow in Figure \ref{Fig :7.2.2}, The vertical length of $\mathcal{W}$ is $j-2$. However, as illustrated by the pink arrows in the same image, the horizontal length of $\mathcal{W}$ is $j+i'$. Therefore, by the Square Window Theorem, $i'=-2$, which is a contradiction. \\

\noindent The proof of the minus case, along with the second part of Lemma \ref{lem: non_zero_row} is fundamentally identical to the first case. The only difference is that one applies rearranged versions of the frame constraints \hyperlink{FC1}{FC1}, \hyperlink{FC2}{FC2} and \hyperlink{FC3}{FC3}.
\end{proof} 
\subsubsection{\textbf{Proof of Lemma \ref{lem: rotate}}}
\begin{proof}[\unskip\nopunct] 
\noindent \textbf{Lemma \ref{lem: rotate} Part {\protect\hyperlink{lem1}{1}}:} Lemma \ref{lem: non_zero_row} shows that the calculation of $\mathfrak{P}_{south}$ depends only on $\textbf{H}$, $\textbf{D}$ and the zeroes on row $m+l+1$. In particular, it does not depend on the value of $m$ nor any other parts of the number wall. Hence, one is able to change every other part of the number wall and, as long as $\textbf{H}$, $\textbf{D}$, and the zeroes on row $m+l+1$ remain in place, the values of every entry in $P_1$ remain unchanged. Therefore, by letting $m=-l-3$ and $n=-1$, and by setting all the entries on row $m'<m+l+3$ to be zero\footnote{The values of $F_{l+1}$ and $G_0$ also need to be changed to $F_{l+1}=S\cdot D_{l+1}$ and $G_0=S^{-1}\cdot D_0$ to match the geometric sequence $(D_i)_{0\le i \le l+1}$}, the calculation of $\mathfrak{P}_{south}$ becomes identical to the calculation of $\mathcal{V}{\left(W_\K^{(S,D_0)}(\textbf{H})[i,j]\right)}$ from the definition of $\mathcal{V}$ (equation (\ref{eqn: reflect_V})).\\

\noindent \textbf{Lemma \ref{lem: rotate} Part {\protect\hyperlink{lem2}{2}}:} As with $\mathfrak{P}_{south}$, Lemma \ref{lem: non_zero_row} shows that $\mathfrak{P}_{east}$ is fully determined given the eastern inner and outer frame of $\mathcal{W}$, along with the eastern most column of $\mathcal{W}$ itself.\\

\noindent For ease of notation, define $\mathfrak{P}_{east}[i,j]:=W_\K(\textbf{S})[m+l+2-i,n+l+3+j]$.\\

\noindent The goal is to establish \begin{equation}\mathfrak{P}_{east}[i,j] = W_\K^{(R,C_0)}(\textbf{G})[j,i]\label{eqn: lem_rot_1}\end{equation} for ${0\le j \le \left\lfloor\frac{l+1}{2}\right\rfloor, j \le i \le l+1-j}$. This is achieved by induction on $j$.\\

\noindent The base case is immediate: namely, $$\mathfrak{P}_{east}[i,0] = W_\K^{(R,C_0)}(\textbf{G})[0,i]=G_i$$ for all $0\le i \le l+1$. Hence, assume that equation (\ref{eqn: lem_rot_1}) holds for all $j\le u\in\N$ and $j\le i \le l+1-j$. \\

\noindent Then, consider the entry $\mathfrak{P}_{east}[i,u+1]$ for some $u+1\le i \le l-u$. This is calculated by rearranging either \hyperlink{FC1}{FC1}, \hyperlink{FC2}{FC2} or \hyperlink{FC3}{FC3}. The plan is to show that the calculation of $W_\K^{(R,C_0)}(\textbf{G})[u+1, i]$ is identical to the calculation of $\mathfrak{P}_{east}[i,u+1]$. This is split into 3 cases, depending on which Frame Constraint is required for the calculation.\\

\noindent \textbf{Case 1:} $\mathfrak{P}_{east}[i,u-1]\neq0$. \\
\noindent In this case, rearranging \hyperlink{FC1}{FC1} implies \begin{align*}
    \mathfrak{P}_{east}[i,u+1]~~&\substack{FC1\\=}~~\frac{\mathfrak{P}_{east}[i,u]^2-\mathfrak{P}_{east}[i-1,u]\cdot \mathfrak{P}_{east}[i+1,u]}{\mathfrak{P}_{east}[i,u-1]}\\
    &=\frac{W_\K^{(R,C_0)}(\textbf{G})[u, i]^2-W_\K^{(R,C_0)}(\textbf{G})[u, i-1]\cdot W_\K^{(R,C_0)}(\textbf{G})[u, i+1]}{W_\K^{(R,C_0)}(\textbf{G})[u-1,i]}\\
    &\substack{FC1\\=}~W_\K^{(R,C_0)}(\textbf{G})[u+1, i]
\end{align*} as required.\\

\noindent \textbf{Case 2:} $\mathfrak{P}_{east}[i,u-1]=0\neq \mathfrak{P}_{east}[i,u]$:\\
\noindent In this case, there is a window $\widetilde{\mathcal{W}}$ of size $1\le l'\le l+2$ in the first $u$ columns of $\mathfrak{P}_{east}$, and $\mathfrak{P}_{east}[i,u+1]$ is in the eastern outer frame of $\widetilde{\mathcal{W}}$. That is, $\mathfrak{P}_{east}[i,u+1]=\widetilde{G}_{\overline{k}}$ for some $0\le \overline{k}\le l+1$, as illustrated below: \begin{figure}[H]
    \centering
    \includegraphics[width=0.34\linewidth]{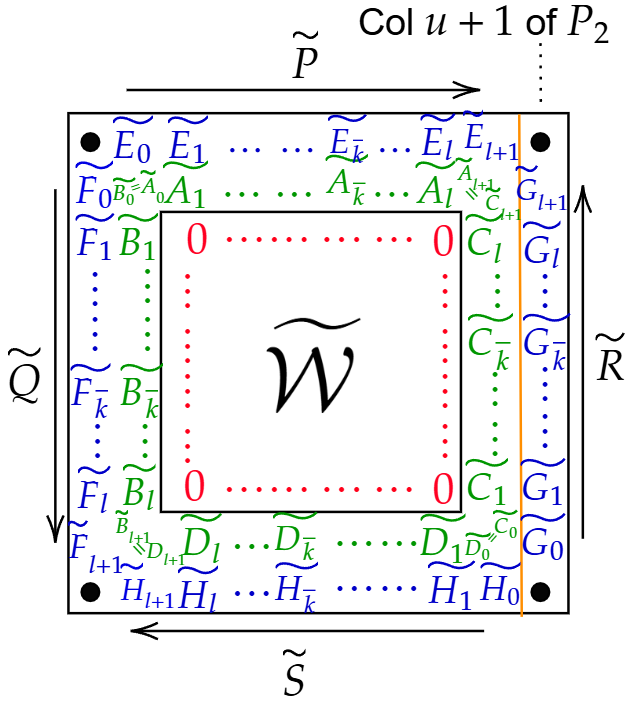}
    \caption{The window $\widetilde{\mathcal{W}}$ and its inner and outer frames. For now, ignore the orange line.}\label{Fig: 7.4.1_alt}
\end{figure}  
\noindent From the induction hypothesis and equation (\ref{eqn: lem_rot_1}), there is also a window $\mathcal{W}'$ of equal size in the first $u$ rows of $W_\K^{(R,C_0)}(\textbf{G})$ and $W_\K^{(R,C_0)}(\textbf{G})[u+1,i]=H'_{k'}$ for some $0\le k' \le l+1$. \begin{figure}[H]
    \centering
    \includegraphics[width=0.8\linewidth]{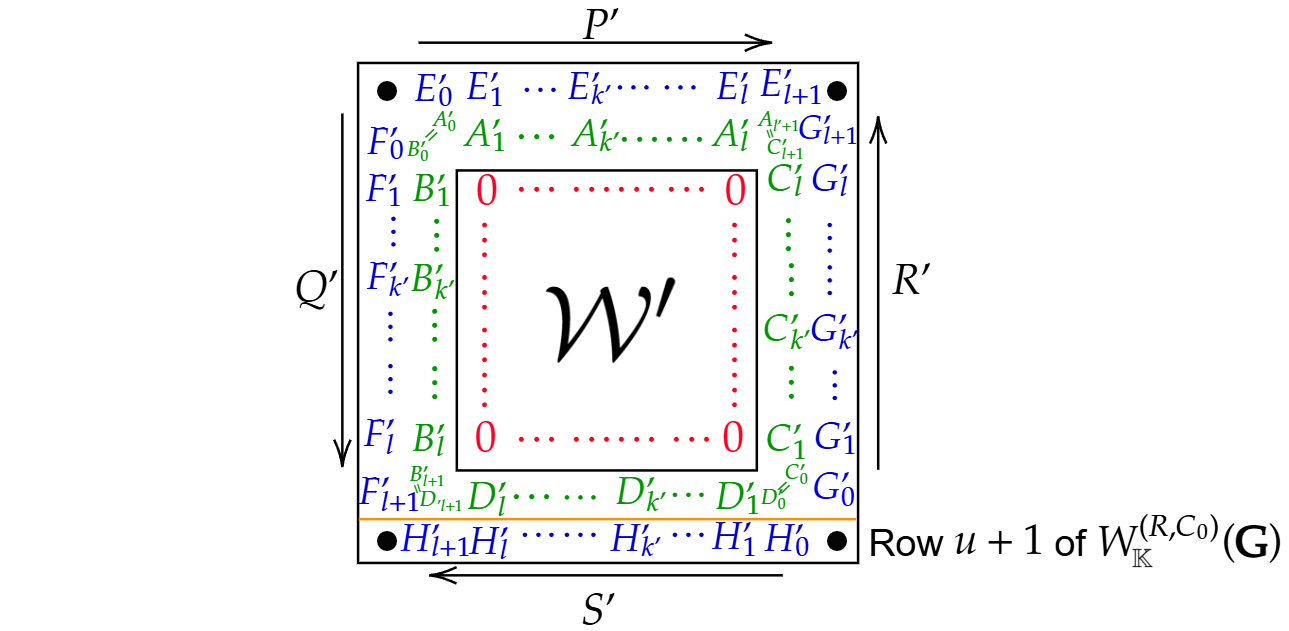}
    \caption{The window $\mathcal{W}'$ and its inner frame. For now, ignore the orange line.}
    \label{Fig: 7.4.2}
\end{figure} \noindent Also by the induction hypothesis, the portion of Figure \ref{Fig: 7.4.1_alt} to the left of the orange line is the $90^\circ$ anticlockwise rotation of the portion of Figure \ref{Fig: 7.4.2} above the orange line, as illustrated below: \begin{figure}[H]
    \centering
    \includegraphics[width=0.8\linewidth]{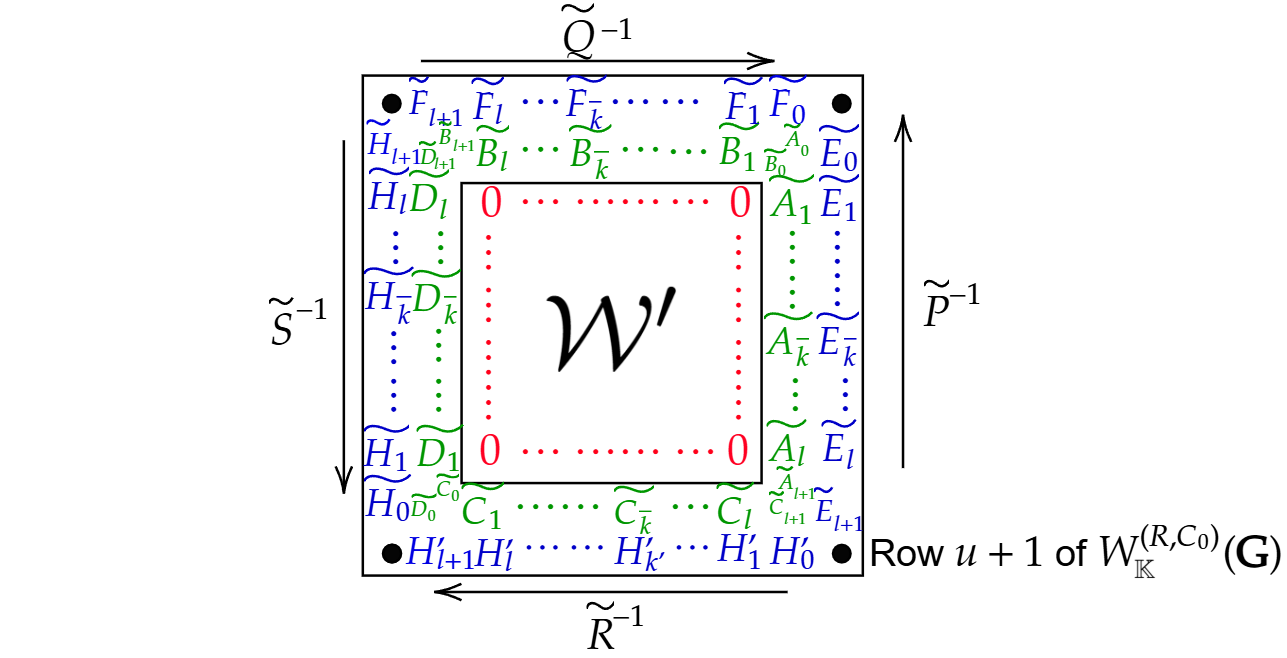}
    \caption{The window $\mathcal{W}'$ with its inner and outer frame relabelled with notation from $\widetilde{\mathcal{W}}$. Note, the south outer frame of $\mathcal{W}'$ (the sequence $(H_i')_{0\le i \le l+1}$) has not been relabelled.  }
\end{figure} Therefore, $\overline{k}=l+1-k'$ and
\begin{align*}
    &\widetilde A_{{\overline{k}}}=C'_{k'}&
    &\widetilde B_{{\overline{k}}}=A_{k'}'& &\widetilde C_{{\overline{k}}}=D'_{k'}& &\widetilde D_{{\overline{k}}}=B_{k'}'&  \\
    & \widetilde E_{{\overline{k}}}=G_{k'}'&
    &\widetilde F_{{\overline{k}}}=E_{k'}'& && &\widetilde H_{{\overline{k}}}=F_{k'}'& \\
    &\widetilde P^{-1}=R'&
    &\widetilde Q^{-1}=P'&  & \widetilde R^{-1}=S'&&\widetilde S^{-1}=Q'& 
\end{align*}
\noindent To complete the inductive step in this case, one needs to confirm that $\widetilde{G}_{\overline{k}}=H'_{k'}$. Indeed, \begin{align}
    \widetilde{G}_{\overline{k}} &~\substack{FC3\\=}~ \frac{(-1)^{\overline{k}}\frac{\widetilde Q\widetilde{ E}_{\overline{k}}}{\widetilde{ A}_{\overline{k}}}+\frac{\widetilde P\widetilde{ F}_{\overline{k}}}{\widetilde{ B}_{\overline{k}}}-(-1)^{\overline{k}}\frac{\widetilde R\widetilde{ H}_{\overline{k}}}{\widetilde{ D}_{\overline{k}}}}{\frac{\widetilde S}{\widetilde{ C}_{\overline{k}}}} \nonumber\\
    &=~ \frac{(-1)^{\overline{k}}\frac{(P')^{-1}G'_{k'}}{C'_{k'}}+\frac{(R')^{-1}E'_{k'}}{A'_{k'}}-(-1)^{\overline{k}}\frac{(S')^{-1}F'_{k'}}{B'_{k'}}}{\frac{(Q')^{-1}}{D'_{k'}}}. \label{eqn: rot_lem_2}
\end{align}

\noindent Substituting in $(P')^{-1}=(-1)^l\cdot \frac{S'}{R'Q'}$ from Theorem \ref{ratio ratio} and multiplying both the numerator and denominator by $R'Q'$ gives 
\begin{align}
    (\ref{eqn: rot_lem_2})&=\frac{(-1)^{\overline{k}-l}\frac{S'G'_{k'}}{C'_{k'}}+\frac{Q'E'_{k'}}{A'_{k'}}-(-1)^{\overline{k}}\frac{R'Q'F'_{k'}}{S'B'_{k'}}}{\frac{R'}{D'_{k'}}}.\label{eqn: rot_lem_3}
\end{align}
\noindent One again using $P'(-1)^{-l}=\frac{R'Q'}{S'}$ from Theorem \ref{ratio ratio} yields \begin{align}
    (\ref{eqn: rot_lem_3})&=\frac{(-1)^{\overline{k}+l}\frac{S'G'_{k'}}{C'_{k'}}+\frac{Q'E'_{k'}}{A'_{k'}}-(-1)^{\overline{k}-l}\frac{P'F'_{k'}}{B'_{k'}}}{\frac{R'}{D'_{k'}}}.\label{eqn: rot_lem_4}
\end{align} 
\noindent Finally, using that $(-1)^{l-\overline{k}}=(-1)^{\overline{k}-l}=(-1)^{\overline{k}+l}$, one obtains 
\begin{align*}
     (\ref{eqn: rot_lem_4})&=\frac{\frac{Q'E'_{k'}}{A'_{k'}}+(-1)^{l+1-\overline{k}}\frac{P'F'_{k'}}{B'_{k'}}-(-1)^{l+1-\overline{k}}\frac{S'G'_{k'}}{C'_{k'}}}{\frac{R'}{D'_{k'}}}\\
     &\substack{FC3\\=}~H'_{\bar{k}}
\end{align*}
\noindent as required, since $\overline{k}=l+1-k'$.\\

\noindent \textbf{Case 3:}  $\mathfrak{P}_{east}[i,u-1]=0= \mathfrak{P}_{east}[i,u]$:\\
\noindent This case is very similar to the previous one, with the only difference being that \hyperlink{FC2}{FC2} is used in place of \hyperlink{FC3}{FC3}. Once again, there is a window $\widetilde{\mathcal{W}}$ of size $1\le l'\le l+2$ in the first $u$ columns of $\mathfrak{P}_{east}$, but this time $\mathfrak{P}_{east}[i,u+1]$ is in the eastern \textit{inner} frame of $\widetilde{\mathcal{W}}$. That is, $\mathfrak{P}_{east}[i,u+1]=\widetilde{C}_{\overline{k}}$ for some $0\le \overline{k}\le l+1$. Also as before, the induction hypothesis implies and equation (\ref{eqn: lem_rot_1}) imply that there is also a window $\mathcal{W}'$ of equal size in the first $u$ rows of $W_\K^{(R,C_0)}(\textbf{G})$ and $W_\K^{(R,C_0)}(\textbf{G})[u+1,i]=D'_{k'}$ for some $0\le k' \le l+1$.\\

\noindent Therefore, the goal is to show that $\widetilde{C}_{\overline{k}}=D'_{k'}$, where $k'=l+1-\overline{k}$. This is achieved via the \hyperlink{FC2}{FC2}:\begin{align}
    \widetilde{C}_{\overline{k}}&~~\substack{FC2\\=}~~\frac{(-1)^{-l\cdot \overline{k}}\widetilde{A}_{\overline{k}}\cdot \widetilde{D}_{\overline{k}}}{\widetilde{B}_{\overline{k}}}\nonumber\\
    &~~=~~\frac{(-1)^{-l\cdot \overline{k}}C'_{k'}\cdot B'_{k'}}{A'_{k'}}\cdotp\label{eqn: rot_lem_5}
\end{align}
\noindent Next, note that by splitting into cases where $l$ is odd or even, one obtains that $(-1)^{l\cdot \overline{k}} = (-1)^{l(l+1-\overline{k})}=(-1)^{l\cdot k'}$. Hence
\begin{align*}
    (\ref{eqn: rot_lem_5})&=\frac{(-1)^{l\cdot k'}C'_{k'}\cdot B'_{k'}}{A'_{k'}}\\
    &\substack{FC2\\=}~~ D'_{k'}
\end{align*}as required.\\

\noindent \textbf{Lemma \ref{lem: rotate} Parts 3 and 4:} \\
\noindent The remainder of the proof is omitted, as is it extremely similar to that of $\mathfrak{P}_{east}$: once again, each part is split into three cases depending on which frame constraint is used. The first case is exactly the same method, up to the rearrangement of \hyperlink{FC1}{FC1}. Both the second and third cases are similar in method, as they imply the existence of two windows $\widetilde{\mathcal{W}}$ and $\mathcal{W}'$. The only difference is in the relationships between the inner and outer frames of $\widetilde{\mathcal{W}}$ and $\mathcal{W}'$. In part 3, one now has that  \begin{align*}
    &\widetilde A_{k}=D'_{k}&
    &\widetilde B_{k}=C_{k}'& &\widetilde C_{k}=B'_{k}& &\widetilde D_{k}=A_{k}'&  \\
    & &
    &\widetilde F_{k}=G_{k}'&  &\widetilde G_{k}=F_{l}'& &\widetilde H_{k}=E_{k}'& \\
    &\widetilde P=S'&
    &\widetilde Q=R'&  & \widetilde R=Q'&&\widetilde S=P'& 
\end{align*}and the goal is to show that $\widetilde{E}_k=H'_{k}$. Note that, in this case $\overline{k}=k'$ and so both notations are replaced by $k$. Similarly for part 4, one has that \begin{align*}
    &\widetilde A_{\overline{k}}=B'_{k'}&
    &\widetilde B_{\overline{k}}=D_{k'}'& &\widetilde C_{\overline{k}}=A'_{k'}& &\widetilde D_{\overline{k}}=C_{k'}'&  \\
    & \widetilde E_{\overline{k}}=F_{k'}'&
    && &\widetilde G_{\overline{k}}=E_{k'}'& &\widetilde H_{\overline{k}}=G_{k'}'&\\
    &\widetilde P^{-1}=Q'&
    &\widetilde Q^{-1}=S'&  & \widetilde P^{-1}=S'&&\widetilde R^{-1}=Q'& 
\end{align*} and the goal is to show that $\widetilde{F}_{\overline{k}}=H'_{k'}$.
\end{proof}

\subsection{Number Walls of One-Sided Sequences}\hfill\\

\noindent The following notation is introduced to save space on upcoming diagrams. \\

\noindent \textbf{Notation:} Let $n\in\N\cup\{\infty\}$ and let $\textbf{S}=(s_i)_{0\le i \le n}$ be a sequence. Furthermore, let $r,a\in\K$. Then, let \begin{equation}
    s^{(r,a)}_i=a\cdot r^i\cdot s_i.\nonumber
\end{equation}
\noindent As seen earlier in Section \ref{Sect: 3.4}, one can associate a sequence $\textbf{S}=(s_i)_{i\in\N}$ over $\K$ to a Laurent series $\Theta(t)=\sum_{i=0}^\infty s_i t^{-i}\in\K((t^{-1}))$. As $\K((t^{-1}))$ is a field, there exists a unique $\Xi(t)\in\K((t^{-1}))$ such that $\Theta(t)\cdot \Xi(t)=1$. Let $\textbf{U}$ be the sequence over $\K$ associated to $\Xi(t)$.\\

\noindent Recall the doubly infinite sequence $\textbf{S}^L$ defined in equation (\ref{eqn: ^L}). The goal of this section is to see how the number wall $W_\K(\textbf{S}^L)$ relates to $W_\K(\textbf{U})$.

\subsubsection{\textbf{Columns 0 and 1 of }$W_\K^{(r_0,a_0)}\left(\mathbf{S}{(r_1,a_1)}\right)$}\hfill\\
\noindent To begin, it is shown that the sequence $\textbf{U}$ appears in $W_\K(\textbf{S})$. 
\begin{lemma}\label{inv}
    Let $a_0,r_0,a_1,r_1\in\K$ such that $a_0,r_0\neq0$ and let $\mathbf{S}=(s_i)_{i\in\Z}$ be a sequence in $\K$ such that $s_0\neq0$ and $s_i=0$ for $i<0$. Furthermore, let $\textbf{U}=(u_i)_{i\ge0}$ be as above. Then,  \begin{equation}\label{eqn: inv1}W_\K^{(r_0,a_0)}\left(\left(\mathbf{S}{(r_1,a_1)}\right)^L\right)[m,0]= a_1s_0 \cdot \left(\frac{a_1s_0}{a_0}\right)^m\end{equation}and\begin{equation}c_{m}:=W_\K^{(r_0,a_0)}\left(\left(\mathbf{S}{(r_1,a_1)}\right)^L\right)[m,1]=-a_1r_1s_0^2\left(\frac{-s_0a_1r_1}{r_0a_0}\right)^{m} u_{m+1}.\label{eqn: inv2}\end{equation}
\end{lemma}
\begin{proof}
\noindent The setup is depicted below.
    \begin{figure}[H]
    \centering
    \includegraphics[width=0.8\linewidth]{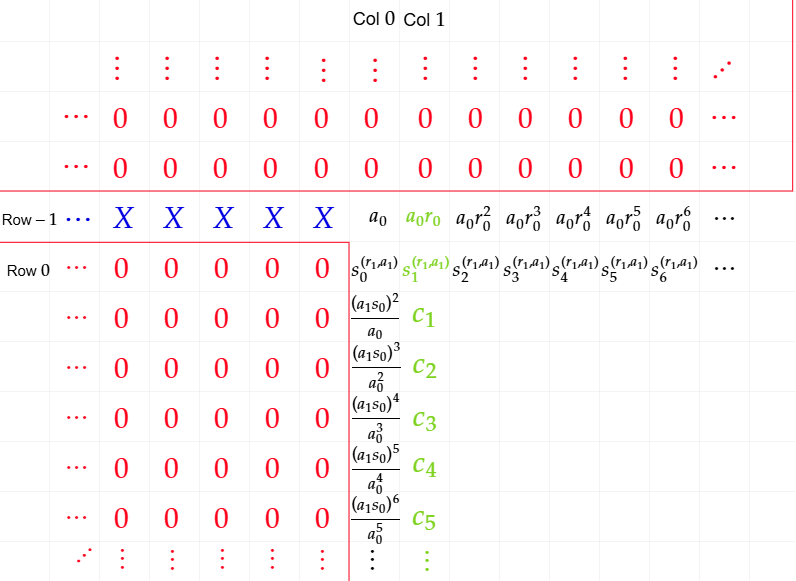}
    \caption{Each square in the grid is an entry in the number wall. The value of the entries denoted with a blue $X$ are easily calculable, but are not ever used in this proof.}\label{invwall1}
    \end{figure}
    \noindent Equation (\ref{eqn: inv1}) follows immediately from the fact that $\left(W_\K^{(r_0,a_0)}\left(\left(\mathbf{S}{(r_1,a_1)}\right)^L\right)[m,0]\right)_{m\ge0}$ is the eastern inner frame of the infinite window $\left(W_\K^{(r_0,a_0)}\left(\left(\mathbf{S}{(r_1,a_1)}\right)^L\right)[m,n]\right)_{n<0, m\ge0}$ and is hence a geometric sequence with initial values $W_\K^{(r_0,a_0)}\left(\left(\mathbf{S}{(r_1,a_1)}\right)^L\right)[-1,0]=a_0$ and \newline$W_\K^{(r_0,a_0)}\left(\left(\mathbf{S}{(r_1,a_1)}\right)^L\right)[0,0]=a_1s_0$.\\

    \noindent Moving on to equation (\ref{eqn: inv2}), one uses the determinant definition of a number wall  (Definition \ref{nw}) and Lemma \ref{rswall} to show that  \begin{equation}\label{invlem1}c_m=a_1r_1\left(\frac{a_1r_1}{r_0a_0}\right)^{m}\begin{vmatrix}
s_1 & s_2 & \dots & \dots & s_{m} & s_{m+1}\\
s_0 & s_1 & \ddots && &s_{m}\\
0 & s_0 & \ddots & \ddots&& \vdots\\
\vdots &\ddots & \ddots & \ddots&\ddots&\vdots\\
\vdots && \ddots & \ddots & \ddots &  s_2\\
0 & \dots & \dots & 0 & s_0 & s_1
    \end{vmatrix}.\end{equation}
    \noindent To conclude the proof, the value of the matrix determinant in (\ref{invlem1}) is shown to be $(-1)^{m+1}s_0^{m+2}u_{m+1}$ by induction. This formulae is already known as Wronski's Formulae \cite[Theorem 1.3]{Hen}, but, since it is elementary, the proof is repeated here for the sake of completion.\\
    
    \noindent First, it is elementary from the definition of $\textbf{U}$ to see that $u_i$ is generated by the following recurrence relation: \begin{equation}\label{birecur}
        u_0=s_0^{-1},~~~~~~~~~u_m=-\frac{\sum_{j=1}^{m} s_j u_{m-j}}{s_0}.
    \end{equation}The initial conditions of the induction are given by the case $m=0$: $$\det(s_1)=s_1\stackrel{(\ref{birecur})}{=}-\frac{s_0u_1}{u_0}=-s_0^{2}u_1.$$ For the inductive step, repeatedly take the determinant along the first column. This yields \begin{align}&\begin{vmatrix}
s_1 & s_2 & \dots & \dots & s_{m} & s_{m+1}\\
s_0 & s_1 & \ddots && &s_{m}\\
0 & s_0 & \ddots & \ddots&& \vdots\\
\vdots &\ddots & \ddots & \ddots&\ddots&\vdots\\
\vdots && \ddots & \ddots & \ddots &  s_2\\
0 & \dots & \dots & 0 & s_0 & s_1
    \end{vmatrix}\nonumber\\&=s_1 (-1)^{m}s_0^{m+1}u_{m} - s_0\begin{vmatrix}
s_2 & s_3 & \dots & \dots & s_{m} & s_{m+1}\\
s_0 & s_1 & \dots &\dots& \dots&s_{m-1}\\
0 & s_0 & \ddots & && \vdots\\
\vdots &\ddots & \ddots & \ddots&\ddots&\vdots\\
\vdots && \ddots & \ddots & \ddots &  s_2\\
0 & \dots & \dots & 0 & s_0 & s_1
    \end{vmatrix}\nonumber\\&=s_1(-1)^{m}s_0^{m+1}u_{m}- s_2s_0 (-1)^{m-1}s_0^{m} u_{m-1}+ s_0^2\begin{vmatrix}s_3 & s_4 & s_5&\dots &s_i\\
    s_0 & s_1 & s_2&\dots &s_{m}\\
    0 & s_0 & s_1 & \dots & s_{m-1}\\
    \vdots &&\ddots&\ddots&\vdots\\
    0&0&\dots&s_0&s_1\end{vmatrix}\nonumber\\
    &= (-1)^{m}s_0^{m+1}\sum_{j=1}^{m+1} s_j  u_{m+1-j}. \label{invlem2} \end{align}
    \noindent By the recurrence relation (\ref{birecur}),
    \begin{equation*}(\ref{invlem2})= (-1)^{m+1} s_0^{m+2}u_{m+1}.\end{equation*}
    Hence, equation (\ref{invlem1}) becomes \[c_m=-\left(\frac{-s_0a_1r_1}{r_0a_0}\right)^{m} a_1r_1s_0^2u_{m+1}\] as required.
\end{proof}
\subsubsection{\textbf{Lemma \ref{inv} Applied to Finite Windows}}\hfill\\
\noindent Lemma \ref{inv} stipulates that the sequence $\mathbf{S}{(r_1,a_1)}$ appear in row zero of the number wall. By thinking of the rows with index $m<-1$ as an infinite window $\mathcal{W}$, the sequence $\mathbf{S}{(r_1,a_1)}$ can be interpreted as the southern outer frame of $\mathcal{W}$. This motivates the following natural generalisation of Lemma \ref{inv}, in which $\mathbf{S}{(r_1,a_1)}$ is a finite sequence appearing in the southern outer frame of a finite window. 
\begin{corollary}\label{fininv}
    \noindent Let $\textbf{X}$ be a sequence over $\K$, and let $l\ge3$ be a natural number. Assume that $W_\K(\textbf{X})$ has two windows, $\mathcal{W}_1$ and $\mathcal{W}_2$, of size $l-2$ and $l$ respectively, such that the north east corner of the inner frame of $\mathcal{W}_1$ overlaps with the south west corner of the inner frame of $\mathcal{W}_2$, as depicted in Figure \ref{inv_wall0}. \\

\noindent Next, let $a_0,r_0,a_1,r_1\in\K$ be such that $a_0,r_0\neq0$, and let $\mathbf{S}=(s_i)_{i\ge0}$ be a sequence over a field $\K$ with $s_0\neq0$. Additionally, let $m$ and $n$ be the respective row and column index of the entry of the number wall that is in the inner frame of both $\mathcal{W}_1$ and $\mathcal{W}_2$. Assume that $$(W_\K(\textbf{X})[m,n+i])_{0\le i \le l}=(a_0r_0^i)_{0\le i\le l}~~\text{ and }~~(W_\K(\textbf{X})[m+1,n+i])_{0\le i \le l}=(a_1 r_1^is_i)_{0\le i \le l}.$$

\noindent Then, \begin{equation}\label{eqn: fininv1.1} \left(W_\K(\textbf{X})[m+i,n]\right)_{0\le i \le l-1}=\left(a_0 \cdot \left(\frac{a_1s_0}{a_0}\right)^i\right)_{0\le i \le l-1}\end{equation}and  \begin{equation}\label{eqn: fininv2.1}(d_{i})_{0\le i \le l}:= \left(W_\K(\textbf{X})[m+i,n+1]\right)_{0\le i \le l}=\left(a_0r_0s_0u_i\left(\frac{-s_0a_1r_1}{r_0a_0}\right)^{i} \right)_{0\le i \le l},\end{equation}\label{inv2}
\noindent where $\textbf{U}=(u_i)_{i\ge0}$ is the unique sequence over $\K$ such that $\sum_{i=0}^\infty s_i t^{-i} \cdot \sum_{j=0}^\infty u_it^{-i}=1$. Finally, one has 
\begin{align}
    &W_\K(\textbf{X})[m+j,n+i+1]=W^{\left(\frac{a_1s_0}{a_0},a_1s_0\right)}_\K\left(\textbf{U}^{\left(\frac{-s_0a_1r_1}{r_0a_0},a_0r_0s_0\right)}\right)[i,j]\label{eqn: fininv3.1}
\end{align}
\noindent for ${0\le i \le \left\lfloor\frac{l+1}{2}\right\rfloor}$ and $ i\le j \le l-i$.
\end{corollary} 
\begin{figure}[H]
    \centering
    \includegraphics[width=0.85\linewidth]{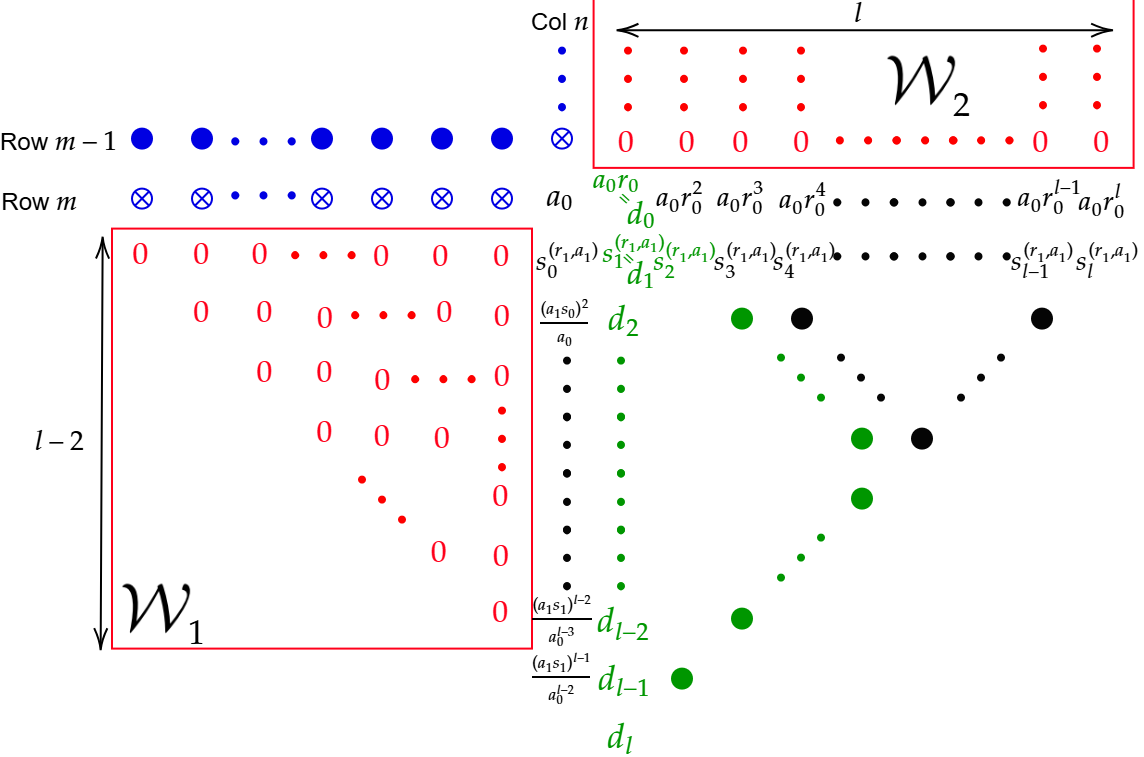}\caption{A drawing of the setup of Corollary \ref{inv2}. Above, $\mathcal{W}_1$ and $\mathcal{W}_2$ are the left an right red squares, respectively, and the crossed dots indicate that an unknown entry is known to not be equal to zero.}\label{inv_wall0}
\end{figure}
\begin{proof}
\noindent First, note that equation (\ref{eqn: fininv3.1}) is achieved by applying Lemma \ref{lem: rotate} with $\textbf{C}$ given by equation (\ref{eqn: fininv1.1}) and $\textbf{G}$ given by equation (\ref{eqn: fininv2.1}). With notation from Lemma \ref{lem: rotate}, the portion of the number wall detailed in equation (\ref{eqn: fininv3.1}) is $\mathfrak{P}_{east}$.\\

\noindent Furthermore, equation (\ref{eqn: fininv1.1}) is immediate as it was in Lemma \ref{inv}. \\

\noindent Therefore, all that remains is to establish equation (\ref{eqn: fininv2.1}). To do this, it is sufficient to show that the value of $d_i$ is independent of the choice of $m$, $n$ and any value denoted by a blue dot in Figure \ref{inv_wall0} (those being any entry of $(W_\K(\textbf{X})[i,j])_{i\le m-1, j\le n-1}$, $(W_\K(\textbf{X})[m,j])_{j\le n-1}$ or $(W_\K(\textbf{X})[i,n])_{i\le m-1}$). To see why this suffices, note that if this is the case, one can take $m=-1$, $n=0$ and pick the values of the blue dots so that Corollary \ref{inv2} follows immediately from Lemma \ref{inv} (where $d_i=c_{i-1}$ with $c_{i-1}$ being from Lemma \ref{inv}). \\

\noindent The proof proceeds by induction. In the base case, the calculation of $d_2$ is independent of $m$, $n$ or any entries marked with a blue dot by \hyperlink{FC1}{FC1}, since $a_0r_0\neq0$.\\

\noindent Assume by induction that $d_2,\dots,d_i$ are calculated without the use of $m$, $n$ and the entries marked by a blue dot. If $d_{i-1}\neq0$, the the independence of $d_{i+1}$ follows from \hyperlink{FC1}{FC1}.\\

\noindent Hence, assume $d_{i-1}=0$. For now, also assume that $d_i\neq0$. Therefore, $d_{i+1}$ is contained in the southern outer frame of a window, $\mathcal{W}_3$. Note that $d_{i+1}$ is in column $n$ and row $m+i+1$. The entry on the same row but column $n-1$ is nonzero (as it is in the eastern inner frame of $\mathcal{W}_3$), and therefore, $d_{i-1}$ is the bottom left corner of a window. Applying \hyperlink{FC3}{FC3} shows that the values of the inner and outer frame that are used in the calculation of $d_{i+1}$ do not include the blue dots. This is depicted below. \begin{figure}[H]
    \centering
    \includegraphics[width=0.85\linewidth]{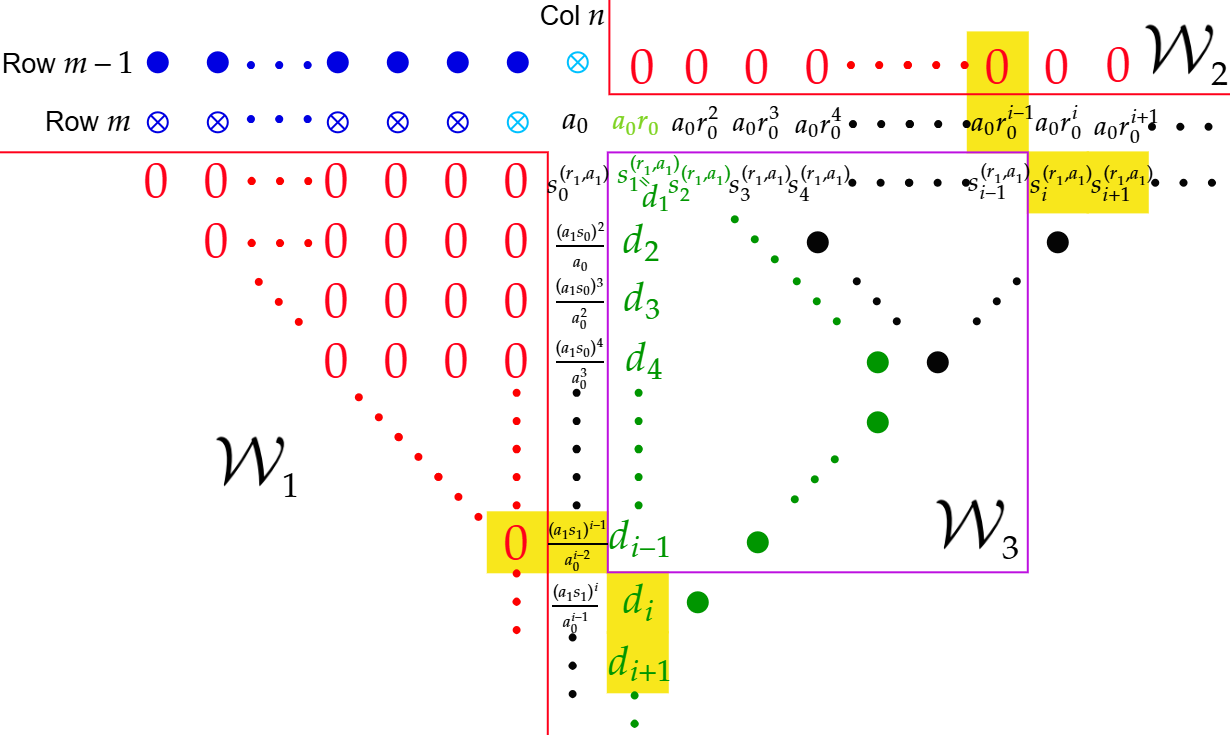}
    \caption{The purple square corresponds to $\mathcal{W}_3$. It is known that $d_{i-1}$ is the bottom right corner of $\mathcal{W}_3$. Depicted above is the \textit{maximum} possible size of $\mathcal{W}_3$ - its actual size is unknown. The entries of the inner and outer frame of $\mathcal{W}_3$ that are used to calculate $d_{i+1}$ in this case are highlighted in yellow. }
\end{figure}\noindent Let $l_3$ be the side length of $\mathcal{W}_3$. With notation from Figure \ref{Fig: basicwindow}, $d_{i+1}$ corresponds to $H_{l_3}$ in the outer frame of $\mathcal{W}_3$. Furthermore, the only blue dots that are contained in the inner or outer frame of $\mathcal{W}_3$ are the ones coloured in light blue ($W_\K(\textbf{X})[m,n-1]$ and $W_\K(\textbf{X})[m-1,n]$). These correspond to $E_0$ and $F_0$. By \hyperlink{FC3}{FC3}, the calculation of $H_{k}$ involves $E_0$ or $F_0$ if and only if $k=0$, which in this case it is not. Hence, $d_{i+1}$ is independent of any blue dot. \\

\noindent When both $d_{i}$ and $d_{i-1}$ are equal to zero, this either means that $d_{i+1}$ is within a window or one must apply \hyperlink{FC2}{FC2}. In the former case, $d_{i+1}=0$, and in the latter case the proof is identical to the case where $d_{i}\neq0$ but with \hyperlink{FC2}{FC2} applied instead of \hyperlink{FC3}{FC3}.\\

\noindent Therefore, the values of $d_i$ are independent of the entries in the number wall denoted by blue dots, as well as the choices of $m$ and $n$, as required.
\end{proof}
\subsubsection{\textbf{Lemma \ref{fininv} Adapted to Infinite Number Walls}}\hfill\\
\noindent By letting $a_0=r_0=a_1=r_1=1$ and taking $l\to\infty$, the following corollary\footnote{This corollary was discovered independently by Erez Nesharim and Viktoria Rudykh at a similar time. A similar result was also shown by Dress \cite[Page 5]{Dress} in 1965.} of Lemma \ref{fininv} is established. 
\begin{corollary}\label{lem: hor_inv_nw}
     Let $\mathbf{S}=(s_i)_{i\in\N}$ be a sequence over $\K$ such that $s_0=1$. Furthermore, define the sequence $\textbf{U}=(u_i)_{\ge0}$ such that \[\sum_{i=0}^\infty s_it^{-i} \cdot \sum_{i=0}^\infty u_i t^{-i}=1.\]Then, \begin{align*}
         &\left(W_\K\left(\mathbf{S}^L\right)[j,i]\right)_{j\ge-1, i\ge0}=\left(W_\K\left(\textbf{U}^{(-1,1)}\right)[i-1,j+1]\right)_{ j\ge-1,i\ge0}.
     \end{align*}
\end{corollary}
\noindent Lemma \ref{lem: hor_inv_nw} is illustrated below.\begin{figure}[H]
    \centering
    \includegraphics[width=0.6\linewidth]{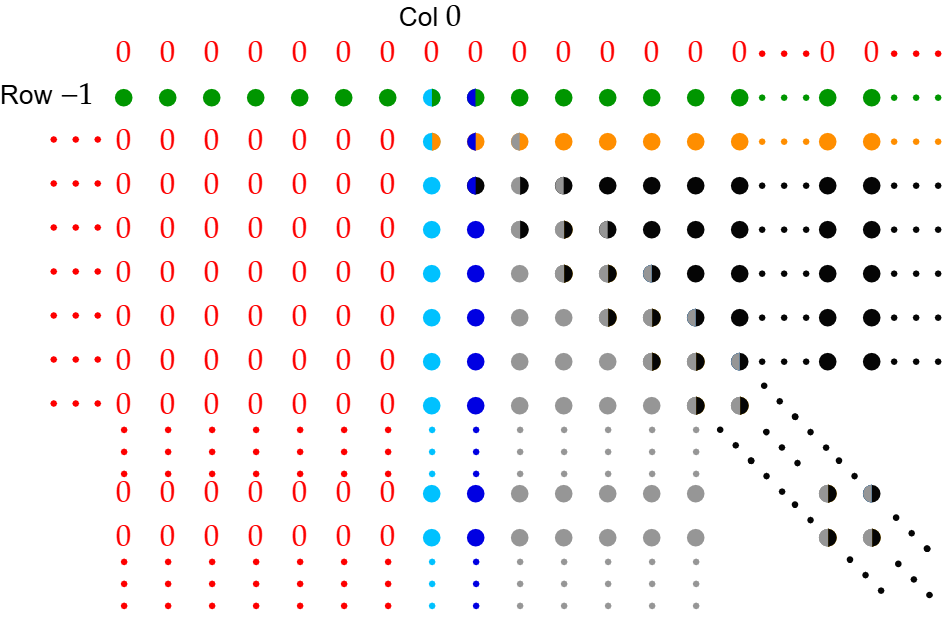}
    \caption{Each dot corresponds to an entry of $W_\K\left(\mathbf{S}^L\right)$. Note, some dots have multiple colours. Each green dot is equal to 1.  The orange dots are the sequence $\mathbf{S}$ and the black dots are the entries of $W_\K\left(\mathbf{S}\right)$. Then, Lemma \ref{lem: hor_inv_nw} states that the light blue dots are all equal to 1, the dark blue dots are the sequence $(-1)^i\cdot  u_i$ and the grey dots the entries of $W_\K\left(U^{\left(-1,1\right)}\right)$, reflected horizontally and then rotated clockwise $90^\circ$.}
\end{figure}

\section{Lemmata on the Number Walls of the $p$-Cantor and $p$-Singer Sequences}\label{Sect: 8}
\noindent This section specialises the lemmata from Section \ref{sect: 7} to the case where the input sequence is either the $p$-Cantor or the $p$-Singer sequence. This additional specificity allows for stronger results that build upon the work in Section \ref{sect: 7} and are crucial for the proof of Theorem \ref{thm: morphism}. From now on, it is necessary to replace the generic field $\K$ with the finite field $\F_p$. Because of this, the following classical result becomes relevant. 
\begin{theorem}[Fermat's Little Theorem (\hypertarget{FLT}{FLT})]
    Let $p$ be a prime. Then, for every $h\in\N$ and $x\in\F_p\backslash\{0\}$, $x^{p^h-1}=1$.
\end{theorem}

\subsection*{A Note on Presentation}
\noindent Some of the upcoming lemmata have complex statements. To make each lemma as simple as possible, such results are presented in bullet point format. Additionally, every lemma is illustrated immediately after it is stated. The reader is strongly encouraged to look at the corresponding picture alongside reading the statement of the lemma.  Finally, for $h\in\N$, recall that $\textbf{C}^{(p)}_h=(c^{(p)}_i)_{0\le i <p^h}$ and define \begin{equation}\textbf{S}_h^{(p)}=(s^{(p)}_i)_{0\le i <p^h+2}.\label{eqn: S_n}\end{equation} 
\noindent All the sequences (and their variants) that are used in this paper can be found in Appendix \hyperlink{Sect: Appendix_A}{A}.
\subsection{Number Wall of the One-Sided $p$-Cantor Sequence}\hfill\\
\noindent The first order of business is to establish a version of Lemma \ref{fininv} that is specific to the $p$-Cantor sequence. Both Lemma \ref{fininv} and the upcoming lemma are used in the proof of Theorem \ref{thm: morphism}. 
\begin{lemma}\label{lem: cant_inv_wall}Let\begin{itemize}
    \item $\textbf{X}$ be a sequence over $\F_p$ and $h\in\N$,
    \item $a_0,r_0,a_1,r_1\in\F_p$ such that $a_0,r_0\neq0$,
    \item  $\mathcal{W}_1$ and $\mathcal{W}_2$ be two windows in $W_p(\textbf{X})$ of size $p^h$ and $p^h-2$ respectively, such that \begin{itemize}
        \item  $\mathcal{W}_1$ and $\mathcal{W}_2$ are positioned so that the north east corner of the inner frame of $\mathcal{W}_1$ overlaps with the south west corner of the inner frame of $\mathcal{W}_2$,
        \item $m,n\in\N$ be the row and column index, respectively, of the entry of $W_p(\textbf{X})$ which is in the inner frame of both $\mathcal{W}_1$ and $\mathcal{W}_2$,
        \item the south inner frame of $\mathcal{W}_2$ is the $(r_0,a_0)$ geometric sequence with increasing indices (the sequence goes from \textup{left to right}). That is, $$(W_p(\textbf{X})[m,n+i])_{0\le i \le p^h-1}=(a_0r_0^i)_{0\le i\le p^h-1}.$$
        \item the south outer frame of $\mathcal{W}_2$ is the $(r_1,a_1)$ geometric transform of $\textbf{C}^{(p)}_h$ with increasing indices (the sequence goes from \textup{left to right}). That is, $$(W_p(\textbf{X})[m+1,n+i])_{0\le i \le p^h+1}=(a_1 r_1^ic^{(p)}_i)_{0\le i \le p^h+1}.$$
    \end{itemize}

\end{itemize}
\noindent Then\begin{itemize}
    \item  for every $0\le i \le p^h+1$, one has \begin{equation}\label{eqn: cantinv1} W_\K(\textbf{X})[m+i,n]=a_0\cdot \left(\frac{a_1}{a_0}\right)^i;\end{equation}
    \item for every $0\le i \le p^h+1$, one has\begin{equation}\label{eqn: cantinv2}d_{i}:= W_p(\textbf{X})[m+i,n+1]=a_0r_0\left(\frac{-a_1r_1}{r_0a_0}\right)^{i} s^{(p)}_{i};\end{equation}
    \item for every ${0\le i \le \left\lfloor\frac{p^h+2}{2}\right\rfloor}$ and $ i\le j \le p^h+1-i$, one has\begin{align}
    &W_p(\textbf{X})[m+j,n+i+1]=W^{\left(\frac{a_1}{a_0},a_0\right)}_p\left(\textbf{S}_h^{(p)}{\left(\frac{-a_1r_1}{r_0a_0},a_0r_0\right)}\right)[i,j].\label{eqn: cantinv3}
\end{align}
\end{itemize}

\end{lemma}
\noindent For the purposes of Figure \ref{fig: cantinv1}, $c^{(r_1,a_1)}_i=c_i^{(p)}\cdot r_1^i\cdot a_1$.
\begin{figure}[H]
    \centering
    \includegraphics[width=0.85\linewidth]{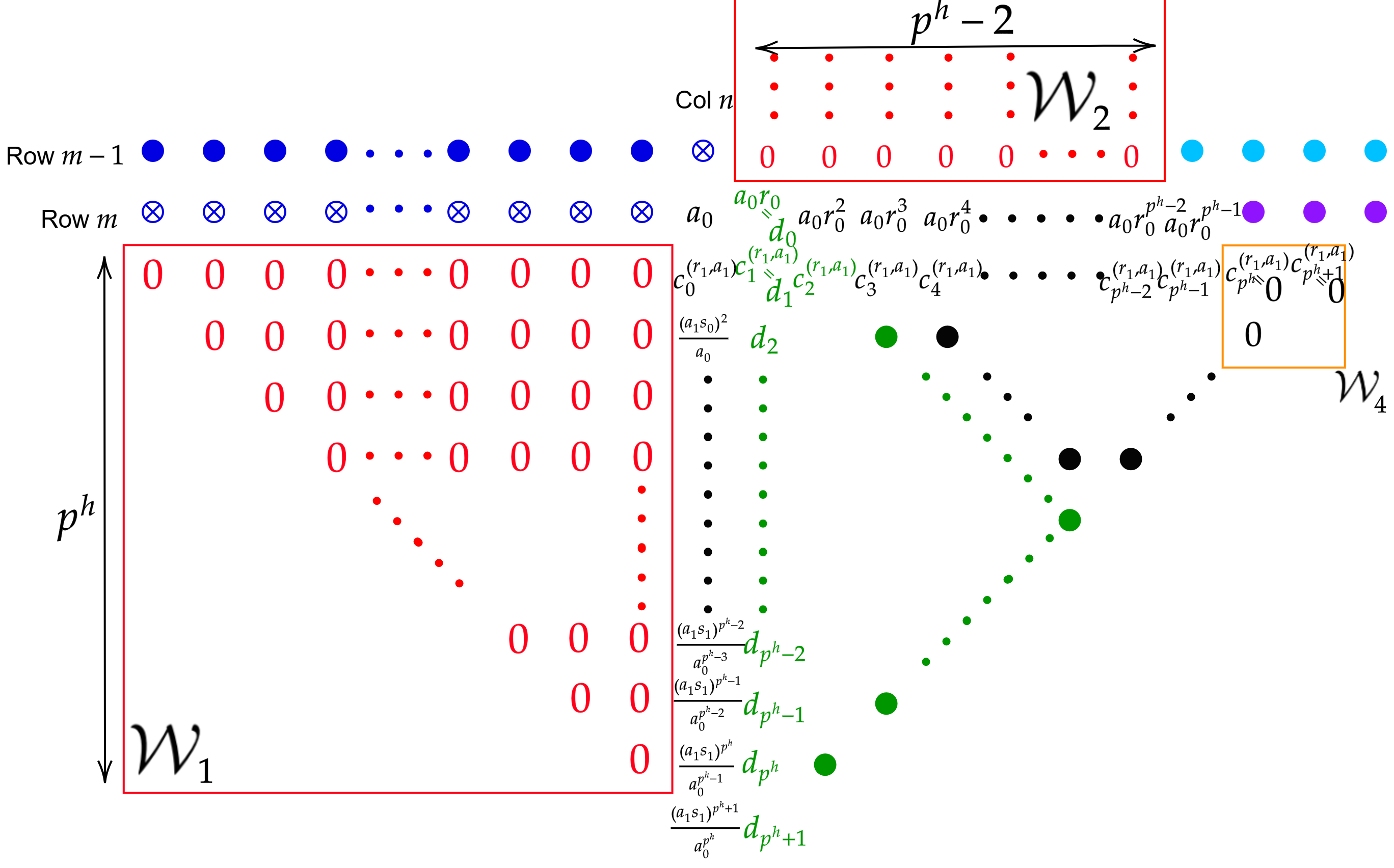}\caption{ An illustration of Lemma \ref{lem: cant_inv_wall}. Above, the crossed dots indicate that an unknown entry is known to not be equal to zero. The significance of the orange square is explained later.}\label{fig: cantinv1}
\end{figure}
\begin{proof}
    To begin, equation (\ref{eqn: cantinv1}) follows from the inner frame of $\mathcal{W}_1$ being a geometric sequence, and equation (\ref{eqn: cantinv3}) follows from equation (\ref{eqn: cantinv2}) and Lemma \ref{lem: rotate}. \\
    
    \noindent The proof of equation (\ref{eqn: cantinv2}) is very similar to that of equation (\ref{eqn: inv2}) in Corollary \ref{fininv} and is therefore mostly omitted, with only the differences being written. The goal is to show that the values of the light-blue, dark-blue or purple dots in Figure \ref{fig: cantinv1} are not used in the calculation of $(W_p(\textbf{X})[m+i,n+j])_{0\le i \le p^h+1,~ 0\le j \le p^h+1-i}$ (that is, the green and black dots in Figure \ref{fig: cantinv1}). This suffices, as in this case one can pick the light-blue dots to be equal to zero and the purple dots (from left to right) to equal $a_0r_0^{p^h}$, $a_0r_0^{p^h+1}$ and $a_0r_0^{p^h+2}$ respectively and then the conclusion follows from Corollary \ref{fininv}. \\

    \noindent There are two main differences from the proof of Corollary \ref{fininv}. The first comes from Lemma \ref{sym} and Corollary \ref{cor: p^h=0} which state that $$c^{(p)}_{p^h-1}=1\neq0= c^{(p)}_{p^h}=c^{(p)}_{p^h+1}.$$ That is, $c^{(p)}_{p^h}$ and $c^{(p)}_{p^h+1}$ are part of a window (call it $\mathcal{W}_4$ and see the orange square in Figure \ref{fig: cantinv1}), implying $(W_p(\textbf{X})[m+i, n+p^h-1])_{i=2,3}$ are in the inner frame of $\mathcal{W}_4$ and hence their value does not depend on the light-blue or purple entries from Figure \ref{fig: cantinv1}. The second difference is that the maximum size of $\mathcal{W}_3$ (recall the proof of Corollary \ref{fininv}) has decreased by 1, since $c_{p^h-1}\neq0$.
\end{proof}
~\\~

\subsection{The $p$-Cantor Sequence between Two Windows}\hfill\\

\noindent Recall the definition of $\oplus$ from equation (\ref{eqn: oplus}). The following lemma considers the number wall of a sequence of the form $$\widetilde{ \mathbf{C}}^{(p)}_h(r,a):=\{0\}_{0\le i \le p^h-1}\oplus \{a\cdot r^i\cdot c^{(p)}_i\}_{0 \le i <p^h} \oplus\{0\}_{0\le i \le p^h-1}.$$ 
\noindent Additionally, recall that $p^{(h)}_2:=\frac{p^h-1}{2}$.
\begin{lemma}\label{invwind}
Let\begin{itemize}
    \item $\textbf{X}$ be a sequence over $\F_p$ of arbitrary length;
    \item $l,m,n$ and $h$ be natural numbers,
    \item $r_0,a_0,r_1,a_1\in\F_p\backslash\{0\}$,
    \item $\mathcal{W}_1$, $\mathcal{W}_2$ and $\mathcal{W}_3$ be windows in $W_p(\textbf{X})$ with side lengths $p^h$, $p^h-2$ and $p^h$, respectively, such that
    \begin{itemize}
        \item $\mathcal{W}_1$ and $\mathcal{W}_3$ have their north west corner on row $m$ and that $\mathcal{W}_1$ is west of $\mathcal{W}_3$,
        \item the south west corner of the inner frame of $\mathcal{W}_2$ overlap with the north east entry of the inner frame of $\mathcal{W}_1$, and let this intersection take place on row $m-1$ and column $n$,
        \item the south east entry in the inner frame of $\mathcal{W}_2$ overlaps with the north west entry of the inner frame of $\mathcal{W}_3$,
        \item  the southern inner frame of $\mathcal{W}_2$ is given by the $(r_0,a_0)$ geometric sequence with decreasing indices (the sequence goes from \textup{right to left}). That is, for $0\le i <p^h$, $$W_p(\textbf{X})[m-1,n+i]=a_0r_0^{p^h-1-i}$$
        \item the southern outer frame of $\mathcal{W}_2$ is given by the $(r_1,a_1)$ geometric transform of $\textbf{C}^{(p)}_h$ with decreasing indices (the sequence goes from \textup{right to left}). That is, for $0\le i <p^h$, $$W_p(\textbf{X})[m,n+i]=a_1r_1^{p^h-1-i}c_{p^h-1-i}^{(p)}.$$
    \end{itemize}

\end{itemize}
Then,
\begin{enumerate}
\hypertarget{lem2.1}{\item}  The eastern outer frame of $\mathcal{W}_1$ (the sequence $(W_p(\textbf{X})[m+i-1,n+1])_{0\le i \le p^h+1}$) and the western outer frame of $\mathcal{W}_3$ (the sequence $(W_p(\textbf{X})[m+i-1,n+p^h-2])_{0\le i \le p^h+1}$) are both geometric transforms of $\textbf{S}_h^{(p)}$.
\hypertarget{lem2.2}{\item}  There is another window in $W_p(\textbf{X})$, denoted by $\mathcal{W}_4$, of size $p^h-2$ such that the north west (north east, respectively) corner of its inner frame overlaps with the south east (south west, respectively) corner of the inner frame of $\mathcal{W}_1$ (\,$\mathcal{W}_3$, respectively).
\hypertarget{lem2.3}{\item}The north outer frame of $\mathcal{W}_4$ is given by the $(r_1^{-1}, a_1)$ geometric transform of $\textbf{C}^{(p)}_h$ with increasing indices (the sequence goes from \textit{left to right}). That is, for $0\le i \le p^h-1$, $$W_p(\textbf{X})[m+p^h-1,n+i]=a_1r_1^{-i}c_i^{(p)}.$$ 
\vspace{-0.5cm}\hypertarget{lem2.4}{\item}  The north inner frame of $\mathcal{W}_4$ is given by the $\left(\frac{r_0}{r_1^2}, \frac{a_1^2}{a_0}\right)$ geometric sequence with increasing indices (the sequence goes from \textit{left to right}). That is, for $0\le i \le p^h-1$, $$W_p(\textbf{X})[m+p^h,n+i]=\frac{a_1^2r_0^i}{a_0r_1^{2i}}.$$
\hypertarget{lem2.5}{\item} when $a_0,r_0,a_1,r_1=1$, the region $(W_p (\textbf{X})[m+i,n+j])_{-1\le i \le p^h,~0\le j \le p^h-1}$ has horizontal and vertical symmetry. That is, for all $0\le i \le p^h+1$ and $0\le j \le p^h-1$, one has\[W_p (\textbf{X})[m+i-1,n+p^h-1-j]=W_p (\textbf{X})[m+i-1,n+j]\]and\[W_p (\textbf{X})[m+i-1,n+j]=W_p (\textbf{X})[m+p^h-i,n+j].\] 
\end{enumerate}

\end{lemma}

\begin{figure}[H]
    \centering
    \includegraphics[width=1\linewidth]{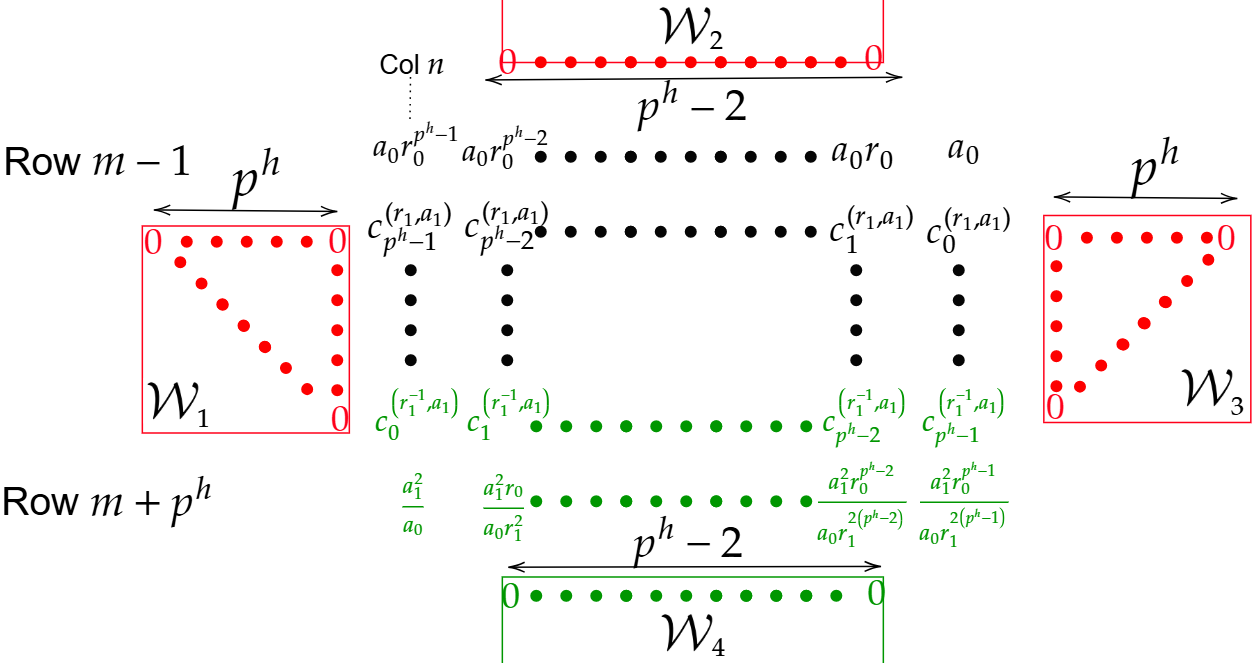}
    \caption{The windows $\mathcal{W}_1$, $\mathcal{W}_2$ and $\mathcal{W}_3$ are drawn in red, and the green entries are the ones to be calculated. The window $\mathcal{W}_4$ is also drawn in green. For the sake of not filling the diagram with entry space, the windows are not drawn to scale. }
    \label{Fig: invwind}
\end{figure}
\begin{proof}
\noindent As the proof is long, it is split into subsections. To begin, assume that $r_0,a_0,r_1,a_1=1$. The full version of Lemma \ref{invwind} is shown to follow from this case later.

\subsubsection*{\textbf{The Eastern Outer Frame of} $\mathcal{W}_1$}\hfill\\
\noindent Note that $c^{(p)}_{p^h}=c^{(p)}_{p^h+1}=0$ for all $h\in\N$ (Corollary \ref{cor: p^h=0}), and hence the conditions of Corollary \ref{lem: cant_inv_wall} are satisfied with $a_0,r_0,a_1,r_1=1$.  Therefore, one has $$W_p(\mathbf{X})[m+i-1,n+1]=(-1)^is^{(p)}_i$$ for $0\le i \le p^h+1$. That is, the first $p^h+1$ entries of the eastern outer frame of $\mathcal{W}_1$ (from top to bottom) is the $(-1,1)$-transform of the first $p^h-1$ entries of the $p$-Singer sequence. This establishes the part \hyperlink{lem2.1}{1} of Lemma \ref{invwind} for $\mathcal{W}_1$, in the case $r_0,a_0,r_1,a_1=1$. Similarly, \[W_\K(\textbf{X})[m+i-1,n]=1\] for $0\le i \le p^h+1$, as it is the geometric sequence with first two entries given by $W_p(\textbf{X})[m-1,n]=1$ and $W_p(\textbf{X})[m,n]=c^{(p)}_0=1$. 
\subsubsection*{\textbf{The Western Outer Frame of }$\mathcal{W}_3$}\hfill\\

\noindent Next, from the vertical symmetry of the $p$-Cantor sequence (Lemma \ref{sym}) and the inherited vertical symmetry of finite number walls (Lemma \ref{reflect}), one has that \[W_p(\mathbf{X})[m+i,n+j]=W_p(\mathbf{X})[m+i,n+p^h-1-j]\] for all $-1\le i \le p^h$ and $0\le j \le p^h+1$. That is, the portion of the finite number wall depicted by Figure \ref{Fig: invwind} satisfies vertical symmetry around its midpoint. Hence, one also has that \[W_p(\mathbf{X})[m+i-1,n+p^h-2]=W_p(\mathbf{X})[m+i-1,n+1]=(-1)^is^{(p)}_i\] for $0\le i \le p^h+1$. This establishes the remainder of part \hyperlink{lem2.1}{1} of Lemma \ref{invwind} in the case $r_0,a_0,r_1,a_1=1$.\\

\noindent Next, note that since $s^{(p)}_{2i+1}=0$ for all $i\in\N$, $\mathbf{S}{(-1,1)}(p)=\textbf{S}(p)$. Therefore, the setup is described by the following illustration.
\begin{figure}[H]
    \centering
    \includegraphics[width=1\linewidth]{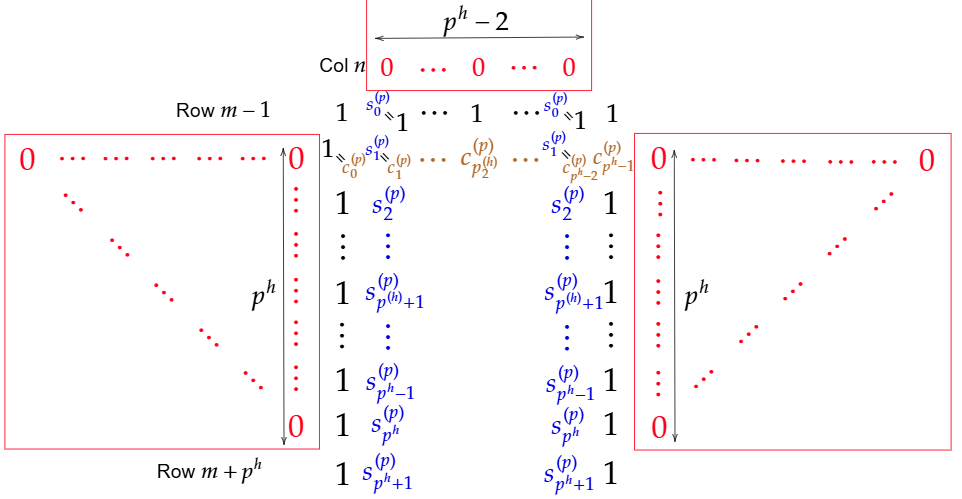}
    \caption{The same setup as Figure \ref{Fig: invwind}, with the additional information that the east/west outer frame of $\mathcal{W}_1$/$\mathcal{W}_3$ is given by the $p$-Singer sequence (blue).}
    \label{fig: 8.1.1}
\end{figure}
\subsubsection*{\textbf{Filing the Space Between Windows}}\hfill\\
\noindent Next, Lemma \ref{lem: rotate} is applied to the east (west, respectively) inner frame of $\mathcal{W}_1$ ($\mathcal{W}_3$, respectively) and to the south inner frame of $\mathcal{W}_2$. This implies that for all $0\le u\le p_2^{(h)}+1$ and for all $u\le v \le p^h+1-u$, one has \[W_p(\mathbf{X})[m+v-1,n+u+1]=W_p(\mathbf{X})[m+v-1,n+p^h-u-2]=W_p\left(\textbf{S}_h^{(p)}\right)[u,v],\] and similarly, for all $0\le u \le p_2^{(h)}$, $u\le v \le p^h-1-u$, one has \[W_p(\mathbf{X})[m+u,n+v]=W_p\left(\textbf{C}^{(p)}_h\right)[u,v].\] \noindent The portion of the number wall given by $(W_p(\mathbf{X})[m+v,n+u])_{-1\le v \le p^h,~ 0\le u \le p^h-1}$  is depicted below. 

\begin{figure}[H]
    \centering
    \includegraphics[width=0.7\linewidth]{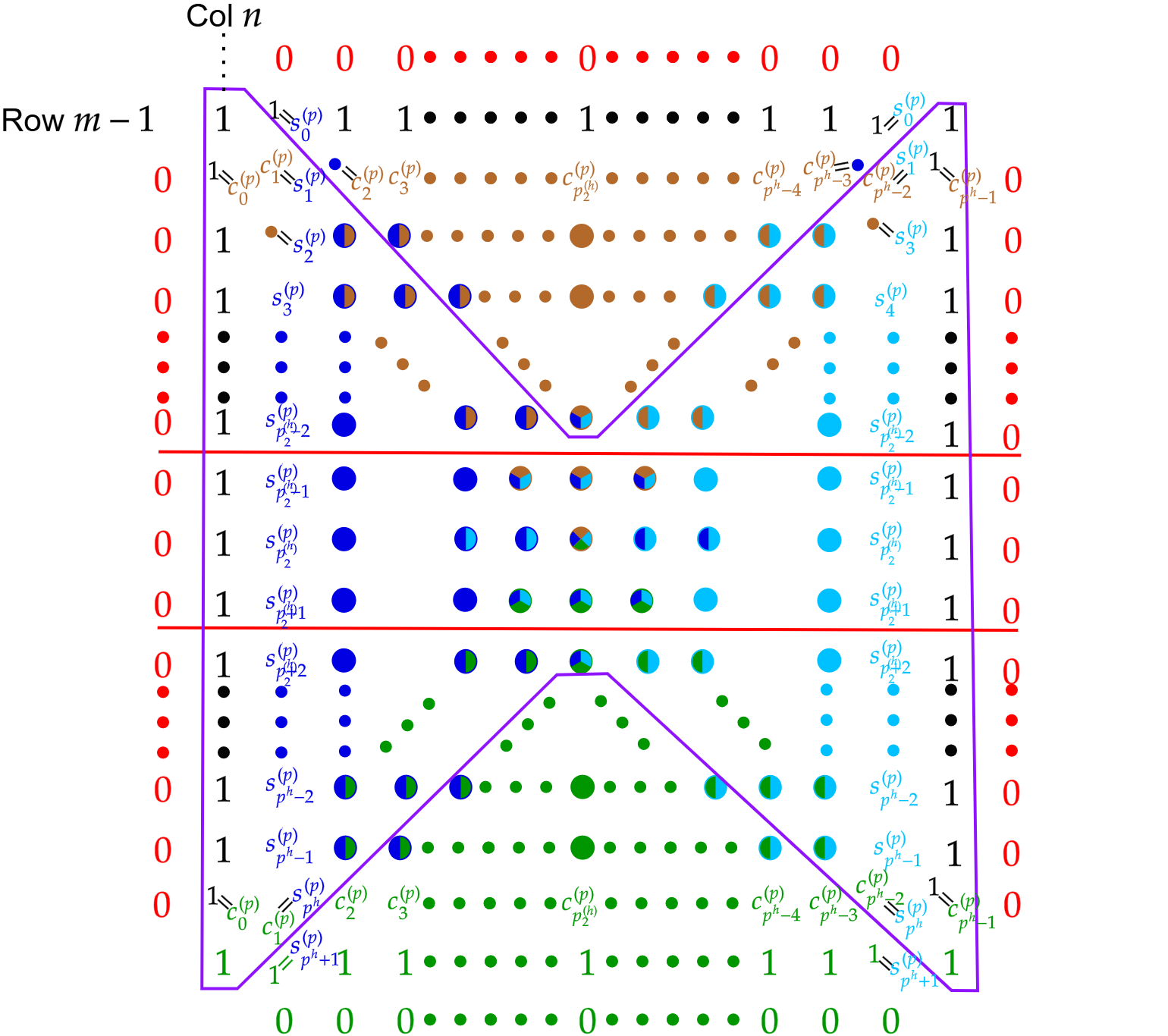}
    \caption{The portion of the number wall given by $(W_p(\mathbf{X})[m+v-1,n+u])_{-1\le v \le p^h,~ 0\le u \le p^h-1}$. The details of this figure are written below.}
    \label{big_n_complex}
\end{figure}
\noindent Figure \ref{big_n_complex} is split into 4 sections, denoted by the brown, light-blue, dark-blue and green entries. Some entries are given explicitly, whilst others are denoted by dots. Furthermore, some dots have more then one colour, indicating they are in more than one section. The brown part is given by $W_p(\textbf{C}^{(p)}_h)$. By Lemma \ref{lem: rotate}, the light and dark blue parts are given by $W_p(\textbf{S}_h^{(p)})$, rotated $90^\circ$ in opposite directions. Lemma \ref{between_two_windows} states that the green section is given by $W_p(\textbf{C}^{(p)}_h)$. Any of the green section outside the purple boundary is yet to be calculated. The inner frame of $\mathcal{W}_1$, $\mathcal{W}_2$ and $\mathcal{W}_3$ is black and the windows themselves are red. The red and purple lines have additional meanings, which are stated later in the proof. 
\subsubsection*{\textbf{Filling out Figure }\ref{big_n_complex}}\hfill\\
\noindent The goal is to establish part \hyperlink{lem2.5}{5} of Lemma \ref{between_two_windows}: that is, the green section of Figure \ref{big_n_complex} is a horizontal reflection of the brown section (which is given by $W_p\left(\textbf{C}^{(p)}_h\right)$. From this, parts \hyperlink{lem2.2}{2}, \hyperlink{lem2.3}{3}, and \hyperlink{lem2.4}{4} of Lemma \ref{between_two_windows} follow (when $r_0,a_0,r_1,a_1=1$). This is done by induction on the rows.\\

\noindent \textbf{The Base Case:} Lemma \ref{sym} states that $s^{(p)}_i=s^{(p)}_{p^h+1-i}$ for $0\le i \le p^h+1$. Hence, by Lemma \ref{reflect} $W_p\left(\textbf{S}_h^{(p)}\right)$ has vertical symmetry, and therefore, by Lemma \ref{lem: rotate}, the blue sections\footnote{Here, ``blue sections" refers to both the light and dark blue sections.} of Figure \ref{big_n_complex} are horizontally symmetric. That is, everything inside the purple border of Figure \ref{big_n_complex} has horizontal and vertical symmetry. In particular, this implies that the three rows in between the red lines in Figure \ref{big_n_complex} have horizontal symmetry. That is, row $m+p^{(h)}_2+1$ is identical to row $m+p^{(h)}_2-1$. This serves as the beginning of an induction on the rows.\\

\noindent \textbf{The Inductive Step:} For some $r\in\N$, assume the rows $m+p^{(h)}_2 \pm r$ rows are horizontally symmetrical. That is, row $m+p^{(h)}_2-r+i$ is identical to row $m+p^{(h)}_2+r-i$ for every $0\le i \le r$. Consider the rows with index $m+p^{(h)}_2\pm(r+1)$. There are two cases: to state them, let $x=W_p(S)[m+p^{(h)}_2+r+1,n+u]$ for some $0\le u \le p^h-1$ be an entry on row $m+p_2^{(h)}+r+1$. \\

\noindent \textbf{Case 1}: The entries of the number wall needed to calculate $x$ are contained between rows $m+p^{(h)}_2-1$ and $m+p^{(h)}_2+r$. \\

\noindent \textbf{Case 2:} Some of the entries of the number wall needed to calculate $x$ are found on a row with index less than $m+p^{(h)}_2-1$.

\subsubsection*{\textbf{Case 1:}}
\noindent Three further sub-cases are needed depending on which Frame Constraint is required to calculate $x$. For the sake of brevity, only the third subcase (corresponding to \hyperlink{FC3}{FC3}) is completed here, as this is the most complex of the three and the other two are very similar. The setup is illustrated below:\begin{figure}[H]
    \centering
    \includegraphics[width=0.45\linewidth]{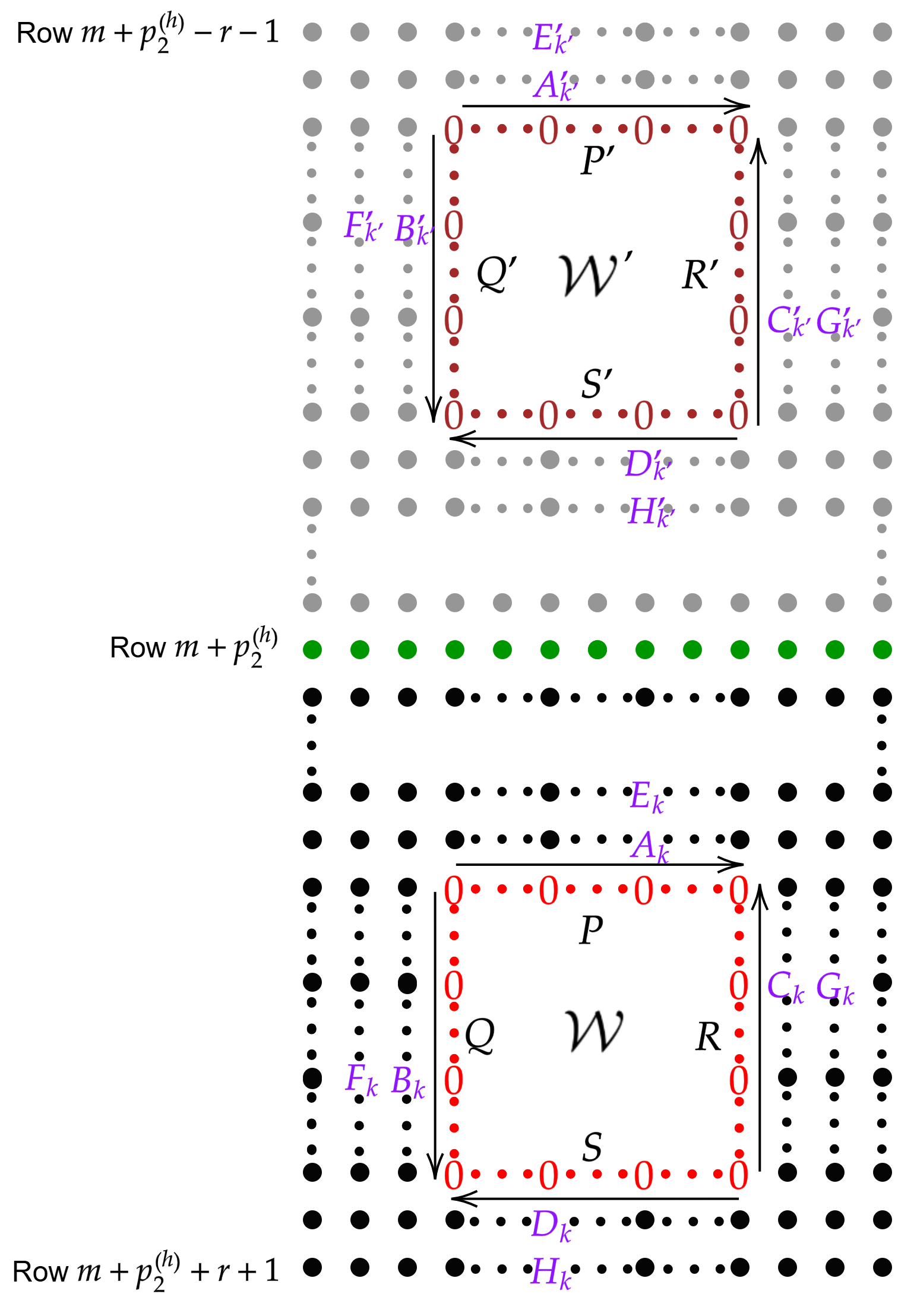}
    \caption{Row $m+p^{(h)}_2$ is coloured in green, with rows $m+p^{(h)}_2\pm i$ in black and grey respectively. Since {{\protect\hyperlink{FC3}{FC3}}} is required, $x$ is equal to $H_k$ in the outer frame of some window $\mathcal{W}$ (light red). By the symmetry from the induction hypothesis, there is also a window $\mathcal{W}'$ (dark red) as depicted. The direction of multiplication in the inner frames (recall, these are geometric sequences) are indicated with the black arrows.}
\label{fig: case_1_invwind}
\end{figure}
\noindent Using the notation from Figure \ref{fig: case_1_invwind}, the induction hypothesis about horizontal symmetry implies that \begin{align*}
    &A_k=D_{k'}'& &B_k=B_{k'}'& &C_k=C_{k'}'& &D_k=A'_{k'}&\\&E_k=H_{k'}'& &F_k=F_{k'}'& &G_k=G_{k'}'& \\
    &P=(S')^{-1}& &Q=(Q')^{-1}& &R=(R')^{-1}& &S=(P')^{-1}&
\end{align*}
\noindent where $k'=l+1-k$ and $l$ is the side length of $\mathcal{W}$. The goal is to show that $H_k=E_k'.$ Indeed, \begin{align}
    H_k&~\substack{\hyperlink{FC3}{FC3}\\=}~\frac{\frac{QE_k}{A_k}+(-1)^{k}\frac{PF_k}{B_k}-(-1)^k\frac{SG_k}{C_k}}{R/D_k}\nonumber\\ &=\frac{\frac{H'_{k'}}{Q'D'_{k'}}+(-1)^k \frac{F'_{k'}}{S'B'_{k'}}-(-1)^k\frac{G'_{k'}}{P'C'_{k'}}}{\frac{1}{R'A'_{k'}}}\cdot\label{eqn: invwind_case_1}
\end{align}Multiplying the numerator and denominator by $R'Q'$ yields \begin{align}
    (\ref{eqn: invwind_case_1})&= \frac{\frac{R'H'_{k'}}{D'_{k'}}+(-1)^k \frac{R'Q'F'_{k'}}{S'B'_{k'}}-(-1)^k\frac{R'Q'G'_{k'}}{P'C'_{k'}}}{\frac{Q'}{A'_{k'}}}\nonumber\\
    &\substack{\text{Theorem }\ref{ratio ratio}\\=} \frac{\frac{R'H'_{k'}}{D'_{k'}}+(-1)^{k-l} \frac{P'F'_{k'}}{B'_{k'}}-(-1)^{k-l}\frac{S'G'_{k'}}{C'_{k'}}}{\frac{Q'}{A'_{k'}}}~\substack{\hyperlink{FC3}{FC3}\\=}~E_{k'}'\nonumber
\end{align}where the final equality comes from $k'=l+1-k$ and $\hyperlink{FC3}{FC3}$.\\
\subsubsection*{\textbf{Case 2:}}
\noindent In this case, one has the same picture as in Figure \ref{fig: case_1_invwind}, but now the top row of $\mathcal{W}$ has index less than $m+p^{(h)}_2$ (that is, the top row is above the green dots). By horizontal symmetry, this means the bottom row of $\mathcal{W}'$ has index greater than $m+p^{(h)}_2$, and hence $\mathcal{W}$ and $\mathcal{W'}$ intersect. By the square window theorem (Theorem \ref{window}), $\mathcal{W}$ and $\mathcal{W}'$ are the same window with top row with index $m+p^{(h)}_2+r-1$ and bottom row with index $m+p^{(h)}_2-r+1$, as illustrated below:
\begin{figure}[H]
    \centering
    \includegraphics[width=0.6\linewidth]{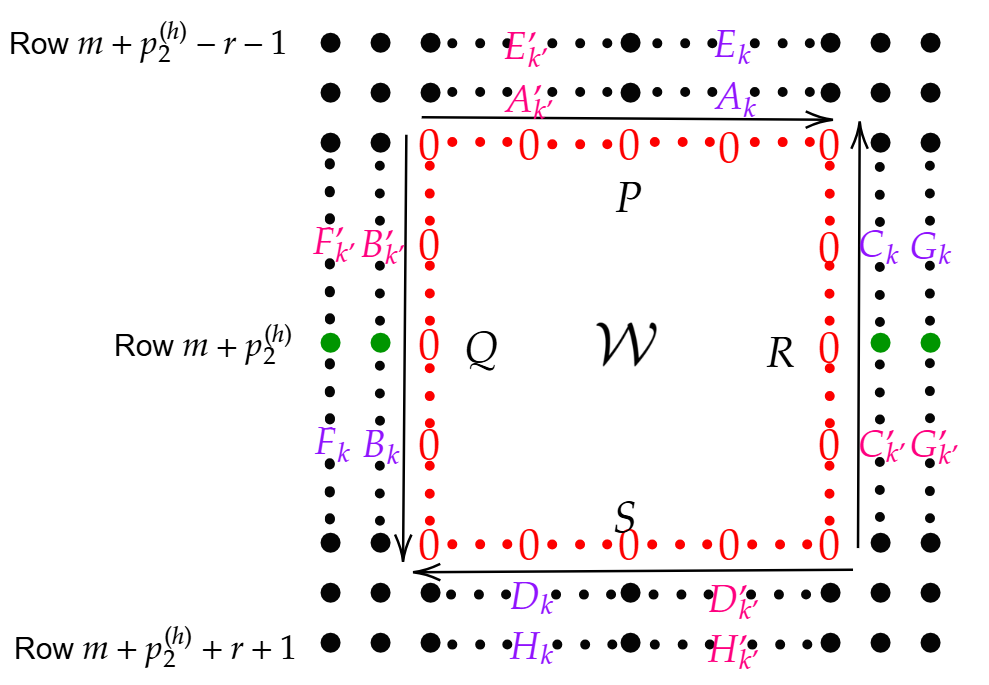}
    \label{fig: case_1_invwind_2}
    \caption{One window of size $2r+1$, denoted $\mathcal{W}$, has its `middle row' on row $m+p^{(h)}_2$ (green). Two different sets of inner and outer frame variables are labelled: one in purple and one in pink.  }
\end{figure}

\noindent In contrast to Case 1, there are only two sub-cases to be considered, given by which of \hyperlink{FC2}{FC2} and \hyperlink{FC3}{FC3} are needed to calculate $x$. For the same reasons as in Case 1, only the sub-case corresponding to \hyperlink{FC3}{FC3} is completed here. Indeed, in this sub-case $x$ is equal to $H_k$ for the window $\mathcal{W}$. The goal is to show that $H_k=E'_{k'}$, where $$k'=l+1-k=2r-k.$$
\noindent Using the notation from Figure \ref{fig: case_1_invwind_2}, the induction hypothesis implies that \begin{align*}
    &A_k=D_{k'}'& &B_k=B_{k'}'& &C_k=C_{k'}'& &D_k=A'_{k'}&\\&E_k=H_{k'}'& &F_k=F_{k'}'& &G_k=G_{k'}'& \\
    &P=S^{-1}&  &Q^2=1& &R^2=1&
\end{align*}
\noindent The latter two equalities come from the entries of the east/west inner frame of $\mathcal{W}$ on rows $m+p_2^{(k)}+i$ for $i\in\{-1,0,1\}$ and the horizontal symmetry from the inductive hypothesis. \\

\noindent Once again, the goal is to show $H_k=E_k'.$ Indeed, \begin{align}
    H_k&~\substack{\hyperlink{FC3}{FC3}\\=}~\frac{\frac{QE_k}{A_k}+(-1)^{k}\frac{PF_k}{B_k}-(-1)^k\frac{SG_k}{C_k}}{R/D_k}\nonumber\\ &=\frac{\frac{QH'_{k'}}{D'_{k'}}+(-1)^k \frac{F'_{k'}}{SB'_{k'}}-(-1)^k\frac{G'_{k'}}{PC'_{k'}}}{\frac{R}{A'_{k'}}}\cdot\label{eqn: invwind_case_3}
\end{align}Multiplying the numerator and denominator by $RQ$ and recalling that $R^2=Q^2=1$ yields \begin{align}
    (\ref{eqn: invwind_case_3})&= \frac{\frac{RH'_{k'}}{D'_{k'}}+(-1)^k \frac{RQF'_{k'}}{SB'_{k'}}-(-1)^k\frac{RQG'_{k'}}{PC'_{k'}}}{\frac{Q}{A'_{k'}}}\nonumber\\
    &\substack{\text{Theorem }\ref{ratio ratio}\\=} \frac{\frac{RH'_{k'}}{D'_{k'}}+(-1)^{k-l} \frac{PF'_{k'}}{B'_{k'}}-(-1)^{k-l}\frac{SG'_{k'}}{C'_{k'}}}{\frac{Q}{A'_{k'}}}=E_{k'}'\nonumber
\end{align}as in Case 1.\\
\subsubsection*{\textbf{Other Values of }$r_0,a_0,r_1$ and $a_1$:}\hfill\\ 
\noindent One now establishes parts \hyperlink{lem2.1}{1}, \hyperlink{lem2.2}{2}, \hyperlink{lem2.3}{3} and \hyperlink{lem2.4}{4} of Lemma \ref{invwind} in full generality. \\

\noindent The plan is to apply Lemma \ref{rswall} to the statement of Lemma \ref{invwind} when $r_0,a_0,r_1,a_1=1$. To do this, note that the proof of Lemma \ref{invwind} so far has been independent of the value of the row index $m$ and the column index $n$. Therefore, it is sufficient to establish the claims for any specific choice of $m$ and $n$. Indeed, let $m=0$ and $n=p^h$. In this case, one is considering the $\left(r_0^{-1},a_0r_0^{2p^h-1}\right)$-number wall of the $\left(r_1^{-1}, a_1r_1^{2p^h-1}\right)$-geometric transform of the sequence \[\widetilde{\mathbf{C}}^{(p)}_h:=\{0\}_{0\le i < p^h}\oplus \{c^{(p)}_i\}_{0 \le i <p^h} \oplus\{0\}_{0\le i < p^h},\] which itself is denoted $\widetilde{\mathbf{C}}_h^{(p)}{\left(r_0^{-1},a_0r_0^{2p^h-1}\right)}.$  For $0\le i \le p^h-1$ one has that\begin{align*}
        W_p^{\left(r_0^{-1},a_0r_0^{2p^h-1}\right)}&\left(\widetilde{\mathbf{C}}_h^{(p)}{\left(r_0^{-1},a_0r_0^{2p^h-1}\right)}\right)\left[p^h-1,p^h+i\right]\\ ~~\substack{\text{Lemma}~ \ref{rswall}\\=} &\frac{r_1^{-(p^h+i)(p^h)}(a_1r_1^{2p^h-1})^{p^h}}{r_0^{-(p^h+i)(p^h-1)}(a_0r_0^{2p^h-1})^{p^h-1}}W_p(\widetilde{\mathbf{C}}_h^{(p)})[p^h-1,p^h+i]\\\substack{\hyperlink{FLT}{FLT}\\=}~~~~&a_1 r_1^{-i}W_p(\widetilde{\mathbf{C}}_h^{(p)})[p^h-1,p^h+i]\\
        \substack{\text{Lemma }\ref{invwind}\\=}&~~a_1r_1^{-i} c^{(p)}_i,
    \end{align*}
    and similarly\begin{align*}
        W_p^{\left(r_0^{-1},a_0r_0^{2p^h-1}\right)}&\left(\widetilde{\mathbf{C}}^{(p)}{\left(r_1^{-1}, a_1r_1^{2p^h-1}\right)}\right)\left[p^h,p^h+i\right]\\ ~~\substack{\text{Lemma} \ref{rswall}\\=} &\frac{r_1^{-(p^h+i)(p^h+1)}(a_1r_1^{2p^h-1})^{p^h+1}}{r_0^{-(p^h+i)(p^h)}(a_0r_0^{2p^h-1})^{p^h}}W_p(\widetilde{\mathbf{C}}^{(p)})[p^h,p^h+i]\\\substack{\hyperlink{FLT}{FLT}\\=}~~~~&\frac{r_1^{-2i}a_1^2}{r_0^{-i}a_0}W_p(\widetilde{\mathbf{C}}^{(p)})[p^h-1+i]\\
        \substack{\text{Lemma }\ref{invwind}\\=}&~~\frac{r_0^{i}a_1^2}{r_1^{2i}a_0}
    \end{align*} as required.
\end{proof}

\subsection{The $p$-Singer Sequence Between two Windows}\hfill\\
\noindent The following lemma swaps the roles of the $p$-Cantor and $p$-Singer sequences in Lemma \ref{invwind}. As in Lemma \ref{invwind} for the purposes of saving space in diagrams, let $s^{(r,a)}_i:=r^i\cdot a\cdot s_i^{(p)}$.
\begin{lemma}
\label{Singer_invwind}
Let\begin{itemize}
    \item $\textbf{X}$ be a sequence over $\F_p$ of arbitrary length,
    \item $l,m,n$ and $h$ be natural numbers,
    \item $r_0,a_0,r_1,a_1\in\F_p\backslash\{0\}$,
    \item $\mathcal{W}_1$, $\mathcal{W}_2$ and $\mathcal{W}_3$ be windows in $W_p(\textbf{X})$ with side lengths $p^h-2$, $p^h$ and $p^h-2$, respectively, such that \begin{itemize}
        \item $\mathcal{W}_1$ and $\mathcal{W}_3$ have their north west corner on row $m$ and that $\mathcal{W}_1$ is west of $\mathcal{W}_3$;
        \item the south west corner of the inner frame of $\mathcal{W}_2$ overlaps with the north east entry of the inner frame of $\mathcal{W}_1$, and that this intersection takes place on row $m-1$ and column $n$,
        \item the south east entry in the inner frame of $\mathcal{W}_2$ overlap with the north west entry of the inner frame of $\mathcal{W}_3$,
        \item  The south inner frame of $\mathcal{W}_2$ is given by the $(r_0,a_0)$ geometric sequence with decreasing indices (the sequence goes from \textup{right to left}). That is, for $0\le i \le p^h+1$, $$W_p(\textbf{X})[m-1,n+i]=a_0r_0^{p^h+1-i}.$$
        \item  The south outer frame of $\mathcal{W}_2$ is given by the $(r_1,a_1)$ geometric transform of $\textbf{C}^{(p)}_h$ with decreasing indices (the sequence goes from \textup{right to left}). That is, for $0\le i \le p^h+1$, $$W_p(\textbf{X})[m,n+i]=a_1r_1^{p^h+1-i}s_{p^h+1-i}^{(p)}.$$
    \end{itemize}
\end{itemize}
Then,
\begin{enumerate}
\item  the eastern outer frame of $\mathcal{W}_1$ (the sequence $(W_p(\textbf{X})[m+i-1,n+1])_{0\le i \le p^h-1}$) and the western outer frame of $\mathcal{W}_3$ (the sequence $(W_p(\textbf{X})[m+i-1,n+p^h])_{0\le i \le p^h-1}$) are both geometric transforms of $\textbf{C}^{(p)}_h$;
\item there is another window in $W_p(\textbf{X})$, denoted $\mathcal{W}_4$, of size $p^h$ such that the north west (north east, respectively) corner of its inner frame overlaps with the south east (south west, respectively) corner of the inner frame of $\mathcal{W}_1$ (\,$\mathcal{W}_3$, respectively);
\item the north outer frame of $\mathcal{W}_4$ is given by the $\left(\frac{r_1}{r_0^2}, \frac{a_0^2r_0^4}{a_1r_1^2}\right)$ geometric transform of $\textbf{S}_h^{(p)}$ with increasing indices (the sequence goes from \textit{left to right}). That is,  for $0\le i \le p^h+1$,  $$W_p(\textbf{X})[m+p^h-3,n+i]=\frac{a_0^2r_0^4}{a_1r_1^2}\cdot \left(\frac{r_1}{r_0^2}\right)^is_i^{(p)}.$$
\item  the north inner frame of $\mathcal{W}_4$ is given by the $\left(r_0^{-1}, r_0^2a_0\right)$ geometric sequence with increasing indices (the sequence goes from \textit{left to right}). That is, for  for $0\le i \le p^h+1$, $$W_p(\textbf{X})[m+p^h,n+i]=r_0^{2-i}a_0.$$
\item when $a_0,r_0,a_1,r_1=1$, the region $(W_p (\textbf{X})[m+i,n+j])_{-1\le i \le p^h-2,~1\le j \le p^h}$ has horizontal and vertical symmetry;
\end{enumerate}
\end{lemma}
\noindent Lemma \ref{Singer_invwind} is illustrated below. Solely for the purposes of this diagram, let $r_2:=\frac{r_1}{r_0^2}$ and $a_2=\frac{a_0^2r_0^4}{a_1r_1^2}$
\begin{figure}[H]
    \centering
    \includegraphics[width=1\linewidth]{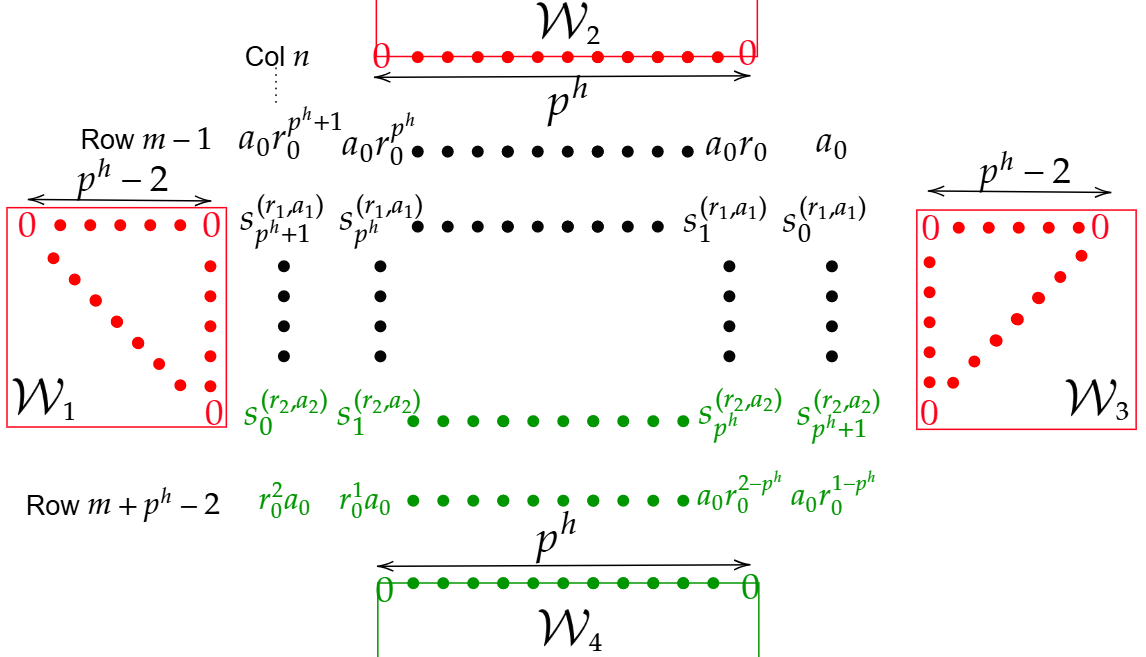}
    \caption{The windows $\mathcal{W}_1$, $\mathcal{W}_2$ and $\mathcal{W}_3$ are drawn in red. The green entries are the ones to be calculated. }
\end{figure}
\noindent  The proof is almost identical to that of Lemma \ref{invwind} and is hence almost entirely omitted. 
\begin{proof}
    \noindent The only difference is that Corollary \ref{fininv} is applied in place of Lemma \ref{lem: cant_inv_wall}.
\end{proof}
\subsubsection{\textbf{Relating Lemmata }\ref{invwind} \textbf{and} \ref{Singer_invwind}\textbf{ through rotation.}}\hfill\\
\noindent A simple corollary of Lemmata \ref{invwind} and \ref{Singer_invwind} is that $\chi\left(W_p\left(\textbf{C}_h^{(p)}\right)\right)$ and $\chi\left(W_p\left(\textbf{S}_h^{(p)}\right)\right)$ are related by a rotation. To state this, recall $\rho$ from Definition \ref{def:rotate}.

\begin{corollary}\label{cor: rotate_profile}
    Let $\textbf{X}_S$ be a sequence over $\F_p$ such that $W_p(\textbf{X}_S)$ has setup identical to that of Lemma \ref{Singer_invwind}. Similarly, let $\textbf{X}_C$ be a sequence over $\F_p$ such that $W_p(\textbf{X}_C)$ has setup identical to that of Lemma \ref{invwind}. Then, \[\chi((W_p(\textbf{X}_C)[m+i,n+j])_{0\le i,j <p^h})=\rho(\chi((W_p(\textbf{X}_S)[m+i-1,n+j+1])_{0\le i,j<p^h}))\]
\end{corollary}
\begin{proof}
    First, recall that Corollary \ref{cor: profile} shows that for any sequence $\textbf{S}$ over $\F_p$ and for any $r_0,a_0,r_1,a_1\in\F_p\backslash\{0\}$ one has \[\chi\left(W^{(r_0,a_0)}_p\left(\mathbf{S}{(r_1,a_1)}\right)\right)=\chi(W_p(\textbf{S})).\] Therefore, it is sufficient to prove Corollary \ref{cor: rotate_profile} when $r_0,a_0,r_1,a_1=1$. \\
    
    \noindent Just as in the proof of Lemma \ref{invwind}, $(W_p(\textbf{X}_C)[m+i,n+j])_{0\le i,j <p^h}$ appears as the purple square below  \begin{figure}[H]
        \centering
        \includegraphics[width=0.8\linewidth]{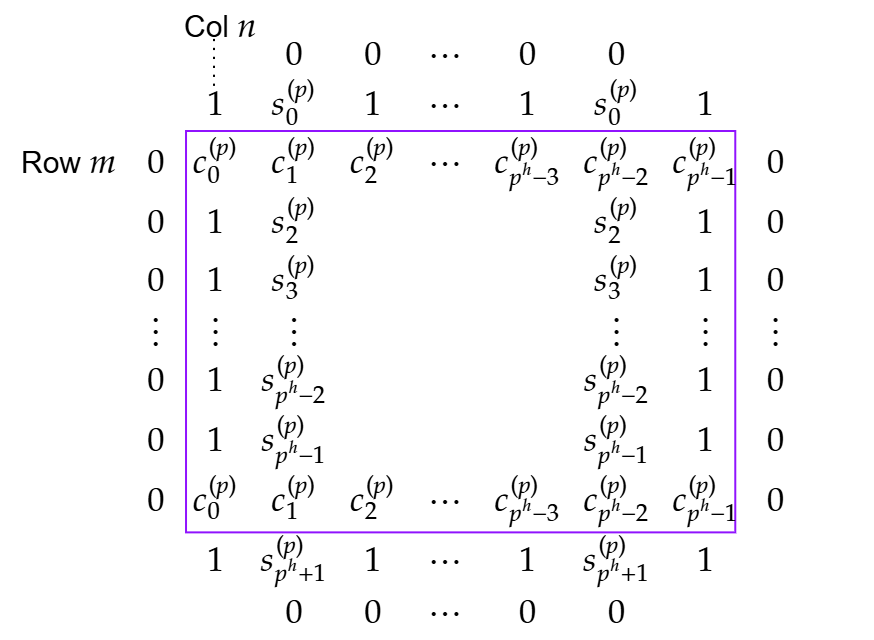}
    \end{figure}\noindent where the interior of this square is filled in by Lemma \ref{lem: rotate} with number walls of $(s_i)_{1\le i \le p^h}$ and $\textbf{C}^{(p)}_h$ as in Figure \ref{big_n_complex}.\\

    \noindent Similarly, Lemma \ref{Singer_invwind} states that $(W_p(\textbf{X}_S)[m+i-1,n+j+1])_{0\le i,j <p^h}$ appears as the green square below.\begin{figure}[H]
        \centering
        \includegraphics[width=0.8\linewidth]{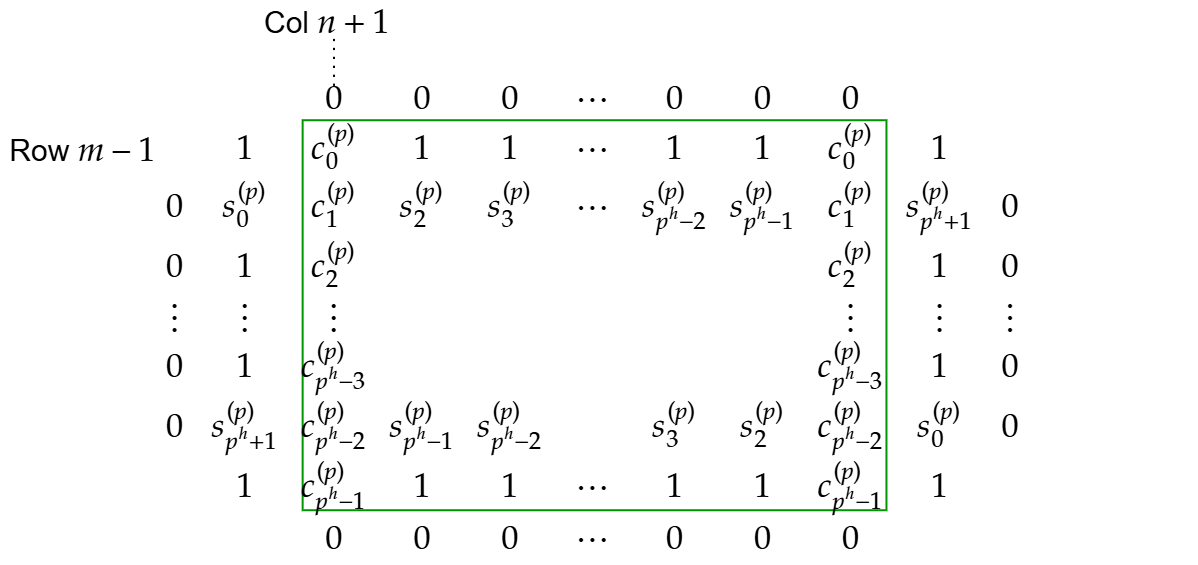}
        \caption{The green square is the $90^\circ$ rotation of the purple square above.}
        \label{Fig: Green_square}
    \end{figure}
    \noindent Since $c_0^{(p)}=c_{p^h-1}^{(p)}=s_{0}^{(p)}=s_{p^h+1}^{(p)}=1$,  $c_1^{(p)}=c_{p^h-2}^{(p)}=s_{1}^{(p)}=s_{p^h}^{(p)}=0$, and the $p$-Cantor and $p$-Singer sequences are symmetric (Lemmas \ref{sym} and \ref{sym_sing}), one sees that the purple square is the $90^\circ$ rotation of the green square. Lemma \ref{lem: rotate} fills in both squares and completes the proof.
\end{proof}

\subsection{A Window Between Two $p$-Cantor Sequences}\hfill\\
\noindent The following lemma is equally as important to the proof of Theorem \ref{thm: morphism} as Lemma \ref{invwind}. Lemma \ref{between_two_windows} is illustrated immediately after it is stated and it is strongly recommended to keep this figure in mind whilst reading.
\begin{lemma}\label{between_two_windows}
\noindent Let
\begin{itemize}
    \item $\textbf{X}$ be a sequence of arbitrary length over $\F_p$,
    \item $i\in\{1,2,3\}$, let $h,m,n\in\N$,
    \item $r_i,a_i,r_L,a_L,r_A,a_A,r_R,a_R\in\F_q\backslash\{0\}$,
    \item $W_p(\textbf{X})[m+1,n+j]=0$, where $j\in\{-p^h, -p^h-1, 2p^h+1, 2p^h+2\}$ (these are the red zeroes on Figure \ref{between_two_windows_fig}),
    \item $\mathcal{W}_i$ be a window in $W_p(\textbf{X})$ of size $p^h-2$ for $i=1,3$ and of size $p^h$ for $i=2$, such that
    \begin{itemize}
        \item  $\mathcal{W}_1$ and $\mathcal{W}_3$ have their bottom left corner on row $m-1$,
        \item  the south east corner of the inner frame of $\mathcal{W}_1$ overlaps with the north west corner of the inner frame of $\mathcal{W}_2$. Let $n$ be the column index of this overlap,
        \item  the south west corner of the inner frame of $\mathcal{W}_3$ overlap with the north east corner of the inner frame of $\mathcal{W}_2$,
         \item the south inner frame of $\mathcal{W}_1$ is the $(r_L,a_L)$ geometric sequence with decreasing indices (the sequence goes from \textup{right to left}). That is, $W_p(\textbf{X})[m,n-i]=a_A\cdot r_L^i$ for $0\le i \le p^h-1$,
        \item the south outer frame of $\mathcal{W}_1$ is the $(r_1,a_1)$-geometric transform of $\textbf{C}^{(p)}_h$ with decreasing indices (the sequence goes from \textup{right to left}). That is, $W_p(\textbf{X})[m+1,n-i]=a_1\cdot r_1^i\cdot c_i^{(p)}$ for $0\le i \le p^h-1$,
        \item the north inner frame of $\mathcal{W}_2$ is the $(r_A,a_A)$ geometric sequence with increasing indices (the sequence goes from \textup{left to right}). That is, $W_p(\textbf{X})[m,n+i]=a_A\cdot r_A^i$ for $0\le i \le p^h+1$,
        \item the north outer frame of $\mathcal{W}_2$ is the $(r_2,a_2)$-geometric transform of $\textbf{S}_h^{(p)}$ with increasing indices (the sequence goes from \textup{left to right}). That is, $W_p(\textbf{X})[m-1,n+i]=a_2\cdot r_2^i\cdot s^{(p)}_i$ for $0\le i \le p^h+1$,
        \item  the south inner frame of $\mathcal{W}_3$ is the $(r_R,a_R)$ geometric sequence with decreasing indices (the sequence goes from \textup{right to left}). That is, $W_p(\textbf{X})[m,n+2p^h-i]=a_R\cdot r_R^i$ for $0\le i \le p^h-1$,
        \item the south outer frame of $\mathcal{W}_3$ is the $(r_3,a_3)$-geometric transform of $\textbf{C}^{(p)}_h$ with decreasing indices (the sequence goes from \textup{right to left}). That is, $W_p(\textbf{X})[m+1,n+2p^h-i]=a_3\cdot r_3^i\cdot c_i^{(p)}$ for $0\le i \le p^h-1$.
    \end{itemize}

\end{itemize}
\noindent Then
\begin{enumerate}
    \hypertarget{lem3.1}{\item} the southern inner frame of $\mathcal{W}_2$ is given by the $\left(\frac{-a_1\cdot a_R}{r_A\cdot a_3\cdot a_A},\frac{a_3^2}{a_R}\right)$ geometric sequence with decreasing indices (the sequence goes from \textup{right to left}). That is, for $0\le i \le p^h+1,$ $$W_p(\textbf{X})[m+p^h+1,n+p^h+1-i]=\frac{a_3^2}{a_R}\cdot \left(\frac{-a_1\cdot a_R}{r_A\cdot a_3\cdot a_A}\right)^i.$$ 
    \hypertarget{lem3.2}{\item} for $0\le i \le p^h+1$, the $i^\nth$ term of the southern outer frame of $\mathcal{W}_2$ is equal to \[W_p(\textbf{X})[m+p^h+2,n+p^h+1-i]=s^{(p)}_i\left(\frac{\frac{a_1\cdot a_2}{a_A^2}\cdot \left(\frac{r_2}{r_A}\right)^i+r_A\cdot r_L\cdot \left(\frac{r_1}{r_L}\right)^i+\frac{a_1\cdot a_R\cdot r_R}{r_A\cdot a_A\cdot a_3\cdot r_3^2}\cdot \left(\frac{r_3}{r_R}\right)^i}{\frac{a_R^2}{a_3^3}\cdot \left(\frac{-a_1\cdot a_R}{r_A\cdot a_3\cdot a_A}\right)^{-i} }\right)\cdotp\]
\end{enumerate}

\end{lemma}
\begin{figure}[H]
        \centering
        \includegraphics[width=1\linewidth]{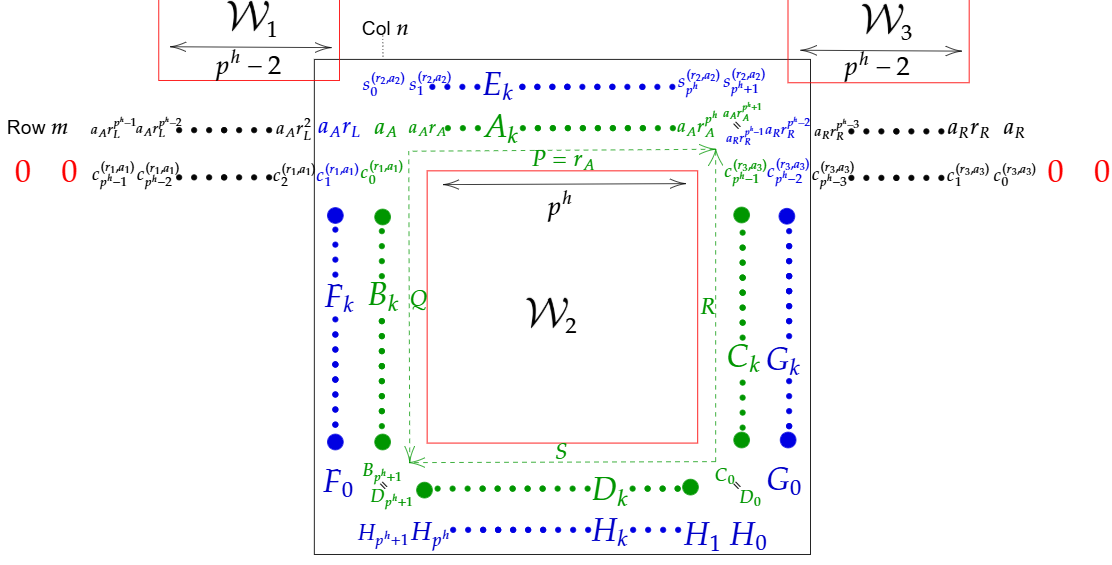}
        \caption{Each window is depicted with a red square. The inner and outer frame of $\mathcal{W}_2$ are coloured in green and blue, respectively. The southern outer frame of $\mathcal{W}_2$ is given by the sequence $(H_i)_{0\le i \le p^h+1}$. The green dotted arrows and their labels show the direction of the geometric sequences comprising the inner frame of $\mathcal{W}_2$.}
        \label{between_two_windows_fig}
        \end{figure}
\begin{proof}
    Using notation form Figure \ref{Fig: basicwindow} applied to $\mathcal{W}_2$, the goal of parts \hyperlink{lem3.1}{1} and \hyperlink{lem3.2}{2} is to calculate the sequences $(D_k)_{0\le k \le p^h+1}$ and $(H_k)_{0\le k \le p^h+1}$, respectively. To calculate each entry in the above sequences, \hyperlink{FC2}{FC2} and \hyperlink{FC3}{FC3} imply that one needs to know the values of $A_k,B_k,C_k,D_k,E_k,F_k,G_k,P,Q,R$ and $S$. Each of these are either known by assumption or are calculated from the previous lemmas in this section. From the statement of the lemma, the following variables are known: \begin{align*}
        &P=r_A& &A_k=a_A\cdot r_A^k& &E_k=s^{(r_2,a_2)}_k=s^{(p)}_k\cdot a_2\cdot r_2^k.&
    \end{align*}
    \noindent Furthermore, for some specific values of $k$, the values of $B_k$ and $C_k$ are known.\begin{align*}
        &B_0=A_0=a_A& &B_1=C^{(\mathbf{r}_1)}_0=a_1& \\&C_{p^h+1}=a_R\cdot r_R^{p^h-1}=a_R& &C_{p^h}=C^{(\mathbf{r}_3)}_{p^h-1}=a_3r_3^{p^h-1}=a_3.&
    \end{align*}
    \noindent Using these, one obtains that \begin{align*}
        &Q=B_1\cdot B_0^{-1}=a_1\cdot a_A^{-1}& &R= C_{p^{h}+1}\cdot C_{p^h}^{-1}=a_R \cdot a_3^{-1}.&
    \end{align*}
    \noindent Therefore, \begin{align}
        &B_k=B_0\cdot Q^k = a_1^k\cdot a_A^{1-k}&\nonumber\\ &C_k=C_{p^h+1}\cdot R^{k-p^h-1}=a_R^{p^h-1}\cdot (a_3\cdot a_R^{-1})^{k-p^h-1}=a_3^{2-k}\cdot a_R^{k-1}.&\label{C_k}
    \end{align}
    \noindent The next variable to be calculated is $F_k$. By applying Lemma \ref{reflect} (vertical reflection of number walls), one is now able to apply Lemma \ref{lem: cant_inv_wall}. To this end, ``$\leftrightarrow$" notation is introduced to make it clear which of the variables from Lemma \ref{lem: cant_inv_wall} correspond to each variable in Lemma \ref{between_two_windows}: the left side of each $\leftrightarrow$ is notation from Lemma \ref{lem: cant_inv_wall} and the right-side is notation from this proof.\begin{align*}
        &a_0\leftrightarrow a_A& &r_0\leftrightarrow r_L& &a_1\leftrightarrow c_0^{(r_1,a_1)}=a_1& &r_1\leftrightarrow r_1& &s_0\leftrightarrow c^{(p)}_0=1& &b_k\leftrightarrow s^{(p)}_k.&
    \end{align*}
    \noindent Therefore, by Lemma \ref{lem: cant_inv_wall}, \begin{align*}
        F_k=a_A\cdot r_L \cdot \left(\frac{-a_1\cdot r_1}{a_A\cdot r_L}\right)^k\cdot s^{(p)}_k.
    \end{align*}
    \noindent Similarly, one uses  Lemma \ref{lem: cant_inv_wall} (without needing to apply Lemma \ref{reflect} first) to calculate $G_k$. Here, \begin{align*}
        &a_0\leftrightarrow A_{p^h+1}=a_R\cdot r_R^{p^h-1}\substack{\hyperlink{FLT}{FLT}\\=}a_R& &r_0\leftrightarrow r_R^{-1}& &a_1\leftrightarrow c^{(r_3,a_3)}_{p^h-1}=a_3\cdot r_3^{p^h-1}\substack{\hyperlink{FLT}{FLT}\\=}a_3& \\&r_1\leftrightarrow r_3^{-1}& &s_0\leftrightarrow c^{(p)}_{p^h-1}=1& &b_k\leftrightarrow s^{(p)}_k.&
    \end{align*} However, note that the index $k$ in $G_k$ increases as the row index decreases. Therefore, one has that \begin{align}
        G_k&=a_R\cdot r_R^{-1}\cdot \left(\frac{-a_3\cdot r_3^{-1}}{r_R^{-1}\cdot a_R}\right)^{p^h+1-k} s^{(p)}_{p^h+1-k}\nonumber\\&\stackrel{\text{ Lemma \ref{sym}}}{=}(a_R\cdot r_R^{-1})^{-1}\cdot (-a_3\cdot r_3^{-1})^{2} \left(\frac{r_R^{-1}\cdot a_R}{-a_3\cdot r_3^{-1}}\right)^{k} s^{(p)}_{k}.\label{G_k}
    \end{align}

    \noindent Next, the values of $D_k$ and $S$ are computed. Using Theorem \ref{ratio ratio}, one has  \begin{equation*}
        S=\frac{QR}{P}\cdot (-1)^{p^h}=\frac{-a_1\cdot a_R}{r_A\cdot a_3\cdot a_A}\cdotp
    \end{equation*} Therefore, using that $D_0=C_0=C_{p^h+1}\cdot R^{-p^h-1}$, one obtains that \begin{align}
        D_k&=D_0\cdot S^k=a_R\cdot \left(\frac{a_R}{a_3}\right)^{-p^h-1} \cdot \left(\frac{-a_1\cdot a_R}{r_A\cdot a_3\cdot a_A}\right)^k\nonumber\\
        &=\frac{a_3^{2}}{a_R}\cdot \left(\frac{-a_1\cdot a_R}{r_A\cdot a_3\cdot a_A}\right)^k\label{D_k}
    \end{align}
\noindent This establishes claim \hyperlink{lem3.1}{1} of Lemma \ref{between_two_windows}. For claim \hyperlink{lem3.2}{2}, recall \hyperlink{FC3}{FC3} which states \begin{align}
        H_k&=\frac{\frac{Q\cdot E_k}{A_k}+(-1)^k\cdot \frac{P\cdot F_k}{B_k}-(-1)^k\cdot \frac{S\cdot G_k}{C_k}}{\frac{R}{D_k}}\nonumber\\
        &=s^{(p)}_k\left(\frac{\frac{a_1\cdot a_2}{a_A^2}\cdot \left(\frac{r_2}{r_A}\right)^k+r_A\cdot r_L\cdot \left(\frac{r_1}{r_L}\right)^k+\frac{a_1\cdot a_R\cdot r_R}{r_A\cdot a_A\cdot a_3\cdot r_3^2}\cdot \left(\frac{r_3}{r_R}\right)^k}{\frac{a_R^2}{a_3^3}\cdot \left(\frac{-a_1\cdot a_R}{r_A\cdot a_3\cdot a_A}\right)^{-k} }\right)\label{between_two_windows_result}
    \end{align}
\end{proof}
\noindent With an additional assumption, Lemma \ref{between_two_windows} is simplified. 
\begin{corollary} \label{between_two_windows_cor}
    Consider the same setup as in Lemma \ref{between_two_windows}, but with the additional condition that there exists an $x\in\F_{p}\backslash\{0\}$ such that \begin{align*}
    \frac{r_1}{r_L}=\frac{r_2}{r_A}=\frac{r_3}{r_R}=x.
    \end{align*}
    \noindent Then, the southern outer frame of $\mathcal{W}_2$ (that is, the sequence $(H_i)_{0\le i \le p^h+1}$ from Figure \ref{between_two_windows_fig}) is a geometric transform of $(s^{(p)}_i)_{0\le i \le p^h+1}$. In particular, \[H_k=s_k^{(p)}\cdot \left(\frac{-x\cdot r_A\cdot a_3\cdot a_A}{a_1\cdot a_R}\right)^k \cdot \frac{a_3^3}{a_R^2}\cdot \left(1+\frac{a_1\cdot a_R\cdot r_R}{r_A\cdot a_A\cdot a_3\cdot r_3^2}\right).\]
\end{corollary}
\begin{proof}
    \noindent With the additional assumption of $r_1,r_2,r_3,r_L,r_A$ and $r_R$, one factorises equation (\ref{between_two_windows_result}) to obtain \begin{align}
        H_k=s_k^{(p)}\cdot \left(\frac{-x\cdot r_A\cdot a_3\cdot a_A}{a_1\cdot a_R}\right)^k \cdot\underbrace{\left(\frac{\frac{a_1\cdot a_2}{a_A^2} +r_A\cdot r_L +\frac{a_1\cdot a_R\cdot r_R}{r_A\cdot a_A\cdot a_3\cdot r_3^2} }{\frac{a_R^2}{a_3^3} }\right)}_{(*)}\label{eqn: H_0}
    \end{align}
    \noindent When $k=0$, equation (\ref{eqn: H_0}) is equal to $H_0$, which implies that $(*)=H_0$. Using \hyperlink{FC1}{FC1} to calulcate $H_0$ directly, one arrives at \begin{align}
        H_0&=\frac{D_0^2\cdot D_1G_0}{C_1}\cdotp \label{H_0}
    \end{align}
    The proof is completed by applying equations, (\ref{G_k}), (\ref{C_k}) and (\ref{D_k}) to equation (\ref{H_0}).
\end{proof}
\subsection{A Window Between Two $p$-Singer Sequences}\hfill\\
\noindent The following Lemma is similar to Lemma \ref{between_two_windows}, but the roles of the $p$-Cantor and $p$-Singer sequence have been swapped. 
\begin{lemma}\label{between_to_windows_singer}
Let
\begin{itemize}
    \item $\textbf{X}$ be a sequence of arbitrary length over $\F_p$;
    \item $i\in\{1,2,3\}$, let $h,m,n\in\N$;
    \item $r_i,a_i,r_L,a_L,r_A,a_A,r_R,a_R\in\F_q\backslash\{0\}$,
    \item $\mathcal{W}_i$ be a window in $W_p(\textbf{X})$ of size $p^h$ for $i=1,3$ and of size $p^h-2$ for $i=2$, such that\begin{itemize}
         \item $\mathcal{W}_1$ and $\mathcal{W}_3$ have their bottom left corner on row $m-1$;
        \item the south east corner of the inner frame of $\mathcal{W}_1$ overlaps with the north west corner of the inner frame of $\mathcal{W}_2$. Let $n$ be the column index of this overlap;
        \item the south west corner of the inner frame of $\mathcal{W}_3$ overlaps with the north east corner of the inner frame of $\mathcal{W}_2$;
        \item the south inner frame of $\mathcal{W}_1$ is given by the $(r_L,a_L)$ geometric sequence with decreasing indices (the sequence goes from \textup{right to left}). That is, $W_p(\textbf{X})[m,n-i]=a_A\cdot r_L^i$ for $0\le i \le p^h+1$;
        \item the south outer frame of $\mathcal{W}_1$ is the $(r_1,a_1)$-geometric transform of $\textbf{S}^{(p)}_h$ with decreasing indices (the sequence goes from \textup{right to left}). That is, $W_p(\textbf{X})[m+1,n-i]=a_1\cdot r_1^i\cdot s_i^{(p)}$ for $0\le i \le p^h+1$;
        \item the north inner frame of $\mathcal{W}_2$ is given by the $(r_A,a_A)$ geometric sequence with increasing indices (the sequence goes from \textup{left to right}). That is, $W_p(\textbf{X})[m,n+i]=a_A\cdot r_A^i$ for $0\le i \le p^h-1$;
        \item the north outer frame of $\mathcal{W}_2$ is the $(r_2,a_2)$-geometric transform of $\textbf{C}^{(p)}_h$ with increasing indices (the sequence goes from \textup{left to right}). That is, $W_p(\textbf{X})[m-1,n+i]=a_2\cdot r_2^i\cdot c^{(p)}_i$ for $0\le i \le p^h-1$;
        \item  the south inner frame of $\mathcal{W}_3$ is given by the $(r_R,a_R)$ geometric sequence with decreasing indices (the sequence goes from \textup{right to left}).. That is, $W_p(\textbf{X})[m,n+2p^h-i]=a_R\cdot r_R^i$ for $0\le i \le p^h+1$;
        \item the south outer frame of $\mathcal{W}_3$ is the $(r_3,a_3)$-geometric transform of $\textbf{S}_h^{(p)}$ with decreasing indices (the sequence goes from \textup{right to left}). That is, $W_p(\textbf{X})[m+1,n+2p^h-i]=a_3\cdot r_3^i\cdot s_i^{(p)}$ for $0\le i \le p^h+1$;
    \end{itemize}

\end{itemize}
\noindent Then,
\begin{itemize}
    \item the southern inner frame of $\mathcal{W}_2$ is the $\left(\frac{-a_1\cdot a_R\cdot r_R^2}{r_A\cdot a_3\cdot a_A\cdot r_3^2},a_Rr_R^2\right)$ geometric sequence with decreasing indices (the sequence goes from \textup{right to left}). That is, for $0\le i \le p^h-1$, $$W_p(\textbf{X})[m+p^h+1,n+p^h+1-i]=a_Rr_R^2\cdot \left(\frac{-a_1\cdot a_R\cdot r_R^2}{r_A\cdot a_3\cdot a_A\cdot r_3^2}\right)^i,$$ 
    \item  for $0\le i \le p^h-1$, the $i^\nth$ term of the south outer-frame of $\mathcal{W}_2$ is equal to\[W_p[m+p^h,n+p^h-1-i]=c^{(p)}_k\left(\frac{\frac{a_1\cdot a_2}{a_A^2}\cdot \left(\frac{r_2}{r_A}\right)^k+r_A\cdot r_L\cdot \left(\frac{r_1}{r_L}\right)^k+\frac{a_1\cdot a_R\cdot r_R}{a_3\cdot r_3^{2}\cdot r_A\cdot a_A}\cdot \left(\frac{r_3}{r_R}\right)^k}{\frac{1}{r_3^2\cdot a_3}\cdot \left(\frac{-a_1\cdot a_R\cdot r_R^2}{a_3\cdot r_3^2\cdot r_A\cdot a_A}\right)^{-k} }\right)\cdotp\]
\end{itemize}

\end{lemma}
    \begin{figure}[H]
        \centering
        \includegraphics[width=1\linewidth]{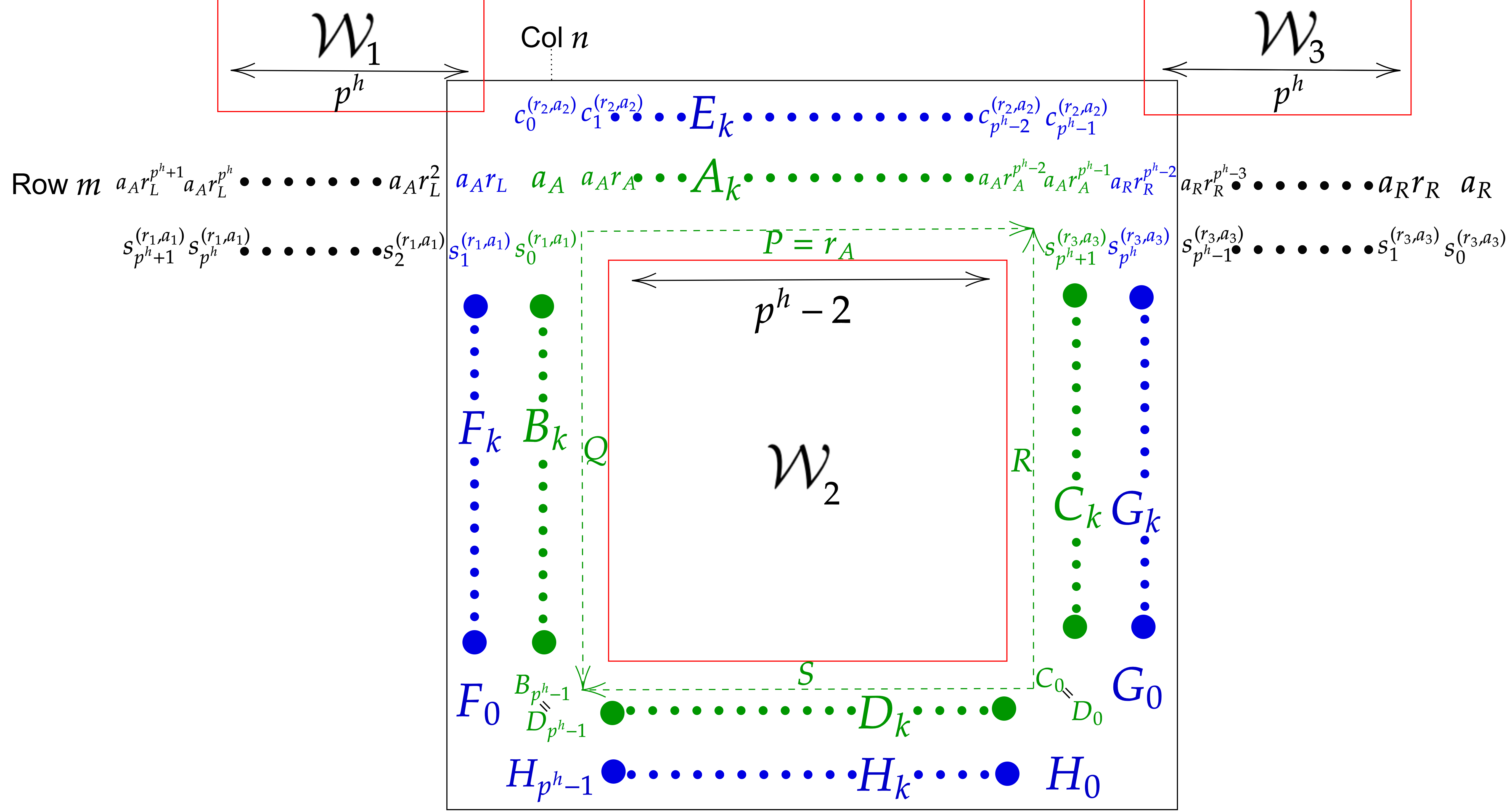}
        \caption{Each window is depicted with a red square. The inner and outer frame of $\mathcal{W}_2$ are coloured in green and blue, respectively. The southern outer frame of $\mathcal{W}_2$ is given by the sequence $(H_i)_{0\le i \le p^h+1}$. The green dotted arrows and their labels show the direction of the geometric sequences comprising the inner frame of $\mathcal{W}_2$.}
        \label{between_two_windows_singer_fig}
        \end{figure}
\begin{proof}
    The proof of Lemma \ref{between_to_windows_singer} is identical to that of Lemma \ref{between_two_windows}, with the exception of Lemma \ref{lem: cant_inv_wall} being replaced with Lemma \ref{fininv}. Hence, the values of the entries in the inner and outer frames are stated without proof. 
    \begin{align*}
        &A_k=a_Ar_A^k& &E_k=a_2r_2^k c_k^{(p)}& &P=r_A&\\
        &B_k=a_A\cdot \left(\frac{a_1}{a_A}\right)^k& &F_k=a_Ar_L \cdot\left(\frac{-a_1r_1}{r_La_A}\right)^kc_k^{(p)}& &Q=\frac{a_1}{a_A}&\\
        &C_k=a_Rr_R^2 \left(\frac{a_Rr_R^2}{a_3r_3^2}\right)^k& &G_k=a_Rr_R \cdot \left(\frac{-a_Rr_R}{a_3r_3}\right)\cdot c_k^{(p)}& &R=\frac{a_R r_R^2}{a_3r_3^2}&\\
        &D_k=a_Rr_R^2 \left(\frac{a_Ra_1r_R^2}{a_3a_Ar_3^2r_A}\right)^k& && &S=-\frac{a_Ra_1r_R^2}{a_3r_3^2a_Ar_A}\cdotp&\\
    \end{align*}
\end{proof}
\noindent Just as with Corollary \ref{between_two_windows_cor}, one can simplify Lemma \ref{between_to_windows_singer}.
\begin{corollary} \label{between_two_windows_singer_cor}
    Consider the same setup as in Lemma \ref{between_to_windows_singer}, but with the additional condition that there exists an $x\in\F_{p}\backslash\{0\}$ such that \begin{align*}
    \frac{r_1}{r_L}=\frac{r_2}{r_A}=\frac{r_3}{r_R}=x.
    \end{align*}
    \noindent Then, the southern outer frame of $\mathcal{W}_2$ (that is, the sequence $(H_i)_{0\le i \le p^h-1}$ from Figure \ref{between_two_windows_singer_fig}) is a geometric transform of $(c^{(p)}_i)_{0\le i \le p^h-1}$. In particular, \[H_k=c_k^{(p)}\cdot \left(\frac{-x\cdot a_3\cdot r_3^2\cdot r_A\cdot a_A}{a_1\cdot a_R\cdot r_R^2}\right)^k \cdot \left(a_3\cdot r_3^{2}+\frac{a_1\cdot a_R\cdot r_R}{ r_A\cdot a_A}\right).\]
\end{corollary}

\section{The Construction Lemmata}\label{Sect: 9}
\noindent The goal of this section is to provide two lemmata that combine to give the proof of Theorem \ref{thm: morphism}. These are known as the Construction Lemmata, and they build upon the results from the previous section. Before they are stated, some additional notation is defined.\\

\noindent \textbf{Notation:} For $h\in\N$ and $r,a\in\F_p$, let\begin{align}
    &\textbf{G}_{h,s}(r,a)=\left(a\cdot r^i\right)_{0\le i \le p^h-1}& &\textbf{G}_{h,\ell}(r,a)=\left(a\cdot r^i\right)_{0\le i \le p^h+1}&\nonumber
\end{align}
\noindent Above, the $s$ in $\textbf{G}_{h,s}(r,a)$ and the $\ell$ in $\textbf{G}_{h,l}(r,a)$ stand for ``short'' and ``long'', respectively. 
\subsection{Construction Lemma 1}\hfill\\
\noindent Recall that $p_2:=\frac{p-1}{2}$. Construction Lemma 1 combines Lemma \ref{invwind} and Lemma \ref{between_two_windows}. It is illustrated immediately after it is stated, and it is \textit{strongly} advised that the reader keep this figure in mind when reading the statement of Lemma \ref{constrution_lemma_1}.
\begin{lemma}[Construction Lemma \hypertarget{con_lem_1}{1}]\label{constrution_lemma_1}
    Let\begin{itemize}
        \item $h\ge1$ and $m\ge0$ be natural numbers,
        \item $\textbf{S}$ be a sequence over $\F_p$ of length $p^{h+1}$,
        \item $\widetilde{\textbf{S}}$ be the sequence of length $3p^{h+1}$ defined by $\widetilde{\textbf{S}}:=\{0\}_{0\le i <p^{h+1}}\oplus \textbf{S} \oplus\{0\}_{0\le i <p^{h+1}}$,
        \item $r_{S,i}\in \F_p\backslash\{0\}$ and $a_{S,i}\in\F_p$ for $0\le i <p_2$,
        \item $r_{C,i}\in\F_p\backslash\{0\}$ and  $a_{C,i}\in\F_p$ for $0 \le i \le p_2$,
        \item $a_0\in\F_p\backslash\{0\}$, and define \begin{equation}a_{i}:=\prod^{i-1}_{k=0}r_{S,k}^{2}.\label{eqn: a_i_chapt_9}\end{equation}
    \end{itemize} Assume that rows $m-2$, $m-1$ and $m$ of $W_p(\widetilde{\textbf{S}})$ have the following structure: \begin{itemize}
        \item Row $m-2$: \begin{itemize}
            \item $W_p(\widetilde{\textbf{S}})[m-2,p^{h+1}]\neq0$, and $ W_p(\widetilde{\textbf{S}})[m-2,2p^{h+1}-1]\neq 0$,
            \item $W_p(\widetilde{\textbf{S}})[m-2, p^{h+1}+2ip^h+j]_{1\le j \le p^h-2}=0$ for $0 \le i \le p_2$,
            \item $W_p(\widetilde{\textbf{S}})[m-2, p^{h+1}+(2i+1)p^h-1+j]=a_{S,i}r_{S,i}^j s^{(p)}_j$ for ${0\le j \le p^h+1}$ and $0 \le i \le p_2-1$.
        \end{itemize}
        \item Row $m-1$: \begin{itemize}
            \item $W_p(\widetilde{\textbf{S}})[m-1, p^{h+1}+2ip^h+j]=a_ir_{C,i}^{p^h-1-j}$ for $0 \le j \le p^h-1$ and $0 \le i \le p_2$,
            \item $W_p(\widetilde{\textbf{S}})[m-1, p^{h+1}+(2i+1)p^h-1+j]=a_ir_{S,i}^{j}$ for $0 \le j \le p^h+1$ and $0 \le i \le p_2-1$.
        \end{itemize}
        \item Row $m$:\begin{itemize}
            \item $W_p(\widetilde{\textbf{S}})[m,p^{h+1}+2ip^h+j]=a_{C,i}\cdot r_{C,i}^{p^h-1-j} \cdot c_j^{(p)}$ for $0\le j \le p^h-1$ and $0\le i \le p_2$,
            \item$W_p(\widetilde{\textbf{S}})[m,(2i+1)p^h+j]=0$ for $0\le j \le p^h-1$ and $0\le i \le p_2-1$.
        \end{itemize}
    \end{itemize} 
    Then, rows $m+p^h-1$, $m+p^h$ and $m+p^h+1$ have the following structure:
    \begin{enumerate}
       \item Row $m+p^h-1$:\begin{enumerate}
            \hypertarget{CL11a}{\item} $W_p(\widetilde{\textbf{S}})[m+p^h-1,p^{h+1}+2ip^h+j]=a_{C,i}\cdot r_{C,i}^{-j} \cdot c_j^{(p)}$ for $0\le j \le p^h-1$ and $0\le i \le p_2$,
            \hypertarget{CL11b}{\item}$W_p(\widetilde{\textbf{S}})[m+p^h-1,p^{h+1}+(2i+1)p^h+j]=0$ for $0\le j \le p^h-1$ and $0\le i \le p_2-1$.
        \end{enumerate}
        \item Row $m+p^h$: \begin{enumerate}
            \hypertarget{CL12a}{\item} $W_p(\widetilde{\textbf{S}})[m+p^h, p^{h+1}+2ip^h+j]=\frac{a_{C,i}^2}{a_i}\cdot \left(r_{C,i}\right)^{-j}$ for $0 \le j \le p^h-1$ and $0 \le i \le p_2$,
            \hypertarget{CL12b}{\item} $W_p(\widetilde{\textbf{S}})[m-1, p^{h+1}+(2i+1)p^h-1+j]= \frac{a_{C,i+1}^2}{a_{i+1}}\cdot \left(\frac{-a_{C,i}\cdot a_{i+1}}{r_{S,i}\cdot a_{C,i+1}\cdot a_i}\right)^{p^h+1-j}$ for ${0\le j \le p^h+1}$.
        \end{enumerate}
        \item Row $m+p^h+1$: \begin{enumerate}
            \hypertarget{CL13a}{\item} $W_p(\widetilde{\textbf{S}})[m-2,p^{h+1}]\neq0$ and $W_p(\widetilde{\textbf{S}})[m-2,2p^{h+1}-1]\neq 0$,
            \hypertarget{CL13b}{\item} $W_p(\widetilde{\textbf{S}})[m-2, p^{h+1}+(2i+1)p^h-1+j]=s^{(p)}_j\cdot  x_i \cdot \left(\frac{-a_{C,i}\cdot a_{i+1}}{r_{C,i+1}\cdot a_{C,i+1}\cdot a_i}\right)^{p^n+1-j}$, where $0\le j \le p^{h+1}$, $0\le i \le p_2-1$ and \[x_i=\frac{a_{C,i+1}^3}{a_{i+1}^2}\cdot\left(1+\frac{a_{C,i}\cdot a_{i+1}}{r_{S,i}\cdot a_i\cdot a_{C,i+1}\cdot r_{C,i+1}}\right);\]
            \hypertarget{CL13c}{\item}  $W_p(\widetilde{\textbf{S}})[m+p^h+1, p^{h+1}+2ip^h+j]_{1\le j \le p^h-2}=0$ for $0 \le i \le p_2$.
    \end{enumerate}
    \end{enumerate} 
\end{lemma}
\begin{figure}[H]
    \centering
    \includegraphics[width=1\linewidth]{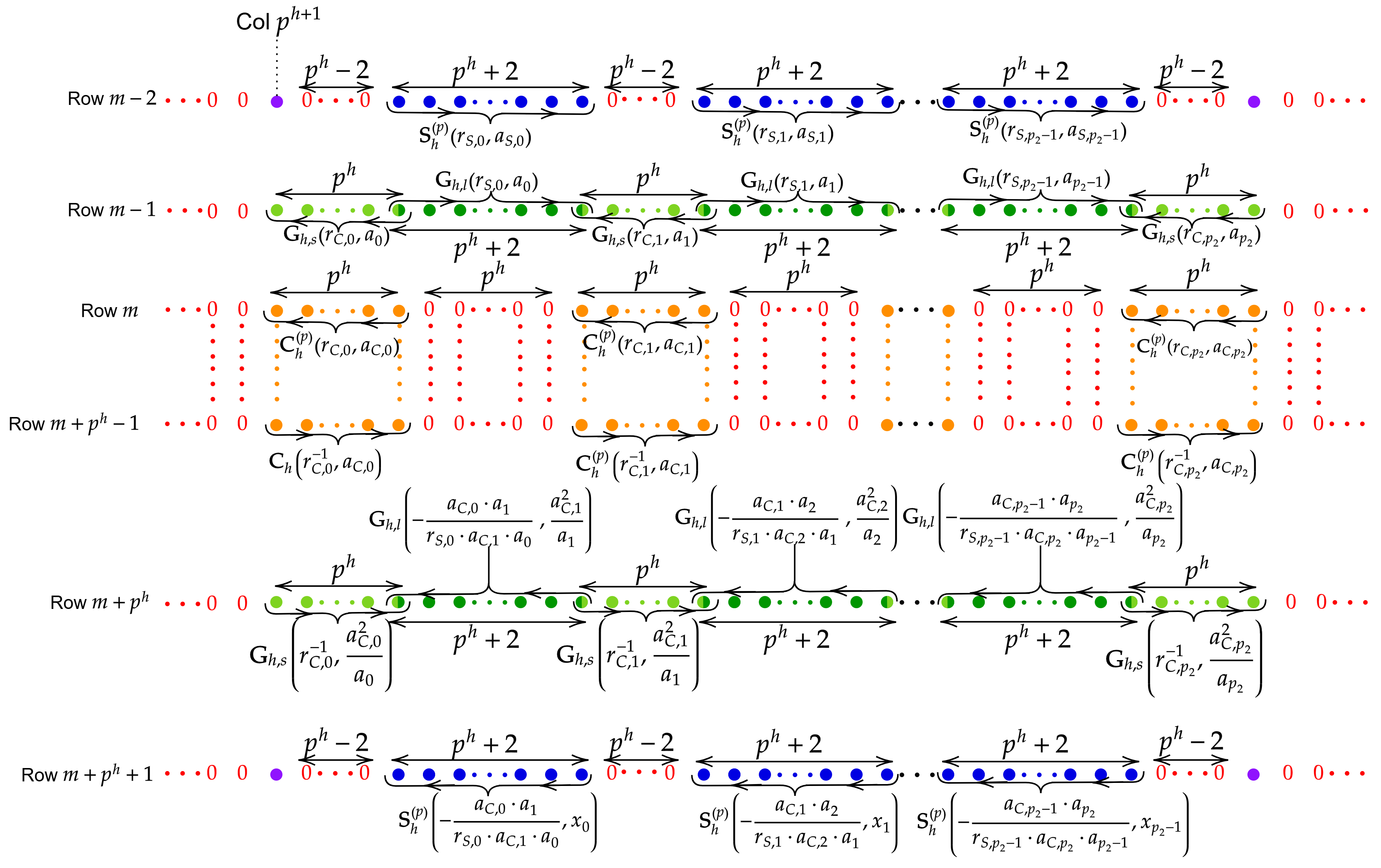}\label{fig: 9.1}
    \caption{The setup of Lemma \ref{constrution_lemma_1}, with the details described below. }
    \label{fig: construciton_lemma_1}
\end{figure}
\noindent Each dot in the above figure is an entry in $W_p(\widetilde{\textbf{S}})$, with the colour representing which sequence that entry belongs to. Some dots have multiple colours. Any geometric transform of the $p$-Singer ($p$-Cantor, respectively) sequence is in blue (orange, respectively). Geometric sequences are in green, where light (dark, respectively) green indicates a short (long, respectively) sequence. The two sided arrows denote the length of a sequence, whereas the one sided `braced' arrows denote the direction of indexing for the geometric sequence/transform.
\begin{remark}
    It is worth verifying that equation (\ref{eqn: a_i_chapt_9}) is consistent with Figure \ref{fig: construciton_lemma_1}. Indeed, since the geometric sequences on row $m-1$ overlap, one can observe that for $0\le i \le p_2$, $$a_i=a_{i-1}\cdot\left(\prod_{k=0}^{p^h+1}r_{S,i-1}\right)\cdot \left(\prod_{k=0}^{p^h-1} r_{C,i}^{-1}\right).$$ Two applications of \hyperlink{FLT}{FLT} and a simple induction on $i$ establishes equation (\ref{eqn: a_i_chapt_9}).
\end{remark}
\begin{proof}The proof is split into parts. 
\subsubsection*{\textbf{Splitting row }$m$}\hfill\\
\noindent First, row $m$ is split into overlapping parts so that each part matches the setup of either Lemma \ref{invwind} or Corollary \ref{between_two_windows_cor}. To this end, two extra variations of the $p$-Cantor sequence are defined. Let $a,r,a',r', a'', r''\in\F_q\backslash\{0\}$. \begin{itemize}
    \item $\widetilde{\textbf{C}}^{(p)}_{h}(r,a):=\{0\}_{0\le i <p^h}\oplus \textbf{C}_h^{(p)}{(r,a)}\oplus\{0\}_{0\le i <p^h}$;
    \item $\overline{\textbf{C}}^{(p)}_h(r',a',r'',a''):= \{0\}_{i=1,2}\oplus \textbf{C}^{(p)}_h(r',a')\oplus\{0\}_{0\le i <p^h}\oplus \textbf{C}^{(p)}_h(r'',a'')\oplus \{0\}_{i=1,2}$
\end{itemize}
\noindent Row $m$ of $W_p(\widetilde{\textbf{S}})$ is split into $p$ overlapping sections as follows:
\begin{figure}[H]
    \centering
    \includegraphics[width=1\linewidth]{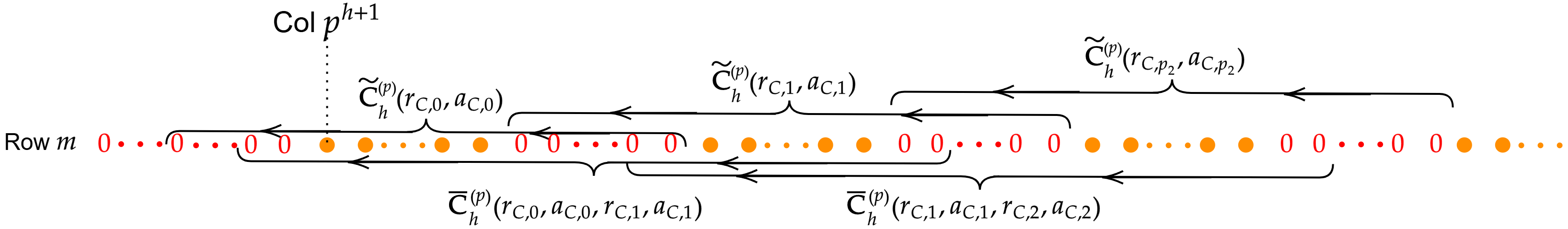}
    \caption{Row $m$ is split into overlapping copies of $\widetilde{\textbf{C}}_n^{(p)}(r_{C,i},a_{C,i})$ and $\overline{\textbf{C}}_n^{(p)}(r_{C,i},a_{C,i},r_{C,i+1},a_{C,i+1})$.}
\end{figure}
\subsubsection*{\textbf{Applying Lemma }\ref{invwind} \textbf{to} $\widetilde C_h^{(p)}(r_{C,i},a_{C,i})$}\hfill\\
\noindent For $0\le i \le p_2$, consider $(W_p(\widetilde{\textbf{S}})[m,p^{h+1}+2ip^h+j])_{-p^h\le j \le 2p^h-1}$ (which is equal to $\widetilde{\textbf{C}}_h^{(p)}(r_{C,i},a_{C,i})$) and $(W_p(\widetilde{\textbf{S}})[m-1,p^{h+1}+2ip^h+j])_{0\le j \le p^h-1}$ (which is equal to $\textbf{G}_{h,s}(r_{L,i},a_{L,i})$). This portion of the number wall satisfies the conditions of Lemma \ref{invwind} with the following values, where the notation on the left of the ``$\leftrightarrow$'' is from Lemma \ref{invwind}, and the notation on the right of the ``$\leftrightarrow$'' is from Lemma \ref{constrution_lemma_1}: \begin{align*}
    &a_0\leftrightarrow a_i& &r_0\leftrightarrow r_{C,i}& &a_1\leftrightarrow a_{C,i}& &r_1\leftrightarrow r_{C,i}& &m\leftrightarrow m.&
\end{align*}
\noindent Therefore, one has the following values on rows $m+p^h-1$ and $m+p^h$: \begin{enumerate}
     \item Row $m+p^h-1$:\begin{enumerate}
            \item[\hyperlink{CL11a}{(1a)}] By Lemma \ref{invwind} (\hyperlink{lem2.3}{3}), $W_p(\widetilde{\textbf{S}})[m+p^h-1,p^{h+1}+2ip^h+j]=a_{C,i}\cdot r_{C,i}^{-j} \cdot c_j^{(p)}$ for $0\le j \le p^h-1$ and $0\le i \le p_2$;
            \item[\hyperlink{CL11b}{(1b)}] By Theorem \ref{window}, $W_p(\widetilde{\textbf{S}})[m+p^h-1,p^{h+1}+(2i+1)p^h+j]=0$ for $0\le j \le p^h-1$ and $0\le i \le p_2-1$.
        \end{enumerate}
    \item Row $m+p^n$: \begin{enumerate}
            \item[\hyperlink{CL12a}{(2a)}] By Lemma \ref{invwind} (\hyperlink{lem2.4}{4}), $W_p(\widetilde{\textbf{S}})[m+p^h, p^{h+1}+2ip^h+j]=\frac{a_{C,i}^2}{a_i}\cdot \left(r_{C,i}\right)^{-j}$ for $0 \le j \le p^h-1$ and $0 \le i \le p_2$,
        \end{enumerate}
    \item Row $m+p^h+1$:  \begin{enumerate}\item[\hyperlink{CL13c}{(3c)}] By Lemma \ref{invwind} (\hyperlink{lem2.2}{2}), $W_p(\widetilde{\textbf{S}})[m+p^h+1, p^{h+1}+2ip^h+j]_{1\le j \le p^h-2}=0$ for $0 \le i \le p_2-1$.\end{enumerate}
\end{enumerate}
\noindent This is illustrated below:
\begin{figure}[H]
    \centering
    \includegraphics[width=1\linewidth]{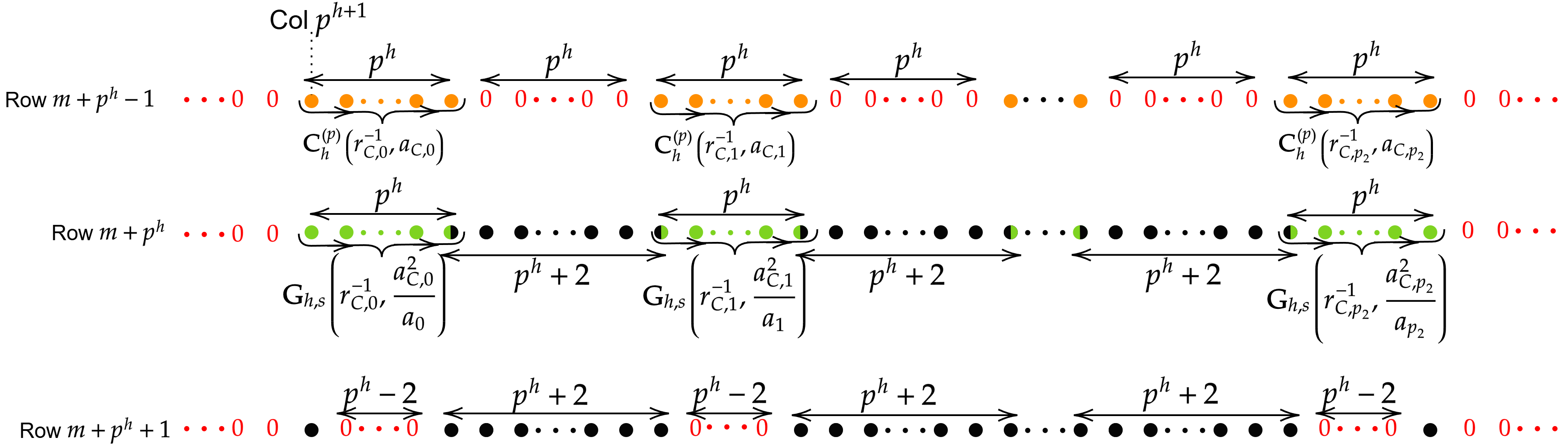}
    \caption{An illustration of rows $m+p^n-1$, $m+p^n$ and $m+p^n+1$ akin to Figure \ref{fig: construciton_lemma_1}. Here, the black dots represent entries that are yet to be calculated.}

\end{figure}
\subsubsection*{\textbf{Applying Lemma }\ref{between_two_windows} \textbf{and Corollary} \ref{between_two_windows_cor} \textbf{to} $\overline{\textbf{C}}^{(p)}_h(r_{C,i},a_{C,i},r_{C,i+1},a_{C,i+1})$}\hfill\\

\noindent Much like how Lemma \ref{invwind} was just applied to $\widetilde{\textbf{C}}_h^{(p)}(r_{C,i},a_{C,i})$, one now aims to utilise both Lemma \ref{between_two_windows} and Corollary \ref{between_two_windows_cor} with $\overline{\textbf{C}}^{(p)}_h(r_{C,i},a_{C,i},r_{C,i+1},a_{C,i+1})$. In particular, consider the following parts of $W_p(\widetilde{\textbf{S}})$: \begin{itemize}
    \item Row $m-2$:  \begin{itemize}
        \item $(W_p(\widetilde{\textbf{S}})[m-2, p^{h+1}+(2i+1)p^h-1+j])_{0\le j \le p^h+1}$. That is, the sequence $S^{(p)}_h(r_{S,i},a_{S,i})$,
    \end{itemize}
    \item Row $m-1$: \begin{itemize}
        \item $(W_p(\widetilde{\textbf{S}})[m-1,p^{h+1}+2ip^h+j])_{0\le j \le 3p^h-1}$. That is, the sequence \[r_{C,i}^{p^h-1}a_i~, \dots, ~r_{C,i}a_i, ~a_i, ~r_{S,i}a_i,~\cdots, ~r_{S,i}^{p^h+1}a_i=r_{C,i+1}^{p^h-1}a_{i+1},~\dots, ~a_{i+1}; \]
    \end{itemize}
    \item Row $m$: \begin{itemize}
        \item $(W_p(\widetilde{\textbf{S}})[m,p^{h+1}+2ip^h+j])_{0\le j \le 3p^h-1}$. That is, $\overline{\textbf{C}}^{(p)}_h(r_{C,i},a_{C,i},r_{C,i+1},a_{C,i+1}))$
    \end{itemize} 
\end{itemize}

\noindent This portion of the number wall satisfies the conditions of Corollary \ref{between_two_windows_cor} with the following values. Like before, the notation on the left is from Corollary \ref{between_two_windows_cor}, and the notation on the right is from Lemma \ref{constrution_lemma_1}.
\begin{align*}
    &r_L\leftrightarrow r_{C,i}& &r_A\leftrightarrow r_{S,i}& &a_A\leftrightarrow a_i& &r_R\leftrightarrow r_{C,i+1}& &a_R\leftrightarrow a_{i+1}&\\ &r_1\leftrightarrow r_{C,i}& &a_1\leftrightarrow a_{C,i}& &r_2\leftrightarrow r_{S,i}& &a_2\leftrightarrow a_{S,i}& &r_3\leftrightarrow r_{C,i+1}& &a_{3}\leftrightarrow a_{C,i+1}&.
\end{align*}
\noindent Then, by Corollary \ref{between_two_windows_cor}, one has that \begin{enumerate}
    \item[(2)] Row $m+p^h$: \begin{enumerate}
        \item[\hyperlink{CL12b}{(2b)}] By Lemma \ref{between_two_windows} (\hyperlink{lem3.1}{1}), $W_p(\widetilde{\textbf{S}})[m-1, p^{h+1}+(2i+1)p^h-1+j]= \frac{a_{C,i+1}^2}{a_{i+1}}\cdot \left(\frac{-a_{C,i}\cdot a_{i+1}}{r_{S,i}\cdot a_{C,i+1}\cdot a_i}\right)^{p^h+1-j}$ for ${0\le j \le p^h+1}$.
    \end{enumerate}
    \item[(3)] Row $m+p^h+1$: \begin{enumerate}
        \item[\hyperlink{CL13b}{(3b)}] By Corollary \ref{between_two_windows_cor}, $W_p(\widetilde{\textbf{S}})[m-2, p^{h+1}+(2i+1)p^h-1+j]=s^{(p)}_j\cdot x_i \cdot \left(\frac{-a_{C,i}\cdot a_{i+1}}{r_{C,i+1}\cdot a_{C,i+1}\cdot a_i}\right)^{p^h+1-j}$, where $0\le j \le p^{h}+1$ and \[x_i=\frac{a_{C,i+1}^3}{a_{i+1}^2}\cdot\left(1+\frac{a_{C,i}\cdot a_{i+1}}{r_{S,i}\cdot a_i\cdot a_{C,i+1}\cdot r_{C,i+1}}\right).\]
    \end{enumerate}
\end{enumerate}

\noindent Finally, \hyperlink{CL13a}{(3a)} follows because $W_p(\widetilde{\textbf{S}})[m-2,p^{h+1}]$ and $W_p(\widetilde{\textbf{S}})[m-2,2p^{h+1}-1]$ appear in the inner frame of the windows of size $p^{h+1}$ at the start and end of $\widetilde{\textbf{S}}$ respectively, they are both nonzero. This results in Figure \ref{fig: construciton_lemma_1} and completes the proof.
\end{proof}
\subsection{Construction Lemma 2}
\noindent This lemma is similar to Construction Lemma 1 (Lemma \ref{constrution_lemma_1}), but where the roles of the $p$-Cantor and $p$-Singer sequences are reversed. 
\begin{lemma}[Construction Lemma \hypertarget{con_lem_2}{2}]\label{constrution_lemma_2}
    Let\begin{itemize}
        \item $h\ge1$ be a natural number,
        \item $\textbf{S}$ be a sequence over $\F_p$ of length $p^h$,
        \item $\widetilde{\textbf{S}}$ be the sequence of length $3p^{h+1}$ defined as $$\widetilde{\textbf{S}}:=\{0\}_{0\le i <p^{h+1}}\oplus \textbf{S} \oplus \{0\}_{0\le i <p^{h+1}},$$
        \item $r_{S,i}\in \F_p\backslash\{0\}$ and $a_{S,i}\in\F_p$ for $0\le i <p_2$,
        \item $r_{C,i}\in\F_p\backslash\{0\}$ and  $a_{C,i}\in\F_p$ for $0 \le i \le p_2$,
        \item $a_0\in\F_p\backslash\{0\}$ and define \begin{equation}a_{i}:=\prod^{i-1}_{k=0}r_{S,k}^{-2}.\label{eqn: a_i_chapt_9}\end{equation}
    \end{itemize} Assume that rows $m-2$, $m-1$ and $m$ of $W_p(\widetilde{\textbf{S}})$ have the following structure: \begin{itemize}
        \item Row $m-2$:\begin{itemize}
            \item $W_p(\widetilde{\textbf{S}})[m-2,p^{h+1}+2ip^h+j]=a_{C,i}\cdot r_{C,i}^{j} \cdot c_j^{(p)}$ for $0\le j \le p^h-1$ and $0\le i \le p_2$;
            \item$W_p(\widetilde{\textbf{S}})[m-2,p^{h+1}+(2i+1)p^h+j]=0$ for $0\le j \le p^h-1$ and $0\le i \le p_2-1$.
        \end{itemize}

        \item Row $m-1$: \begin{itemize}
            \item $W_p(\widetilde{\textbf{S}})[m-1, p^{h+1}+2ip^h+j]=a_ir_{C,i}^{j}$ for $0 \le j \le p^h-1$ and $0 \le i \le p_2$,
            \item $W_p(\widetilde{\textbf{S}})[m-1, p^{h+1}+(2i+1)p^h-1+j]=a_{i+1}r_{S,i}^{p^h-1-j}$ for $0 \le j \le p^h+1$ and $0 \le i \le p_2-1$.
        \end{itemize}
        \item Row $m$: \begin{itemize}
            \item $W_p(\widetilde{\textbf{S}})[m,p^{h+1}]\neq0$ and  $W_p(\widetilde{\textbf{S}})[m-2,2p^{h+1}-1]\neq 0$,
            \item $W_p(\widetilde{\textbf{S}})[m, p^{h+1}+2ip^h+j]$ for ${1\le j \le p^h-2}=0$ and $0 \le i \le p_2$,
            \item $W_p(\widetilde{\textbf{S}})[m, p^{h+1}+(2i+1)p^h-1+j]=a_{S,i}r_{S,i}^{p^h+1-j} s^{(p)}_j$ for ${0\le j \le p^h+1}$ and $0 \le i \le p_2-1$.
        \end{itemize}
    \end{itemize} 
    \noindent Then, rows $m+p^h-3$, $m+p^h-2$ and $m+p^h-1$ have the following structure:
    \begin{enumerate}
    \item Row $m+p^h-3$: \begin{enumerate}
            \hypertarget{CL21a}{\item} $W_p(\widetilde{\textbf{S}})[m+p^h-3,p^{h+1}]\neq0$ and $W_p(\widetilde{\textbf{S}})[m-2,2p^{h+1}-1]\neq 0$,
            \hypertarget{CL21b}{\item} $W_p(\widetilde{\textbf{S}})[m+p^{h}-3, p^{h+1}+(2i+1)p^h-1+j]=s^{(p)}_j\cdot\frac{a_{i+1}^2\cdot r_{S,i}^2}{a_{S,i}}  \cdot r_{S,i}^{-j}$, with $0\le j \le p^{h}+1$ and $0\le i \le p_2-1$,
            \hypertarget{CL21c}{\item} $W_p(\widetilde{\textbf{S}})[m+p^n-3, p^{h+1}+2ip^h+j]=0$ for ${1\le j \le p^h-2}$ and $0 \le i \le p_2$.
    \end{enumerate}
        \item Row $m+p^h-2$: \begin{enumerate}
            \hypertarget{CL22a}{\item} $W_p(\widetilde{\textbf{S}})[m+p^h, p^{h+1}+2ip^h+j]=\left(-\frac{a_{S,i-1}\cdot a_{i+1}}{a_i\cdot a_{S,i}\cdot r_{C,i}}\right)^{p^h-1-j}\cdot a_{i+1}\cdot r_{S,i}^2$ for $0 \le j \le p^h-1$ and $0 \le i \le p_2$,
            \hypertarget{CL22b}{\item} $W_p(\widetilde{\textbf{S}})[m-1, p^{h+1}+(2i+1)p^h-1+j]= r_{S,i}^2\cdot a_{i+1}\cdot r_{S,i}^{-j}$ for ${0\le j \le p^h+1}.$
        \end{enumerate}
        \item Row $m+p^h-1$:\begin{enumerate}
            \hypertarget{CL23a}{\item}$W_p(\widetilde{\textbf{S}})[m+p^h-1,p^{h+1}+(2i+1)p^h+j]=0$ for $0\le j \le p^h-1$ and $0\le i \le p_2-1$,
            \hypertarget{CL23b}{\item} $W_p(\widetilde{\textbf{S}})[m+p^h-1,p^{h+1}+2ip^h+j]=y_i\cdot \left(-\frac{a_{S,i-1}\cdot a_{i+1}}{a_i\cdot a_{S,i}\cdot r_{C,i}}\right)^{p^h-1-j}\cdot c^{(p)}_{p^h-1-j}$ for $0\le j \le p^h-1$ and $0\le i \le p_2$, where \[y_i=r_{S,i}^2\cdot a_{S,i}\left(1+\frac{a_{S,i-1}\cdot a_{i+1}}{a_{S,i}\cdot r_{S,i}\cdot r_{C,i}\cdot a_{i}}\right).\]
            
        \end{enumerate}
        
    \end{enumerate} 
\end{lemma}
\begin{figure}[H]
    \centering
    \includegraphics[width=1\linewidth]{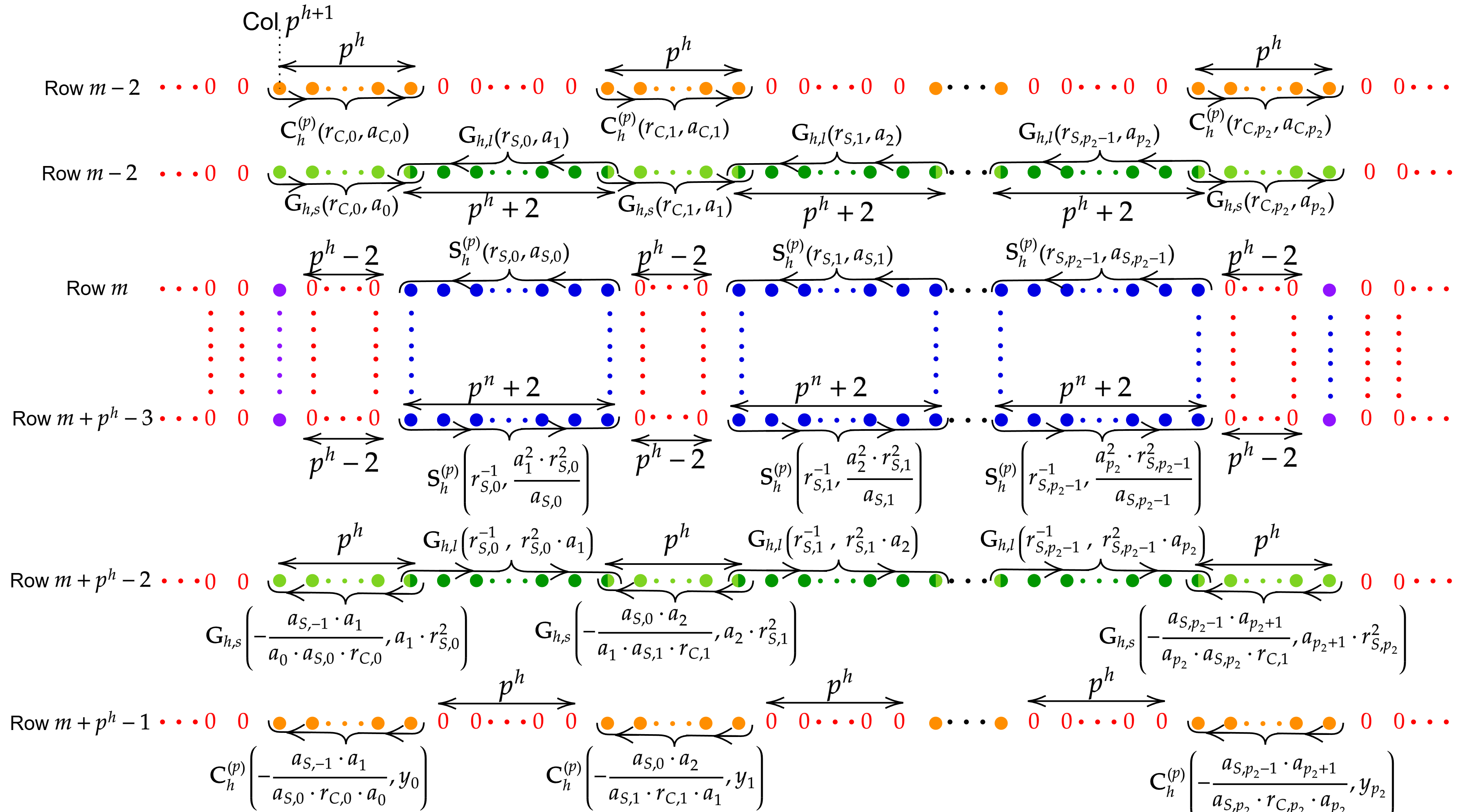}\label{fig: 9.4}
    \caption{The setup of Lemma \ref{constrution_lemma_1}. Each dot is an entry in $W_p(\widetilde{\textbf{S}})$, with the colour representing which sequence that entry belongs to. Some dots have multiple colours. Any geometric transform of the $p$-Singer ($p$-Cantor, respectively) sequence is in blue (orange, respectively). Geometric sequences are in green, with the shade being used to distinguish one geometric sequence from another. The two sided arrows denote the length of a sequence, whereas the one sided arrows denote the direction of indexing for the geometric sequence/transform.}
    \label{fig: construciton_lemma_2}
\end{figure}
\begin{proof}
    The proof is identical to that of Construction Lemma \hyperlink{con_lem_1}{1}, with Lemma \ref{Singer_invwind}, Lemma \ref{between_to_windows_singer} and Corollary \ref{between_two_windows_singer_cor} used in place of Lemma \ref{invwind}, Lemma \ref{between_two_windows} and Corollary \ref{between_two_windows_cor} respectively.
\end{proof}
\section{Proof of Theorem \ref{thm: morphism}} \label{Sect: 10}
\noindent Recall that  $\widetilde{\textbf{C}}_h^{(p)}:=\{0\}_{0\le i <p^h}\oplus \textbf{C}_h^{(p)}\oplus \{0\}_{0\le i <p^h}$. Theorem \ref{thm: morphism} is established by showing that for every $h\ge0$, one has \begin{equation}\label{eqn: sect_10_ind}\chi\left(\left(W_p\left(\widetilde {\textbf{C}}_h^{(p)}\right)[m,n]\right)_{0\le m < p^h,~ p^h\le n <2p^h}\right)=\Pi(\Phi^h(A)).\end{equation} This is done by induction on $h$.

\subsection{The Base Case}
\noindent When $h=0$, equation (\ref{eqn: sect_10_ind}) is trivial. In general, the $h=i$ case follows from applying Construction Lemmata \hyperlink{con_lem_1}{1} and \hyperlink{con_lem_2}{2} to the $h=i-1$ case. However, this does not hold in the $h=1$ case, since this would involve windows of size $-2$. Therefore, the $h=1$ case is established independently and serves as the base case for the induction. This is achieved by explicitly calculating the value of every entry in $\left(W_p\left(\widetilde {\textbf{C}}_1^{(p)}\right)[m,n]\right)_{-1\le m < p,~ p\le n <2p}$. \\

\noindent To this end, each entry in $W_p\left(\mathcal{C}_1\right)$ is categorised into one of four families of variables, depending on the parity of its row and column index modulo 2. In particular, let \begin{align*}
    &a_{i,j}:=W_p\left(\widetilde {\textbf{C}}_1^{(p)}\right)[2j-1,p+2i]~~\text{ for }~~0\le i\le p_2~~\text{ and }~~0\le j \le p_2, &\\
    &b_{i,j}:=W_p\left(\widetilde {\textbf{C}}_1^{(p)}\right)[2j-1,p+2i+1]~~\text{ for }~~0\le i\le p_2-1~~\text{ and }~~0\le j \le p_2, &\\
    &c_{i,j}:=W_p\left(\widetilde {\textbf{C}}_1^{(p)}\right)[2j,p+2i]~~\text{ for }~~0\le i \le p_2~~\text{ and }~~0\le j \le p_2, &\\
    &d_{i,j}:=W_p\left(\widetilde {\textbf{C}}_1^{(p)}\right)[2j,p+2i+1]~~\text{ for }~~0\le i\le p_2-1~~\text{ and }~~0\le j \le p_2. &
\end{align*}
\noindent That is, $ \left(W_p\left(\widetilde {\textbf{C}}_1{(p)}\right)[m,n]\right)_{-1\le m <p, ~p\le n <2p}$ appears as follows: \begin{equation*}
    \begin{matrix}
        \text{Row }-1&a_{0,0}&b_{0,0}&a_{1,0}&b_{1,0}&\cdots &a_{p_2-1,0}&b_{p_2-1,0}&a_{p_2,0}\\
        \text{Row }0&c_{0,0}&d_{0,0}&c_{1,0}&d_{1,0}&\cdots &c_{p_2-1,0}&d_{p_2-1,0}& c_{p_2,0}\\
        &a_{0,1}&b_{0,1}&a_{1,1}&b_{1,1}&\cdots &a_{p_2-1,1}&b_{p_2-1,1}&a_{p_2,1}\\
        &c_{0,1}&d_{0,1}&c_{1,1}&d_{1,1}&\cdots &c_{p_2-1,1}&d_{p_2-1,1}&c_{p_2,1}\\
        &\vdots&\vdots&\vdots&\vdots&&\vdots&\vdots&\vdots\\
        &a_{0,p_2-1}&b_{0,p_2-1}&a_{1,p_2-1}&b_{1,p_2-1}&\cdots &a_{p_2-1,p_2-1}&b_{p_2-1,p_2-1}&a_{p_2,p_2-1}\\
        &c_{0,p_2-1}&d_{0,p_2-1}&c_{1,p_2-1}&d_{1,p_2-1}&\cdots &c_{p_2-1,p_2-1}&d_{p_2-1,p_2-1}&c_{p_2,p_2-1}\\
        &a_{0,p_2}&b_{0,p_2}&a_{1,p_2}&b_{1,p_2}&\cdots &a_{p_2-1,p_2}&b_{p_2-1,p_2}&a_{p_2,p_2}\\
        \text{Row }p-1&c_{0,p_2}&d_{0,p_2}&c_{1,p_2}&d_{1,p_2}&\cdots &c_{p_2-1,p_2}&d_{p_2-1,p_2}&c_{p_2,p_2}
    \end{matrix}
\end{equation*}
\begin{lemma}
    Let $a_{i,j}$, $b_{i,j}$, $c_{i,j}$ and $d_{i,j}$ be as above and let $b_{-1,j},b_{p_2,j},d_{-1,j},d_{p_2,j}=0$ for all $0\le j \le p_2$. Then, \begin{itemize}
        \item For $0\le i\le p_2$ and $0\le j \le p_2$, \begin{equation}
            a_{i,j}= \left(\left(\prod_{k=0}^{j-1} {p_2+k \choose i+k}\right)\cdot \left(\prod_{k=1}^{j-1}\left(\frac{k}{p_2+k-i}\right)^{j-k}\right)\right)^2,\label{eqn: base_case_a}
        \end{equation}
        \item For $0\le i\le p_2-1$ and $0\le j \le p_2$,\begin{equation}
            b_{i,j}= (-1)^j\cdot\left(\prod_{k=0}^{j-1}{p_2+k\choose i+k+1}^2\cdot\frac{i+k+1}{(p_2-i+k)^{2(j-k)-1}}\right)\cdot \prod_{k=1}^{j-1} k^{2(j-k)},\label{eqn: base_case_b}
        \end{equation}
        \item For $0\le i\le p_2$ and $0\le j \le p_2$, \begin{equation}
            c_{i,j}= {p_2+j\choose i+j} \cdot \prod_{k=0}^{j-1} {p_2+k\choose i+k}^2\cdot \left(\frac{k+1}{p_2+k-i+1}\right)^{2(j-k)-1},\label{eqn: base_case_c}
        \end{equation}
        \item For $0\le i\le p_2-1$ and $0\le j \le p_2$, \begin{equation}
            d_{i,j}=0.\label{eqn: base_case_d}
        \end{equation}

    \end{itemize} 
\end{lemma}
\begin{proof}

\noindent These formulae are proven simultaneously by induction on $j$, with the base case ($j=0$) given by  \begin{align*}
&a_{i,0}=1,& &b_{i,0}=1,& &c_{i,0}={p_2\choose i},& &d_{i,0}=0.&
\end{align*}The above values follow immediately from the definition of row $-1$ of a number wall (Definition \ref{nw}) and the definition of the $p$-Cantor sequence (Definition (\ref{pcant}).\\

\noindent For the induction step, let $0\le j\le p_2$ and assume that \begin{enumerate}
    \hypertarget{base_case_ind_1}{\item} equations (\ref{eqn: base_case_a}-\ref{eqn: base_case_d}) hold for all $0\le j'<j$,
    \hypertarget{base_case_ind_2}{\item} for all $0\le j'<j$ and $0\le i \le p_2$, $a_{i,j'}$ and $c_{i,j'}$ are \textit{nonzero},
    \hypertarget{base_case_ind_3}{\item} for all $0\le j'<j$ and $0\le i \le p_2-1$, $b_{i,j'}$ is \textit{nonzero}.
\end{enumerate} 

\noindent By assumptions (\hyperlink{base_case_ind_2}{2}) and (\hyperlink{base_case_ind_3}{3}), $ \left(W_p\left(\widetilde {\textbf{C}}_1^{(p)}\right)[m,n]\right)_{-1\le m <2j-1, ~p\le n <2p}$ appears as follows:
\begin{figure}[H]
    \centering
    \includegraphics[width=0.5\linewidth]{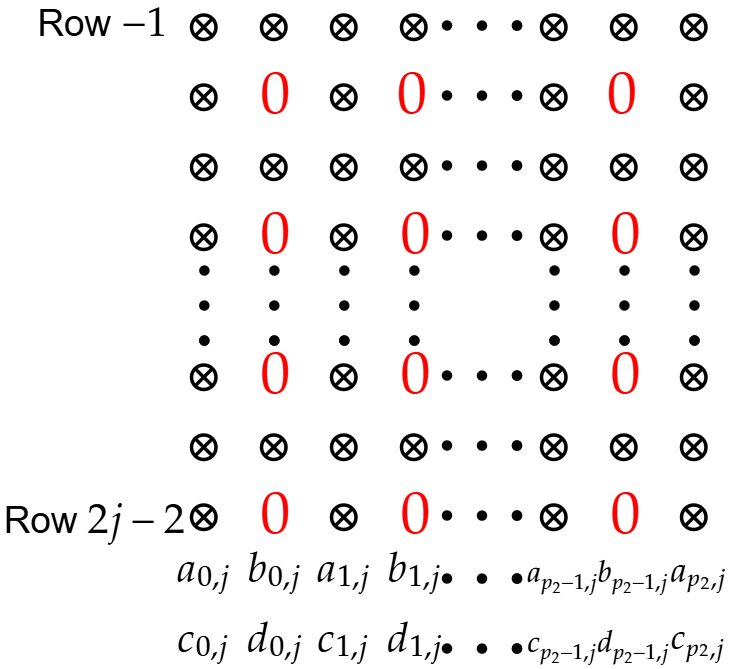}
    \caption{The number wall $W_p\left(\widetilde {\textbf{C}}_1^{(p)}\right)[m,n]$ for $-1 \le m \le j$ and $p\le n <2p$. The crossed dots indicate that an entry is nonzero.}
\end{figure}
\noindent As seen above, $d_{i,j}$ is in the southern outer frame of a window of size one, given by $d_{i,j-1}$. More specifically, $d_{i,j}=H_1$ (with notation from Figure \ref{window}). Furthermore, one has that \begin{align*}
    &E_1=d_{i,j-2}=0,& &F_1=d_{i-1,j-1}=0,& &G_1=d_{i+1,j-1}=0&
\end{align*}and hence, by \hyperlink{FC3}{FC3}, $d_{i,j}=0$. \\

\noindent Next, assumptions (\hyperlink{base_case_ind_2}{2}) and (\hyperlink{base_case_ind_3}{3}) also imply that $a_{i,j}$, $b_{i,j}$ and $c_{i,j}$ are all calculated using \hyperlink{FC1}{FC1}. Therefore, one obtains the following recurrence relations:
\begin{align}
        &a_{i,j}=\frac{c_{i,j-1}^2}{a_{i,j-1}},& &b_{i,j}=-\frac{c_{i,j-1}\cdot{c_{i+1,j-1}}}{b_{i,j-1}},& &c_{i,j}=\frac{a_{i,j}^2-b_{i,j}\cdot b_{i-1,j}}{c_{i,j-1}}\cdotp&\label{eqn: base_case_eqn}
    \end{align}
\noindent These are used to verify equations (\ref{eqn: base_case_a}), (\ref{eqn: base_case_b}) and (\ref{eqn: base_case_c}). As the calculations are elementary, only an outline is provided for each case.

\subsubsection*{\textbf{Verifying equation} (\ref{eqn: base_case_a}):}\hfill\\
\noindent By the induction hypothesis, equations (\ref{eqn: base_case_a}) and (\ref{eqn: base_case_c}), and the first part of equation (\ref{eqn: base_case_eqn}), \begin{align}
    a_{i,j}=\left(\prod_{k=0}^{j-1} {p_2+k\choose i+k}^2\right) \cdot \frac{\prod_{k=0}^{j-2}\left(\frac{k+1}{p_2-i+k+1}\right)^{4(j-1-k)-2}}{\prod_{k=1}^{j-2}\left(\frac{k}{p_2-i+k}\right)^{2(j-k)-2}}\label{eqn: base_canse_ind_a_1}
\end{align}
\noindent Separating the $k=j-2$ term of the numerator and re-indexing to match the denominator gives \begin{align*}
    (\ref{eqn: base_canse_ind_a_1})&=\left(\prod_{k=0}^{j-1} {p_2+k\choose i+k}^2\right) \cdot \left(\frac{j-1}{p_2-i+j-1}\right)^2 \cdot \prod_{k=0}^{j-2}\left(\frac{k}{p_2-i+k}\right)^{2(j-k)}
\end{align*} \noindent which is equal to $a_{i,j}$ upon combining the terms into one product.
\subsubsection*{\textbf{Verifying equation} (\ref{eqn: base_case_b}):}\hfill\\
\noindent By the induction hypothesis, equations (\ref{eqn: base_case_b}) and (\ref{eqn: base_case_c}), and the second part of equation (\ref{eqn: base_case_eqn}),
\begin{align}
    b_{i,j}&=(-1)^j\cdot {p_2+j-1\choose i+j}\nonumber\\
    &\cdot \underbrace{{p_2+j-1\choose i+j-1}\cdot \left(\prod_{k=0}^{j-2}{p_2+k\choose i+k}^2\cdot \frac{1}{(p_2+k-i+1)^{2(j-k)-3}}\cdot \frac{1}{i+k+1}\right)}_X\nonumber\\
    &\cdot \underbrace{\left(\prod_{k=1}^{j-2} k^{2(j-k)-2}\right)^{-1}\cdot \left(\prod_{k=0}^{j-2} (k+1)^{2(j-k)-3}\right)^2}_{Y}.\label{eqn: base_case_ind_b_1}
\end{align}
\noindent The part of the above equation labelled $X$ is dealt with first. Indeed, using the factorial definition of the choose function and factorising out every power of $p_2-i$ that appears, one has that \begin{align*}
    X&= \underbrace{\frac{1}{(p_2-i)^{2j-1}}\left( \prod_{k=0}^{j-2} \frac{1}{(p_2+k-i+1)^{(2j-k)-3}}\right)}_{X_1} \\&\cdot \underbrace{\frac{(p_2+j-1)!}{(i+j-1)!\cdot (p_2-i-1)!}
    \cdot \prod_{k=0}^{j-2} \left(\frac{(p_2+k)!}{(i+k)!\cdot (p_2-i-1)!}\right)^2\cdot\frac{1}{i+k+1}}_{X_2} .
\end{align*}
\noindent By re-indexing the product in $X_1$, one obtains \begin{equation}
    X_1=\prod_{k=0}^{j-1} \frac{1}{(p_2-i+k)^{2(j-k)-1}} \cdotp
\end{equation}
\noindent Next, note that \begin{align*}
    X_2&={p_2+j-1\choose i+j}\cdot (i+j) \cdot \prod_{k=0}^{j-2} {p_2+k\choose i+k+1}^2\cdot (i+k+1)
\end{align*} 
\noindent Now $Y$ is dealt with. Indeed, by re-indexing the second product to begin at $k=1$ one obtains that \begin{equation*}
    Y= \prod_{k=1}^{j-1}k^{2(j-k)}.
\end{equation*}
\noindent Inserting $X_1$, $X_2$ and $Y$ back into equation (\ref{eqn: base_case_ind_b_1}) completes the proof in this case.
\subsubsection*{\textbf{Verifying equation} (\ref{eqn: base_case_c}):}\hfill\\
\noindent By splitting the third part of equation (\ref{eqn: base_case_eqn}) into two summands, applying equations (\ref{eqn: base_case_a}), (\ref{eqn: base_case_b}) and (\ref{eqn: base_case_c}) and simplifying, one obtains
\begin{align}
    c_{i,j}&= {p_2+j-1 \choose i+j-1}^3\cdot \left(\prod_{k=0}^{j-2} {p_2+k\choose i+k}^2\right) \cdot \left(\prod_{k=1}^{j-1}\left(\frac{k}{p_2-i+k}\right)^{2(j-k)+1}\right)\label{eqn: base_case_ind_c_1}\\
    &-{p_2+j-1 \choose i+j-1}\cdot \frac{1}{(p_2-i)^{2j-1}}\cdot \left(\prod_{k=0}^{j-1} {p_2+k\choose i+k+1}^2 \cdot \frac{(i+k+1)(i+k)}{(p_2-i+k+1)^{2(j-k)-1}}\right)\cdot \prod_{k=1}^{j-1} k^{2(j-k)+1}.\nonumber
\end{align}
\noindent Next, factorise to get \begin{align}
    (\ref{eqn: base_case_ind_c_1})&= {p_2+j-1\choose i+j-1}\cdot \left(\prod_{k=1}^{j-1}\left(\frac{k}{p_2-i+k}\right)^{2(j-k)+1}\right)\label{eqn: base_case_ind_c_2}\\
    &\cdot\left(\left(\prod_{k=0}^{j-1} {p_2+k\choose i+k}^2\right)-\frac{1}{(p_2-i)^{2j-1}\cdot (p_2-i+j)}\cdot \left(\prod_{k=0}^{j-1} {p_2+k\choose i+k+1}^2 \cdot (i+k+1)(i+k)\right)\right)\nonumber
\end{align}
\noindent Applying the factorial definition of the choose function, one factorises further:
\begin{align}
    (\ref{eqn: base_case_ind_c_2}) &= {p_2+j-1\choose i+j-1}\cdot \left(\prod_{k=1}^{j-1}\left(\frac{k}{p_2-i+k}\right)^{2(j-k)+1}\right) \label{eqn: base_case_ind_c_3}\\&\cdot \left(\frac{1}{((p_2-i)!)^{2j-1}\cdot (p_2-i-1)!}\cdot\prod_{k=0}^{j-1}\left(\frac{(p_2+k)!}{(i+k)!}\right)^2\right)\cdot \left(\frac{1}{p_2-i}-\frac{i}{(i+j)\cdot (p_2-i+j)}\right)\nonumber
\end{align}
\noindent Using that \[\frac{1}{p_2-i}-\frac{i}{(i+j)\cdot (p_2-i+j)}=\frac{j\cdot(p_2+j)}{(p_2-i)\cdot (i+j)\cdot (p_2-i+j)},\]
one obtains \begin{align*}
    (\ref{eqn: base_case_ind_c_3}) &={p_2+j\choose i+j} \cdot \frac{j}{p_2-i+j} \cdot \left(\prod_{k=1}^{j-1}\left(\frac{k}{p_2-i+k}\right)^{2(j-k)+1}\right) \cdot \prod_{k=0}^{j-1}{p_2+k\choose i+k}^2.
\end{align*}
\noindent Re-indexing the first product above to begin at $k=0$ establishes equation (\ref{eqn: base_case_c}).

\subsubsection*{\textbf{Completing the Base Case}}\hfill\\
\noindent Finally, assumptions \hyperlink{base_case_ind_2}{2} and \hyperlink{base_case_ind_3}{3} are shown to hold. Indeed, $a_{i,j}$, $b_{i,j}$ and $c_{i,j}$ are all products of nonzero terms over $\F_p$ and are therefore nonzero themselves.
\end{proof}
\subsection{The Inductive Step}
\noindent Now that the base case has been established, the induction step is proved. 
\subsubsection{\textbf{Preliminaries to the Induction Step}}\hfill\\
\noindent Assume that (\ref{eqn: sect_10_ind}) holds up to some natural number $h$. The goal is now to show that \begin{align*}&\chi\left(\left(W_p\left(\widetilde {\textbf{C}}_{h+1}^{(p)}\right)[m,n]\right)_{0\le m < p^h,~ p^h\le n <2p^h}\right)=\Pi\left(\Phi_p^{h+1}(A)\right)\\&=\Pi\left(\begin{matrix}
        \Phi_p^{h}(A)&\Phi_p^{h}(0)&\Phi_p^{h}(A)&\Phi_p^{h}(0)&\Phi_p^{h}(A) &\cdots& \Phi_p^{h}(A)&\Phi_p^{h}(0)&\Phi_p^{h}(A)\\
        \Phi_p^h(F)&\Phi_p^h(B)&\Phi_p^h(F)&\Phi_p^h(B)&\Phi_p^h(F)&\cdots &\Phi_p^h(F)&\Phi_p^h(B)&\Phi_p^h(F)\\
       \Phi_p^{h}(A)&\Phi_p^{h}(0)&\Phi_p^{h}(A)&\Phi_p^{h}(0)&\Phi_p^{h}(A) &\cdots& \Phi_p^{h}(A)&\Phi_p^{h}(0)&\Phi_p^{h}(A)\\
        \Phi_p^h(F)&\Phi_p^h(B)&\Phi_p^h(F)&\Phi_p^h(B)&\Phi_p^h(F)&\cdots &\Phi_p^h(F)&\Phi_p^h(B)&\Phi_p^h(F)\\
        \Phi_p^{h}(A)&\Phi_p^{h}(0)&\Phi_p^{h}(A)&\Phi_p^{h}(0)&\Phi_p^{h}(A) &\cdots& \Phi_p^{h}(A)&\Phi_p^{h}(0)&\Phi_p^{h}(A)\\
        \vdots &\vdots&\vdots&\vdots&\vdots&&\vdots&\vdots&\vdots\\
        \Phi_p^{h}(A)&\Phi_p^{h}(0)&\Phi_p^{h}(A)&\Phi_p^{h}(0)&\Phi_p^{h}(A) &\cdots& \Phi_p^{h}(A)&\Phi_p^{h}(0)&\Phi_p^{h}(A)\\
        \Phi_p^h(F)&\Phi_p^h(B)&\Phi_p^h(F)&\Phi_p^h(B)&\Phi_p^h(F)&\cdots &\Phi_p^h(F)&\Phi_p^h(B)&\Phi_p^h(F)\\
        \Phi_p^{h}(A)&\Phi_p^{h}(0)&\Phi_p^{h}(A)&\Phi_p^{h}(0)&\Phi_p^{h}(A) &\cdots& \Phi_p^{h}(A)&\Phi_p^{h}(0)&\Phi_p^{h}(A)\\
    \end{matrix}\right)\end{align*}\noindent Recall the image of $0,F,E_N,E_W,E_S,E_E,C_{NE},C_{NW},C_{SW}$ and $C_{SE}$ under $\Phi_p$ from Section \ref{subsect: morphism}. An easy induction shows that for all $k\in\N$,\begin{figure}[H]
        \centering
        \includegraphics[width=1\linewidth]{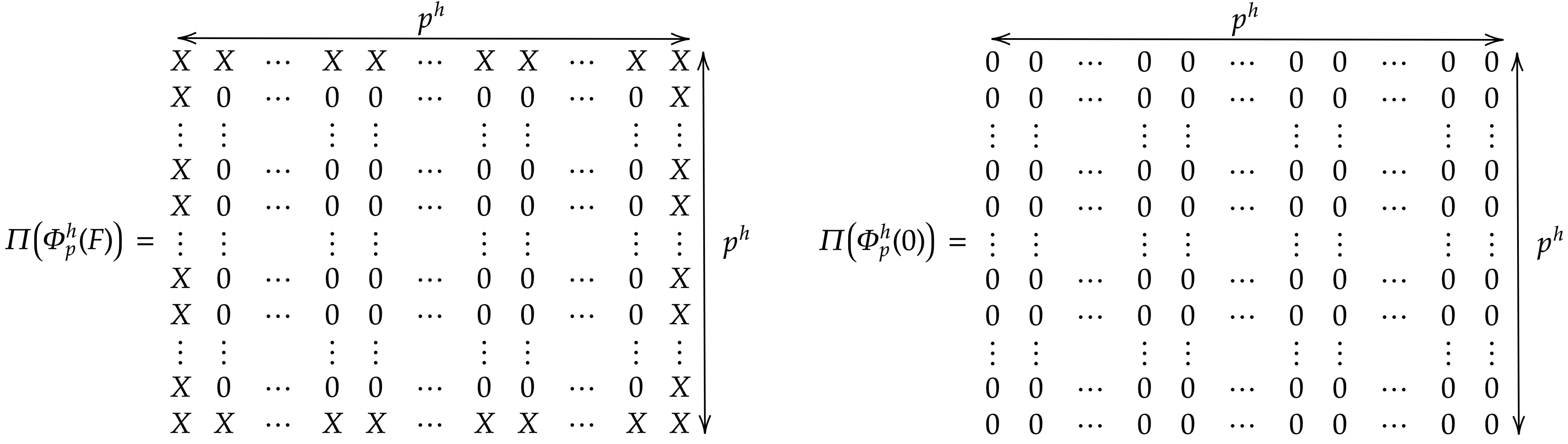}
    \end{figure}
 \noindent Recall $\rho$ from equation (\ref{def:rotate}). Another simple induction shows that \begin{equation*}
     \Pi\left(\Phi_p^h(A)\right)=\rho\left(\Pi\left(\Phi_p^h(B)\right)\right)~~~\text{ and }~~~\Pi\left(\Phi_p^h(B)\right)=\rho\left(\Pi\left(\Phi_p^h(A)\right)\right).
 \end{equation*} Furthermore, define \[\widetilde{\textbf{S}}_h^{(p)}:=\{0\}_{0\le i \le p^h+1}\oplus \textbf{S}_h^{(p)}\oplus \{0\}_{0\le i \le p^h+1}.\]Then, \begin{align}
     \Pi\left(\Phi_p^h(B)\right)&=\rho\left(\Pi\left(\Phi_p^h(A)\right)\right)\nonumber\\
     &=\rho\left(\chi\left(\left(W_p\left(\widetilde {\textbf{C}}_{h}^{(p)}\right)[m,n]\right)_{0\le m < p^h,~ p^h\le n <2p^h}\right)\right)\nonumber\\
     &\substack{\text{Corollary \ref{cor: rotate_profile}}\\=}\chi\left(\left(W_p\left(\widetilde{\textbf{S}}_h^{(p)}\right)[i,j]\right)_{-1\le i <p^h-1,~ p^h+3\le j <p^h+3}\right)\label{eqn: pink_square}
 \end{align}
\noindent In other words, the profile of the green square in Figure \ref{Fig: Green_square} (with $m=0$ and $n=p^h+2$) is given by $\Pi\left(\Phi_p^h(B)\right)$.
\subsubsection{\textbf{The Strategy for the Induction Step}}\hfill\\
\noindent The plan is to repeatedly apply Construction Lemma \hyperlink{con_lem_1}{1} followed by Construction Lemma \hyperlink{con_lem_1}{2}, where the output of Construction Lemma \hyperlink{con_lem_1}{1} (\hyperlink{con_lem_1}{2}, respectively) is the input for Construction Lemma \hyperlink{con_lem_1}{2} (\hyperlink{con_lem_1}{1}, respectively). To this end, define $$r_{S,i,j},~a_{S,i,j}\in\F_p ~~\text{ for }~~0\le i <p_2 ~~\text{ and }~~ 0\le j \le p$$ and $$r_{C,i,j},~a_{C,i,j},~\text{ and }a_{i,j}\in\F_p ~~\text{ for }~~0\le i \le p_2 ~~\text{ and }~~0\le j \le p.$$ These are intended to match the notation from Construction Lemmas \hyperlink{con_lem_1}{1} and \hyperlink{con_lem_1}{2}, with the third index denoting how many times either of the Construction Lemmas has been applied before. \begin{remark}
Throughout this proof, it is assumed that $r_{S,i,j},a_{S,i,j},a_i,r_{C,i,j}$ and $a_{C,i,j}$ are not equal to zero (with the exception of $a_{S,i,0}$). This is verified in Section \ref{Sect: 10.3}\label{remark: non_zero}\end{remark}
\subsubsection*{\textbf{Applying Construction Lemma 1}}
\noindent To begin, recall that ${p\choose i}=0$ for all $i\not\in\N$. Lemma \ref{pcant_geom} states that $$\textbf{C}^{(p)}_{h+1}=\bigoplus_{i=0}^{p-1}\textbf{C}_h^{(p)}{\left(1,{p_2\choose i/2}\right)}.$$ Therefore, rows $-2$, $-1$ and 0 of $W_p\left(\widetilde{\textbf{C}}_{h+1}^{(p)}\right)$ appear as follows:
\begin{figure}[H]
    \centering
    \includegraphics[width=1\linewidth]{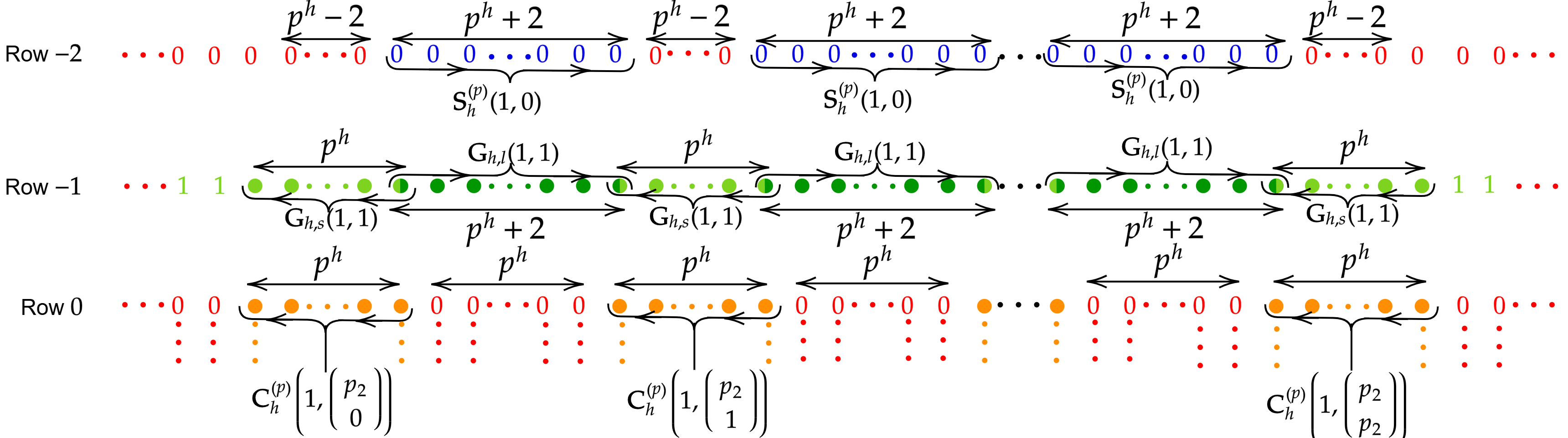}
    \caption{Rows $-2$, $-1$ and 0 of $W_p(\widetilde{\textbf{C}}_{h+1}(p))$, presented in the same form as in Construction Lemma 1. }
\end{figure}
\noindent Hence, the profile of rows $0$ to $p^h-1$ of $W_p(\widetilde{\textbf{C}}_{h+1}(p))$ is:\begin{figure}[H]
    \centering
    \includegraphics[width=1\linewidth]{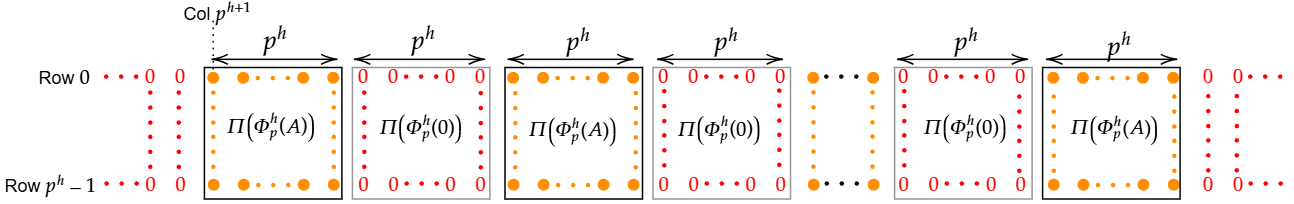}
    \caption{The profile of rows $0$ to $p^h-1$ of $W_p\left(\widetilde{\textbf{C}}_{h+1}^{(p)}\right)$, with colours matching the presentation of Construction Lemma \protect\hyperlink{con_lem_1}{1}. The black and grey squares denote the parts of the number wall whose profile is given by $\Pi\left(\Phi^k_p(A)\right)$ and $\Pi\left(\Phi^k_p(0)\right)$, respectively. }
    \label{fig: profile_rows}
\end{figure}
\noindent Above, \begin{itemize}
    \item the profile of the entries in the black squares from Figure \ref{fig: profile_rows} being $\Pi(\Phi_p^h(A))$ follows from Lemma \ref{invwind} and the induction hypothesis,
    \item the entries in the grey squares from Figure \ref{fig: profile_rows} being all zero follows from the Square Window Theorem (Theorem \ref{window}).
\end{itemize}

\noindent Furthermore, by Construction Lemma \hyperlink{con_lem_1}{1} with \begin{align*}
    &r_{S,i}\leftrightarrow r_{S,i,0}=1,& &a_{S,i}\leftrightarrow a_{S,i,0}=0,& &a_i\leftrightarrow a_{i,0}=1,& &r_{C,i}\leftrightarrow r_{C,i,0}=1,& &a_{C,i}\leftrightarrow a_{C,i,0}={p_2\choose i},& 
\end{align*} rows $p^h-1, p^h$ and $p^h+1$ appear as follows: 
\begin{figure}[H]
    \centering
    \includegraphics[width=1\linewidth]{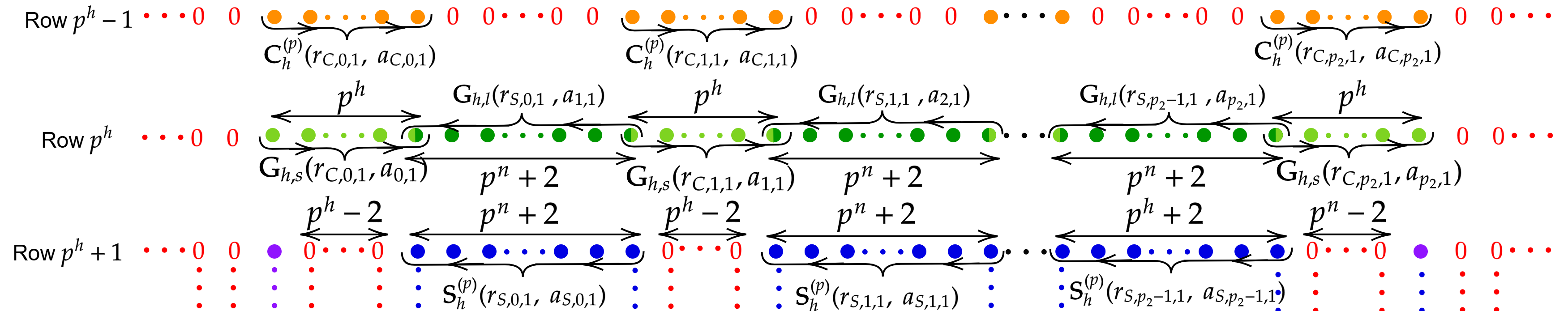}
    \caption{Rows $p^h-1$, $p^h$ and $p^h+1$ of $W_p\left(\widetilde{\textbf{C}}_{h+1}^{(p)}\right)$, presented in the same form as Construction Lemma \protect\hyperlink{con_lem_1}{1}}
\end{figure}
\noindent Here, by Construction Lemma \hyperlink{con_lem_1}{1},
\begin{align}
    &r_{S,i,1}\substack{\hyperlink{CL13b}{(3b)}\\=}-{p_2\choose i}{p_2\choose i+1}^{-1},& &a_{S,i,1}\substack{\hyperlink{CL13b}{(3b)}\\=}{p_2\choose i+1}^3\cdot \left(1+{p_2\choose i}{p_2 \choose i+1}^{-1}\right),& &a_{i,1}\substack{\hyperlink{CL12a}{(2a)}\\=}{p_2\choose i}^2,&\nonumber \\
    &r_{C,i,1}\substack{\hyperlink{CL11a}{(1a)}\\=}1,& &a_{C,i,1}\substack{\hyperlink{CL11a}{(1a)}\\=}{p_2 \choose i}.& \label{eqn: j=1}
\end{align}
\subsubsection{\textbf{Applying Construction Lemma 2}}\hfill\\
\noindent Now, apply Construction Lemma 2 with \begin{align*}
    &m\leftrightarrow p^h+1& &r_{S,i}\leftrightarrow r_{S,i,1},& &a_{S,i}\leftrightarrow a_{S,i,1},& &a_i\leftrightarrow a_{i,1},& &r_{C,i}\leftrightarrow r_{C,i,1},& &a_{C,i}\leftrightarrow a_{C,i,1},&
\end{align*} \noindent where the values on the right-hand side are from equation (\ref{eqn: j=1}).\\

\noindent Hence, $\left(W_p\left(\widetilde{\textbf{C}}^{(p)}\right)[m,n]\right)_{0\le n < p^{h+1}, p^h\le m \le 2p^h-1}$ appears as \begin{figure}[H]
    \centering
    \includegraphics[width=1\linewidth]{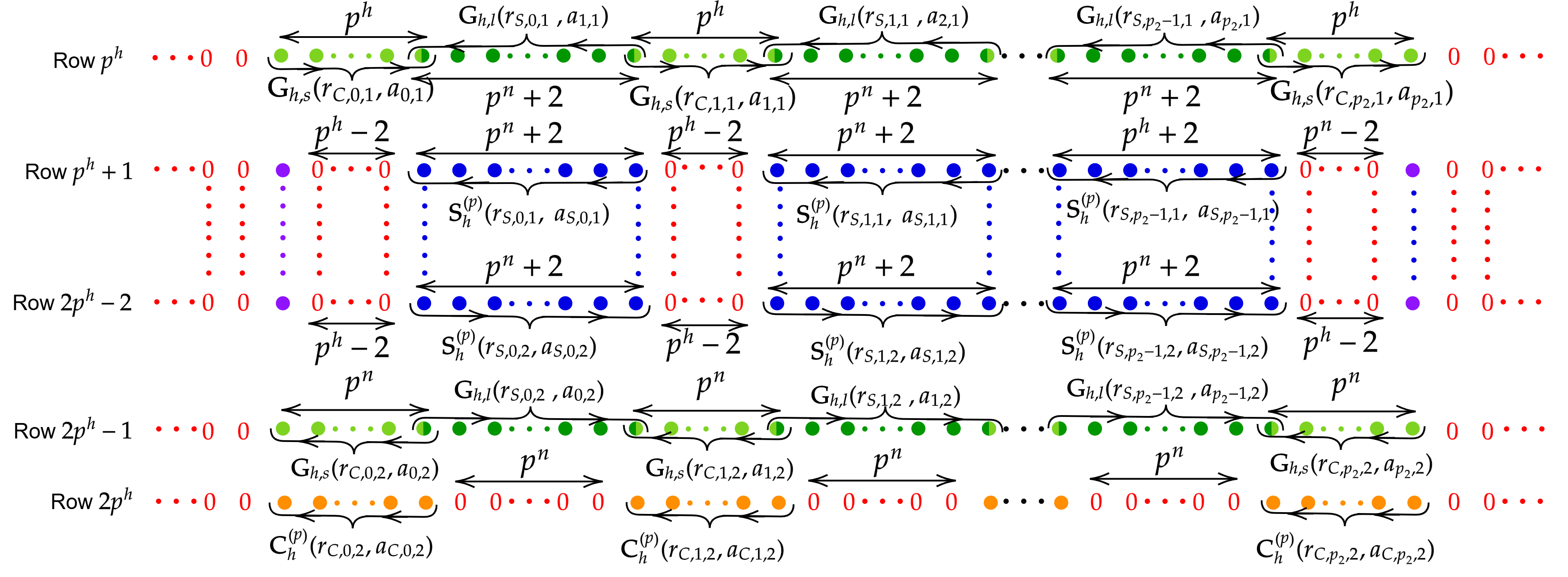}
    \caption{Rows $p^h$ to $2p^h$ of $W_p\left(\widetilde{\textbf{C}}_{h+1}^{(p)}\right)$, presented in the same form as Construction Lemma 2. The green, blue and orange dots represent geometric sequences, the $p$-Singer sequence and the $p$-Cantor sequence respectively.}
\end{figure}
\noindent where, by Construction Lemma \hyperlink{con_lem_2}{2},  \begin{align}\label{eqn: j=2}
    &r_{S,i,2}\substack{\hyperlink{CL21a}{(1a)}\\=}-{p_2\choose i+1}{p_2\choose i}^{-1}& &a_{S,i,2}\substack{\hyperlink{CL21a}{(1a)}\\=}{p_2\choose i}^1\cdot {p_2\choose i+1}^{-1}\cdot \left(1+{p_2\choose i}{p_2 \choose i+1}^{-1}\right)& &a_{i,2}\substack{\hyperlink{CL22a}{(2a)}\\=}{p_2\choose i}^2& \\
    &r_{C,i,2}\substack{\hyperlink{CL23b}{(3b)}\\=}\frac{-{p_2\choose i}\left(1+{p_2\choose i-1}{p_2\choose i}^{-1}\right)}{-{p_2\choose i+1}\left(1+{p_2\choose i}{p_2\choose i+1}^{-1}\right)}& &a_{C,i,2}\substack{\hyperlink{CL23b}{(3b)}\\=}{p_2 \choose i}\left({p_2\choose i}^2-{p_2\choose i-1}{p_2\choose i+1}\right).&\nonumber
\end{align}
\noindent In particular, the profile of rows $p^h$ to $2p^h-1$ of $W_p\left(\widetilde{\textbf{C}}_{h+1}^{(p)}\right)$ is given by \begin{figure}[H]
    \centering
    \includegraphics[width=1\linewidth]{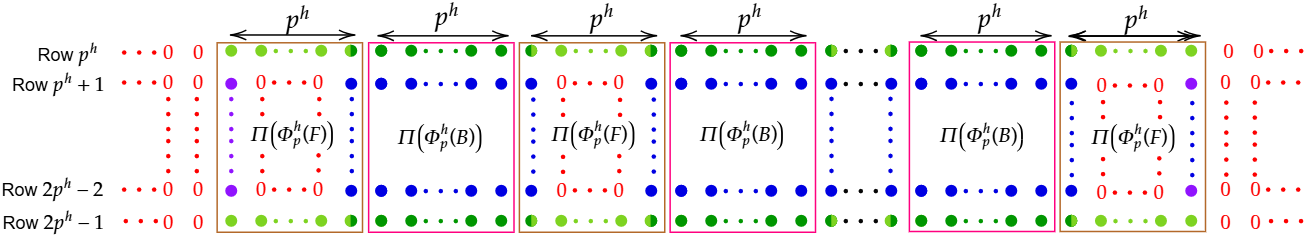}
    \caption{Rows $p^h$ to $2p^h-1$ of $W_p\left(\widetilde{\textbf{C}}^{(p)}_{h+1}\right)$, with colours matching that of Figure \ref{fig: construciton_lemma_2}. In particular, green denotes geometric sequences and blue denotes a geometric transform of the $p$-Singer sequence. }
    \label{fig: profile_rows_2}
\end{figure} \noindent where the contents of the pink squares follow from equation (\ref{eqn: pink_square}). Furthermore, note that rows $2p^h-2$, $2p^h-1$ and $2p^h$ of $W_p\left(\widetilde{\textbf{C}}_h^{(p)}\right)$ satisfy the conditions of Construction Lemma \hyperlink{con_lem_1}{1}.
\subsubsection{\textbf{Applying Construction Lemma \protect\hyperlink{con_lem_1}{1} and \protect\hyperlink{con_lem_2}{2} Repeatedly}}\hfill\\
\noindent One now applies Construction Lemma \hyperlink{con_lem_1}{1} and \hyperlink{con_lem_2}{2} in an alternating fashion, each time using the output of the previous Construction Lemma as the input of the next one. That is, for $0\le j \le p_2$, each time Construction Lemma \hyperlink{con_lem_2}{2} is applied, the variables $a_{S,i,2j}$, $r_{S,i,2j}$, $a_{i,2j}$, $a_{C,i,2j}$, and $r_{C,i,2j}$ are outputted. These are then are used in Construction Lemma \hyperlink{con_lem_1}{1} as \begin{align*}
    &m\leftrightarrow 2j\cdot p^h& &r_{S,i}\leftrightarrow r_{S,i,2j},& &a_{S,i}\leftrightarrow a_{S,i,2j},& &a_i\leftrightarrow a_{i,2j},& &r_{C,i}\leftrightarrow r_{C,i,2j},& &a_{C,i}\leftrightarrow a_{C,i,2j}.&
\end{align*}Similarly, for $0\le j < p_2$, each time Construction Lemma \hyperlink{con_lem_1}{1} is applied it outputs variables $a_{S,i,2j+1}$, $r_{S,i,2j+1}$, $a_{i,2j+1}$, $a_{C,i,2j+1}$, and $r_{C,i,2j+1}$ which are fed into Construction Lemma 2 as inputs \begin{align*}
    &m\leftrightarrow (2j+1)\cdot p^h +1& &r_{S,i}\leftrightarrow r_{S,i,2j+1},& &a_{S,i}\leftrightarrow a_{S,i,2j+1},&\\ &a_i\leftrightarrow a_{i,2j+1},& &r_{C,i}\leftrightarrow r_{C,i,2j+1},& &a_{C,i}\leftrightarrow a_{C,i,2j+1}.&
\end{align*} 
\noindent Furthermore, \begin{itemize}
    \item each time Construction Lemma \hyperlink{con_lem_1}{1} is applied, the profile of rows $2j\cdot p^h$ to $(2j+1)\cdot p^h-1$ are given by Figure \ref{fig: profile_rows},
    \item  each time Construction Lemma \hyperlink{con_lem_2}{2} is applied, the profile of rows $(2j+1)\cdot p^h$ to $(2j+2)\cdot p^h-1$ are given by Figure \ref{fig: profile_rows_2}.
\end{itemize}

\noindent The proof concludes when Construction Lemma 1 is applied for the $p_2^\nth$ time. This completes the proof of Theorem \ref{thm: morphism} modulo Remark \ref{remark: non_zero}.

\subsection{Verifying Remark \ref{remark: non_zero}}\label{Sect: 10.3}\hfill\\
\noindent Finally, it is necessary to verify that every time either Construction Lemma is applied, none of the inputs are zero (with the exception of $a_{S,i,0}$). In this subsection, the values of $a_{S,i,j}, r_{S,i,j},a_{i,j},r_{C,i,j}$ and $a_{C,i,j}$ are computed explicitly. \\

\noindent To begin, note that $a_{i,j}$ can be written in terms of $r_{S,i,j}$. Indeed, $a_{i,2j+1}$ only appears in the output to Construction Lemma \hyperlink{con_lem_1}{1} on row $(2j+1)\cdot p^h$ and similarly $a_{i,2j}$ only appears in the output of Construction Lemma \hyperlink{con_lem_2} on row $2j\cdot p^h-1$. These two rows appear as follows: \begin{figure}[H]
    \centering
    \includegraphics[width=1\linewidth]{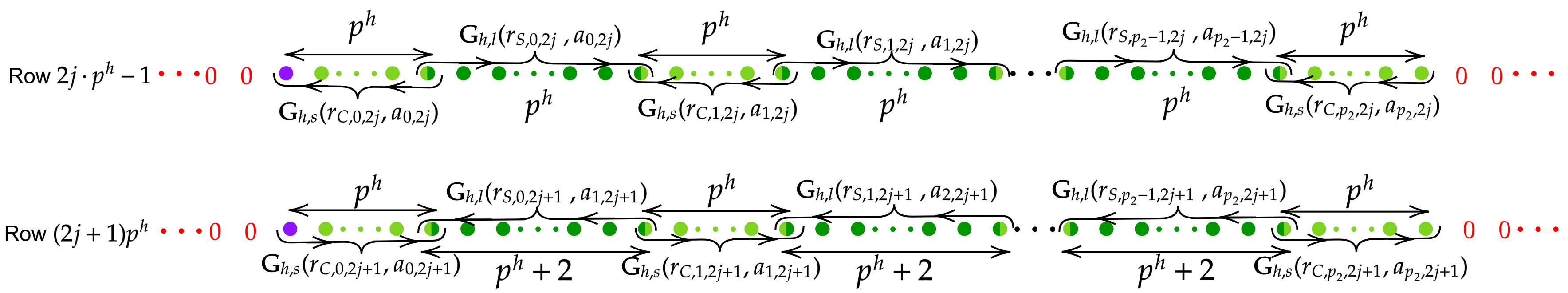}
    \caption{Rows $2jp^h-1$ ($(2j+1)p^h$, respectively) are the only rows in which $a_{i,2j}$ ($a_{i,2j+1}$, respectively) appear.}
\end{figure} \noindent Furthermore, $W_p(\widetilde{\textbf{C}}_h(p))[m,p^{h+1}]=1$ for all $-1\le m \le p^{h+1}$, as it is in the inner frame of the window generated by first $p^{h+1}$ zeroes of $W_p(\widetilde{\textbf{C}}_h(p))$. Therefore, by \hyperlink{FLT}{FLT},\begin{equation}a_{i,2j}=\prod_{k=0}^{i-1} r_{S,k,2j}^{2}~~~\text{ and }~~~a_{i,2j+1}=\prod^{i-1}_{k=0}r_{S,k,2j+1}^{-2}.\label{eqn: a_i}\end{equation}

\noindent Next, Construction Lemma 1 gives recurrence relations for $a_{S,i,2j+1}, r_{S,i,2j+1},r_{C,i,2j+1}$ and $a_{C,i,2j+1}$ in terms of $a_{S,i,2j}, r_{S,i,2j},a_{i,2j},r_{C,i,2j}$ and $a_{C,i,2j}$. Similarly, Construction Lemma 2 gives recurrence relations for $a_{S,i,2j}, r_{S,i,2j},r_{C,i,2j}$ and $a_{C,i,2j}$ in terms of $a_{S,i,2j+1}, r_{S,i,2j+1},a_{i,2j+1},r_{C,i,2j+1}$ and $a_{C,i,2j+1}$. Combining these with equation (\ref{eqn: a_i}) gives the following 8 recurrence relations:\begin{align}
&r_{C,i,2j+1}=r_{C,i,2j}^{-1}\label{eqn: recur_1}&\\ &r_{C,i,2j+2}=-\frac{a_{S,i-1,2j+1}}{a_{S,i,2j+1}\cdot r_{C,i,2j+1}\cdot r_{S,i,2j+1}^2}&\label{eqn: recur_2}\\
&a_{C,i,2j+1}=a_{C,i,2j}&\label{eqn: recur_3}\\ &a_{C,i,2j+2}=a_{S,i,2j+1}\cdot r_{S,i,2j+1}^2\cdot \left(1+\frac{a_{S,i-1,2j+1}}{r_{S,i,2j+1}^3\cdot a_{S,i,2j+1}\cdot r_{C,i,2j+1}}\right)&\label{eqn: recur_4}\\
&r_{S,i,2j+1}=-\frac{a_{C,i,2j}\cdot r_{S,i,2j}}{a_{C,i+1,2j}}&\label{eqn: recur_5}\\ &r_{S,i,2j+2}=r_{S,i,2j+1}^{-1}&\label{eqn: recur_6}\\
&a_{S,i,2j+1}=a_{C,i+1,2j}^3\cdot \left(\prod_{k=0}^ir_{S,k,2j}^{-4}\right)\cdot\left(1+\frac{a_{C,i,2j}\cdot r_{S,i,2j}}{a_{C,i+1,2j}\cdot r_{C,i+1,2j}}\right)& \label{eqn: recur_7}\\&a_{S,i,2j+2}=\frac{r_{S,i,2j+1}^2\cdot \prod_{k=0}^ir_{S,k,2j+1}^{-4}}{a_{S,i,2j+1}}&\label{eqn: recur_8}
\end{align}
with initial values \begin{align*}
    &r_{C,i,0}=1& &a_{C,i,0}={p_2\choose i}& &r_{S,i,0}=1& &a_{S,i,0}=0.&
\end{align*}
\noindent It is now established that the following formulae hold:

\begin{align} 
    &r_{C,i,2j}=(-1)^j\cdot \prod_{k=1}^j \frac{i+k}{p_2-i+k}&\label{eqn: j_recur_1}\\&r_{C,i,2j+1}=(-1)^j\cdot \prod_{k=1}^j \frac{p_2-i+k}{i+k}& \label{eqn: j_recur_2}\\
    &a_{C,i,2j+1}=a_{C,i,2j}={p_2+j\choose i+j}\cdot \left(\prod_{k=0}^{j-1}{p_2+k\choose i+k}^2\cdot \left(\frac{k+1}{p_2+1-i+k}\right)^{2j-1-2k}\right)&\label{eqn: j_recur_3}\\
    &r_{S,i,2j}=(-1)^j\cdot \left(\prod_{k=0}^{j-1}\frac{p_2-i+k}{i+1+k}
    \right)&\label{eqn: j_recur_4}\\
    &r_{S,i,2j+1}=(-1)^{j+1}\cdot \left(\prod_{k=0}^j \frac{i+k+1}{p_2-i+k}\right) &\label{eqn: j_recur_5}\\
    &a_{S,i,2j}=\begin{cases}0&\text{ if }j=0\\{p_2+j\choose i+j}^{-1}\cdot \left(\prod_{k=0}^{j-1}{p_2+k\choose i+k}\right)^2\cdot \left(\prod_{k=0}^{j-2} \left(\frac{k+1}{p_2+k-i+1}\right)^{2j-3-2k}\right)&\text{ otherwise}\end{cases}&\label{eqn: j_recur_6}\\
    &a_{S,i,2j+1}={p_2+j\choose i+j+1}\cdot \frac{p_2+j+1}{j+1}\cdot\left(\prod_{k=0}^j{p_2+k\choose i+k+1}^2\cdot \left(\frac{k+1}{p_2-i+k}\right)^{2j+1-2k}\right)\label{eqn: j_recur_7}&
\end{align}
\noindent It is easy to verify that each of these values is nonzero for $1\le j\le p_2$.
\subsubsection{\textbf{Establishing Equations (\ref{eqn: j_recur_1} - \ref{eqn: j_recur_7})}}\hfill\\
\noindent Each of the above formulae is proved by induction, where the bases cases are given by the $j=1$ case established in equations (\ref{eqn: j=1}) and (\ref{eqn: j=2}). In both the inductive step and the base case, the following elementary identity is used frequently: for $x,y\in\Z$ with $0\le y <x$ \begin{equation}
  {x\choose y}+{x\choose y+1}= {x+1\choose y+1}.\label{eqn: binom_identity}  
\end{equation}
\noindent With the exception of identity (\ref{eqn: binom_identity}), there is no special technique required to prove equations (\ref{eqn: j_recur_1} - \ref{eqn: j_recur_7}). Therefore, an abridged version of the proof is provided for each equation.\\

\noindent \textbf{Proof of Equation (\ref{eqn: j_recur_1}) and hence Equation (\ref{eqn: j_recur_2}}). \\
    \noindent By applying recurrence relation (\ref{eqn: recur_2}) and the induction hypothesis and then completing some elementary cancellations, one obtains\begin{align}
        r_{C,i,2j+2}&=(-1)^{j+1}\cdot \frac{{p_2+j\choose i+j}}{{p_2+j\choose i+j+1}}\left(\prod_{k=0}^j\frac{{p_2+k\choose i+k}^2(p_2-i+k)^{2j+1-2k}}{{p_2+k\choose i+k+1}^2(p_2-i+k+1)^{2j+1-2k}}\right)\nonumber\\ &~~~~~~~~~~~~~~~~~~~~\cdot \left(\prod^j_{k=1}\frac{i+k}{p_2-i+k}\right)\cdot\left(\prod_{k=0}^j\frac{p_2-i+k}{i+k+1}\right)^2 \label{r_{C,i,2n}_eqn_1}
    \end{align}
    \noindent Then, using the definition of the choose function and combining the products yields \begin{align}
        (\ref{r_{C,i,2n}_eqn_1})&=(-1)^{j+1}\cdot \frac{(i+j+1)}{(p_2-i)^{2j+1}} \left(\prod_{k=1}^{j}\left(\frac{(p_2-i+k)^{2j+2-2k}(i+k)}{(p_2-i+k+1)^{2j+1-2j}}\right)\right)\label{r_{C,i,2n}_eqn_2}
    \end{align} 
    \noindent Expanding out the product and making cancellations gives \begin{align*}
        (\ref{r_{C,i,2n}_eqn_2})&=(-1)^{j+1}\prod_{k=1}^{j+1}\frac{i+k}{p_2-i+k},
    \end{align*} establishing equation (\ref{eqn: j_recur_1}). Equation (\ref{eqn: j_recur_2}) follows immediately from recurrence relation (\ref{eqn: recur_1}).\\
    
    \noindent \textbf{Proof of Equation (\ref{eqn: j_recur_3}).} \\
    \noindent Using recurrence relation (\ref{eqn: recur_4}) and combining the products yields \begin{align}
        &a_{C,i,2j+2}={p_2+j\choose i+j+1}\cdot \frac{p_2+j+1}{n+1}\nonumber\\&\cdot \left(\prod_{k=0}^j {p_2+k\choose i+k+1}^2 \left(\frac{k+1}{p_2-i+k}\right)^{2j+1-2k}\left(\frac{i+k+1}{p_2-i+k}\right)^2\right)\nonumber\\ &\cdot \left(1-\frac{{p_2+j\choose i+j}}{{p_2+j\choose i+j+1}} \left(\prod_{k=0}^j \frac{{p_2+k\choose i+k}^2}{{p_2+k\choose i+k+1}^2}\left(\frac{p_2-i+k}{p_2-i+k+1}\right)^{2j+1-2k}\left(\frac{p_2-i+k}{i+k+1}\right)^3\right)\left(\prod_{k=1}^j \frac{i+1}{p_2-i+k}\right)\right)\label{a_{C,i,2n}_eqn_1}
    \end{align}
    \noindent Applying the definition of the choose function, factorising out the $k=0$ term of the second product and making cancellations gives
    \begin{align}
        (\ref{a_{C,i,2n}_eqn_1})&= \frac{(p_2+j)!}{(i+j+1)! (p_2-i-1)!}\cdot \frac{p_2+j+1}{j+1}\nonumber\\&\cdot\left(\prod_{k=0}^j \left(\frac{(p_2+k)!}{(i+k+1)! (p_2-i-1)!}\right)^2\frac{(k+1)^{2j+1-2k}(i+k+1)^2}{(p_2-i+k)^{2j+3-2k}}\right)\nonumber\\ &\cdot \left(1-\frac{(i+j+1)(p_2-i)}{(i+1)(p_2-i+1)^{2j+1}}\prod_{k=1}^j\frac{(p_2-i+k)^{2j+3-2k}(i+k)}{(p_2-i+k+1)^{2j+1-2k} (i+k+1)}\right)\nonumber\\
        &= {p_2+j+1\choose i+j+1} \frac{(p_2-i)^{2j+3}}{j+1} \cdot \left(\prod_{k=0}^j {p_2+k\choose i+k}^2\frac{(k+1)^{2j+1-2k}}{(p_2-i+k)^{2j+3-2k}}\right)\nonumber\\\cdot&\left(1-\frac{(i+j+1)(p_2-i)}{(i+1)(p_2-i+1)^{2j+1}}\prod_{k=1}^j\frac{(p_2-i+k)^{2j+3-2k}(i+k)}{(p_2-i+k+1)^{2j+1-2k} (i+k+1)}\right).\label{a_{C,i,2n}_eqn_2}
    \end{align}
    \noindent Next, factorise out the $k=0$ term from the denominator of the first product and reindex. Additionally, expand the second product and make cancellations. This yields \begin{align}
        (\ref{a_{C,i,2n}_eqn_2})&={p_2+j+1\choose i+j+1} \frac{p_2-i+j+1}{j+1} \cdot \left(\prod_{k=0}^j {p_2+k\choose i+k}^2\left(\frac{k+1}{p_2-i+k+1}\right)^{2j+1-2k}\right)\nonumber\\\cdot&\left(1-\frac{(p_2-i)}{(p_2-i+j+1)}\right)\nonumber\\
        &={p_2+j+1\choose i+j+1}\cdot \left(\prod_{k=0}^j {p_2+k\choose i+k}^2\left(\frac{k+1}{p_2-i+k+1}\right)^{2j+1-2k}\right),\nonumber
    \end{align}
    as required.
    \\~\\
    \noindent \textbf{Proof of equation (\ref{eqn: j_recur_4}) and (\ref{eqn: j_recur_5})}\\
    \\ By applying the recurrence relation (\ref{eqn: recur_5}), \begin{align}
        r_{S,i,2j+1}&=(-1)^{j+1}\cdot \frac{{p_2+j\choose i+j}}{{p_2+j\choose i+j+1}}\prod_{k=0}^{j-1} \frac{{p_2+k\choose i+k}^2}{{p_2+k\choose i+k+1}^2} \frac{p_2-i+k}{i+k+1} \left(\frac{p_2-i+k}{p_2-i+k+1}\right)^{2j-1-2k}.\label{r_{S,i,2n+1}_eqn_1}
    \end{align}
    \noindent Applying the definition of the choose function and simplifying gives \begin{align}
        (\ref{r_{S,i,2n+1}_eqn_1})&= (-1)^{j+1} \frac{i+j+1}{p_2-i} \prod_{k=0}^{j-1} \frac{(i+k+1)(p_2-i+k)}{(p_2-i)^2}\cdot \left(\frac{p_2-i+k}{p_2-i+k+1}\right)^{2j-1-2k}.\label{r_{S,i,2n+1}_eqn_2}
    \end{align} 
    \noindent Expanding the above product gives \begin{equation*}
        (\ref{r_{S,i,2n+1}_eqn_2})=(-1)^{j+1}\prod_{k=0}^j \frac{i+k+1}{p_2-i-k},
    \end{equation*} as required. Equation (\ref{eqn: j_recur_6}) follows from equation (\ref{eqn: recur_5}) and equation (\ref{eqn: j_recur_4}).\\

    \noindent \textbf{Proof of Equation (\ref{eqn: j_recur_6})}. \\
    \noindent Using recurrence relation (\ref{eqn: recur_8}), one obtains
    \begin{align}
        a_{S,i,2j+2}&= \frac{\left(\prod_{k=0}^j \frac{i+k+1}{p_2-i+k}\right)^2 \cdot \overbrace{\left(\prod_{h=0}^i\prod_{k=0}^j \frac{h+k+1}{p_2-h+j}\right)^{-4}}^{(\times)}}{{p_2+j\choose i+j+1}\cdot \frac{p_2+j+1}{j+1}\cdot \left(\prod_{k=0}^j {p_2+k\choose i+k+1}^2\cdot \left(\frac{k+1}{p_2-i+k}\right)^{2j+1-2k}\right)}\label{a_{S,i,2n+2}_eqn_1}
    \end{align}
    \noindent Similarly to the $a_{S,i,2j+1}$ calculation, one has that \begin{align*}
        (\times)&=\left(\prod_{k=0}^j \frac{(i+k+1)!(p_2-i+k-1)!}{(p_2+k)!k!}\right)^{-4}\\
        &=\left(\prod_{k=0}^j {p_2+k\choose i+k}^{-1}\cdot (i+1+k)\cdot \frac{(p_2-i+k-1)!}{(p_2-i)!k!}\right)^{-4}\\
        &= \left(\prod_{k=0}^j {p_2+k\choose i+k}^{-1}\cdot (i+1+k)\right)^{-4}\cdot \left(\frac{1}{p_2-i}\left(\prod_{k=1}^{j-1} (p_2-i+k)^{j-k}\right) \cdot \left(\prod_{k=0}^{j-1} (k+1)^{-(j-k)}\right)\right)^{-4}
    \end{align*}
    \noindent Reinserting this into equation (\ref{a_{S,i,2n+2}_eqn_1}) gives \begin{align}
        (\ref{a_{S,i,2n+2}_eqn_1})&=\frac{(p_2-i)^4\prod_{k=0}^j \left(\frac{1}{p_2-i+k}\right)^2 \cdot {p_2+k\choose i+k}^4 (i+k+1)^{-2}}{{p_2+j\choose i+j+1}\cdot \frac{p_2+j+1}{j+1}\cdot \left(\prod_{k=0}^j {p_2+k\choose i+k+1}^2\cdot \left(\frac{k+1}{p_2-i+k}\right)^{2j+1-2k}\right)}\nonumber\\ \cdot& \frac{1}{\left(\prod_{k=1}^{j-1} (p_2-i+k)^{4j-4k}\right) \cdot \left(\prod_{k=0}^{j-1} (k+1)^{-(4j-4k)}\right)}\nonumber\\
        &=Z\cdot W\nonumber
    \end{align}
    where \begin{align*}
        Z&= \frac{(p_2-i)^4\prod_{k=0}^j {p_2+k\choose i+k}^4 (i+k+1)^{-2}}{{p_2+j\choose i+j+1}\cdot (p_2+j+1)\cdot \left(\prod_{k=0}^j {p_2+k\choose i+k+1}^2\right)}\\
        &=(p_2-i)^{-2j+1}{p_2+j+1\choose i+j+1}^{-1}\prod_{k=0}^j{p_2+k\choose i+k}^2
    \end{align*}
    and \begin{align*}
        W&=\frac{(j+1)\prod_{k=0}^j \frac{(p_2-i+k)^{2j-1-2k}}{(k+1)^{2j+1-k}} }{\left(\prod_{k=1}^{j-1} (p_2-i+k)^{4j-4k}\right) \cdot \left(\prod_{k=0}^{j-1} (k+1)^{-(4j-4k)}\right)}\nonumber\\
        &=\left(\prod_{k=0}^{j-1} (k+1)^{2j-1-2k}\right) \cdot \left(\prod_{k=1}^{j} (p_2-i+k)^{-(2j-1-2k)}\right) \cdot (p_2-i)^{2j-1}.
    \end{align*}
    Reindexing the latter product in $W$ to go from $k=0$ to $j-1$ and multiplying $Z$ and $W$ together completes the proof.\\

    \noindent \textbf{Proof of Equation (\ref{eqn: j_recur_7})}. \\
    \noindent Using recurrence relation (\ref{eqn: recur_7}), one obtains \begin{align}
        a_{S,i,2j+1}&= X\cdot Y,\nonumber\end{align} \noindent where \begin{align}X&={p_2+j \choose i+j+1}^3 \cdot \left(\prod_{k=0}^{j-1} {p_2+k \choose i+k+1}^6\cdot \left(\frac{k+1}{p_2-i+k}\right)^{6j-3-6k}\right)\cdot \underbrace{\left(\prod_{h=0}^i \prod_{k=0}^{j-1} \frac{p_2-h+k}{h+k+1}\right)^{-4}}_{(*)}\nonumber\\Y&=\left(1+\frac{{p_2+j\choose i+j}\left(\prod_{k=0}^{j-1}{p_2+k\choose i+k}^2\left(\frac{k+1}{p_2-i+k+1}\right)^{2j-1-2k}\right)}{{p_2+j\choose i+j+1}\left(\prod_{k=0}^{j-1}{p_2+k\choose i+k+1}^2\left(\frac{k+1}{p_2-i+k}\right)^{2j-1-2k}\right)}\frac{(-1)^j\cdot \prod_{k=0}^{j-1} \frac{p_2-i+k}{i+k+1}}{(-1)^j\cdot \prod_{k=1}^{j}\frac{i+k+1}{p_2-i+k-1}}\right)
    \end{align}
    \noindent The values of $X$ and $Y$ are now simplified separately, with the former being done first. Indeed, one swaps the order of the products in $(*)$ to get \begin{align}
        (*)&= \left(\prod_{k=0}^{j-1}\prod_{h=0}^i\frac{p_2-h+k}{h+k+1}\right)^{-4}\nonumber\\&=\left(\prod_{k=0}^{j-1} \frac{(p_2+k)!\cdot k!}{(p_2-i+k-1)!(i+k+1)!}\right)^{-4}\nonumber\\
        &=\left(\prod_{k=0}^{j-1}{p_2+k\choose i+k+1}\cdot \frac{k! (p_2-i-1)!}{(p_2-i+k-1)!}\right)^{-4} \label{*}
    \end{align}
    \noindent Reinserting equation (\ref{*}) into $X$ gives \begin{align}
        X&={p_2+j \choose i+j+1}^3  \left(\prod_{k=0}^{j-1} {p_2+k \choose i+k+1}^2\right)\left(\prod_{k=0}^{j-1} \left(\frac{k+1}{p_2-i+k}\right)^{6j-3-6k}\right)\left(\underbrace{\prod_{k=0}^{j-1}\left(\frac{k! (p_2-i-1)!}{(p_2-i+k-1)!}\right)^{-4}}_{(**)}\right)\nonumber
    \end{align}
    \noindent Next, by noting that the first term of $(**)$ is equal to $1$, reindexing and combing powers of like terms from the factorials in $(**)$, one arrives at \begin{align}
        (**)&=\prod_{j=0}^{n-2} \left(\frac{(j+1)! (p_2-i-1)!}{(p_2-i+j)!}\right)^{-4}\nonumber\\
        &=\left(\prod_{j=0}^{n-2} \left(\frac{j+1}{p_2-i+j}\right)^{n-1-j}\right)^{-4} 
    \end{align}
    \noindent Reinserting $(**)$ into $X$ and using that $6n-3-6j=2n+1-2j$ when $j=n-1$, one has\begin{align}
        X&={p_2+j \choose i+j+1}^3  \left(\prod_{k=0}^{j-1} {p_2+k \choose i+k+1}^2\right) \left(\prod_{k=0}^{j-1}\left(\frac{k+1}{p_2-i+k}\right)^{2j+1-2k}\right).\nonumber
    \end{align}
    \noindent Now $Y$ is dealt with. The first step towards this is to apply the factorial definition of the choose function and make some cancellations: \begin{align}
    Y&= 1+\left(\frac{i+j+1}{p_2-i}\left(\prod_{k=0}^{j-1}\left(\frac{i+k+1}{p_2-i}\right)^2\cdot \left(\frac{p_2-i+k}{p_2-i+k+1}\right)^{2j-2k-1}\right)\right.\nonumber\\&\cdot \left.\left(\prod_{k=1}^{j-1}\frac{(p_2-i+k)(p_2-i+k+1)}{(i+k+1)^2}\right)\frac{(p_2-i)(p_2-i+j-1)}{(i+1)(i+j+1)}\right)\nonumber\\ &=1+\left(\frac{(p_2-i+j-1)(i+1)}{(p_2-i)(p_2+1-i)^{2j-1}}\cdot \prod_{k=1}^{j-1} \frac{(p_2-i+k)^{2j-2k}\cdot(p_2-i+k-1)}{(p_2+1-i+k)^{2j-1-2k}}\right)\nonumber\\
    &=1+\frac{i+1}{p_2-i+j}\cdotp\nonumber
    \end{align}
    \noindent Finally, combining $X$ and $Y$ gives \begin{align*}
        a_{S,i,2j+1}&={p_2+j \choose i+j+1}  \left(\prod_{k=0}^j {p_2+k \choose i+k+1}^2\right) \left(\prod_{k=0}^{j-1}\left(\frac{k+1}{p_2-i+j}\right)^{2j+1-2k}\right) \cdot \frac{p_2+j+1}{p_2-i+j}\\&=\frac{p_2+j+1}{j+1}{p_2+j \choose i+j+1}  \left(\prod_{k=0}^j {p_2+k \choose i+k+1}^2\left(\frac{k+1}{p_2-i+k}\right)^{2j+1-2k}\right) 
    \end{align*}
    as required.\\~\\
\bibliography{Ref}
\bibliographystyle{amsalpha}
\newpage
\section*{Appendix A: Variations on a Sequence}\label{Sect: Appendix_A}
\noindent Throughout this work, numerous versions of any given sequence are used. The purpose of this appendix is to help the reader keep track of them.

 \subsubsection*{\textbf{Transformations on Sequences of Arbitrary Length}}\hfill\\

\noindent Let $n\in\N\cup\infty$ and let $\textbf{S}=(s_i)_{0\le i \le n}$ be a sequence over some integral domain $\mathcal{R}$. Then, \begin{itemize}
    \item For $k\in\mathcal{R}$, $k\cdot \mathbf{S}=(k\cdot s_i)_{0\le i \le n}$,
\end{itemize} 

 \subsubsection*{\textbf{Transformations of finite sequences}}\hfill \\
 \noindent Let $\textbf{S}=(s_i)_{0\le i \le l}$ be finite sequence of length $l+1\in\N$ over some integral domain $\id$. Furthermore, recall that for $\oplus$ denotes concatenation of sequences. Then, \begin{itemize}

    \item $\widetilde{\textbf{S}}:=\{0\}_{0\le i \le l}\oplus \textbf{S}\oplus\{0\}_{0\le i \le l}$;
    \item $\overset{\text{\tiny$\bm\leftrightarrow$}}{\textbf{\textup{S}}}=(\overset{\text{\tiny$\bm\leftrightarrow$}}{s}_{i})_{0\le i \le n}$, where $\overset{\text{\tiny$\bm\leftrightarrow$}}{s}_i=s_{l-i}$.
\end{itemize}
 \subsubsection*{\textbf{Transformations of Infinite Sequence}}\hfill\\ 
 \noindent Let $\textbf{R}=(r_i)_{i\ge0}$ be an infinite sequence over some integral domain $\id$. Then, the following variations on $\textbf{R}$ are defined:\begin{itemize}\item $\textbf{R}^L:= (x_i)_{i\in\Z}$ such that $x_i=\begin{cases}
    r_i &\text{ if }i\ge0\\ 0&\text{ otherwise,}
\end{cases}$
\item $\textbf{C}^{(p,L)}=\left(c_i^{(p)}\right)_{i\in\Z}$ where $c^{(p)}_i=0$ for all $i<0$.
\end{itemize}
\subsubsection*{\textbf{Geometric Sequences}}\hfill\\
\noindent For $h\in\N$ and $r,a\in\F_p$, let \begin{itemize}
    \item $\textbf{G}_{h,s}(r,a)=\left(a\cdot r^i\right)_{0\le i \le p^h-1}$,
    \item $\textbf{G}_{h,\ell}(r,a)=\left(a\cdot r^i\right)_{0\le i \le p^h+1}.$
\end{itemize}
 \subsubsection*{\textbf{Transformation of the }$p$\textbf{-Cantor Sequence}}\hfill\\
 \noindent Some variations and transformations only apply to the $p$-Cantor sequence, $\textbf{C}(p)=\left(c^{(p)}_i\right)_{i\ge0}$. Let $r,a\in\F_p$ and let $h\in\N$. Then, let\begin{itemize}
    \item $c_i^{(r,a)}=a\cdot r^i\cdot c_i^{(p)}$,
    \item   $\textbf{C}^{(p)}_h{(r,a)}:= \left(ar^ic^{(p)}\right)_{0\le i <p^h}$  
    \item  $\textbf{C}_h^{(p)}=\textbf{C}^{(p)}_h{(1,1)}$,
    \item  $\widetilde{ \mathbf{C}}_h^{(p)}{(r,a)}:=\{0\}_{0\le i \le p^h-1}\oplus \{a\cdot r^i\cdot c^{(p)}_i\}_{0 \le i <p^h} \oplus\{0\}_{0\le i \le p^h-1}$,
    \item  $\widetilde{ \mathbf{C}}_h^{(p)}=\widetilde{ \mathbf{C}}_h^{(p)}{(1,1)}$.
    \item $\overline{\textbf{C}}^{(p)}_h(r',a',r'',a''):= \{0\}_{i=1,2}\oplus \textbf{C}^{(p)}_h(r',a')\oplus\{0\}_{0\le i <p^h}\oplus \textbf{C}^{(p)}_h(r'',a'')\oplus \{0\}_{i=1,2}$
\end{itemize}

\subsubsection*{\textbf{Transformation of the} $p$\textbf{-Singer Sequence}}\hfill\\
\noindent Similarly to above, the following notation is specific to the $p$-Singer sequence. \begin{itemize}
\item $s_i^{(r,a)}=a\cdot r^i\cdot s_i^{(p)}$,
    \item $\textbf{S}_h^{(p)}=(s^{(p)}_i)_{0\le i <p^h+2}.$
\item   $\textbf{S}^{(p)}_h{(r,a)}:= \left(ar^is^{(p)}\right)_{0\le i <p^h+2}$  
\item $\widetilde{\textbf{S}}_h^{(p)}:=\{0\}_{0\le i \le p^h+1}\oplus \textbf{S}_h^{(p)}\oplus \{0\}_{0\le i \le p^h+1}.$
\end{itemize} 
\newpage
\section*{Appendix B: Number Wall Pictures}\label{Sect: Appendix_B}
\noindent See below for pictures of the number wall of the $p$-Cantor sequence for $p=3$, $5$ and $7$. Specifically, each picture is of $W_p\left(\textbf{C}^{(p,L)}\right)[m,n]$ for $-1 \le m \le p^h$ and $0\le n \le p^h$ where $h$ is some small natural number. Furthermore, the same colour scheme is used in each instance: zero entries in yellow, nonzero entries in a shade of blue in order from darkest to lightest.  
\begin{figure}[H]
    \centering
    \includegraphics[width=0.475\linewidth]{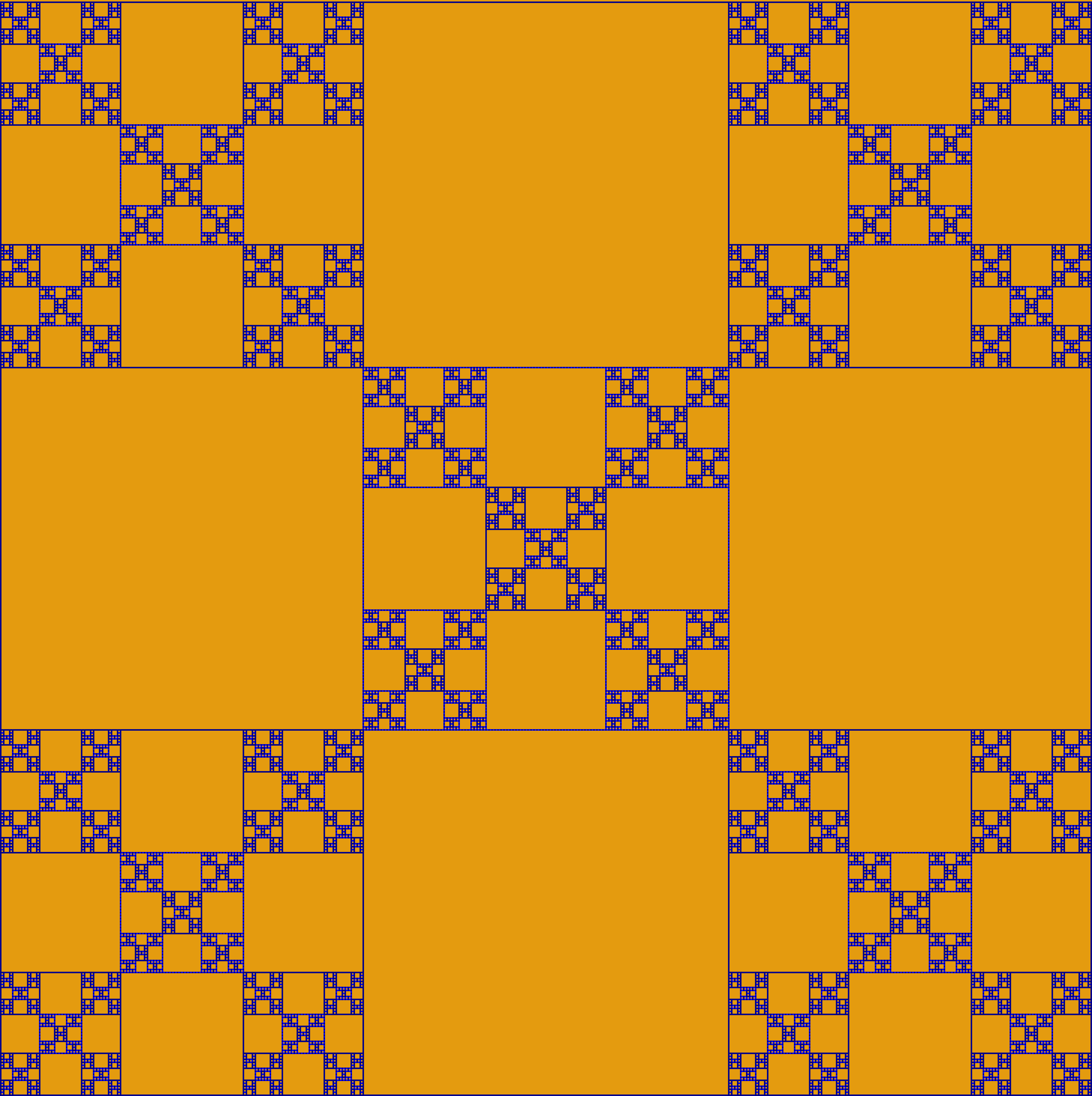}
    \caption{$W_3\left(\textbf{C}^{(3,L)}\right)[m,n]$ for $-2 \le m \le 3^6$ and $0\le n \le 3^6$}
\end{figure}
\begin{figure}[H]
    \centering
    \includegraphics[width=0.475\linewidth]{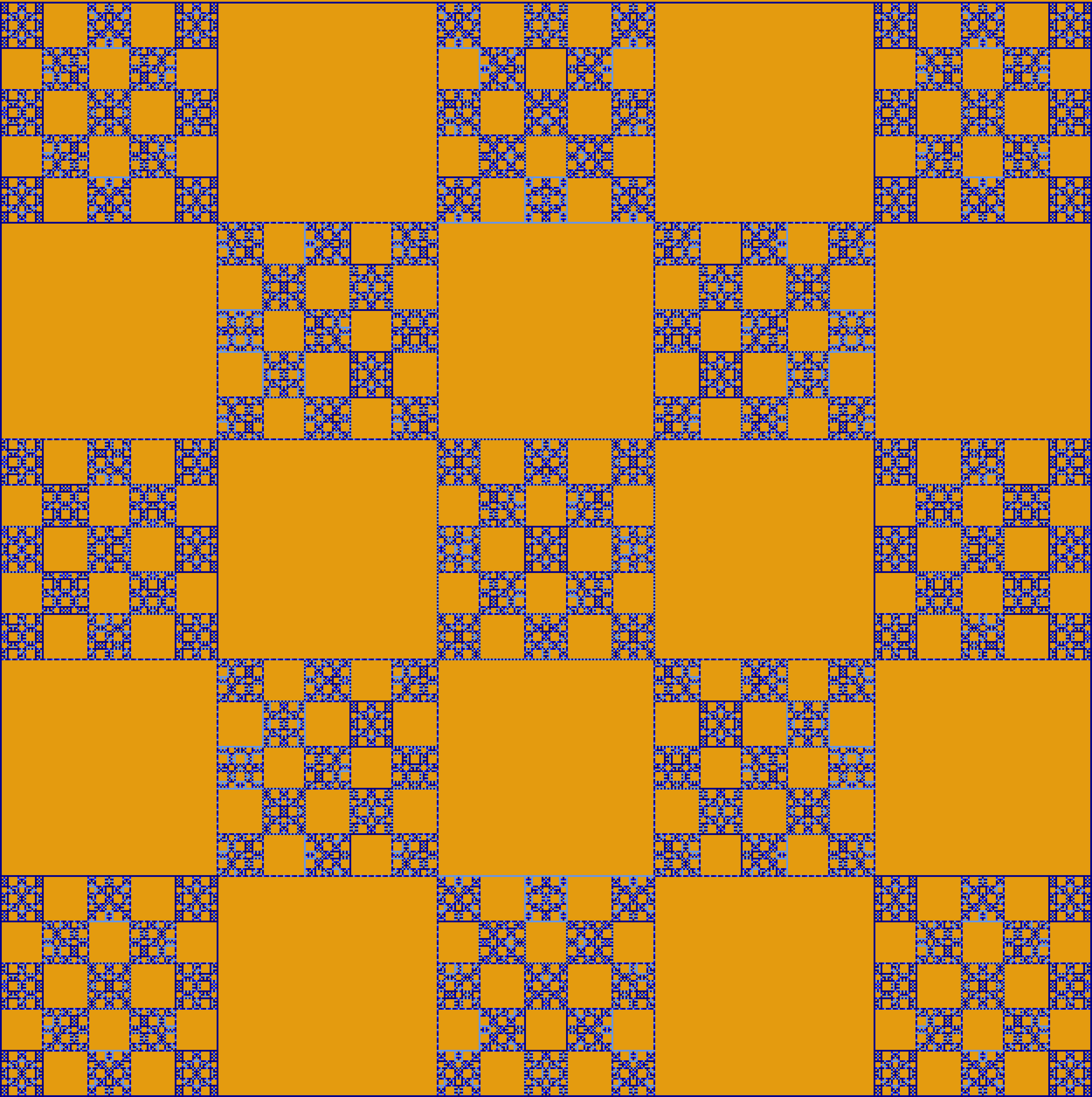}
    \caption{$W_5\left(\textbf{C}^{(5,L)}\right)[m,n]$ for $-2 \le m \le 5^4$ and $0\le n \le 5^4$}
\end{figure}
\begin{figure}[H]
    \centering
    \includegraphics[width=0.475\linewidth]{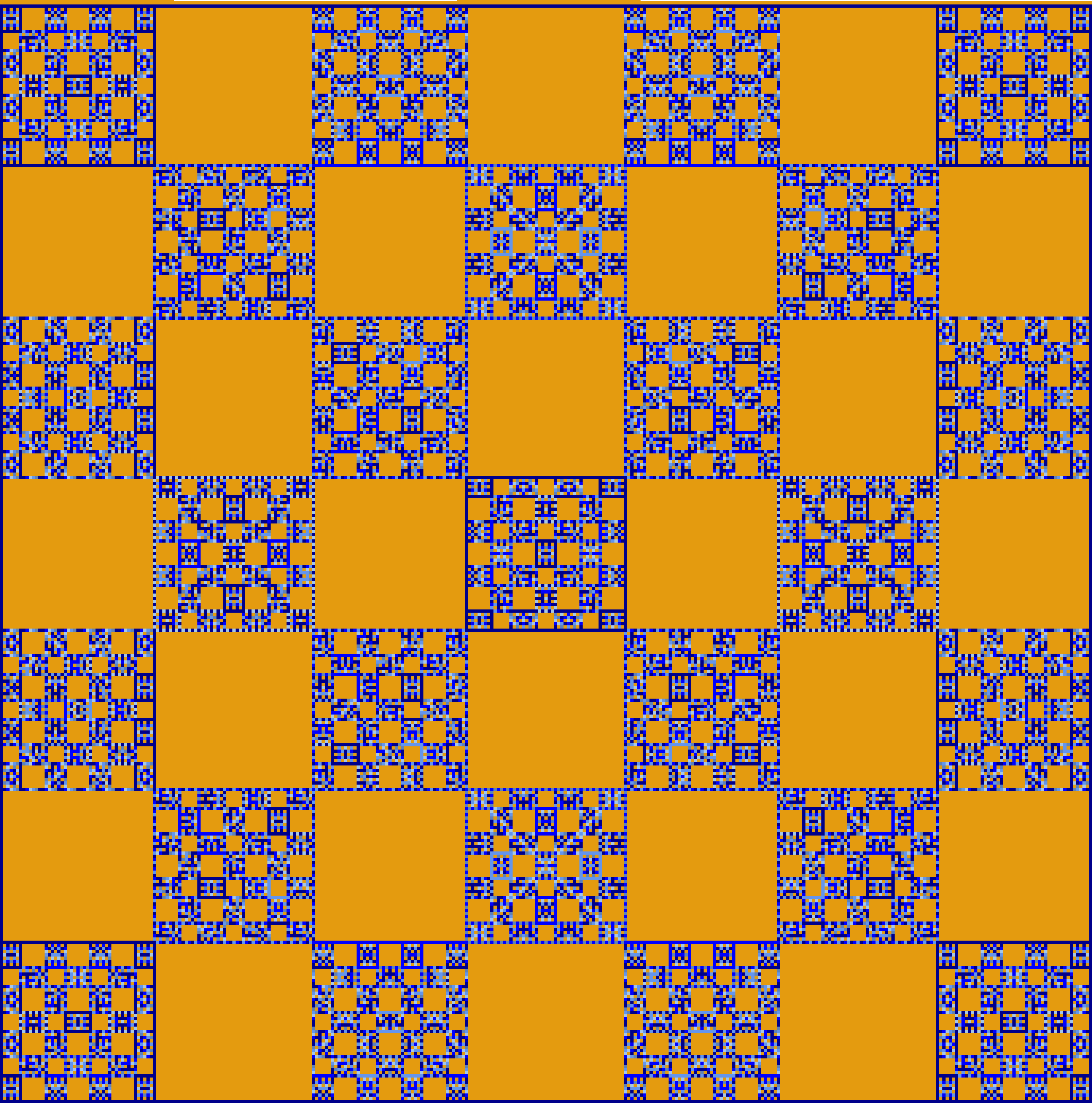}
    \caption{$W_7\left(\textbf{C}^{(7,L)}\right)[m,n]$ for $-2 \le m \le 7^3$ and $0\le n \le 7^3$}
\end{figure}
\end{document}